\pdfoutput=1
\RequirePackage{ifpdf}
\ifpdf % We are running pdfTeX in pdf mode
\documentclass[pdftex]{sigma}
\else
\documentclass{sigma}
\fi

\usepackage{enumerate, mathrsfs, setspace, tikz, stmaryrd, mathdots, youngtab,epstopdf,comment,bm}
\usetikzlibrary{arrows,decorations.markings}
\usetikzlibrary{patterns}

\DeclareMathOperator{\half}{\frac{1}{2}}

\DeclareMathOperator{\SL}{SL}
\DeclareMathOperator{\Z}{\mathbb{Z}}
\DeclareMathOperator{\R}{\mathbb{R}}

\DeclareMathOperator{\C}{\mathbb{C}}
\newcommand{\lat}{L}
\newcommand{\lati}{L_{\textup{init}}}
\newcommand{\zti}{\mathbb{Z}^3_{\textup{init}}}
\newcommand{\xtin}{\mathbf{\tilde{x}}_{\textup{init}}}
\newcommand{\xin}{\mathbf{x}_{\textup{init}}}
\newcommand{\feq}{\mathbb{Z}^3_{\boxdot}}
\newcommand{\cueq}{\varkappa_{\boxdot}}

\newcommand{\zadi}{\mathbb{Z}^d_{a,\textup{init}}}
\newcommand{\fati}{F^{\mathbf{a}}_{\textup{init}}}
\newcommand{\fab}{F^{\mathbf{a}}}
\newcommand{\faeq}{F^{\mathbf{a}}_{\boxdot}}
\newcommand{\gtf}{g}
\newcommand{\hof}{h}

\newcommand{\pos}{\mathbb{R}_{>0}}

\newcommand{\zzzz}{z_{000}}
\newcommand{\zozz}{z_{100}}
\newcommand{\zzoz}{z_{010}}
\newcommand{\zzzo}{z_{001}}
\newcommand{\zzoo}{z_{011}}
\newcommand{\zozo}{z_{101}}
\newcommand{\zooz}{z_{110}}
\newcommand{\zooo}{z_{111}}
\newcommand{\zzhh}{z_{0\frac{1}{2}\frac{1}{2}}}
\newcommand{\zhzh}{z_{\frac{1}{2}0\frac{1}{2}}}
\newcommand{\zhhz}{z_{\frac{1}{2}\frac{1}{2}0}}
\newcommand{\zohh}{z_{1\frac{1}{2}\frac{1}{2}}}
\newcommand{\zhoh}{z_{\frac{1}{2}1\frac{1}{2}}}
\newcommand{\zhho}{z_{\frac{1}{2}\frac{1}{2}1}}

\numberwithin{equation}{section}

\newtheorem{Theorem}{Theorem}[section]
\newtheorem{Corollary}[Theorem]{Corollary}
\newtheorem{Lemma}[Theorem]{Lemma}
\newtheorem{Proposition}[Theorem]{Proposition}
{ \theoremstyle{definition}
\newtheorem{Definition}[Theorem]{Definition}

\newtheorem{Example}[Theorem]{Example}
\newtheorem{Remark}[Theorem]{Remark} }

\begin{document}

\allowdisplaybreaks

\newcommand{\arXivNumber}{1805.04197}

\renewcommand{\PaperNumber}{012}

\FirstPageHeading

\ShortArticleName{The Kashaev Equation and Related Recurrences}

\ArticleName{The Kashaev Equation and Related Recurrences}

\Author{Alexander LEAF}

\AuthorNameForHeading{A.~Leaf}

\Address{Department of Mathematics, University of Michigan, Ann Arbor, MI, 48109, USA}
\Email{\href{mailto:aleaf@umich.edu}{aleaf@umich.edu}}

\ArticleDates{Received May 24, 2018, in final form February 03, 2019; Published online February 21, 2019}

\Abstract{The hexahedron recurrence was introduced by R.~Kenyon and R.~Pemantle in the study of the double-dimer model in statistical mechanics. It describes a relationship among certain minors of a square matrix. This recurrence is closely related to the Kashaev equation, which has its roots in the Ising model and in the study of relations among principal minors of a symmetric matrix. Certain solutions of the hexahedron recurrence restrict to solutions of the Kashaev equation. We characterize the solutions of the Kashaev equation that can be obtained by such a restriction. This characterization leads to new results about principal minors of symmetric matrices. We describe and study other recurrences whose behavior is similar to that of the Kashaev equation and hexahedron recurrence. These include equations that appear in the study of s-holomorphicity, as well as other recurrences which, like the hexahedron recurrence, can be related to cluster algebras.}


\Classification{05E99; 13F60; 15A15; 30G25; 52C22}

\section{Introduction} \label{intro}

The \emph{Kashaev equation} is a polynomial equation involving $8$ numbers indexed by the vertices of a~cube; this equation is invariant under the symmetries of the cube. It originally appeared in the study of the star-triangle move in the Ising model \cite{kashaev}; it also arises as a relation among principal minors of a symmetric matrix~\cite{otherhex}.

We say that a $\C$-valued array indexed by $\Z^3$ satisfies the Kashaev equation if for every unit cube $C$ in $\Z^3$, the $8$ numbers indexed by the vertices of $C$ satisfy the Kashaev equation. The Kashaev equation is quadratic in each of its variables, so we in general have two choices in solving for one value in terms of the remaining seven. If these seven values are all positive, then both solutions are real, and the larger solution is positive. This leads to a recurrence on positive-valued arrays on $\Z^3$ that we call the \emph{positive Kashaev recurrence}; it expresses the value at the ``top vertex'' of each unit cube in terms of the $7$ values underneath it.

Our first observation is that solutions of this positive recurrence satisfy an additional algebraic constraint not implied by the Kashaev equation alone. This constraint involves the values indexed by the $27$ vertices of a $2 \times 2 \times 2$ cube in~$\Z^3$. A solution of the Kashaev equation that satisfies this constraint is called \emph{coherent}.

The \emph{hexahedron recurrence} is a birational recurrence satisfied by an array indexed by the vertices and (centers of) two-dimensional faces of the standard tiling of $\R^3$ with unit cubes. This recurrence was introduced by Kenyon and Pemantle \cite{mainkp} in the context of statistical mechanics as a way to count ``taut double-dimer configurations'' of certain graphs. It also describes a~relationship among principal and ``almost principal'' minors of a square matrix \cite{otherhex}.

A key observation of Kenyon and Pemantle \cite{mainkp} was that restricting an array satisfying the hexahedron recurrence to the vertices of the standard tiling of $\R^3$ with cubes (i.e., to $\Z^3$) yields an array satisfying the Kashaev equation. However, not all solutions of the Kashaev equation can be obtained this way. Our main result (Theorem~\ref{galoisfromintro}) states that, modulo some natural technical conditions, a solution of the Kashaev equation can be extended to a solution of the hexahedron recurrence if and only if it is coherent.

We then generalize this result to a certain subclass of $3$-dimensional cubical complexes. We show that a suitable generalization of Theorem~\ref{galoisfromintro} holds for these complexes (Proposition~\ref{cc_easy_direction} and Theorem~\ref{cc_main_generalization}), but that the corresponding statement can be false for cubical complexes outside this subclass (Theorem~\ref{when_theorem_fails}).

We use this generalization to study the relations among principal minors of symmetric matrices. Given a symmetric matrix $M$, we associate principal minors of $M$ to the vertices of a~cubical complex, so that the resulting array is a coherent solution of the Kashaev equation. Conversely, for any generic coherent solution of the Kashaev equation, there exists a symmetric matrix whose principal minors appear as the entries of the given array. This leads to Theorem~\ref{main_all_minors_simple}, which provides a simple test for whether a $2^n$-tuple of complex numbers (satisfying certain genericity conditions) arises as a collection of principal minors of an $n \times n$ symmetric matrix. An alternative criterion was given by L.~Oeding \cite{oeding}.

Going in another direction, we develop an axiomatic setup for pairs of recurrences whose behavior is similar to that of the Kashaev equation and the hexahedron recurrence, respectively. Theorem~\ref{maintheoremgen} generalizes Theorem~\ref{galoisfromintro} to this class of recurrences.

Among the applications of this generalization, we study a set of equations that appear in the context of \emph{s-holomorphicity} in discrete complex analysis. We introduce an equation~\eqref{qcdef}, similar to the Kashaev equation for arrays indexed by $\Z^2$, along with equations~\eqref{2dh1}--\eqref{2dh3}, similar to the hexahedron recurrence for arrays indexed by the edges and vertices of the standard tiling of~$\R^2$ with unit squares. The equations~\eqref{2dh2}--\eqref{2dh3} for the edge values are independent of the values on the vertices, and can be used (with small modifications) to define s-holomorphic functions on the tiling of~$\R^2$ with unit squares. While the equations~\eqref{qcdef} and~\eqref{2dh1}--\eqref{2dh3} have been studied before (cf.~\cite{chelkaksmirnov}), our main novelty is the notion of coherence similar to that for the Kashaev equation.

As another application, we introduce additional recurrences exhibiting hexahe\-dron-like behavior that have their origins in the theory of cluster algebras. Whereas the connections with cluster algebras are to be discussed elsewhere, the definitions of coherence for these recurrences are provided herein.

The paper is organized as summarized in the table below:

\begin{center}
\begin{tabular}{ | l | c | c | }
\hline & \begin{tabular}{ c } Definitions\\ and results \end{tabular} & \begin{tabular}{ c } Proofs and\\ generalizations \end{tabular}\\
\hline Kashaev equation in $\Z^3$ & Section \ref{main-results} & Section \ref{mainproof} \\
\hline Kashaev equation for cubical complexes & Sections \ref{preliminaries}, \ref{principal_minors_section} & Sections \ref{cc_proof_section}, \ref{matrix_proofs} \\
\hline Other Kashaev-like recurrences & Sections \ref{s_holo_section}, \ref{gen_from_ca} & Section \ref{klikerecurrences} \\ \hline
\end{tabular}
\end{center}

We next review the content of each section of the paper. Section~\ref{main-results} introduces the basic concepts. Its main result is Theorem~\ref{galoisfromintro}, which has been discussed above. The results from Section~\ref{main-results} are proved in Section~\ref{mainproof}.

While Sections~\ref{main-results}~and~\ref{mainproof} are necessary for the rest of the paper, Sections~\ref{preliminaries}, \ref{principal_minors_section}, \ref{cc_proof_section}, \ref{matrix_proofs} are independent of Sections~\ref{s_holo_section}, \ref{gen_from_ca}, \ref{klikerecurrences}, and vice versa. In Section~\ref{preliminaries}, we discuss some combinatorial tools involving cubical complexes and zonotopal tilings that we use in Sections~\ref{principal_minors_section},~\ref{cc_proof_section}, and~\ref{matrix_proofs}. In Section~\ref{principal_minors_section}, we review the background from Kenyon and Pemantle \cite{otherhex} on the use of the hexahedron recurrence and the Kashaev equation in the study of principal and almost principal minors. In that section, we also state a version of Theorem~\ref{galoisfromintro} for certain cubical complexes, and then apply this result to the study of principal minors of symmetric matrices. In Section~\ref{cc_proof_section}, we extend Theorem~\ref{galoisfromintro} to the setting of cubical complexes, and in the process prove some results from Section~\ref{principal_minors_section}. In Section~\ref{matrix_proofs}, we prove the remaining results from Section~\ref{principal_minors_section}.

In Section~\ref{s_holo_section}, we discuss a condition similar to the Kashaev equation that arises in the context of s-holomorphicity. In Section~\ref{gen_from_ca}, we discuss some additional recurrences with behavior similar to the Kashaev equation and hexahedron recurrence, which are related to cluster algebras. Sections~\ref{s_holo_section} and~\ref{gen_from_ca} can be read independently of each other. In Section~\ref{klikerecurrences}, we describe an axiomatic setup for equations with properties similar to those of the Kashaev equation, and prove a more general version of Theorem~\ref{galoisfromintro}. In the process, we prove all of the results from Sections~\ref{s_holo_section}--\ref{gen_from_ca}.

This paper is a slightly edited version of the author's Ph.D.\ Thesis~\cite{Leaf-Thesis}.

\section[The Kashaev equation in $\mathbb{Z}^3$]{The Kashaev equation in $\boldsymbol{\mathbb{Z}^3}$} \label{main-results}

In this section, we introduce the Kashaev equation, the hexahedron recurrence, and the K-hexahedron equations. We then state our main results (Theorems~\ref{galoisfromintro}--\ref{extensions_that_agree}) about the Kashaev equation for arrays indexed by $\Z^3$.

\begin{Definition} \label{cubekashdef}Let $\zzzz, \dots, \zooo \in \C$ be $8$ numbers indexed by the vertices of a cube, as shown in Fig.~\ref{labeledcube}. We say that these $8$ numbers satisfy the \emph{Kashaev equation} if
\begin{gather} \label{kashforcube}
2\big(a^2 + b^2 + c^2 + d^2\big) - (a + b + c + d)^2 - 4 (s + t) = 0,
\end{gather}
where $a$, $b$, $c$, $d$, $s$, $t$ are the monomials defined in Fig.~\ref{labeledcube}. Notice that the equation~\eqref{kashforcube} is invariant under the symmetries of the cube. Thus, reindexing the $8$ values using an isomorphic labeling of the cube does not change the Kashaev equation.

\begin{figure}[ht]\centering
\begin{tikzpicture}
\draw (0,0)--(2.5,0)--(2.5,2.5)--(0,2.5)--(0,0);
\draw (1,1)--(3.5,1)--(3.5,3.5)--(1,3.5)--(1,1);
\draw (0,0)--(1,1);
\draw (2.5,0)--(3.5,1);
\draw (2.5, 2.5)--(3.5, 3.5);
\draw (0,2.5)--(1,3.5);

\filldraw (0,0) circle (1.5pt);
\filldraw (2.5,2.5) circle (1.5pt);
\filldraw (3.5,1) circle (1.5pt);
\filldraw (1,3.5) circle (1.5pt);

\filldraw[black] (2.5, 0) circle (1.5pt);
\filldraw[white] (2.5, 0) circle (1pt);
\filldraw[black] (0, 2.5) circle (1.5pt);
\filldraw[white] (0, 2.5) circle (1pt);
\filldraw[black] (1, 1) circle (1.5pt);
\filldraw[white] (1, 1) circle (1pt);
\filldraw[black] (3.5, 3.5) circle (1.5pt);
\filldraw[white] (3.5, 3.5) circle (1pt);

\draw (0,0) node[anchor=east]{$\zzzz$};
\draw (0,2.5) node[anchor=east]{$\zzoz$};
\draw (1,1) node[anchor=east]{$\zzzo$};
\draw (1,3.5) node[anchor=east]{$\zzoo$};
\draw (2.5,0) node[anchor=west]{$\zozz$};
\draw (2.5,2.5) node[anchor=west]{$\zooz$};
\draw (3.5,1) node[anchor=west]{$\zozo$};
\draw (3.5,3.5) node[anchor=west]{$\zooo$};

\draw (10,3.5) node[anchor=north]{\begin{tabular}{c c c}
$a = \zzzz \zooo,$ \\
\\
$b = \zozz \zzoo,$ & \qquad\qquad& $s = \zzzz \zzoo \zozo \zooz,$ \\
\\
$c = \zzoz \zozo,$ & & $t = \zooo \zozz \zzoz \zzzo.$ \\
\\
$d = \zzzo \zooz,$ &
\end{tabular}};

\end{tikzpicture}
\caption{Notation used in Definition~\ref{cubekashdef}. The quantities $a$, $b$, $c$, and $d$ are the products of the values at opposite vertices of the cube, and $s$ and $t$ are the products corresponding to the two inscribed tetrahedra.} \label{labeledcube}
\end{figure}
\end{Definition}

\begin{Definition}\label{kashaevdef}We say that a $3$-dimensional array $\mathbf{x} \in \C^{\Z^3}$ satisfies the Kashaev equation if its components labeled by the vertices of any unit cube in $\Z^3$ satisfy~\eqref{kashforcube}. More formally, given a unit cube $C$ in~$\Z^3$, define $K^C \colon \C^{\Z^3} \rightarrow \C$ by
\begin{gather*}
K^C(\mathbf{x}) = 2\big(a^2 + b^2 + c^2 + d^2\big) - (a + b + c + d)^2 - 4 (s + t),
\end{gather*}
where $a$, $b$, $c$, $d$, $s$, $t$ are the monomials in the components of $\mathbf{x}$ at the vertices of $C$, defined as in Fig.~\ref{labeledcube}. We then say that~$\mathbf{x}$ satisfies the Kashaev equation if $K^C(\mathbf{x}) = 0$ for every unit cube~$C$ in~$\Z^3$.
\end{Definition}

The Kashaev equation was originally introduced by R.~Kashaev \cite{kashaev} in the study of the star-triangle move in the Ising model. It also appears as an identity involving principal minors of a~symmetric matrix~\cite{otherhex}; this connection is discussed in Section~\ref{principal_minors_section}. Furthermore, up to changes of sign, the Kashaev equation can be interpreted as the vanishing of Cayley's hyperdeterminant of a~$2 \times 2 \times 2$ hypermatrix; this connection is also discussed in Section~\ref{principal_minors_section}. The Kashaev equation is also related to the theory of cluster algebras and to Descartes's formula for Apollonian circles, connections that we will explore in later work.

\begin{Remark} \label{twosolutionsremark}The left-hand side of equation~\eqref{kashforcube} is a quadratic polynomial in each of the variables~$z_{ijk}$. Solving for $\zooo$ in terms of the other $z_{ijk}$, we obtain
\begin{gather} \label{pmz000}
\zooo = \frac{ A \pm 2 \sqrt{D}}{\zzzz^2},
\end{gather}
where
\begin{gather}
A = 2 \zozz \zzoz \zzzo + \zzzz (\zozz \zzoo + \zzoz \zozo + \zzzo \zooz ),\nonumber\\
D = (\zzzz \zzoo + \zzoz \zzzo)(\zzzz \zozo + \zozz \zzzo)(\zzzz \zooz + \zozz \zzoz),\label{aandbs}
\end{gather}
and $\sqrt{D}$ denotes any of the two square roots of $D$. Notice that if all $7$ values $z_{ijk}$ contributing to the right-hand side of~\eqref{pmz000} are positive, then $D > 0$, so both solutions for $\zooo$ in~\eqref{pmz000} are real; moreover, the larger of these two solutions is positive. This observation suggests the following definition.
\end{Remark}

\begin{Definition} \label{poskashdef}We say that a $3$-dimensional array $\mathbf{x} \in (\pos)^{\Z^3}$ satisfies the \emph{positive Kashaev recurrence} if for every $(v_1, v_2, v_3) \in \Z^3$, we have
\begin{gather} \label{kashaevwithsr}
\zooo = \frac{A + 2 \sqrt{D}}{\zzzz^2},
\end{gather}
where $z_{ijk}$ denotes the component of $\mathbf{x}$ at $(v_1 + i, v_2 + j, v_3 + k)$, for $i, j, k \in \{0,1\}$, and we use the notation introduced in~\eqref{aandbs}, with the conventional meaning of the square root.
\end{Definition}

\begin{Remark}By Remark~\ref{twosolutionsremark}, any solution of the positive Kashaev recurrence is a positive real solution of the Kashaev equation. However, the converse is false; there exist arrays $\mathbf{x} \in (\pos)^{\Z^3}$ satisfying the Kashaev equation which do not satisfy the positive Kashaev recurrence. (There exist positive $z_{ijk}$ such that both solutions for $\zooo$ in~\eqref{pmz000} are positive.)
\end{Remark}

Any solution of the positive Kashaev recurrence must satisfy certain algebraic equations which are not implied by the Kashaev equation.

\begin{Definition} \label{kcdef}Let $\mathbf{x} \in \C^{\Z^3}$. Let $v$, $w$ be two opposite vertices in a unit cube $C$ in~$\Z^3$. We set
\begin{gather*}
K^C_v(\mathbf{x}) = \frac{1}{4} \frac{\partial K^C}{\partial x_w}(\mathbf{x})
= \frac{1}{2} \big(\zooo \zzzz^2 - \zzzz (\zozz \zzoo + \zzoz \zozo + \zzzo \zooz)\big) - \zozz \zzoz \zzzo,
\end{gather*}
where we use a labeling of the components of $\mathbf{x}$ on the vertices of $C$ as in Fig.~\ref{labeledcube}, with $\zzzz$ corresponding to the component of~$\mathbf{x}$ at~$v$.
\end{Definition}

\begin{Definition}Given $v \in \Z^3$ and $i_1, i_2, i_3 \in \{-1, 1\}$, define $C_v(i_1, i_2, i_3)$ to be the unique unit cube containing the vertices~$v$ and $v + (i_1, i_2, i_3)$.
\end{Definition}

\begin{Proposition} \label{A2B2proposition}
Suppose that $\mathbf{x} = (x_s) \in \C^{\Z^3}$ satisfies the Kashaev equation. Then for any $v \in \Z^3$,
\begin{gather} \label{precoherencecondition}
\begin{split}
\left( \prod_{C \ni v} K^C_v (\mathbf{x}) \right)^2
= \left( \prod_{S \ni v} (x_v x_{v_2} + x_{v_1} x_{v_3}) \right)^2,
\end{split}
\end{gather}
where
\begin{itemize}\itemsep=0pt
\item the first product is over the $8$ unit cubes $C$ incident to the vertex~$v$,
\item the second product is over the $12$ unit squares $S$ incident to~$v$ $($cf.\ Fig.~{\rm \ref{latticesquaresv})}, and
\item $v$, $v_1$, $v_2$, $v_3$ are the vertices of such a unit square $S$ listed in cyclic order.
\end{itemize}
Moreover, the following strengthening of~\eqref{precoherencecondition} holds:
\begin{gather} \label{altprecoherencecondition}
\left( \prod_{\substack{C = C_v(i_1, i_2, i_3)\\ i_1, i_2, i_3 \in \{-1,1\} \\ i_1 i_2 i_3 = 1}} K_v^C(\mathbf{x}) \right)^2 = \left( \prod_{\substack{C = C_v(i_1, i_2, i_3)\\ i_1, i_2, i_3 \in \{-1,1\} \\ i_1 i_2 i_3 = -1}} K_v^C(\mathbf{x}) \right)^2 = \prod_{S \ni v} (x_v x_{v_2} + x_{v_1} x_{v_3}),
\end{gather}
where the rightmost product is the same as in~\eqref{precoherencecondition}.
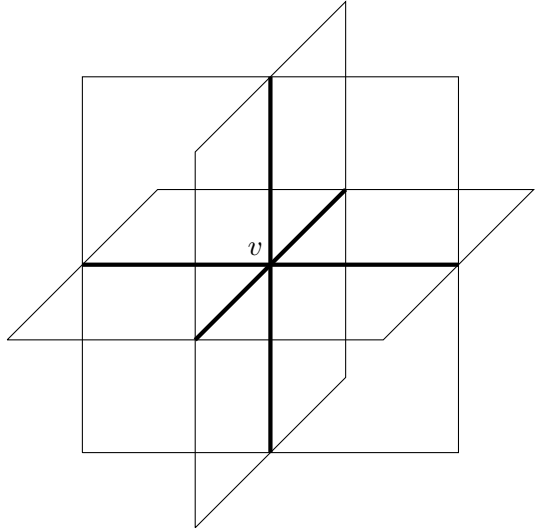
\begin{figure}[ht]\centering
\begin{tikzpicture}

\draw (-3.5,-1)--(1.5,-1)--(3.5,1)--(-1.5,1)--(-3.5,-1);
\draw (-1,-3.5)--(-1,1.5)--(1,3.5)--(1,-1.5)--(-1,-3.5);
\draw (-2.5,-2.5)--(-2.5,2.5)--(2.5,2.5)--(2.5,-2.5)--(-2.5,-2.5);
\draw[ultra thick] (-2.5,0)--(2.5,0);
\draw[ultra thick] (0,-2.5)--(0,2.5);
\draw[ultra thick] (-1,-1)--(1,1);
\draw (-.2,0) node[anchor=south]{$v$};

\end{tikzpicture}
\caption{The $12$ unit squares incident to $v \in \Z^3$.} \label{latticesquaresv}
\end{figure}
\end{Proposition}

\begin{Theorem} \label{ABtheorem}
Suppose that $\mathbf{x} = (x_s) \in (\pos)^{\Z^3}$ satisfies the positive Kashaev recurrence. Then for any $v \in \Z^3$,
\begin{gather} \label{coherencecondition}
\begin{split}
\prod_{C \ni v} K^C_v (\mathbf{x})
= \prod_{S \ni v} (x_v x_{v_2} + x_{v_1} x_{v_3}),
\end{split}
\end{gather}
where the notational conventions are the same as in equation~\eqref{precoherencecondition}.
\end{Theorem}

Proposition~\ref{A2B2proposition} asserts that the expressions being squared in equation~\eqref{precoherencecondition} are equal up to sign; in the case of the positive Kashaev recurrence, Theorem~\ref{ABtheorem} states that the signs must match.

\begin{Definition} \label{coherentkashaevdef}
We say that a solution $\mathbf{x}$ of the Kashaev equation is \emph{coherent} if it satisfies~\eqref{coherencecondition} for every $v \in \Z^3$. Equivalently, $\mathbf{x}$ is coherent if
\begin{gather} \label{coherenceconditionalt}
\prod_{\substack{C = C_v(i_1, i_2, i_3)\\ i_1, i_2, i_3 \in \{-1,1\} \\ i_1 i_2 i_3 = 1}} K_v^C(\mathbf{x}) =\prod_{\substack{C = C_v(i_1, i_2, i_3)\\ i_1, i_2, i_3 \in \{-1,1\} \\ i_1 i_2 i_3 = -1}} K_v^C(\mathbf{x})
\end{gather}
(cf.~\eqref{altprecoherencecondition}).
\end{Definition}

By Theorem~\ref{ABtheorem}, any solution of the positive Kashaev recurrence is a coherent solution of the Kashaev equation.

\begin{Remark} \label{rationalin26others}
If $\mathbf{x} = (x_s)_{s \in \Z^3}$ is a coherent solution of the Kashaev equation,
then for any $v \in \Z^3$, each of the formulas~\eqref{coherencecondition} and~\eqref{coherenceconditionalt} represent $x_{v + (1,1,1)}$ as a rational expression in the $26$ values $x_{v + (\beta_1, \beta_2, \beta_3)}$ for $(\beta_1, \beta_2, \beta_3) \in \{-1,0,1\}^3 \setminus \{(1,1,1)\}$.
\end{Remark}

Coherent solutions of the Kashaev equation are closely related to (a special case of) the hexahedron recurrence, introduced and studied by Kenyon and Pemantle \cite{mainkp}. We next discuss this important construction, which plays a central role in this paper.

\begin{Definition}Let $\lat$ be the subset of $\big(\frac12\Z\big)^3$ defined by
\begin{gather*}
\lat = \big\{(i,j,k) \in \R^3 \colon 2i, 2j, 2k, i + j + k \in \Z\big\}\\
\hphantom{\lat}{} = \Z^3 + \textstyle{\big\{ (0,0,0), \big( 0, \frac{1}{2}, \frac{1}{2} \big), \big( \frac{1}{2}, 0, \frac{1}{2} \big), \big( \frac{1}{2}, \frac{1}{2}, 0 \big) \big\}}.
\end{gather*}
Thus, $\lat$ contains $\Z^3$, together with the centers of unit squares with vertices in~$\Z^3$.
\end{Definition}

Kenyon and Pemantle \cite{mainkp} made the following important observation, which can be verified by direct computation.

\begin{Proposition}[{\cite[Proposition 7.6 and Corollary 7.7]{mainkp}}] \label{easyfromkp} \quad
\begin{enumerate}\itemsep=0pt
\item[$(a)$] Let $\mathbf{x} = (x_s) \in (\pos)^{\Z^3}$ satisfy the positive Kashaev recurrence. Extend $\mathbf{x}$ to an array $\mathbf{\tilde{x}} = (x_s) \in (\pos)^{\lat}$ by setting
\begin{gather} \label{khexspec}
x_{s}^2 = x_{v_1} x_{v_3} + x_{v_2} x_{v_4},
\end{gather}
for all $s \in \lat - \Z^3$, where $v_1,v_2,v_3,v_4 \in \Z^3$ appear in cyclic order along the unit square corresponding to $s$; see Fig.~{\rm \ref{facewithcenter}}. In other words, for all $v \in \Z^3$,
\begin{gather*}
x_{v + \left( 0, \half, \half \right)} = \sqrt{x_{v} x_{v + ( 0, 1, 1)} + x_{v +( 0, 1, 0 )} x_{v + ( 0, 0, 1)}},\\
x_{v + \left( \half, 0, \half \right)} = \sqrt{x_{v} x_{v + ( 1, 0, 1)} + x_{v + ( 1, 0, 0)} x_{v + ( 0, 0, 1)}},\\
x_{v + \left( \half, \half, 0 \right)} = \sqrt{x_{v} x_{v + ( 1, 1, 0 )} + x_{v + ( 1, 0, 0)} x_{v + ( 0, 1, 0 )}}.
\end{gather*}
Then for all $v \in \Z^3$, we have
\begin{gather}
\label{hexeqn1} \zohh = \frac{\zzhh \zhzh \zhhz + \zozz \zzoz \zzzo + \zzzz \zozz \zzoo}{\zzzz \zzhh},\\
\label{hexeqn2} \zhoh= \frac{\zzhh \zhzh \zhhz + \zozz \zzoz \zzzo + \zzzz \zzoz \zozo}{\zzzz \zhzh},\\
\label{hexeqn3} \zhho = \frac{\zzhh \zhzh \zhhz + \zozz \zzoz \zzzo + \zzzz \zzzo \zooz}{\zzzz \zhhz},\\
\label{hexeqnmain} \zooo = \frac{\zzhh^2 \zhzh^2 \zhhz^2 + A \zzhh \zhzh \zhhz + D}{\zzzz^2 \zzhh \zhzh \zhhz},
\end{gather}
where $z_{ijk}$ denotes the component of $\mathbf{\tilde{x}}$ at $v + (i,j,k)$, and $A$ and~$D$ are given by~\eqref{aandbs}.

\item[$(b)$] Conversely, suppose $\mathbf{\tilde{x}} = (x_s) \in (\pos)^\lat$ satisfies~\eqref{hexeqn1}--\eqref{hexeqnmain} together with~\eqref{khexspec}. Then the restriction of $\mathbf{\tilde{x}}$ to $\Z^3$ satisfies the positive Kashaev recurrence.
\end{enumerate}
\begin{figure}[ht]\centering
\begin{tikzpicture}
\draw (0,0)--(2.5,0)--(2.5,2.5)--(0,2.5)--(0,0);
\filldraw (0,0) circle (1.5pt);
\filldraw (2.5,0) circle (1.5pt);
\filldraw (0,2.5) circle (1.5pt);
\filldraw (2.5,2.5) circle (1.5pt);
\filldraw (1.25, 1.25) circle (1.5pt);
\draw (0,0) node[anchor=east]{$v_1$};
\draw (0,2.5) node[anchor=east]{$v_4$};
\draw (2.5, 0) node[anchor=west]{$v_2$};
\draw (2.5,2.5) node[anchor=west]{$v_3$};
\draw (1.25,1.25) node[anchor=north]{$s$};
\end{tikzpicture}
\caption{The points involved in equation~\eqref{khexspec}.}\label{facewithcenter}
\end{figure}
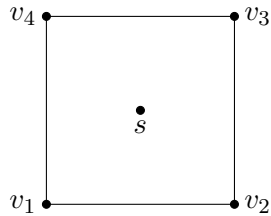
\end{Proposition}

\begin{Definition}[\cite{mainkp}] \label{hexdef} We say that an array $\mathbf{\tilde{x}} \in (\C^*)^{\lat}$ satisfies the \emph{hexahedron recurrence} if for any $v \in \Z^3$, $\mathbf{\tilde{x}}$ satisfies equations~\eqref{hexeqn1}--\eqref{hexeqnmain}. Notice that equations~\eqref{hexeqn1}--\eqref{hexeqnmain} involve the components of~$\mathbf{\tilde{x}}$ at the 14 points in~$L$ located at the boundary of the unit cube in $\Z^3$ with the vertices $v+(i,j,k)$, for $i,j,k\in\{0,1\}$, namely the 8 vertices of the cube, and the 6 centers of its faces.
\end{Definition}

The hexahedron recurrence was introduced in \cite{mainkp} in the context of statistical mechanics as a way to count ``taut double-dimer configurations'' of certain graphs. This recurrence also describes a relationship among principal and ``almost principal'' minors of a square matrix \cite{otherhex}, a connection we will discuss in Section~\ref{principal_minors_section}.

\begin{Remark} \label{directionality_of_hex}
The equations for the hexahedron recurrence, like those for the positive Kashaev recurrence above (and unlike the original Kashaev equation~\eqref{kashforcube}), have a ``preferred direction'', \emph{viz.}, the direction of increase of all three coordinates. While replacing the direction $(1,1,1)$ by the opposite direction $(-1,-1,-1)$ does not change these equations, using any of the six remaining directions $(\pm 1, \pm 1, \pm 1)$ yields a different recurrence. See Remark~\ref{addressing_directionality_of_hex}.
\end{Remark}

We now extend Proposition~\ref{easyfromkp} to complex-valued solutions of the hexahedron recurrence.

\begin{Theorem} \label{galoisfromintro1} \quad
\begin{enumerate}\itemsep=0pt
\item[$(a)$] Let $\mathbf{x} = (x_s) \in (\C^*)^{\Z^3}$ be a coherent solution of the Kashaev equation, with
\begin{gather} \label{conditionnonzerofaces}
x_v x_{v + e_i + e_j} + x_{v + e_i} x_{v + e_j} \not= 0
\end{gather}
for all $v \in \Z^3$ and all distinct $i, j \in \{1,2,3\}$. Then $\mathbf{x}$ can be extended to an array $\mathbf{\tilde{x}} = (x_s) \in (\C^*)^{\lat}$ satisfying the hexahedron recurrence along with~\eqref{khexspec}.
\item[$(b)$] Conversely, suppose $\mathbf{\tilde{x}} = (x_s) \in (\C^*)^{\lat}$ satisfies the hexahedron recurrence along with~\eqref{khexspec}. Then the restriction of $\mathbf{\tilde{x}}$ to $\Z^3$ is a coherent solution of the Kashaev equation and satisfies condition~\eqref{conditionnonzerofaces}.
\end{enumerate}
\end{Theorem}

\begin{Remark}Theorem~\ref{ABtheorem} follows from Theorem~\ref{galoisfromintro1}(b), because a solution of the positive Kashaev recurrence can be extended to a solution of the hexahedron recurrence that satisfies~\eqref{khexspec} (by Proposition~\ref{easyfromkp}).
\end{Remark}

\begin{Remark}If $\mathbf{x}$ doesn't satisfy condition~\eqref{conditionnonzerofaces}, and an array $\mathbf{\tilde{x}}$ extending $\mathbf{x} \in \C^{\lat}$ satisfies~\eqref{khexspec}, then at least one of the face variables for $\mathbf{\tilde{x}}$ equals $0$, requiring us to divide by $0$ when we apply the hexahedron recurrence. On the other hand, if $\mathbf{\tilde{x}} \in \C^{\lat}$ satisfies equations~\eqref{hexeqn1}--\eqref{hexeqnmain} with the denominators multiplied out (so that the denominators can equal~$0$), then the restriction of $\mathbf{\tilde{x}}$ to $\Z^3$ doesn't necessarily satisfy the Kashaev equation.
\end{Remark}

The following statement is straightforward to check.

\begin{Proposition} \label{likehexprop}Let $\mathbf{\tilde{x}} = (x_s) \in (\C^*)^\lat$ be an array satisfying~\eqref{khexspec}, for any~$s \in \lat - \Z^3$. Then the following are equivalent:
\begin{itemize}\itemsep=0pt
\item $\mathbf{\tilde{x}}$ satisfies the hexahedron recurrence,
\item for any $v \in \Z^3$, we have
\begin{gather}
\label{likehex1d} \zohh = \frac{\zhzh \zhhz + \zzhh \zozz}{\zzzz},\\
\label{likehex2d} \zhoh = \frac{\zzhh \zhhz + \zhzh \zzoz}{\zzzz},\\
\label{likehex3d} \zhho = \frac{\zzhh \zhzh + \zhhz \zzzo}{\zzzz},\\
\label{likehex4d} \zooo = \frac{A + 2 \zzhh \zhzh \zhhz}{\zzzz^2},
\end{gather}
where, as before, $z_{ijk}$ denotes the component of $\mathbf{\tilde{x}}$ at $v + (i,j,k)$, and $A$ is given by~\eqref{aandbs}.
\end{itemize}
\end{Proposition}

\begin{Definition} \label{khexeqndef}
Let $\mathbf{\tilde{x}} = (x_s) \in \C^{\lat}$ be an array with $x_s \not= 0$ for $s \in \Z^3$. We say that $\mathbf{\tilde{x}}$ satisfies the \emph{K-hexahedron equations} (a shorthand for ``Kashaev-hexahedron equations'') if $\mathbf{\tilde{x}}$ satisfies equation~\eqref{khexspec} for all $s \in \lat - \Z^3$, and satisfies equations~\eqref{likehex1d}--\eqref{likehex4d} for all $v \in \Z^3$.%
\end{Definition}

\begin{Remark}
By Proposition~\ref{likehexprop}, if $\mathbf{\tilde{x}} \in (\C^*)^{\lat}$, i.e., $\mathbf{\tilde{x}}$ has all nonzero components, then the following are equivalent:
\begin{itemize}\itemsep=0pt
\item $\mathbf{\tilde{x}}$ satisfies the K-hexahedron equations,
\item $\mathbf{\tilde{x}}$ satisfies the hexahedron recurrence, along with equation~\eqref{khexspec} for $s \in \lat - \Z^3$.
\end{itemize}
\end{Remark}

We next restate Theorem~\ref{galoisfromintro1} (and slightly strengthen part~(b) thereof) using the notion of the K-hexahedron equations.

\begin{Theorem} \label{galoisfromintro} \quad
\begin{enumerate}\itemsep=0pt
\item[$(a)$] Any coherent solution $\mathbf{x} = (x_s) \!\in\! (\C^*)^{\Z^3}$ of the Kashaev equation satisfying condi\-tion~\eqref{conditionnonzerofaces} can be extended to an array $\mathbf{\tilde{x}} = (x_s) \in \C^{\lat}$ satisfying the K-hexahedron equations.
\item[$(b)$] Conversely, suppose that $\mathbf{\tilde{x}} = (x_s) \in \C^\lat$ $($with $x_s \not= 0$ for all $s \in \Z^3)$ satisfies the K-hexahedron equations. Then the restriction of $\mathbf{\tilde{x}}$ to $\Z^3$ is a coherent solution of the Kashaev equation.
\end{enumerate}
\end{Theorem}

The extension from $\mathbf{x}$ to $\mathbf{\tilde{x}}$ in Theorem~\ref{galoisfromintro}(a) is not unique. The theorem below clarifies the relationship between different extensions.

\begin{Theorem} \label{extensions_that_agree}Let $\mathbf{\tilde{x}} = (x_s) \in (\C^*)^{\lat}$ be an array satisfying the K-hexahedron equations.
\begin{enumerate}\itemsep=0pt
\item[$(a)$] For all $i \in \Z$, let $\alpha_i, \beta_i, \gamma_i \in \{-1, 1\}$. Let $\mathbf{\tilde{y}} = (y_s) \in (\C^*)^{\lat}$ be defined by
\begin{gather} \label{agree_on_vertices}
y_{( a,b,c)} = x_{(a,b,c)},\\
y_{\left( a, b + \half, c + \half \right)} = \beta_b \gamma_c x_{\left( a, b + \half, c + \half \right)},\\
y_{\left( a + \half, b, c + \half \right)} = \alpha_a \gamma_c x_{\left( a + \half, b, c + \half \right)},\\ \label{agree_on_L_3}
y_{\left( a + \half, b + \half, c \right)} = \alpha_a \beta_b x_{\left( a + \half, b + \half, c \right)},
\end{gather}
for $( a, b, c) \in \Z^3$. Then $\mathbf{\tilde{y}}$ satisfies the K-hexahedron equations.
\item[$(b)$] Conversely, suppose $\mathbf{\tilde{y}} = (y_s) \in (\C^*)^{\lat}$ is an array satisfying the K-hexahedron equations that agrees with $\mathbf{\tilde{x}}$ on $\Z^3$ $($cf.~\eqref{agree_on_vertices}$)$. Then there exist signs $\alpha_i, \beta_i, \gamma_i \in \{-1, 1\}$, $i \in \Z$, such that $\mathbf{\tilde{y}}$ is given by~\eqref{agree_on_vertices}--\eqref{agree_on_L_3} for $( a, b, c) \in \Z^3$.
\end{enumerate}
\end{Theorem}

Theorems~\ref{galoisfromintro}--\ref{extensions_that_agree} are proved in Section~\ref{mainproof}.%

\section[Combinatorial preliminaries on cubical complexes and zonotopes]{Combinatorial preliminaries on cubical complexes\\ and zonotopes} \label{preliminaries}

In this section, we review some standard combinatorial background on cubical complexes and zonotopes. This section introduces some unconventional terminology which later sections will use.

\begin{Definition}\label{cc_def}A \emph{cubical complex} is a polyhedral complex whose cells are cubes of various dimensions, see \cite[Definitions~2.39 and~2.42]{kozlov}. We do not require that there exist an embedding of a cubical complex into Euclidean space such that every cell is a polyhedron. A cubical complex~$\varkappa$ is \emph{$d$-dimensional} if the dimension of the largest cube in~$\varkappa$ is~$d$; it is \emph{pure of dimension~$d$} if every cube of $\varkappa$ is either dimension $d$ or a face of some $d$-dimensional cube in~$\varkappa$. A \emph{quadrangulation} of a polygon $R$ in $\R^2$ is a realization of $R$ as a (pure, $2$-dimensional) cubical complex.
\end{Definition}

\begin{Definition} \label{interior_point_definition}Let $\varkappa$ be a cubical complex embedded (as a topological space) into a Euclidean space $\R^d$. A point $v$ in $\varkappa$ is called an \emph{interior point} of $\varkappa$ if $\varkappa$ contains a (small) open ball centered at~$v$. (This notion does not depend on the choice of embedding for a fixed $d$.)
\end{Definition}

\begin{Definition}\label{zonotope_definition}A \emph{$m$-dimensional zonotope} $\mathcal{Z}_{v_1, \dots, v_\ell}$ is the Minkowski sum of line segments $\sum\limits_{j = 1}^\ell [0, v_j]$ for $v_1, \dots, v_\ell \in \R^m$ spanning $\R^m$. A \emph{cubical tiling} of $\mathcal{Z}_{v_1, \dots, v_\ell}$ is a tiling of $\mathcal{Z}_{v_1, \dots, v_\ell}$ with the translates of the Minkowski sums $\sum\limits_{j \in I} [0, v_j]$ over $I \in \binom{[\ell]}{m}$ such that $\{v_j\colon j \in I\}$ is linearly independent. Cubical tilings of zonotopes are examples of cubical complexes.
\end{Definition}

\begin{Definition}\label{def_diamond_pn}We denote by $\mathbf{P}_n$ the regular $(2n)$-gon $\mathcal{Z}_{e_1, \dots, e_n}$ where $e_j = {\rm e}^{\pi {\rm i} (j - 1) / n} \allowbreak \in \C \cong \R^2$ for $j = 1, \dots, n$ using the standard identification between $\C$ and $\R^2$ (see Fig.~\ref{pn_zono_figure}). We denote by~$v_0$ the vertex of $\mathbf{P}_n$ corresponding to the origin. We define a \emph{$\Diamond$-tiling of $\mathbf{P}_n$} to be a cubical tiling of $\mathcal{Z}_{e_1, \dots, e_n}$, i.e., a tiling of $\mathbf{P}_n$ with the $\binom{n}{2}$ rhombi given by the translations of the Minkowski sums $[0, e_i] + [0, e_j]$ for $1 \le i < j \le n$ (see Fig.~\ref{2ngon_fig}). We label the vertices of a $\Diamond$-tiling of $\mathbf{P}_n$ by subsets of $[n]$ as follows: we label a vertex $v$ by $I \subseteq [n]$ if we can reach $v$ from $v_0$ by following the edges of the tiling corresponding to the vectors $e_j$ for $j \in I$ (see Fig.~\ref{2ngon_fig}).
\begin{figure}[ht]\centering
\begin{tabular}{ c c }
\begin{tikzpicture}[scale=1.5]
\draw[gray] (-2,0)--(2,0);
\draw[gray] (0,0)--(0,2);

\draw[ultra thick, ->] (0,0)--(1,0);
\draw[ultra thick, ->] (0,0)--({sqrt(2)/2},{sqrt(2)/2});
\draw[ultra thick, ->] (0,0)--(0,1);
\draw[ultra thick, ->] (0,0)--(-{sqrt(2)/2},{sqrt(2)/2});

\draw (1,0) node[anchor=west]{$e_1$};
\draw ({sqrt(2)/2},{sqrt(2)/2}) node[anchor=south west]{$e_2$};
\draw (0,1) node[anchor=south]{$e_3$};
\draw ({-sqrt(2)/2},{sqrt(2)/2}) node[anchor=south east]{$e_4$};
%\draw[ultra thick] [decoration={markings,mark=at position 1 with
% {\arrow[scale=1,>=stealth]{>}}},postaction={decorate}] (0,0)--(1,0);
\end{tikzpicture} & \hspace{1 cm}
\begin{tikzpicture}[scale=1.5]
\draw (0,0)--(1,0)--({1+sqrt(2)/2},{sqrt(2)/2})--({1+sqrt(2)/2},{1+sqrt(2)/2})--({1},{1+sqrt(2)})--({0},{1+sqrt(2)})--({-sqrt(2)/2},{1+sqrt(2)/2})--({-sqrt(2)/2},{sqrt(2)/2})--(0,0);
\draw ({1/2},{(1+sqrt(2))/2}) node{$\mathbf{P}_4$};
\filldraw (0,0) circle (1.5pt);
\draw (0,0) node[anchor=north east]{$v_0$};
\end{tikzpicture}
\end{tabular}
\caption{On the left, the vectors $e_1$, $e_2$, $e_3$, $e_4$ from Definition~\ref{def_diamond_pn} when $n = 4$. On the right, the regular $8$-gon $\mathbf{P}_4$.} \label{pn_zono_figure}
\end{figure}
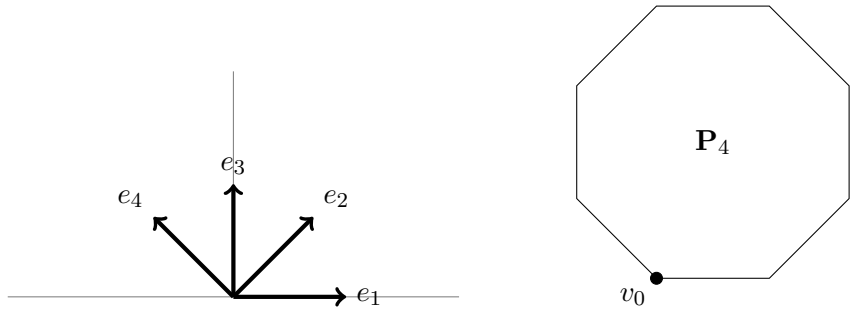
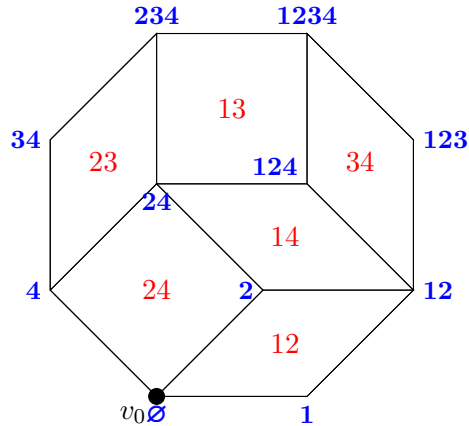
\begin{figure}[ht]\centering
\begin{tikzpicture}[scale=2]
\draw (0,0)--(1,0)--({1+sqrt(2)/2},{sqrt(2)/2})--({1+sqrt(2)/2},{1+sqrt(2)/2})--({1},{1+sqrt(2)})--({0},{1+sqrt(2)})--({-sqrt(2)/2},{1+sqrt(2)/2})--({-sqrt(2)/2},{sqrt(2)/2})--(0,0);

\draw (0,0)--(1,0)--({1+sqrt(2)/2},{sqrt(2)/2})--({sqrt(2)/2},{sqrt(2)/2})--(0,0);
\draw (0,0)--({sqrt(2)/2},{sqrt(2)/2})--(0,{sqrt(2)})--({-sqrt(2)/2},{sqrt(2)/2})--(0,0);
\draw[shift={({-sqrt(2)/2},{sqrt(2)/2})}] (0,0)--({sqrt(2)/2},{sqrt(2)/2})--({sqrt(2)/2},{1+sqrt(2)/2})--(0,1)--(0,0);
\draw[shift={({sqrt(2)/2},{sqrt(2)/2})}] (0,0)--(1,0)--({1 - sqrt(2)/2},{sqrt(2)/2})--({- sqrt(2)/2},{sqrt(2)/2})--(0,0);
\draw[shift={(0,{sqrt(2)})}] (0,0)--(1,0)--(1,1)--(0,1)--(0,0);
\draw[shift={({1+sqrt(2)/2},{sqrt(2)/2})}] (0,0)--(0,1)--({-sqrt(2)/2},{1 + sqrt(2)/2})--({-sqrt(2)/2},{sqrt(2)/2})--(0,0);
\filldraw (0,0) circle (1.5pt);
\draw (0,0) node[anchor=north east]{$v_0$};

{
\color{blue}
\small
\draw (0,0) node[anchor= north]{$\bm{\varnothing}$};
\draw (1,0) node[anchor=north]{$\mathbf{1}$};
\draw ({1+sqrt(2)/2},{sqrt(2)/2}) node[anchor=west]{$\mathbf{12}$};
\draw ({1+sqrt(2)/2},{1+sqrt(2)/2}) node[anchor=west]{$\mathbf{123}$};
\draw ({1},{1+sqrt(2)}) node[anchor=south]{$\mathbf{1234}$};
\draw ({0},{1+sqrt(2)}) node[anchor=south]{$\mathbf{234}$};
\draw ({-sqrt(2)/2},{1+sqrt(2)/2}) node[anchor=east]{$\mathbf{34}$};
\draw ({-sqrt(2)/2},{sqrt(2)/2}) node[anchor=east]{$\mathbf{4}$};

\draw ({sqrt(2)/2},{sqrt(2)/2}) node[anchor=east]{$\mathbf{2}$};
\draw (0,{sqrt(2)}) node[anchor=north]{$\mathbf{24}$};
\draw (1,{sqrt(2)}) node[anchor=south east]{$\mathbf{124}$};

}

{
\color{red}

\draw ({(1+sqrt(2)/2)/2},{sqrt(2)/4}) node{$12$};
\draw (0,{sqrt(2)/2}) node{$24$};
\draw ({sqrt(2)/4 + 1/2},{3*sqrt(2)/4}) node{$14$};
\draw ({-sqrt(2)/4},{sqrt(2)/4 + (1+sqrt(2))/2}) node{$23$};
\draw (1/2,{1/2+sqrt(2)}) node{$13$};
\draw ({1/2 + (1+sqrt(2)/2)/2},{sqrt(2)/2 + (1+sqrt(2)/2)/2}) node{$34$};

}
\end{tikzpicture}
\caption{A $\Diamond$-tiling of $\mathbf{P}_4$. The vertex $v_0$ corresponding to the origin is labeled. In red, we label each rhombus that is a translation of the Minkowski sum $[0, e_i] \times [0, e_j]$ by $ij$. In blue, we label each vertex of the tiling by its corresponding subset of $[4]$.}\label{2ngon_fig}
\end{figure}
\end{Definition}

\begin{Definition}\label{flip_def}Two quadrangulations $T_1$, $T_2$ of a polygon are connected by a \emph{flip} if they are related by a single local move of the form pictured in Fig.~\ref{rhombus_flip}. Note that we can picture a flip as placing a cube on top of the hexagon where the flip occurs.
\begin{figure}[ht]\centering
\begin{tikzpicture}[scale=1.5]
\draw (0,0)--(1,0)--({1 + 1/2},{sqrt(3)/2})--(1,{sqrt(3)})--(0,{sqrt(3)})--({- 1/2},{sqrt(3)/2})--(0,0);
\draw ({1/2},{sqrt(3)/2})--(0,0);
\draw ({1/2},{sqrt(3)/2})--({1 + 1/2},{sqrt(3)/2});
\draw ({1/2},{sqrt(3)/2})--(0,{sqrt(3)});

%\draw[<|-|>,ultra thick] (1.75,{sqrt(3)/2})--(3.25,{sqrt(3)/2});
\draw [decoration={markings,mark=at position 1 with
 {\arrow[scale=3,>=stealth]{>}}},postaction={decorate}] (1.75,{sqrt(3)/2})--(3.25,{sqrt(3)/2});
\draw [decoration={markings,mark=at position 1 with
 {\arrow[scale=3,>=stealth]{>}}},postaction={decorate}] (3.25,{sqrt(3)/2})--(1.75,{sqrt(3)/2});

\draw[shift={(4,0)}] (0,0)--(1,0)--({1 + 1/2},{sqrt(3)/2})--(1,{sqrt(3)})--(0,{sqrt(3)})--({- 1/2},{sqrt(3)/2})--(0,0);
\draw[shift={(4,0)}] ({1/2},{sqrt(3)/2})--(1,0);
\draw[shift={(4,0)}] ({1/2},{sqrt(3)/2})--(1,{sqrt(3)});
\draw[shift={(4,0)}] ({1/2},{sqrt(3)/2})--({- 1/2},{sqrt(3)/2});
\end{tikzpicture}
\caption{A flip.} \label{rhombus_flip}
\end{figure}
\end{Definition}

It will be helpful for us to think about quadrangulations through the dual language of divides.

\begin{Definition}\label{divide_definition}A \emph{divide} $D$ in a polygon $R$ in $\R^2$ is an immersion of a finite set of closed intervals and circles, called \emph{branches}, in $R$, such that
\begin{itemize}\itemsep=0pt
\item the immersed circles do not intersect the boundary of $R$,
\item the immersed intervals have pairwise distinct endpoints on the boundary of $R$, and are otherwise disjoint from the boundary,
\item all intersections and self-intersections of the branches are transversal, and
\item no three branches intersect at a point,
\end{itemize}
all considered up to isotopy. For further details, see \cite[Definition~2.1]{divide}. Given a quadrangula\-tion~$T$ of~$R$, the divide associated to $T$ is the divide in $R$ such that for every tile $Q$ in $T$, branches connect the $2$ pairs of opposite edges in $Q$, and there is a single branch intersection in the interior of $Q$ (see Fig.~\ref{divide_figure}). (All divides considered in the remainder of this paper are associated to quadrangulations.) A \emph{braid move} is a local transformation of divides shown in Fig.~\ref{braid_figure}. A flip in a~quadrangulation corresponds to a braid move in its associated divide.
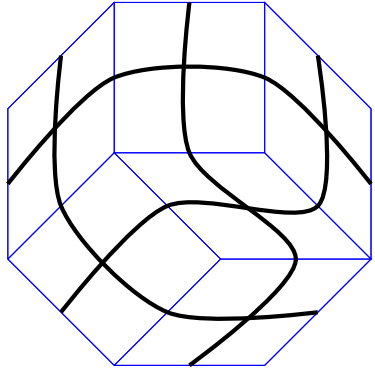
\begin{figure}[ht]\centering
\begin{tikzpicture}[scale=2]
{
\color{blue}
\draw (0,0)--(1,0)--({1+sqrt(2)/2},{sqrt(2)/2})--({1+sqrt(2)/2},{1+sqrt(2)/2})--({1},{1+sqrt(2)})--({0},{1+sqrt(2)})--({-sqrt(2)/2},{1+sqrt(2)/2})--({-sqrt(2)/2},{sqrt(2)/2})--(0,0);

\draw (0,0)--(1,0)--({1+sqrt(2)/2},{sqrt(2)/2})--({sqrt(2)/2},{sqrt(2)/2})--(0,0);
\draw (0,0)--({sqrt(2)/2},{sqrt(2)/2})--(0,{sqrt(2)})--({-sqrt(2)/2},{sqrt(2)/2})--(0,0);
\draw[shift={({-sqrt(2)/2},{sqrt(2)/2})}] (0,0)--({sqrt(2)/2},{sqrt(2)/2})--({sqrt(2)/2},{1+sqrt(2)/2})--(0,1)--(0,0);
\draw[shift={({sqrt(2)/2},{sqrt(2)/2})}] (0,0)--(1,0)--({1 - sqrt(2)/2},{sqrt(2)/2})--({- sqrt(2)/2},{sqrt(2)/2})--(0,0);
\draw[shift={(0,{sqrt(2)})}] (0,0)--(1,0)--(1,1)--(0,1)--(0,0);
\draw[shift={({1+sqrt(2)/2},{sqrt(2)/2})}] (0,0)--(0,1)--({-sqrt(2)/2},{1 + sqrt(2)/2})--({-sqrt(2)/2},{sqrt(2)/2})--(0,0);
}

\draw[ultra thick] plot [smooth] coordinates {(0.5,0) ({sqrt(2)/2 + 0.5},{sqrt(2)/2}) (0.5,{sqrt(2)}) (0.5,{sqrt(2) + 1})};
\draw[ultra thick] plot [smooth] coordinates {({-sqrt(2)/4}, {sqrt(2)/4}) ({sqrt(2)/4},{3*sqrt(2)/4}) ({sqrt(2)/4 + 1},{3*sqrt(2)/4}) ({sqrt(2)/4 + 1},{3*sqrt(2)/4 + 1})};
\draw[ultra thick] plot [smooth] coordinates {({-sqrt(2)/2},{sqrt(2)/2 + 0.5}) ({0},{sqrt(2) + 0.5}) ({1},{sqrt(2) + 0.5}) ({1 + sqrt(2)/2},{sqrt(2)/2 + 0.5})};
\draw[ultra thick] plot [smooth] coordinates {({-sqrt(2)/4},{1+3*sqrt(2)/4}) ({-sqrt(2)/4},{3*sqrt(2)/4}) ({sqrt(2)/4},{sqrt(2)/4}) ({sqrt(2)/4 + 1},{sqrt(2)/4})};

\end{tikzpicture}
\caption{The divide (in black) associated to a quadrangulation (in blue).} \label{divide_figure}
\end{figure}
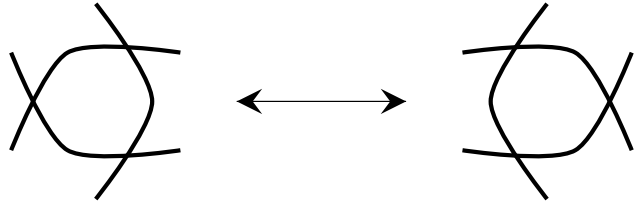
\begin{figure}[ht]\centering
\begin{tikzpicture}[scale=1.5]

\draw[ultra thick] plot [smooth] coordinates {(0.5, 0) (1, {sqrt(3)/2}) (0.5, {sqrt(3)})};
\draw[ultra thick] plot [smooth] coordinates {(-0.25, {sqrt(3)/4}) (0.25, {3*sqrt(3)/4}) (1.25, {3*sqrt(3)/4})};
\draw[ultra thick] plot [smooth] coordinates {(-0.25, {3*sqrt(3)/4}) (0.25, {sqrt(3)/4}) (1.25, {sqrt(3)/4})};

%\draw[<|-|>,ultra thick] (1.75,{sqrt(3)/2})--(3.25,{sqrt(3)/2});
\draw [decoration={markings,mark=at position 1 with
 {\arrow[scale=3,>=stealth]{>}}},postaction={decorate}] (1.75,{sqrt(3)/2})--(3.25,{sqrt(3)/2});
\draw [decoration={markings,mark=at position 1 with
 {\arrow[scale=3,>=stealth]{>}}},postaction={decorate}] (3.25,{sqrt(3)/2})--(1.75,{sqrt(3)/2});

\draw[shift={(4,0)}][ultra thick] plot [smooth] coordinates {(0.5, 0) (0, {sqrt(3)/2}) (0.5, {sqrt(3)})};
\draw[shift={(4,0)}][ultra thick] plot [smooth] coordinates {(-0.25, {sqrt(3)/4}) (0.75, {sqrt(3)/4}) (1.25, {3*sqrt(3)/4})};
\draw[shift={(4,0)}][ultra thick] plot [smooth] coordinates {(-0.25, {3*sqrt(3)/4}) (0.75, {3*sqrt(3)/4}) (1.25, {sqrt(3)/4})};
\end{tikzpicture}
\caption{A braid move.} \label{braid_figure}
\end{figure}
\end{Definition}

\begin{Definition}\label{ps_arr_definition}A divide is called a \emph{pseudoline arrangement} if all of its branches are immersed intervals with no self-intersections, and, moreover, each pair of branches intersects at most once. Note that the class of pseudoline arrangements is closed under braid moves.
\end{Definition}

The following fact is well known \cite[Theorem~6.10]{felsner}.

\begin{Proposition} \label{pn_pseudoline}Let $D$ be a divide in a polygon in $\R^2$. Then the following are equivalent:
\begin{itemize}\itemsep=0pt
\item $D$ is a pseudoline arrangement in which every pair of branches intersects exactly once,
\item $D$ is topologically equivalent to the divide associated to a $\Diamond$-tiling of $\mathbf{P}_n$.
\end{itemize}
\end{Proposition}

\begin{Remark}Pseudoline arrangements of $n$ branches, each pair of which intersects exactly once, are in bijection with commutation-equivalence classes of reduced words for the longest element $w_0 \in S_n$ in the symmetric group. A braid move on the pseudoline arrangement corresponds to a braid move on the reduced word.%
\end{Remark}

\begin{Definition}\label{chambers_pla}Let $T$ be a $\Diamond$-tiling of $\mathbf{P}_n$, and let $D$ be the pseudoline arrangement associated to $T$. We call the connected components of the complement of $D$ the \emph{chambers} of $D$. Note that the chambers of $D$ correspond to the vertices of $T$, and the crossings of $D$ correspond to the tiles of $T$. Label the branches $1, \dots, n$ as in Fig.~\ref{labeled_psa_branches}, by starting at $v_0$ and traveling counterclockwise along the boundary of $\mathbf{P}_n$ (so that branch $j$ intersects the boundary of $\mathbf{P}_n$ at the edges parallel to~$e_j = {\rm e}^{\pi {\rm i} (j - 1)/n}$). Note that the label $I \subseteq [n]$ for a vertex of $T$ is precisely the set of labels for the branches in between the chamber and $v_0$.
\begin{figure}\centering
\begin{tikzpicture}[scale=2]
{
\color{red}
\draw (0,0)--(1,0)--({1+sqrt(2)/2},{sqrt(2)/2})--({1+sqrt(2)/2},{1+sqrt(2)/2})--({1},{1+sqrt(2)})--({0},{1+sqrt(2)})--({-sqrt(2)/2},{1+sqrt(2)/2})--({-sqrt(2)/2},{sqrt(2)/2})--(0,0);

\draw (0,0)--(1,0)--({1+sqrt(2)/2},{sqrt(2)/2})--({sqrt(2)/2},{sqrt(2)/2})--(0,0);
\draw (0,0)--({sqrt(2)/2},{sqrt(2)/2})--(0,{sqrt(2)})--({-sqrt(2)/2},{sqrt(2)/2})--(0,0);
\draw[shift={({-sqrt(2)/2},{sqrt(2)/2})}] (0,0)--({sqrt(2)/2},{sqrt(2)/2})--({sqrt(2)/2},{1+sqrt(2)/2})--(0,1)--(0,0);
\draw[shift={({sqrt(2)/2},{sqrt(2)/2})}] (0,0)--(1,0)--({1 - sqrt(2)/2},{sqrt(2)/2})--({- sqrt(2)/2},{sqrt(2)/2})--(0,0);
\draw[shift={(0,{sqrt(2)})}] (0,0)--(1,0)--(1,1)--(0,1)--(0,0);
\draw[shift={({1+sqrt(2)/2},{sqrt(2)/2})}] (0,0)--(0,1)--({-sqrt(2)/2},{1 + sqrt(2)/2})--({-sqrt(2)/2},{sqrt(2)/2})--(0,0);

\filldraw (0,0) circle (1.5pt);
\draw (0,0) node[anchor=north east]{$v_0$};
}

\draw[ultra thick] plot [smooth] coordinates {(0.5,0) ({sqrt(2)/2 + 0.5},{sqrt(2)/2}) (0.5,{sqrt(2)}) (0.5,{sqrt(2) + 1})};
\draw[ultra thick] plot [smooth] coordinates {({-sqrt(2)/4}, {sqrt(2)/4}) ({sqrt(2)/4},{3*sqrt(2)/4}) ({sqrt(2)/4 + 1},{3*sqrt(2)/4}) ({sqrt(2)/4 + 1},{3*sqrt(2)/4 + 1})};
\draw[ultra thick] plot [smooth] coordinates {({-sqrt(2)/2},{sqrt(2)/2 + 0.5}) ({0},{sqrt(2) + 0.5}) ({1},{sqrt(2) + 0.5}) ({1 + sqrt(2)/2},{sqrt(2)/2 + 0.5})};
\draw[ultra thick] plot [smooth] coordinates {({-sqrt(2)/4},{1+3*sqrt(2)/4}) ({-sqrt(2)/4},{3*sqrt(2)/4}) ({sqrt(2)/4},{sqrt(2)/4}) ({sqrt(2)/4 + 1},{sqrt(2)/4})};

\draw ({0.5},{0}) node[anchor=north]{$1$};
\draw ({1 + sqrt(2)/4},{sqrt(2)/4}) node[anchor=west]{$2$};
\draw ({1 + sqrt(2)/2},{sqrt(2)/2 + 0.5}) node[anchor=west]{$3$};
\draw ({1 + sqrt(2)/4},{3*sqrt(2)/4 + 1}) node[anchor=west]{$4$};

{
\color{blue}
\small
\draw (0,0) node[anchor= north]{$\bm{\varnothing}$};
\draw (1,0) node[anchor=north]{$\mathbf{1}$};
\draw ({1+sqrt(2)/2},{sqrt(2)/2}) node[anchor=west]{$\mathbf{12}$};
\draw ({1+sqrt(2)/2},{1+sqrt(2)/2}) node[anchor=west]{$\mathbf{123}$};
\draw ({1},{1+sqrt(2)}) node[anchor=south]{$\mathbf{1234}$};
\draw ({0},{1+sqrt(2)}) node[anchor=south]{$\mathbf{234}$};
\draw ({-sqrt(2)/2},{1+sqrt(2)/2}) node[anchor=east]{$\mathbf{34}$};
\draw ({-sqrt(2)/2},{sqrt(2)/2}) node[anchor=east]{$\mathbf{4}$};

\draw ({sqrt(2)/2},{sqrt(2)/2}) node[anchor=east]{$\mathbf{2}$};
\draw (0,{sqrt(2)}) node[anchor=north]{$\mathbf{24}$};
\draw (1,{sqrt(2)}) node[anchor=south east]{$\mathbf{124}$};

}
\end{tikzpicture}
\caption{The pseudoline arrangement (in black) associated to a $\Diamond$-tiling of $\mathbf{P}_4$ (in red). The branches are labeled (in black) as described in Definition~\ref{chambers_pla}. Note that the label $I \subseteq [n]$ for the vertices of $T$ (in blue) is precisely the set of labels for the branches in between the chamber and $v_0$.} \label{labeled_psa_branches}
\end{figure}
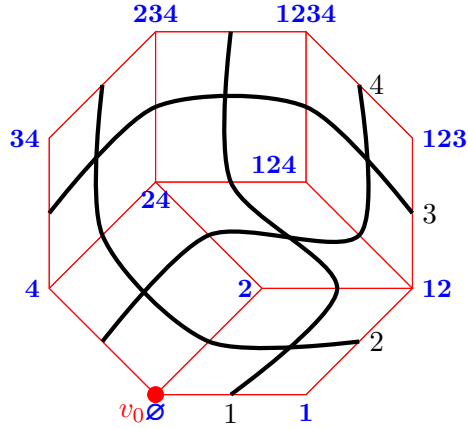
\end{Definition}

\begin{Definition}Let $T$ be a $\Diamond$-tiling of $\mathbf{P}_n$, and let $D$ be the pseudoline arrangement associated to~$T$. Label the branches as in Definition~\ref{chambers_pla}. Given three branches labeled $i < j < k$, we say that $i$, $j$, $k$ have a \emph{$\Delta$-crossing} if the point where branches~$i$ and~$k$ intersect is in a different component of the complement of branch $j$ than the point $v_0$. We say that $i$, $j$, $k$ have a \emph{$\nabla$-crossing} if the point where branches $i$ and $k$ intersect is in the same component of the complement of branch $j$ as the point $v_0$. Note that $i$, $j$, $k$ must either have a $\Delta$-crossing or a $\nabla$-crossing. See Fig.~\ref{delta_crossing} for pictures which make clear the reasoning behind this choice of terminology. Note that when a braid move is performed with $i$, $j$, $k$, the triple $i$, $j$, $k$ switches between having a $\Delta$-crossing and having a $\nabla$-crossing.
\begin{figure}\centering
\begin{tabular}{ c c }
\begin{tikzpicture}[scale=1.5]
\draw (0.5,{sqrt(3)/2}) circle ({sqrt(3)/2});
\draw[ultra thick] plot [smooth] coordinates {(0.5, 0) (1, {sqrt(3)/2}) (0.5, {sqrt(3)})};
\draw[ultra thick] plot [smooth] coordinates {(-0.25, {sqrt(3)/4}) (0.25, {3*sqrt(3)/4}) (1.25, {3*sqrt(3)/4})};
\draw[ultra thick] plot [smooth] coordinates {(-0.25, {3*sqrt(3)/4}) (0.25, {sqrt(3)/4}) (1.25, {sqrt(3)/4})};

\draw (0.5, 0) node[anchor=north]{$i$};
\draw (1.25, {sqrt(3)/4}) node[anchor=west]{$j$};
\draw (1.25, {3*sqrt(3)/4}) node[anchor=west]{$k$};

\draw (0.75, {sqrt(3)/4}) node[anchor=south east]{$ij$};
\draw (0.75, {3*sqrt(3)/4}) node[anchor=north]{$ik$};
\draw (0,{sqrt(3)/2}) node[anchor=west]{$jk$};

\filldraw ({0.5*(1 - sqrt(3)/2)},{sqrt(3)/2*(1 - sqrt(3)/2)}) circle (1.5pt);
\draw ({0.5*(1 - sqrt(3)/2)},{sqrt(3)/2*(1 - sqrt(3)/2)}) node[anchor=north east]{$v_0$};
\end{tikzpicture} & \hspace{2cm}
\begin{tikzpicture}[scale=1.5]
\draw (0.5,{sqrt(3)/2}) circle ({sqrt(3)/2});
\draw[ultra thick] plot [smooth] coordinates {(0.5, 0) (0, {sqrt(3)/2}) (0.5, {sqrt(3)})};
\draw[ultra thick] plot [smooth] coordinates {(-0.25, {sqrt(3)/4}) (0.75, {sqrt(3)/4}) (1.25, {3*sqrt(3)/4})};
\draw[ultra thick] plot [smooth] coordinates {(-0.25, {3*sqrt(3)/4}) (0.75, {3*sqrt(3)/4}) (1.25, {sqrt(3)/4})};

\draw (0.5, 0) node[anchor=north]{$i$};
\draw (1.25, {sqrt(3)/4}) node[anchor=west]{$j$};
\draw (1.25, {3*sqrt(3)/4}) node[anchor=west]{$k$};

\draw (0.25,{sqrt(3)/4}) node[anchor=south]{$ik$};
\draw (0.25,{3*sqrt(3)/4}) node[anchor=north]{$ij$};
\draw (1,{sqrt(3)/2}) node[anchor=east]{$jk$};

\filldraw ({0.5*(1 - sqrt(3)/2)},{sqrt(3)/2*(1 - sqrt(3)/2)}) circle (1.5pt);
\draw ({0.5*(1 - sqrt(3)/2)},{sqrt(3)/2*(1 - sqrt(3)/2)}) node[anchor=north east]{$v_0$};
\end{tikzpicture}\\
$\Delta$-crossing & \hspace{2cm} $\nabla$-crossing
\end{tabular}
\caption{A $\Delta$-crossing and a $\nabla$-crossing. Note that in the $\Delta$-crossing, the triangle formed by the~$3$ intersecting branches points away from $v_0$, while in the $\nabla$-crossing, the triangle formed by the $3$ intersecting branches points towards $v_0$.} \label{delta_crossing}
\end{figure}
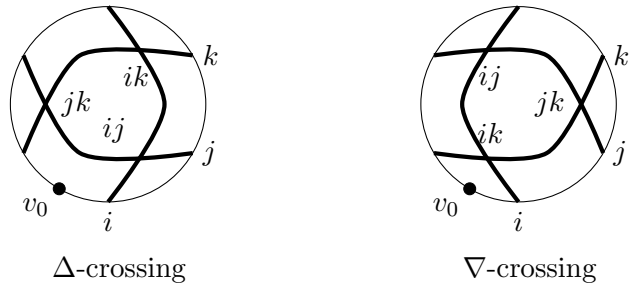
\end{Definition}

\begin{Definition} \label{min_max_tiling_def}Define $T_{\min ,n}$ to be the unique $\Diamond$-tiling of $\mathbf{P}_n$ in which every vertex is labeled by consecutive subsets $I \subseteq [n]$, and define $T_{\max ,n}$ to be the unique $\Diamond$-tiling of $\mathbf{P}_n$ in which every vertex is labeled by a subset $I \subseteq [n]$ whose complement $[n] - I$ is consecutive (see Fig.~\ref{min_max_tiling}). Equivalently, $T_{\min ,n}$ is the $\Diamond$-tiling of $\mathbf{P}_n$ in which every triple $i<j<k$ has a $\Delta$-crossing in its associated pseudoline arrangement, while $T_{\max ,n}$ is the $\Diamond$-tiling of $\mathbf{P}_n$ in which every triple $i<j<k$ has a $\nabla$-crossing in its associated pseudoline arrangement.
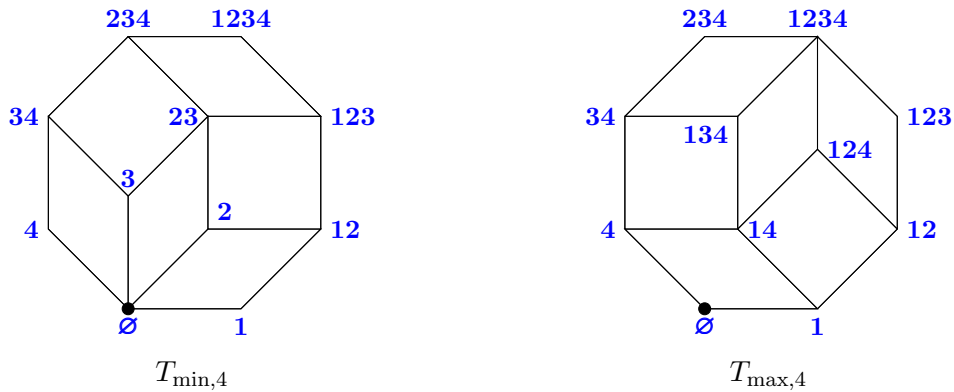
\begin{figure}\centering
\begin{tabular}{ c c }
\begin{tikzpicture}[scale=1.5]
\draw (0,0)--(1,0)--({1+sqrt(2)/2},{sqrt(2)/2})--({1+sqrt(2)/2},{1+sqrt(2)/2})--({1},{1+sqrt(2)})--({0},{1+sqrt(2)})--({-sqrt(2)/2},{1+sqrt(2)/2})--({-sqrt(2)/2},{sqrt(2)/2})--(0,0);

\draw (0,0)--(1,0)--({1+sqrt(2)/2},{sqrt(2)/2})--({sqrt(2)/2},{sqrt(2)/2})--(0,0);
\draw (0,0)--({sqrt(2)/2},{sqrt(2)/2})--({sqrt(2)/2},{1+sqrt(2)/2})--(0,1)--(0,0);
\draw (0,0)--(0,1)--({-sqrt(2)/2},{1 + sqrt(2)/2})--({-sqrt(2)/2},{sqrt(2)/2})--(0,0);
\draw[shift={({sqrt(2)/2},{sqrt(2)/2})}] (0,0)--(1,0)--(1,1)--(0,1)--(0,0);
\draw[shift={(0,1)}] (0,0)--({sqrt(2)/2},{sqrt(2)/2})--(0,{sqrt(2)})--({-sqrt(2)/2},{sqrt(2)/2})--(0,0);
\draw[shift={({sqrt(2)/2},{1+sqrt(2)/2})}] (0,0)--(1,0)--({1 - sqrt(2)/2},{sqrt(2)/2})--({- sqrt(2)/2},{sqrt(2)/2})--(0,0);

\filldraw (0,0) circle (1.5pt);
{
\color{blue}
\small
\draw (0,0) node[anchor= north]{$\bm{\varnothing}$};
\draw (1,0) node[anchor=north]{$\mathbf{1}$};
\draw ({1+sqrt(2)/2},{sqrt(2)/2}) node[anchor=west]{$\mathbf{12}$};
\draw ({1+sqrt(2)/2},{1+sqrt(2)/2}) node[anchor=west]{$\mathbf{123}$};
\draw ({1},{1+sqrt(2)}) node[anchor=south]{$\mathbf{1234}$};
\draw ({0},{1+sqrt(2)}) node[anchor=south]{$\mathbf{234}$};
\draw ({-sqrt(2)/2},{1+sqrt(2)/2}) node[anchor=east]{$\mathbf{34}$};
\draw ({-sqrt(2)/2},{sqrt(2)/2}) node[anchor=east]{$\mathbf{4}$};

\draw ({sqrt(2)/2},{sqrt(2)/2}) node[anchor=south west]{$\mathbf{2}$};
\draw (0,{1}) node[anchor=south]{$\mathbf{3}$};
\draw ({sqrt(2)/2},{1+sqrt(2)/2}) node[anchor=east]{$\mathbf{23}$};
}
\end{tikzpicture} & \hspace{2 cm}
\begin{tikzpicture}[scale=1.5]
\draw (0,0)--(1,0)--({1+sqrt(2)/2},{sqrt(2)/2})--({1+sqrt(2)/2},{1+sqrt(2)/2})--({1},{1+sqrt(2)})--({0},{1+sqrt(2)})--({-sqrt(2)/2},{1+sqrt(2)/2})--({-sqrt(2)/2},{sqrt(2)/2})--(0,0);

\draw[shift={({-sqrt(2)/2},{1+sqrt(2)/2})}] (0,0)--(1,0)--({1+sqrt(2)/2},{sqrt(2)/2})--({sqrt(2)/2},{sqrt(2)/2})--(0,0);
\draw[shift={({1 - sqrt(2)/2},{sqrt(2)/2})}] (0,0)--({sqrt(2)/2},{sqrt(2)/2})--({sqrt(2)/2},{1+sqrt(2)/2})--(0,1)--(0,0);
\draw[shift={({1+sqrt(2)/2},{sqrt(2)/2})}] (0,0)--(0,1)--({-sqrt(2)/2},{1 + sqrt(2)/2})--({-sqrt(2)/2},{sqrt(2)/2})--(0,0);
\draw[shift={({-sqrt(2)/2},{sqrt(2)/2})}] (0,0)--(1,0)--(1,1)--(0,1)--(0,0);
\draw[shift={(1,0)}] (0,0)--({sqrt(2)/2},{sqrt(2)/2})--(0,{sqrt(2)})--({-sqrt(2)/2},{sqrt(2)/2})--(0,0);
\draw (0,0)--(1,0)--({1 - sqrt(2)/2},{sqrt(2)/2})--({- sqrt(2)/2},{sqrt(2)/2})--(0,0);

\filldraw (0,0) circle (1.5pt);
{
\color{blue}
\small
\draw (0,0) node[anchor= north]{$\bm{\varnothing}$};
\draw (1,0) node[anchor=north]{$\mathbf{1}$};
\draw ({1+sqrt(2)/2},{sqrt(2)/2}) node[anchor=west]{$\mathbf{12}$};
\draw ({1+sqrt(2)/2},{1+sqrt(2)/2}) node[anchor=west]{$\mathbf{123}$};
\draw ({1},{1+sqrt(2)}) node[anchor=south]{$\mathbf{1234}$};
\draw ({0},{1+sqrt(2)}) node[anchor=south]{$\mathbf{234}$};
\draw ({-sqrt(2)/2},{1+sqrt(2)/2}) node[anchor=east]{$\mathbf{34}$};
\draw ({-sqrt(2)/2},{sqrt(2)/2}) node[anchor=east]{$\mathbf{4}$};

\draw (1,{sqrt(2)}) node[anchor=west]{$\mathbf{124}$};
\draw ({1 - sqrt(2)/2},{sqrt(2)/2}) node[anchor=west]{$\mathbf{14}$};
\draw ({1-sqrt(2)/2},{1+sqrt(2)/2}) node[anchor=north east]{$\mathbf{134}$};
}
\end{tikzpicture}\\
$T_{\min ,4}$ & \hspace{2cm} $T_{\max ,4}$
\end{tabular}
\caption{The $\Diamond$-tilings $T_{\min ,4}$ and $T_{\max ,4}$.} \label{min_max_tiling}
\end{figure}
\end{Definition}

\begin{Definition}\label{pile_def}We say that $\mathbf{T} = (T_0, \dots, T_\ell)$ is a \emph{pile} of quadrangulations of a polygon if~$T_{i - 1}$ and~$T_i$ are connected by a flip for $i = 1, \dots, \ell$.
\end{Definition}

\begin{Example} \label{octagon_seq_example}Consider the pile $\mathbf{T} = (T_0, \dots, T_8)$ of $\Diamond$-tilings of $\mathbf{P}_4$ shown in Fig.~\ref{oct_cycle_figure}. Note that $T_0 = T_8$. For each tiling $T_i$, there are two possible flips that we can perform; applying one gives us $T_{i - 1}$, and applying the other gives us $T_{i + 1}$ (with indices taken mod $8$).
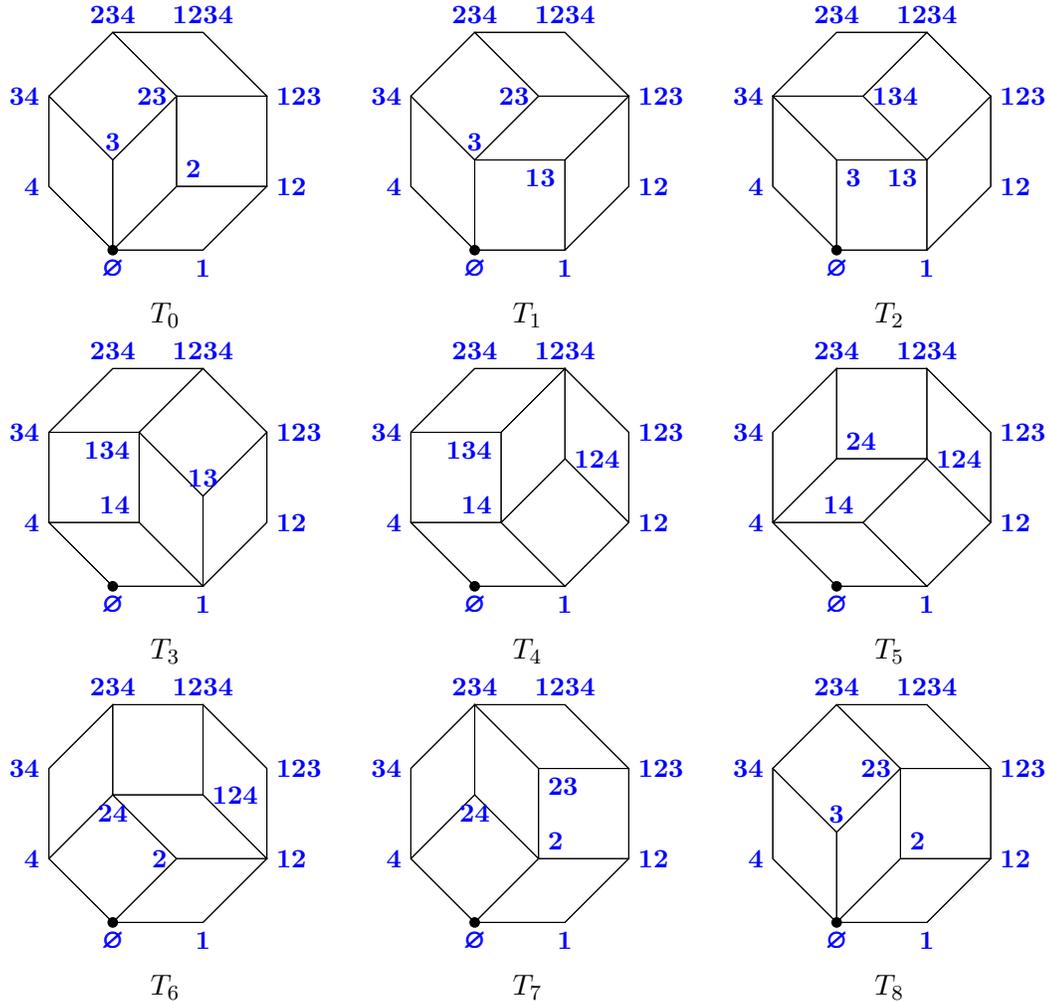
\begin{figure}[ht]\centering
\begin{tabular}{ c c c }
\begin{tikzpicture}[scale=1.2]
\draw (0,0)--(1,0)--({1+sqrt(2)/2},{sqrt(2)/2})--({1+sqrt(2)/2},{1+sqrt(2)/2})--({1},{1+sqrt(2)})--({0},{1+sqrt(2)})--({-sqrt(2)/2},{1+sqrt(2)/2})--({-sqrt(2)/2},{sqrt(2)/2})--(0,0);

\draw (0,0)--(1,0)--({1+sqrt(2)/2},{sqrt(2)/2})--({sqrt(2)/2},{sqrt(2)/2})--(0,0);
\draw (0,0)--({sqrt(2)/2},{sqrt(2)/2})--({sqrt(2)/2},{1+sqrt(2)/2})--(0,1)--(0,0);
\draw (0,0)--(0,1)--({-sqrt(2)/2},{1 + sqrt(2)/2})--({-sqrt(2)/2},{sqrt(2)/2})--(0,0);
\draw[shift={({sqrt(2)/2},{sqrt(2)/2})}] (0,0)--(1,0)--(1,1)--(0,1)--(0,0);
\draw[shift={(0,1)}] (0,0)--({sqrt(2)/2},{sqrt(2)/2})--(0,{sqrt(2)})--({-sqrt(2)/2},{sqrt(2)/2})--(0,0);
\draw[shift={({sqrt(2)/2},{1+sqrt(2)/2})}] (0,0)--(1,0)--({1 - sqrt(2)/2},{sqrt(2)/2})--({- sqrt(2)/2},{sqrt(2)/2})--(0,0);

\filldraw (0,0) circle (1.5pt);
{
\color{blue}
\small
\draw (0,0) node[anchor= north]{$\bm{\varnothing}$};
\draw (1,0) node[anchor=north]{$\mathbf{1}$};
\draw ({1+sqrt(2)/2},{sqrt(2)/2}) node[anchor=west]{$\mathbf{12}$};
\draw ({1+sqrt(2)/2},{1+sqrt(2)/2}) node[anchor=west]{$\mathbf{123}$};
\draw ({1},{1+sqrt(2)}) node[anchor=south]{$\mathbf{1234}$};
\draw ({0},{1+sqrt(2)}) node[anchor=south]{$\mathbf{234}$};
\draw ({-sqrt(2)/2},{1+sqrt(2)/2}) node[anchor=east]{$\mathbf{34}$};
\draw ({-sqrt(2)/2},{sqrt(2)/2}) node[anchor=east]{$\mathbf{4}$};

\draw ({sqrt(2)/2},{sqrt(2)/2}) node[anchor=south west]{$\mathbf{2}$};
\draw (0,{1}) node[anchor=south]{$\mathbf{3}$};
\draw ({sqrt(2)/2},{1+sqrt(2)/2}) node[anchor=east]{$\mathbf{23}$};
}
\end{tikzpicture} &
\begin{tikzpicture}[scale=1.2]
\draw (0,0)--(1,0)--({1+sqrt(2)/2},{sqrt(2)/2})--({1+sqrt(2)/2},{1+sqrt(2)/2})--({1},{1+sqrt(2)})--({0},{1+sqrt(2)})--({-sqrt(2)/2},{1+sqrt(2)/2})--({-sqrt(2)/2},{sqrt(2)/2})--(0,0);

\draw[shift={(0,1)}] (0,0)--(1,0)--({1+sqrt(2)/2},{sqrt(2)/2})--({sqrt(2)/2},{sqrt(2)/2})--(0,0);
\draw[shift={(1,0)}] (0,0)--({sqrt(2)/2},{sqrt(2)/2})--({sqrt(2)/2},{1+sqrt(2)/2})--(0,1)--(0,0);
\draw (0,0)--(0,1)--({-sqrt(2)/2},{1 + sqrt(2)/2})--({-sqrt(2)/2},{sqrt(2)/2})--(0,0);
\draw (0,0)--(1,0)--(1,1)--(0,1)--(0,0);
\draw[shift={(0,1)}] (0,0)--({sqrt(2)/2},{sqrt(2)/2})--(0,{sqrt(2)})--({-sqrt(2)/2},{sqrt(2)/2})--(0,0);
\draw[shift={({sqrt(2)/2},{1+sqrt(2)/2})}] (0,0)--(1,0)--({1 - sqrt(2)/2},{sqrt(2)/2})--({- sqrt(2)/2},{sqrt(2)/2})--(0,0);

\filldraw (0,0) circle (1.5pt);
{
\color{blue}
\small
\draw (0,0) node[anchor= north]{$\bm{\varnothing}$};
\draw (1,0) node[anchor=north]{$\mathbf{1}$};
\draw ({1+sqrt(2)/2},{sqrt(2)/2}) node[anchor=west]{$\mathbf{12}$};
\draw ({1+sqrt(2)/2},{1+sqrt(2)/2}) node[anchor=west]{$\mathbf{123}$};
\draw ({1},{1+sqrt(2)}) node[anchor=south]{$\mathbf{1234}$};
\draw ({0},{1+sqrt(2)}) node[anchor=south]{$\mathbf{234}$};
\draw ({-sqrt(2)/2},{1+sqrt(2)/2}) node[anchor=east]{$\mathbf{34}$};
\draw ({-sqrt(2)/2},{sqrt(2)/2}) node[anchor=east]{$\mathbf{4}$};

\draw (1,1) node[anchor=north east]{$\mathbf{13}$};
\draw (0,{1}) node[anchor=south]{$\mathbf{3}$};
\draw ({sqrt(2)/2},{1+sqrt(2)/2}) node[anchor=east]{$\mathbf{23}$};
}
\end{tikzpicture} &
\begin{tikzpicture}[scale=1.2]
\draw (0,0)--(1,0)--({1+sqrt(2)/2},{sqrt(2)/2})--({1+sqrt(2)/2},{1+sqrt(2)/2})--({1},{1+sqrt(2)})--({0},{1+sqrt(2)})--({-sqrt(2)/2},{1+sqrt(2)/2})--({-sqrt(2)/2},{sqrt(2)/2})--(0,0);

\draw[shift={({-sqrt(2)/2},{1+sqrt(2)/2})}] (0,0)--(1,0)--({1+sqrt(2)/2},{sqrt(2)/2})--({sqrt(2)/2},{sqrt(2)/2})--(0,0);
\draw[shift={(1,0)}] (0,0)--({sqrt(2)/2},{sqrt(2)/2})--({sqrt(2)/2},{1+sqrt(2)/2})--(0,1)--(0,0);
\draw (0,0)--(0,1)--({-sqrt(2)/2},{1 + sqrt(2)/2})--({-sqrt(2)/2},{sqrt(2)/2})--(0,0);
\draw (0,0)--(1,0)--(1,1)--(0,1)--(0,0);
\draw[shift={(1,1)}] (0,0)--({sqrt(2)/2},{sqrt(2)/2})--(0,{sqrt(2)})--({-sqrt(2)/2},{sqrt(2)/2})--(0,0);
\draw[shift={(0,1)}] (0,0)--(1,0)--({1 - sqrt(2)/2},{sqrt(2)/2})--({- sqrt(2)/2},{sqrt(2)/2})--(0,0);

\filldraw (0,0) circle (1.5pt);
{
\color{blue}
\small
\draw (0,0) node[anchor= north]{$\bm{\varnothing}$};
\draw (1,0) node[anchor=north]{$\mathbf{1}$};
\draw ({1+sqrt(2)/2},{sqrt(2)/2}) node[anchor=west]{$\mathbf{12}$};
\draw ({1+sqrt(2)/2},{1+sqrt(2)/2}) node[anchor=west]{$\mathbf{123}$};
\draw ({1},{1+sqrt(2)}) node[anchor=south]{$\mathbf{1234}$};
\draw ({0},{1+sqrt(2)}) node[anchor=south]{$\mathbf{234}$};
\draw ({-sqrt(2)/2},{1+sqrt(2)/2}) node[anchor=east]{$\mathbf{34}$};
\draw ({-sqrt(2)/2},{sqrt(2)/2}) node[anchor=east]{$\mathbf{4}$};

\draw (1,1) node[anchor=north east]{$\mathbf{13}$};
\draw (0,{1}) node[anchor=north west]{$\mathbf{3}$};
\draw ({1-sqrt(2)/2},{1+sqrt(2)/2}) node[anchor=west]{$\mathbf{134}$};
}
\end{tikzpicture} \\
$T_0$ & $T_1$ & $T_2$ \\
\begin{tikzpicture}[scale=1.2]
\draw (0,0)--(1,0)--({1+sqrt(2)/2},{sqrt(2)/2})--({1+sqrt(2)/2},{1+sqrt(2)/2})--({1},{1+sqrt(2)})--({0},{1+sqrt(2)})--({-sqrt(2)/2},{1+sqrt(2)/2})--({-sqrt(2)/2},{sqrt(2)/2})--(0,0);

\draw[shift={({-sqrt(2)/2},{1+sqrt(2)/2})}] (0,0)--(1,0)--({1+sqrt(2)/2},{sqrt(2)/2})--({sqrt(2)/2},{sqrt(2)/2})--(0,0);
\draw[shift={(1,0)}] (0,0)--({sqrt(2)/2},{sqrt(2)/2})--({sqrt(2)/2},{1+sqrt(2)/2})--(0,1)--(0,0);
\draw[shift={(1,0)}] (0,0)--(0,1)--({-sqrt(2)/2},{1 + sqrt(2)/2})--({-sqrt(2)/2},{sqrt(2)/2})--(0,0);
\draw[shift={({-sqrt(2)/2},{sqrt(2)/2})}] (0,0)--(1,0)--(1,1)--(0,1)--(0,0);
\draw[shift={(1,1)}] (0,0)--({sqrt(2)/2},{sqrt(2)/2})--(0,{sqrt(2)})--({-sqrt(2)/2},{sqrt(2)/2})--(0,0);
\draw (0,0)--(1,0)--({1 - sqrt(2)/2},{sqrt(2)/2})--({- sqrt(2)/2},{sqrt(2)/2})--(0,0);

\filldraw (0,0) circle (1.5pt);
{
\color{blue}
\small
\draw (0,0) node[anchor= north]{$\bm{\varnothing}$};
\draw (1,0) node[anchor=north]{$\mathbf{1}$};
\draw ({1+sqrt(2)/2},{sqrt(2)/2}) node[anchor=west]{$\mathbf{12}$};
\draw ({1+sqrt(2)/2},{1+sqrt(2)/2}) node[anchor=west]{$\mathbf{123}$};
\draw ({1},{1+sqrt(2)}) node[anchor=south]{$\mathbf{1234}$};
\draw ({0},{1+sqrt(2)}) node[anchor=south]{$\mathbf{234}$};
\draw ({-sqrt(2)/2},{1+sqrt(2)/2}) node[anchor=east]{$\mathbf{34}$};
\draw ({-sqrt(2)/2},{sqrt(2)/2}) node[anchor=east]{$\mathbf{4}$};

\draw (1,1) node[anchor=south]{$\mathbf{13}$};
\draw ({1 - sqrt(2)/2},{sqrt(2)/2}) node[anchor=south east]{$\mathbf{14}$};
\draw ({1-sqrt(2)/2},{1+sqrt(2)/2}) node[anchor=north east]{$\mathbf{134}$};
}
\end{tikzpicture} &
\begin{tikzpicture}[scale=1.2]
\draw (0,0)--(1,0)--({1+sqrt(2)/2},{sqrt(2)/2})--({1+sqrt(2)/2},{1+sqrt(2)/2})--({1},{1+sqrt(2)})--({0},{1+sqrt(2)})--({-sqrt(2)/2},{1+sqrt(2)/2})--({-sqrt(2)/2},{sqrt(2)/2})--(0,0);

\draw[shift={({-sqrt(2)/2},{1+sqrt(2)/2})}] (0,0)--(1,0)--({1+sqrt(2)/2},{sqrt(2)/2})--({sqrt(2)/2},{sqrt(2)/2})--(0,0);
\draw[shift={({1 - sqrt(2)/2},{sqrt(2)/2})}] (0,0)--({sqrt(2)/2},{sqrt(2)/2})--({sqrt(2)/2},{1+sqrt(2)/2})--(0,1)--(0,0);
\draw[shift={({1+sqrt(2)/2},{sqrt(2)/2})}] (0,0)--(0,1)--({-sqrt(2)/2},{1 + sqrt(2)/2})--({-sqrt(2)/2},{sqrt(2)/2})--(0,0);
\draw[shift={({-sqrt(2)/2},{sqrt(2)/2})}] (0,0)--(1,0)--(1,1)--(0,1)--(0,0);
\draw[shift={(1,0)}] (0,0)--({sqrt(2)/2},{sqrt(2)/2})--(0,{sqrt(2)})--({-sqrt(2)/2},{sqrt(2)/2})--(0,0);
\draw (0,0)--(1,0)--({1 - sqrt(2)/2},{sqrt(2)/2})--({- sqrt(2)/2},{sqrt(2)/2})--(0,0);

\filldraw (0,0) circle (1.5pt);
{
\color{blue}
\small
\draw (0,0) node[anchor= north]{$\bm{\varnothing}$};
\draw (1,0) node[anchor=north]{$\mathbf{1}$};
\draw ({1+sqrt(2)/2},{sqrt(2)/2}) node[anchor=west]{$\mathbf{12}$};
\draw ({1+sqrt(2)/2},{1+sqrt(2)/2}) node[anchor=west]{$\mathbf{123}$};
\draw ({1},{1+sqrt(2)}) node[anchor=south]{$\mathbf{1234}$};
\draw ({0},{1+sqrt(2)}) node[anchor=south]{$\mathbf{234}$};
\draw ({-sqrt(2)/2},{1+sqrt(2)/2}) node[anchor=east]{$\mathbf{34}$};
\draw ({-sqrt(2)/2},{sqrt(2)/2}) node[anchor=east]{$\mathbf{4}$};

\draw (1,{sqrt(2)}) node[anchor=west]{$\mathbf{124}$};
\draw ({1 - sqrt(2)/2},{sqrt(2)/2}) node[anchor=south east]{$\mathbf{14}$};
\draw ({1-sqrt(2)/2},{1+sqrt(2)/2}) node[anchor=north east]{$\mathbf{134}$};
}
\end{tikzpicture} &
\begin{tikzpicture}[scale=1.2]
\draw (0,0)--(1,0)--({1+sqrt(2)/2},{sqrt(2)/2})--({1+sqrt(2)/2},{1+sqrt(2)/2})--({1},{1+sqrt(2)})--({0},{1+sqrt(2)})--({-sqrt(2)/2},{1+sqrt(2)/2})--({-sqrt(2)/2},{sqrt(2)/2})--(0,0);

\draw[shift={({-sqrt(2)/2},{sqrt(2)/2})}] (0,0)--(1,0)--({1+sqrt(2)/2},{sqrt(2)/2})--({sqrt(2)/2},{sqrt(2)/2})--(0,0);
\draw[shift={({- sqrt(2)/2},{sqrt(2)/2})}] (0,0)--({sqrt(2)/2},{sqrt(2)/2})--({sqrt(2)/2},{1+sqrt(2)/2})--(0,1)--(0,0);
\draw[shift={({1+sqrt(2)/2},{sqrt(2)/2})}] (0,0)--(0,1)--({-sqrt(2)/2},{1 + sqrt(2)/2})--({-sqrt(2)/2},{sqrt(2)/2})--(0,0);
\draw[shift={({0},{sqrt(2)})}] (0,0)--(1,0)--(1,1)--(0,1)--(0,0);
\draw[shift={(1,0)}] (0,0)--({sqrt(2)/2},{sqrt(2)/2})--(0,{sqrt(2)})--({-sqrt(2)/2},{sqrt(2)/2})--(0,0);
\draw (0,0)--(1,0)--({1 - sqrt(2)/2},{sqrt(2)/2})--({- sqrt(2)/2},{sqrt(2)/2})--(0,0);

\filldraw (0,0) circle (1.5pt);
{
\color{blue}
\small
\draw (0,0) node[anchor= north]{$\bm{\varnothing}$};
\draw (1,0) node[anchor=north]{$\mathbf{1}$};
\draw ({1+sqrt(2)/2},{sqrt(2)/2}) node[anchor=west]{$\mathbf{12}$};
\draw ({1+sqrt(2)/2},{1+sqrt(2)/2}) node[anchor=west]{$\mathbf{123}$};
\draw ({1},{1+sqrt(2)}) node[anchor=south]{$\mathbf{1234}$};
\draw ({0},{1+sqrt(2)}) node[anchor=south]{$\mathbf{234}$};
\draw ({-sqrt(2)/2},{1+sqrt(2)/2}) node[anchor=east]{$\mathbf{34}$};
\draw ({-sqrt(2)/2},{sqrt(2)/2}) node[anchor=east]{$\mathbf{4}$};

\draw (1,{sqrt(2)}) node[anchor=west]{$\mathbf{124}$};
\draw ({1 - sqrt(2)/2},{sqrt(2)/2}) node[anchor=south east]{$\mathbf{14}$};
\draw ({0},{sqrt(2)}) node[anchor=south west]{$\mathbf{24}$};
}
\end{tikzpicture} \\
$T_3$ & $T_4$ & $T_5$\\
\begin{tikzpicture}[scale=1.2]
\draw (0,0)--(1,0)--({1+sqrt(2)/2},{sqrt(2)/2})--({1+sqrt(2)/2},{1+sqrt(2)/2})--({1},{1+sqrt(2)})--({0},{1+sqrt(2)})--({-sqrt(2)/2},{1+sqrt(2)/2})--({-sqrt(2)/2},{sqrt(2)/2})--(0,0);

\draw (0,0)--(1,0)--({1+sqrt(2)/2},{sqrt(2)/2})--({sqrt(2)/2},{sqrt(2)/2})--(0,0);
\draw[shift={({- sqrt(2)/2},{sqrt(2)/2})}] (0,0)--({sqrt(2)/2},{sqrt(2)/2})--({sqrt(2)/2},{1+sqrt(2)/2})--(0,1)--(0,0);
\draw[shift={({1+sqrt(2)/2},{sqrt(2)/2})}] (0,0)--(0,1)--({-sqrt(2)/2},{1 + sqrt(2)/2})--({-sqrt(2)/2},{sqrt(2)/2})--(0,0);
\draw[shift={({0},{sqrt(2)})}] (0,0)--(1,0)--(1,1)--(0,1)--(0,0);
\draw (0,0)--({sqrt(2)/2},{sqrt(2)/2})--(0,{sqrt(2)})--({-sqrt(2)/2},{sqrt(2)/2})--(0,0);
\draw[shift={({sqrt(2)/2},{sqrt(2)/2})}] (0,0)--(1,0)--({1 - sqrt(2)/2},{sqrt(2)/2})--({- sqrt(2)/2},{sqrt(2)/2})--(0,0);

\filldraw (0,0) circle (1.5pt);
{
\color{blue}
\small
\draw (0,0) node[anchor= north]{$\bm{\varnothing}$};
\draw (1,0) node[anchor=north]{$\mathbf{1}$};
\draw ({1+sqrt(2)/2},{sqrt(2)/2}) node[anchor=west]{$\mathbf{12}$};
\draw ({1+sqrt(2)/2},{1+sqrt(2)/2}) node[anchor=west]{$\mathbf{123}$};
\draw ({1},{1+sqrt(2)}) node[anchor=south]{$\mathbf{1234}$};
\draw ({0},{1+sqrt(2)}) node[anchor=south]{$\mathbf{234}$};
\draw ({-sqrt(2)/2},{1+sqrt(2)/2}) node[anchor=east]{$\mathbf{34}$};
\draw ({-sqrt(2)/2},{sqrt(2)/2}) node[anchor=east]{$\mathbf{4}$};

\draw (1,{sqrt(2)}) node[anchor=west]{$\mathbf{124}$};
\draw ({sqrt(2)/2},{sqrt(2)/2}) node[anchor=east]{$\mathbf{2}$};
\draw ({0},{sqrt(2)}) node[anchor=north]{$\mathbf{24}$};
}
\end{tikzpicture} &
\begin{tikzpicture}[scale=1.2]
\draw (0,0)--(1,0)--({1+sqrt(2)/2},{sqrt(2)/2})--({1+sqrt(2)/2},{1+sqrt(2)/2})--({1},{1+sqrt(2)})--({0},{1+sqrt(2)})--({-sqrt(2)/2},{1+sqrt(2)/2})--({-sqrt(2)/2},{sqrt(2)/2})--(0,0);

\draw (0,0)--(1,0)--({1+sqrt(2)/2},{sqrt(2)/2})--({sqrt(2)/2},{sqrt(2)/2})--(0,0);
\draw[shift={({- sqrt(2)/2},{sqrt(2)/2})}] (0,0)--({sqrt(2)/2},{sqrt(2)/2})--({sqrt(2)/2},{1+sqrt(2)/2})--(0,1)--(0,0);
\draw[shift={({sqrt(2)/2},{sqrt(2)/2})}] (0,0)--(0,1)--({-sqrt(2)/2},{1 + sqrt(2)/2})--({-sqrt(2)/2},{sqrt(2)/2})--(0,0);
\draw[shift={({sqrt(2)/2},{sqrt(2)/2})}] (0,0)--(1,0)--(1,1)--(0,1)--(0,0);
\draw (0,0)--({sqrt(2)/2},{sqrt(2)/2})--(0,{sqrt(2)})--({-sqrt(2)/2},{sqrt(2)/2})--(0,0);
\draw[shift={({sqrt(2)/2},{1+sqrt(2)/2})}] (0,0)--(1,0)--({1 - sqrt(2)/2},{sqrt(2)/2})--({- sqrt(2)/2},{sqrt(2)/2})--(0,0);

\filldraw (0,0) circle (1.5pt);
{
\color{blue}
\small
\draw (0,0) node[anchor= north]{$\bm{\varnothing}$};
\draw (1,0) node[anchor=north]{$\mathbf{1}$};
\draw ({1+sqrt(2)/2},{sqrt(2)/2}) node[anchor=west]{$\mathbf{12}$};
\draw ({1+sqrt(2)/2},{1+sqrt(2)/2}) node[anchor=west]{$\mathbf{123}$};
\draw ({1},{1+sqrt(2)}) node[anchor=south]{$\mathbf{1234}$};
\draw ({0},{1+sqrt(2)}) node[anchor=south]{$\mathbf{234}$};
\draw ({-sqrt(2)/2},{1+sqrt(2)/2}) node[anchor=east]{$\mathbf{34}$};
\draw ({-sqrt(2)/2},{sqrt(2)/2}) node[anchor=east]{$\mathbf{4}$};

\draw ({sqrt(2)/2},{1+sqrt(2)/2}) node[anchor=north west]{$\mathbf{23}$};
\draw ({sqrt(2)/2},{sqrt(2)/2}) node[anchor=south west]{$\mathbf{2}$};
\draw ({0},{sqrt(2)}) node[anchor=north]{$\mathbf{24}$};
}
\end{tikzpicture} &
\begin{tikzpicture}[scale=1.2]
\draw (0,0)--(1,0)--({1+sqrt(2)/2},{sqrt(2)/2})--({1+sqrt(2)/2},{1+sqrt(2)/2})--({1},{1+sqrt(2)})--({0},{1+sqrt(2)})--({-sqrt(2)/2},{1+sqrt(2)/2})--({-sqrt(2)/2},{sqrt(2)/2})--(0,0);

\draw (0,0)--(1,0)--({1+sqrt(2)/2},{sqrt(2)/2})--({sqrt(2)/2},{sqrt(2)/2})--(0,0);
\draw (0,0)--({sqrt(2)/2},{sqrt(2)/2})--({sqrt(2)/2},{1+sqrt(2)/2})--(0,1)--(0,0);
\draw (0,0)--(0,1)--({-sqrt(2)/2},{1 + sqrt(2)/2})--({-sqrt(2)/2},{sqrt(2)/2})--(0,0);
\draw[shift={({sqrt(2)/2},{sqrt(2)/2})}] (0,0)--(1,0)--(1,1)--(0,1)--(0,0);
\draw[shift={(0,1)}] (0,0)--({sqrt(2)/2},{sqrt(2)/2})--(0,{sqrt(2)})--({-sqrt(2)/2},{sqrt(2)/2})--(0,0);
\draw[shift={({sqrt(2)/2},{1+sqrt(2)/2})}] (0,0)--(1,0)--({1 - sqrt(2)/2},{sqrt(2)/2})--({- sqrt(2)/2},{sqrt(2)/2})--(0,0);

\filldraw (0,0) circle (1.5pt);
{
\color{blue}
\small
\draw (0,0) node[anchor= north]{$\bm{\varnothing}$};
\draw (1,0) node[anchor=north]{$\mathbf{1}$};
\draw ({1+sqrt(2)/2},{sqrt(2)/2}) node[anchor=west]{$\mathbf{12}$};
\draw ({1+sqrt(2)/2},{1+sqrt(2)/2}) node[anchor=west]{$\mathbf{123}$};
\draw ({1},{1+sqrt(2)}) node[anchor=south]{$\mathbf{1234}$};
\draw ({0},{1+sqrt(2)}) node[anchor=south]{$\mathbf{234}$};
\draw ({-sqrt(2)/2},{1+sqrt(2)/2}) node[anchor=east]{$\mathbf{34}$};
\draw ({-sqrt(2)/2},{sqrt(2)/2}) node[anchor=east]{$\mathbf{4}$};

\draw ({sqrt(2)/2},{sqrt(2)/2}) node[anchor=south west]{$\mathbf{2}$};
\draw (0,{1}) node[anchor=south]{$\mathbf{3}$};
\draw ({sqrt(2)/2},{1+sqrt(2)/2}) node[anchor=east]{$\mathbf{23}$};
}
\end{tikzpicture}\\
$T_6$ & $T_7$ & $T_8$
\end{tabular}
\caption{The pile of $\Diamond$-tilings of $\mathbf{P}_4$ from Example~\ref{octagon_seq_example}.} \label{oct_cycle_figure}
\end{figure}
\end{Example}

\begin{Definition} \label{cn_seq_def}Define $\mathcal{C}(n)$ to be the set of all piles $\mathbf{T} = (T_0, \dots, T_{\binom{n}{3}})$ with $T_0 = T_{\min ,n}$ and $T_{\binom{n}{3}} = T_{\max ,n}$. Note that the shortest length of a pile starting with $T_{\min ,n}$ and ending with $T_{\max ,n}$ is $\binom{n}{3}$, corresponding to switching from $\Delta$-crossings to $\nabla$-crossings for each of the $\binom{n}{3}$ triples. Every $\Diamond$-tiling $T$ of $\mathbf{P}_n$ appears in at least one pile in $\mathcal{C}(n)$.
\end{Definition}

\begin{Remark}One can put a poset structure on the set of $\Diamond$-tilings of $\mathbf{P}_n$, called the \emph{second higher Bruhat order} $B(n, 2)$ \cite{manin_sch}, as follows: given $\Diamond$-tilings $T_1$, $T_2$ with associated pseudoline arrangements $D_1$, $D_2$, we say that $T_1 \le T_2$ if $i$, $j$, $k$ having a $\Delta$-crossing in $T_2$ implies that $i$,~$j$,~$k$ have a~$\Delta$-crossing in $T_1$ for all $i < j < k$. Note that $T_{\min ,n}$ is the minimum element of $B(n, 2)$, and $T_{\max ,n}$ is the maximum element of $B(n, 2)$.
\end{Remark}

\begin{Example} \label{two_piles_c4}The set $\mathcal{C}(4)$ consists of two piles, namely $(T_0, T_1, T_2, T_3, T_4)$ and $(T_0, T_7, \allowbreak T_6, T_5, T_4)$, where the $T_i$ are the $\Diamond$-tilings of $\mathbf{P}_4$ from Fig.~\ref{oct_cycle_figure}.
\end{Example}

\begin{Definition}\label{dcc_def}A \emph{directed cubical complex} is a $d$-dimensional cubical complex along with a~choice of ``top'' vertex in each $d$-dimensional cube.
\end{Definition}

\begin{Definition}\label{assoc_cubical_complex_def}Given a pile $\mathbf{T} = (T_0, \dots, T_\ell)$ of quadrangulations of a polygon~$R$, we define an associated $3$-dimensional directed cubical complex $\varkappa = \varkappa(\mathbf{T})$ as follows. Start with the $2$-dimensional cubical complex corresponding to $T_0$. For $i = 1, \dots, \ell$, given $T_{i - 1}$ and $T_i$ labeled as in Fig.~\ref{cube_added_from_flip}, add a new vertex $v$ to the complex corresponding to the new vertex in $T_i$, and add the $3$-dimensional cube labeled as in Fig.~\ref{cube_added_from_flip}, with $v$ as its top vertex. Note that a flip cannot be centered at a vertex on the boundary of $R$, so each vertex on the boundary of $R$ corresponds to a single vertex in $\varkappa$. In the special case where $R = \mathbf{P}_n$ and the quadrangulations are $\Diamond$-tilings, we denote the vertex of $\varkappa$ corresponding to~$v_0$ as~$v_0$, by an abuse of notation.
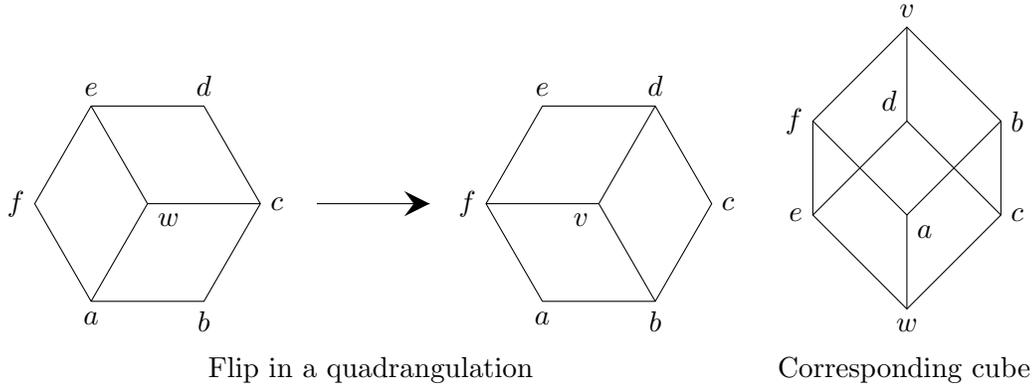
\begin{figure}[ht]\centering
\begin{tabular}{ c c }
\begin{tikzpicture}[scale=1.5]
\draw (0,0)--(1,0)--({1 + 1/2},{sqrt(3)/2})--(1,{sqrt(3)})--(0,{sqrt(3)})--({- 1/2},{sqrt(3)/2})--(0,0);
\draw ({1/2},{sqrt(3)/2})--(0,0);
\draw ({1/2},{sqrt(3)/2})--({1 + 1/2},{sqrt(3)/2});
\draw ({1/2},{sqrt(3)/2})--(0,{sqrt(3)});

{
\draw (0,0) node[anchor=north]{$a$};
\draw (1,0) node[anchor=north]{$b$};
\draw ({1 + 1/2},{sqrt(3)/2}) node[anchor=west]{$c$};
\draw (1,{sqrt(3)}) node[anchor=south]{$d$};
\draw (0,{sqrt(3)}) node[anchor=south]{$e$};
\draw ({- 1/2},{sqrt(3)/2}) node[anchor=east]{$f$};
\draw ({1/2},{sqrt(3)/2}) node[anchor=north west]{$w$};
}

%\draw[-|>,ultra thick] (1.75,{sqrt(3)/2})--(3.25,{sqrt(3)/2});
\draw [decoration={markings,mark=at position 1 with
 {\arrow[scale=3,>=stealth]{>}}},postaction={decorate}] (2,{sqrt(3)/2})--(3,{sqrt(3)/2});

\draw[shift={(4,0)}] (0,0)--(1,0)--({1 + 1/2},{sqrt(3)/2})--(1,{sqrt(3)})--(0,{sqrt(3)})--({- 1/2},{sqrt(3)/2})--(0,0);
\draw[shift={(4,0)}] ({1/2},{sqrt(3)/2})--(1,0);
\draw[shift={(4,0)}] ({1/2},{sqrt(3)/2})--(1,{sqrt(3)});
\draw[shift={(4,0)}] ({1/2},{sqrt(3)/2})--({- 1/2},{sqrt(3)/2});

\draw[shift={(4,0)}] (0,0) node[anchor=north]{$a$};
\draw[shift={(4,0)}] (1,0) node[anchor=north]{$b$};
\draw[shift={(4,0)}] ({1 + 1/2},{sqrt(3)/2}) node[anchor=west]{$c$};
\draw[shift={(4,0)}] (1,{sqrt(3)}) node[anchor=south]{$d$};
\draw[shift={(4,0)}] (0,{sqrt(3)}) node[anchor=south]{$e$};
\draw[shift={(4,0)}] ({- 1/2},{sqrt(3)/2}) node[anchor=east]{$f$};
\draw[shift={(4,0)}] ({1/2},{sqrt(3)/2}) node[anchor=north east]{$v$};
\end{tikzpicture} &
\begin{tikzpicture}[scale=1.25]
\draw (0,0)--(1,1)--(0,2)--(-1,1)--(0,0);
\draw (0,1)--(1,2)--(0,3)--(-1,2)--(0,1);
\draw (0,0)--(0,1);
\draw (1,1)--(1,2);
\draw (0,2)--(0,3);
\draw (-1,1)--(-1,2);

\draw (0,0) node[anchor=north]{$w$};
\draw (0,3) node[anchor=south]{$v$};
\draw (1,2) node[anchor=west]{$b$};
\draw (1,1) node[anchor=west]{$c$};
\draw (-1,1) node[anchor=east]{$e$};
\draw (-1,2) node[anchor=east]{$f$};
\draw (0,1) node[anchor=north west]{$a$};
\draw (0,2) node[anchor = south east]{$d$};
\end{tikzpicture} \\
Flip in a quadrangulation & Corresponding cube
\end{tabular}
\caption{A flip in a quadrangulation on the left, and the corresponding cube added to the directed cubical complex.} \label{cube_added_from_flip}
\end{figure}
\end{Definition}

\begin{Definition}\label{0_02def}Given a cubical complex $\varkappa$, we denote by $\varkappa^i$ the set of $i$-dimensional faces of~$\varkappa$, and set $\varkappa^{02} = \varkappa^0 \cup \varkappa^2$. Similarly, for a pile $\mathbf{T}$ of a quadrangulations, we denote by $\varkappa^i(\mathbf{T})$ the set of $i$-dimensional faces of $\varkappa(\mathbf{T})$, and $\varkappa^{02}(\mathbf{T}) = \varkappa^0(\mathbf{T}) \cup \varkappa^2(\mathbf{T})$.
\end{Definition}

\begin{Definition}\label{three_dim_flip}Two $3$-dimensional directed cubical complexes $\varkappa_1$ and $\varkappa_2$ are related by a \emph{flip} if there exist piles $\mathbf{T}_1 = (T_{1,0}, \dots, T_{1, \ell})$ and $\mathbf{T}_2 = (T_{2,0}, \dots, T_{2, \ell})$ such that
\begin{itemize}\itemsep=0pt
\item $\varkappa_1 = \varkappa(\mathbf{T}_1)$ and $\varkappa_2 = \varkappa(\mathbf{T}_2)$,
\end{itemize}
and there exists $i$ such that
\begin{itemize}\itemsep=0pt
\item $T_{1,j} = T_{2,j}$ for $j = 1, \dots, i$ and $j = i + 4, \dots, \ell$, and
\item $T_{1, j}$ and $T_{2, j}$ for $j = i + 1, i + 2, i + 3$ are related by the local moves shown in Fig.~\ref{3flip_figure}.
\end{itemize}

\begin{figure}[ht]\centering
\begin{tabular}{ c c c c c c }
$\mathbf{T}_1$ &
\begin{tikzpicture}[scale=1]
\draw (0,0)--(1,0)--({1+sqrt(2)/2},{sqrt(2)/2})--({1+sqrt(2)/2},{1+sqrt(2)/2})--({1},{1+sqrt(2)})--({0},{1+sqrt(2)})--({-sqrt(2)/2},{1+sqrt(2)/2})--({-sqrt(2)/2},{sqrt(2)/2})--(0,0);

\draw (0,0)--(1,0)--({1+sqrt(2)/2},{sqrt(2)/2})--({sqrt(2)/2},{sqrt(2)/2})--(0,0);
\draw (0,0)--({sqrt(2)/2},{sqrt(2)/2})--({sqrt(2)/2},{1+sqrt(2)/2})--(0,1)--(0,0);
\draw (0,0)--(0,1)--({-sqrt(2)/2},{1 + sqrt(2)/2})--({-sqrt(2)/2},{sqrt(2)/2})--(0,0);
\draw[shift={({sqrt(2)/2},{sqrt(2)/2})}] (0,0)--(1,0)--(1,1)--(0,1)--(0,0);
\draw[shift={(0,1)}] (0,0)--({sqrt(2)/2},{sqrt(2)/2})--(0,{sqrt(2)})--({-sqrt(2)/2},{sqrt(2)/2})--(0,0);
\draw[shift={({sqrt(2)/2},{1+sqrt(2)/2})}] (0,0)--(1,0)--({1 - sqrt(2)/2},{sqrt(2)/2})--({- sqrt(2)/2},{sqrt(2)/2})--(0,0);
\end{tikzpicture} &
\begin{tikzpicture}[scale=1]
\draw (0,0)--(1,0)--({1+sqrt(2)/2},{sqrt(2)/2})--({1+sqrt(2)/2},{1+sqrt(2)/2})--({1},{1+sqrt(2)})--({0},{1+sqrt(2)})--({-sqrt(2)/2},{1+sqrt(2)/2})--({-sqrt(2)/2},{sqrt(2)/2})--(0,0);

\draw[shift={(0,1)}] (0,0)--(1,0)--({1+sqrt(2)/2},{sqrt(2)/2})--({sqrt(2)/2},{sqrt(2)/2})--(0,0);
\draw[shift={(1,0)}] (0,0)--({sqrt(2)/2},{sqrt(2)/2})--({sqrt(2)/2},{1+sqrt(2)/2})--(0,1)--(0,0);
\draw (0,0)--(0,1)--({-sqrt(2)/2},{1 + sqrt(2)/2})--({-sqrt(2)/2},{sqrt(2)/2})--(0,0);
\draw (0,0)--(1,0)--(1,1)--(0,1)--(0,0);
\draw[shift={(0,1)}] (0,0)--({sqrt(2)/2},{sqrt(2)/2})--(0,{sqrt(2)})--({-sqrt(2)/2},{sqrt(2)/2})--(0,0);
\draw[shift={({sqrt(2)/2},{1+sqrt(2)/2})}] (0,0)--(1,0)--({1 - sqrt(2)/2},{sqrt(2)/2})--({- sqrt(2)/2},{sqrt(2)/2})--(0,0);
\end{tikzpicture} &
\begin{tikzpicture}[scale=1]
\draw (0,0)--(1,0)--({1+sqrt(2)/2},{sqrt(2)/2})--({1+sqrt(2)/2},{1+sqrt(2)/2})--({1},{1+sqrt(2)})--({0},{1+sqrt(2)})--({-sqrt(2)/2},{1+sqrt(2)/2})--({-sqrt(2)/2},{sqrt(2)/2})--(0,0);

\draw[shift={({-sqrt(2)/2},{1+sqrt(2)/2})}] (0,0)--(1,0)--({1+sqrt(2)/2},{sqrt(2)/2})--({sqrt(2)/2},{sqrt(2)/2})--(0,0);
\draw[shift={(1,0)}] (0,0)--({sqrt(2)/2},{sqrt(2)/2})--({sqrt(2)/2},{1+sqrt(2)/2})--(0,1)--(0,0);
\draw (0,0)--(0,1)--({-sqrt(2)/2},{1 + sqrt(2)/2})--({-sqrt(2)/2},{sqrt(2)/2})--(0,0);
\draw (0,0)--(1,0)--(1,1)--(0,1)--(0,0);
\draw[shift={(1,1)}] (0,0)--({sqrt(2)/2},{sqrt(2)/2})--(0,{sqrt(2)})--({-sqrt(2)/2},{sqrt(2)/2})--(0,0);
\draw[shift={(0,1)}] (0,0)--(1,0)--({1 - sqrt(2)/2},{sqrt(2)/2})--({- sqrt(2)/2},{sqrt(2)/2})--(0,0);
\end{tikzpicture} &
\begin{tikzpicture}[scale=1]
\draw (0,0)--(1,0)--({1+sqrt(2)/2},{sqrt(2)/2})--({1+sqrt(2)/2},{1+sqrt(2)/2})--({1},{1+sqrt(2)})--({0},{1+sqrt(2)})--({-sqrt(2)/2},{1+sqrt(2)/2})--({-sqrt(2)/2},{sqrt(2)/2})--(0,0);

\draw[shift={({-sqrt(2)/2},{1+sqrt(2)/2})}] (0,0)--(1,0)--({1+sqrt(2)/2},{sqrt(2)/2})--({sqrt(2)/2},{sqrt(2)/2})--(0,0);
\draw[shift={(1,0)}] (0,0)--({sqrt(2)/2},{sqrt(2)/2})--({sqrt(2)/2},{1+sqrt(2)/2})--(0,1)--(0,0);
\draw[shift={(1,0)}] (0,0)--(0,1)--({-sqrt(2)/2},{1 + sqrt(2)/2})--({-sqrt(2)/2},{sqrt(2)/2})--(0,0);
\draw[shift={({-sqrt(2)/2},{sqrt(2)/2})}] (0,0)--(1,0)--(1,1)--(0,1)--(0,0);
\draw[shift={(1,1)}] (0,0)--({sqrt(2)/2},{sqrt(2)/2})--(0,{sqrt(2)})--({-sqrt(2)/2},{sqrt(2)/2})--(0,0);
\draw (0,0)--(1,0)--({1 - sqrt(2)/2},{sqrt(2)/2})--({- sqrt(2)/2},{sqrt(2)/2})--(0,0);
\end{tikzpicture} &
\begin{tikzpicture}[scale=1]
\draw (0,0)--(1,0)--({1+sqrt(2)/2},{sqrt(2)/2})--({1+sqrt(2)/2},{1+sqrt(2)/2})--({1},{1+sqrt(2)})--({0},{1+sqrt(2)})--({-sqrt(2)/2},{1+sqrt(2)/2})--({-sqrt(2)/2},{sqrt(2)/2})--(0,0);

\draw[shift={({-sqrt(2)/2},{1+sqrt(2)/2})}] (0,0)--(1,0)--({1+sqrt(2)/2},{sqrt(2)/2})--({sqrt(2)/2},{sqrt(2)/2})--(0,0);
\draw[shift={({1 - sqrt(2)/2},{sqrt(2)/2})}] (0,0)--({sqrt(2)/2},{sqrt(2)/2})--({sqrt(2)/2},{1+sqrt(2)/2})--(0,1)--(0,0);
\draw[shift={({1+sqrt(2)/2},{sqrt(2)/2})}] (0,0)--(0,1)--({-sqrt(2)/2},{1 + sqrt(2)/2})--({-sqrt(2)/2},{sqrt(2)/2})--(0,0);
\draw[shift={({-sqrt(2)/2},{sqrt(2)/2})}] (0,0)--(1,0)--(1,1)--(0,1)--(0,0);
\draw[shift={(1,0)}] (0,0)--({sqrt(2)/2},{sqrt(2)/2})--(0,{sqrt(2)})--({-sqrt(2)/2},{sqrt(2)/2})--(0,0);
\draw (0,0)--(1,0)--({1 - sqrt(2)/2},{sqrt(2)/2})--({- sqrt(2)/2},{sqrt(2)/2})--(0,0);
\end{tikzpicture}\\
& \vspace{1cm} $T_{1,i}$ & $T_{1,i + 1}$ & $T_{1,i + 2}$ & $T_{1,i + 3}$ & $T_{1,i + 4}$\\
$\mathbf{T}_2$ &
\begin{tikzpicture}[scale=1]
\draw (0,0)--(1,0)--({1+sqrt(2)/2},{sqrt(2)/2})--({1+sqrt(2)/2},{1+sqrt(2)/2})--({1},{1+sqrt(2)})--({0},{1+sqrt(2)})--({-sqrt(2)/2},{1+sqrt(2)/2})--({-sqrt(2)/2},{sqrt(2)/2})--(0,0);

\draw (0,0)--(1,0)--({1+sqrt(2)/2},{sqrt(2)/2})--({sqrt(2)/2},{sqrt(2)/2})--(0,0);
\draw (0,0)--({sqrt(2)/2},{sqrt(2)/2})--({sqrt(2)/2},{1+sqrt(2)/2})--(0,1)--(0,0);
\draw (0,0)--(0,1)--({-sqrt(2)/2},{1 + sqrt(2)/2})--({-sqrt(2)/2},{sqrt(2)/2})--(0,0);
\draw[shift={({sqrt(2)/2},{sqrt(2)/2})}] (0,0)--(1,0)--(1,1)--(0,1)--(0,0);
\draw[shift={(0,1)}] (0,0)--({sqrt(2)/2},{sqrt(2)/2})--(0,{sqrt(2)})--({-sqrt(2)/2},{sqrt(2)/2})--(0,0);
\draw[shift={({sqrt(2)/2},{1+sqrt(2)/2})}] (0,0)--(1,0)--({1 - sqrt(2)/2},{sqrt(2)/2})--({- sqrt(2)/2},{sqrt(2)/2})--(0,0);
\end{tikzpicture} &
\begin{tikzpicture}[scale=1]
\draw (0,0)--(1,0)--({1+sqrt(2)/2},{sqrt(2)/2})--({1+sqrt(2)/2},{1+sqrt(2)/2})--({1},{1+sqrt(2)})--({0},{1+sqrt(2)})--({-sqrt(2)/2},{1+sqrt(2)/2})--({-sqrt(2)/2},{sqrt(2)/2})--(0,0);

\draw (0,0)--(1,0)--({1+sqrt(2)/2},{sqrt(2)/2})--({sqrt(2)/2},{sqrt(2)/2})--(0,0);
\draw[shift={({- sqrt(2)/2},{sqrt(2)/2})}] (0,0)--({sqrt(2)/2},{sqrt(2)/2})--({sqrt(2)/2},{1+sqrt(2)/2})--(0,1)--(0,0);
\draw[shift={({sqrt(2)/2},{sqrt(2)/2})}] (0,0)--(0,1)--({-sqrt(2)/2},{1 + sqrt(2)/2})--({-sqrt(2)/2},{sqrt(2)/2})--(0,0);
\draw[shift={({sqrt(2)/2},{sqrt(2)/2})}] (0,0)--(1,0)--(1,1)--(0,1)--(0,0);
\draw (0,0)--({sqrt(2)/2},{sqrt(2)/2})--(0,{sqrt(2)})--({-sqrt(2)/2},{sqrt(2)/2})--(0,0);
\draw[shift={({sqrt(2)/2},{1+sqrt(2)/2})}] (0,0)--(1,0)--({1 - sqrt(2)/2},{sqrt(2)/2})--({- sqrt(2)/2},{sqrt(2)/2})--(0,0);
\end{tikzpicture} &
\begin{tikzpicture}[scale=1]
\draw (0,0)--(1,0)--({1+sqrt(2)/2},{sqrt(2)/2})--({1+sqrt(2)/2},{1+sqrt(2)/2})--({1},{1+sqrt(2)})--({0},{1+sqrt(2)})--({-sqrt(2)/2},{1+sqrt(2)/2})--({-sqrt(2)/2},{sqrt(2)/2})--(0,0);

\draw (0,0)--(1,0)--({1+sqrt(2)/2},{sqrt(2)/2})--({sqrt(2)/2},{sqrt(2)/2})--(0,0);
\draw[shift={({- sqrt(2)/2},{sqrt(2)/2})}] (0,0)--({sqrt(2)/2},{sqrt(2)/2})--({sqrt(2)/2},{1+sqrt(2)/2})--(0,1)--(0,0);
\draw[shift={({1+sqrt(2)/2},{sqrt(2)/2})}] (0,0)--(0,1)--({-sqrt(2)/2},{1 + sqrt(2)/2})--({-sqrt(2)/2},{sqrt(2)/2})--(0,0);
\draw[shift={({0},{sqrt(2)})}] (0,0)--(1,0)--(1,1)--(0,1)--(0,0);
\draw (0,0)--({sqrt(2)/2},{sqrt(2)/2})--(0,{sqrt(2)})--({-sqrt(2)/2},{sqrt(2)/2})--(0,0);
\draw[shift={({sqrt(2)/2},{sqrt(2)/2})}] (0,0)--(1,0)--({1 - sqrt(2)/2},{sqrt(2)/2})--({- sqrt(2)/2},{sqrt(2)/2})--(0,0);
\end{tikzpicture} &
\begin{tikzpicture}[scale=1]
\draw (0,0)--(1,0)--({1+sqrt(2)/2},{sqrt(2)/2})--({1+sqrt(2)/2},{1+sqrt(2)/2})--({1},{1+sqrt(2)})--({0},{1+sqrt(2)})--({-sqrt(2)/2},{1+sqrt(2)/2})--({-sqrt(2)/2},{sqrt(2)/2})--(0,0);

\draw[shift={({-sqrt(2)/2},{sqrt(2)/2})}] (0,0)--(1,0)--({1+sqrt(2)/2},{sqrt(2)/2})--({sqrt(2)/2},{sqrt(2)/2})--(0,0);
\draw[shift={({- sqrt(2)/2},{sqrt(2)/2})}] (0,0)--({sqrt(2)/2},{sqrt(2)/2})--({sqrt(2)/2},{1+sqrt(2)/2})--(0,1)--(0,0);
\draw[shift={({1+sqrt(2)/2},{sqrt(2)/2})}] (0,0)--(0,1)--({-sqrt(2)/2},{1 + sqrt(2)/2})--({-sqrt(2)/2},{sqrt(2)/2})--(0,0);
\draw[shift={({0},{sqrt(2)})}] (0,0)--(1,0)--(1,1)--(0,1)--(0,0);
\draw[shift={(1,0)}] (0,0)--({sqrt(2)/2},{sqrt(2)/2})--(0,{sqrt(2)})--({-sqrt(2)/2},{sqrt(2)/2})--(0,0);
\draw (0,0)--(1,0)--({1 - sqrt(2)/2},{sqrt(2)/2})--({- sqrt(2)/2},{sqrt(2)/2})--(0,0);
\end{tikzpicture} &
\begin{tikzpicture}[scale=1]
\draw (0,0)--(1,0)--({1+sqrt(2)/2},{sqrt(2)/2})--({1+sqrt(2)/2},{1+sqrt(2)/2})--({1},{1+sqrt(2)})--({0},{1+sqrt(2)})--({-sqrt(2)/2},{1+sqrt(2)/2})--({-sqrt(2)/2},{sqrt(2)/2})--(0,0);

\draw[shift={({-sqrt(2)/2},{1+sqrt(2)/2})}] (0,0)--(1,0)--({1+sqrt(2)/2},{sqrt(2)/2})--({sqrt(2)/2},{sqrt(2)/2})--(0,0);
\draw[shift={({1 - sqrt(2)/2},{sqrt(2)/2})}] (0,0)--({sqrt(2)/2},{sqrt(2)/2})--({sqrt(2)/2},{1+sqrt(2)/2})--(0,1)--(0,0);
\draw[shift={({1+sqrt(2)/2},{sqrt(2)/2})}] (0,0)--(0,1)--({-sqrt(2)/2},{1 + sqrt(2)/2})--({-sqrt(2)/2},{sqrt(2)/2})--(0,0);
\draw[shift={({-sqrt(2)/2},{sqrt(2)/2})}] (0,0)--(1,0)--(1,1)--(0,1)--(0,0);
\draw[shift={(1,0)}] (0,0)--({sqrt(2)/2},{sqrt(2)/2})--(0,{sqrt(2)})--({-sqrt(2)/2},{sqrt(2)/2})--(0,0);
\draw (0,0)--(1,0)--({1 - sqrt(2)/2},{sqrt(2)/2})--({- sqrt(2)/2},{sqrt(2)/2})--(0,0);
\end{tikzpicture}\\
& $T_{2,i}$ & $T_{2,i + 1}$ & $T_{2,i + 2}$ & $T_{2,i + 3}$ & $T_{2,i + 4}$\\
\end{tabular}
\caption{The two piles $\mathbf{T}_1$ and $\mathbf{T}_2$ in Definition~\ref{three_dim_flip}. The tiles outside the octagon remain in place.} \label{3flip_figure}
\end{figure}
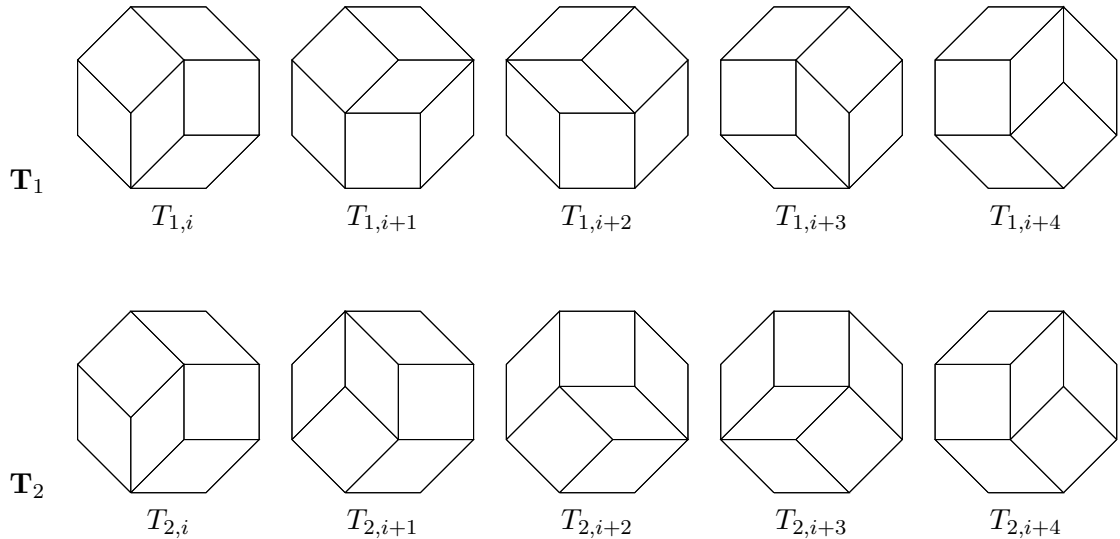
\end{Definition}

\begin{Proposition}[{\cite[Theorem~4.1]{ziegler}}] \label{bn3_connected_by_flips}For any $\mathbf{T}, \mathbf{T}' \in \mathcal{C}(n)$, there exists a sequence of piles $\mathbf{T}_0, \dots, \mathbf{T}_\ell \in \mathcal{C}(n)$ with $\mathbf{T} = \mathbf{T}_0$ and $\mathbf{T}' = \mathbf{T}_\ell$, such that the directed cubical complexes $\varkappa(\mathbf{T}_{i - 1})$ and $\varkappa(\mathbf{T}_i)$ are related by a flip, for $i = 1,\dots,\ell$.
\end{Proposition}

\begin{Definition} \label{admissible_def}Let $\sigma$ be a permutation of $\binom{[n]}{k}$, viewed as a sequence whose elements are $k$-element subsets of $[n]$, each of them appearing exactly once. We call such a permutation $\sigma$ \emph{admissible} if for every $I = \{i_1 < \cdots < i_{k + 1}\} \in \binom{[n]}{k + 1}$, the $k + 1$ sets in $\binom{I}{k}$ appear in $\sigma$ in either
\begin{itemize}\itemsep=0pt
\item lexicographic order, i.e., in the order $\big\{i_1, i_2, \dots, i_k, \widehat{i_{k + 1}}\big\}$, $\big\{i_1, i_2, \dots, \widehat{i_k}, i_{k + 1}\}$, \dots, $\big\{\widehat{i_1}, i_2$, $\dots, i_k, i_{k + 1}\big\}$, or
\item reversed lexicographic order, i.e., in the order $\big\{\widehat{i_1}, i_2, \dots, i_k, i_{k + 1}\big\}$, $\big\{i_1, \widehat{i_2}, \dots,$ $i_k, i_{k + 1}\big\}$, $\dots$, $\big\{i_1, i_2, \dots, i_k, \widehat{i_{k + 1}}\big\}$.
\end{itemize}
Thus, for example, $(\{1,2\}, \{1,3\}, \{2,3\})$ and $(\{2,3\}, \{1,3\}, \{1,2\})$ are the only two admissible permutations of $\binom{[3]}{2}$. The \emph{inversion set} of an admissible permutation $\sigma$ of $\binom{[n]}{k}$ is the subset of $\binom{[n]}{k + 1}$ consisting of those $I \in \binom{[n]}{k + 1}$ for which the elements of $\binom{I}{k}$ appear in $\sigma$ in the reversed lexicographic order.
\end{Definition}

\begin{Definition}\label{perm_from_pile_def}Given a pile $\mathbf{T} = \big( T_0, \dots, T_{\binom{n}{3}} \big) \in \mathcal{C}(n)$, for $i = 1, \dots, \binom{n}{3}$, let $\alpha_i \in \binom{[n]}{3}$ be the set of three indices of the edges of the hexagon involved in the flip between $T_{i - 1}$ and $T_i$. Note that $\binom{[n]}{3} = \big\{ \alpha_1, \dots, \alpha_{\binom{n}{3}} \big\}$, as each $\alpha_i$ indexes a triple which is switching from a $\Delta$-crossing to a $\nabla$-crossing. We say that $\big(\alpha_1, \dots, \alpha_{\binom{n}{3}} \big)$ is the \emph{permutation of $\binom{[n]}{3}$ associated to~$\mathbf{T}$}. Note that $\mathbf{T} \in \mathcal{C}(n)$ is uniquely determined by its permutation of $\binom{[n]}{3}$.
\end{Definition}

\begin{Theorem}[\cite{manin_sch}, {\cite[Definition~2.1, Theorem~4.1]{ziegler}}] \label{admissible_prop}
Let $\sigma$ be a permutation of $\binom{[n]}{3}$. The following are equivalent:
\begin{itemize}\itemsep=0pt
\item $\sigma$ is an admissible permutation of $\binom{[n]}{3}$,
\item there exists a pile $\mathbf{T} \in \mathcal{C}(n)$ whose corresponding permutation of $\binom{[n]}{3}$ is $\sigma$.
\end{itemize}
\end{Theorem}

\section[The coherence condition and principal minors of symmetric matrices]{The coherence condition and principal minors\\ of symmetric matrices} \label{principal_minors_section}

We begin this section by reviewing the earlier work of Kenyon and Pemantle \cite{otherhex} concerning the occurrence of the hexahedron recurrence as a determinantal identity. We then formulate new criteria for the existence of symmetric matrices with prescribed values of certain principal minors. See in particular Corollaries~\ref{cor_with_t_min},~\ref{all_pminors_corollary}, and Theorem~\ref{main_all_minors_simple}. The proofs of these results are given later in Sections~\ref{cc_proof_section}--\ref{matrix_proofs}.

We begin by extending the definitions of the Kashaev equation, positive Kashaev recurrence, hexahedron recurrence, and K-hexahedron equations to arrays indexed on directed cubical complexes in the obvious way. For those readers who skipped Section~\ref{preliminaries}, it may be helpful to review Definitions~\ref{cc_def}, \ref{dcc_def}, and~\ref{0_02def} before processing the following definition.

\begin{Definition}\label{recurrences_on_tilings_def}Fix a directed cubical complex $\varkappa$. An array $\mathbf{x}$ indexed by $\varkappa^0$ satisfies the Kashaev equation if for all $3$-dimensional cubes $C$ of $\varkappa$, $\mathbf{x}$ satisfies~\eqref{kashforcube} with the components of $\mathbf{x}$ labeled on $C$ as in Fig.~\ref{labeledcube}. We say that $\mathbf{x}$ satisfies the positive Kashaev equation if the components of $\mathbf{x}$ are all positive and for all $3$-dimensional cubes $C$ of $\varkappa$, $\mathbf{x}$ satisfies~\eqref{kashaevwithsr} with the components of $\mathbf{x}$ labeled on $C$ as in Fig.~\ref{labeledcube} and~$z_{111}$ corresponding to the component of~$\mathbf{x}$ at the top vertex of~$C$. We say that an array $\mathbf{\tilde{x}}$ indexed by $\varkappa^{02}$ satisfies the hexahedron recurrence (resp., K-hexahedron equations) if for all $3$-dimensional cubes $C$ of $\varkappa$, $\mathbf{\tilde{x}}$ satisfies equations~\eqref{hexeqn1}--\eqref{hexeqnmain} (resp., equations~\eqref{likehex1d}--\eqref{likehex4d}, along with equation~\eqref{khexspec} for all $s \in \varkappa^2$) with the components of~$\mathbf{\tilde{x}}$ labeled on the vertices of $C$ as in Fig.~\ref{labeledcube}, labeled on the $2$-dimensional faces of $C$ by averaging the indices of the vertices on the boundary (so that, for example, $z_{0\frac12\frac12}$ is the component of $\mathbf{\tilde{x}}$ at the face where $\mathbf{\tilde{x}}$ has components $z_{000}$, $z_{010}$, $z_{001}$, and $z_{011}$ at the vertices at its boundary), and with $z_{111}$ corresponding to the component of $\mathbf{\tilde{x}}$ on the top vertex of $C$.
\end{Definition}

\begin{Remark} \label{cc_dep_on_direction}As we saw in Section~\ref{main-results}, the Kashaev equation is independent of a choice of direction on each cube, and hence can be defined for arbitrary $3$-dimensional cubical complexes. On the other hand, the positive Kashaev recurrence, hexahedron recurrence, and K-hexahedron equations depend on a choice of a pair of opposite distinguished vertices in each cube, and hence are defined on directed cubical complexes.
\end{Remark}

\begin{Definition}Given an $n \times n$ matrix $M$ and $I, J \subseteq [n]$, we denote by $M_I^J$ the submatrix of~$M$ obtained by restricting to rows~$I$ and columns~$J$. A \emph{principal minor} of~$M$ is the determinant of a submatrix $\det M_I^I$, where $I \subseteq [n]$. We follow the convention that $M_\varnothing^\varnothing = 1$. An \emph{almost-principal minor} of $M$ is the determinant of a submatrix $\det M_{I \cup \{i\}}^{I \cup \{j\}}$ for $I \subset [n]$, and distinct $i, j \not\in I$. We say that an almost-principal minor $M_{I \cup \{i\}}^{I \cup \{j\}}$ is \emph{odd} if $(i - j) (-1)^{|I|} > 0$, and \emph{even} if $(i - j) (-1)^{|I|} < 0$.
\end{Definition}

Before proceeding with the following definition, the reader may want to review Definitions~\ref{def_diamond_pn}, \ref{chambers_pla}, \ref{pile_def}, and~\ref{assoc_cubical_complex_def}.%ref

\begin{Definition}\label{matrix_to_arrays_map}Given a $\Diamond$-tiling $T$ of $\mathbf{P}_n$ and an $n \times n$ complex-valued matrix $M$, define the array $\mathbf{x}_T(M) = (x_s)_{s \in \varkappa^0(T)}$, where if vertex $s$ of $T$ is labeled by $I \subseteq [n]$,
\begin{gather*}
x_s = (-1)^{\lfloor | I | / 2 \rfloor} M_I^I.
\end{gather*}
Similarly, define the array $\mathbf{\tilde{x}}_T(M) = (x_s)_{s \in \varkappa^{02}(T)}$, where
\begin{itemize}\itemsep=0pt
\item if $s \in \varkappa^0(T)$: given that vertex $s$ of $T$ is labeled by $I \subseteq [n]$, set
\begin{gather*}
x_s = (-1)^{\lfloor |I| / 2 \rfloor} M_I^I,
\end{gather*}
\item if $s \in \varkappa^2(T)$: given that tile $s$ of $T$ has vertices labeled by $I, I \cup \{i\}, I \cup \{j\}, I \cup \{i, j\} \subseteq [n]$, where $i$ and $j$ are chosen so that $M_{I \cup \{i\}}^{I \cup \{j\}}$ is the odd almost principal minor, set
\begin{gather*} %\label{signed_odd_principal}
x_s = (-1)^{\lfloor (|I| +1) / 2 \rfloor} M_{I \cup \{i\}}^{I \cup \{j\}}.
\end{gather*}
\end{itemize}
More generally, if $\mathbf{T} = (T_0, \dots, T_\ell)$ is a pile of $\Diamond$-tilings of $\mathbf{P}_n$, define $\mathbf{x}_{\varkappa(T)}(M)$ to be the array indexed by $\varkappa^0(\mathbf{T})$ whose restriction to $\varkappa^0(T_i)$ is $\mathbf{x}_{T_i}(M)$, and define $\mathbf{\tilde{x}}_{\varkappa(T)}(M)$ to be the array indexed by $\varkappa^{02}(\mathbf{T})$ whose restriction to $\varkappa^{02}(T_i)$ is $\mathbf{\tilde{x}}_{T_i}(M)$.
\end{Definition}

\begin{Remark}For any $\Diamond$-tiling $T$ of $\mathbf{P}_n$, the vertex $v_0$ is labeled by $\varnothing \subset [n]$. In Definition~\ref{matrix_to_arrays_map}, because of the convention that $M_\varnothing^\varnothing = 1$, $x_{v_0} = 1$ independent of the matrix~$M$.
\end{Remark}

\begin{Definition}Given a $\Diamond$-tiling $T$ of $\mathbf{P}_n$, we say that a complex-valued array $\mathbf{\tilde{x}} = (x_s)_{s \in \varkappa^{02}(T)}$ is \emph{standard} if $x_{v_0} = 1$. Furthermore, given a pile of $\Diamond$-tilings of $\mathbf{P}_n$, $\mathbf{T} = (T_0, \dots, T_\ell)$, and $\varkappa = \varkappa(\mathbf{T})$, we say that a complex-valued array $\mathbf{\tilde{x}} = (x_s)_{s \in \varkappa^{02}}$ is \emph{standard} with respect to $\mathbf{T}$ if $x_{v_0} = 1$.
\end{Definition}

\begin{Definition}\label{generic_array_def}Given a $\Diamond$-tiling $T$ of $\mathbf{P}_n$, we say that a complex-valued array $\mathbf{\tilde{x}}$ indexed by~$\varkappa^{02}(T)$ is \emph{generic} if for any sequence of flips applied to $T$ accompanied by applications of the hexahedron recurrence to $\mathbf{\tilde{x}}$, the resulting coefficients are all nonzero.
\end{Definition}

\begin{Definition}We say that a square matrix is \emph{generic} if all of its principal minors and odd almost-principal minors are non-zero. Let $M_n^*(\C)$ denote the set of $n \times n$ generic complex-valued matrices.
\end{Definition}

We can now provide some important results of Kenyon and Pemantle~\cite{otherhex}.

\begin{Theorem}[{\cite[Theorem 4.2]{otherhex}}] \label{tiles_to_matrix_bij}Given a $\Diamond$-tiling $T$ of $\mathbf{P}_n$, the map $\mathbf{\tilde{x}}_T(\cdot)$ establishes a bijective correspondence between $M_n^*(\C)$ and standard, generic, complex-valued arrays on $\varkappa^{02}(T)$.%
\end{Theorem}

Before proceeding with the following theorem, the reader may want to review Definition~\ref{min_max_tiling_def}.%ref

\begin{Theorem}[{\cite[Theorem~4.4]{otherhex}}] \label{tmin_laurent}Let $\mathbf{\tilde{x}} = (x_s) \in (\C^*)^{\varkappa^{02}(T_{\textup{min},n})}$ be a standard array. Then there is a unique matrix $M$ such that $\mathbf{\tilde{x}} = \mathbf{\tilde{x}}_{\varkappa(T_{\textup{min},n})}(M)$. The entries of this matrix $M$ are Laurent polynomials in the components of $\mathbf{\tilde{x}}$.
\end{Theorem}

\begin{Proposition}[{\cite[Lemma 2.1]{otherhex}}] \label{flip_hex}Suppose $M$ is an $n \times n$ matrix, and $\mathbf{T}$ is a pile of $\Diamond$-tilings of $\mathbf{P}_n$ such that the components of $\mathbf{\tilde{x}}_{\varkappa(\mathbf{T})}(M)$ are all nonzero. Then $\mathbf{\tilde{x}}_{\varkappa(\mathbf{T})}(M)$ satisfies the hexahedron recurrence.
\end{Proposition}

\begin{Theorem}[{\cite[Theorem 5.2]{otherhex}}] \label{sym_mat_net}Let $T$ be a $\Diamond$-tiling of $\mathbf{P}_n$. A matrix $M \in M_n^*(\C)$ is symmetric if and only if for every tile $s$ in $T$ with vertices $s_1$, $s_2$, $s_3$, $s_4$ in cyclic order,
\begin{gather} \label{tiling_cond_for_sym}
(\mathbf{\tilde{x}}_T(M)_s)^2 = \mathbf{\tilde{x}}_T(M)_{s_1} \mathbf{\tilde{x}}_T(M)_{s_3} + \mathbf{\tilde{x}}_T(M)_{s_2} \mathbf{\tilde{x}}_T(M)_{s_4}.
\end{gather}
Furthermore, if $M$ is any $n \times n$ symmetric matrix, condition~\eqref{tiling_cond_for_sym} holds for every tile of~$T$.
\end{Theorem}

\begin{Remark}Theorem~\ref{sym_mat_net} is not stated explicitly in~\cite{otherhex}, but follows immediately from the cited theorem. The original theorem concerns Hermitian matrices, and a slightly modified version of the hexahedron recurrence in which some complex conjugates are taken.
\end{Remark}

The next result follows from Propositions~\ref{likehexprop} and~\ref{flip_hex} and Theorem~\ref{sym_mat_net}:

\begin{Corollary} \label{sym_matrix_k_hex_arb}Let $\mathbf{T}$ be a pile of $\Diamond$-tilings of $\mathbf{P}_n$, with $\varkappa = \varkappa(\mathbf{T})$. Let $M$ be an $n \times n$ symmetric matrix with nonzero principal minors. Then $\mathbf{x}_{\varkappa(\mathbf{T})}(M)$ satisfies the K-hexahedron equations.
\end{Corollary}

The next corollary follows immediately from Theorem~\ref{sym_mat_net} and the fact that the entries of~$M$ are Laurent polynomials in the components of $\mathbf{\tilde{x}}_{\varkappa(T_{\textup{min},n})}(M)$:

\begin{Corollary} \label{sym_for_min_tiling}Let $M$ be an $n \times n$ matrix such that the components of $\mathbf{\tilde{x}}_{\varkappa(T_{\textup{min},n})}(M)$ are nonzero. Then $M$ is symmetric if and only if condition~\eqref{tiling_cond_for_sym} holds.
\end{Corollary}

Hence, the following is immediate from Proposition~\ref{flip_hex} and Corollary~\ref{sym_for_min_tiling}:

\begin{Corollary} \label{sym_matrix_k_hex}Let $\mathbf{T}$ be a pile of $\Diamond$-tilings of $\mathbf{P}_n$ containing $T_{\textup{min},n}$, with~$\varkappa = \varkappa(\mathbf{T})$. Let $\mathbf{x} = (x_s) \in (\C^*)^{\varkappa^0}$ be a standard array satisfying the property that
\begin{gather} \label{nonzero_face_condition}
x_{v_1} x_{v_3} + x_{v_2} x_{v_4} \not= 0
\end{gather}
for all $2$-dimensional faces of $\varkappa$ with vertices $v_1$, $v_2$, $v_3$, $v_4$ in cyclic order. Then the following are equivalent:
\begin{itemize}\itemsep=0pt\samepage
\item $\mathbf{x}$ can be extended to a standard array indexed by $\varkappa^{02}$ satisfying the K-hexahedron equations;
\item there exists a symmetric matrix $M$ such that $\mathbf{x} = \mathbf{x}_{\varkappa(\mathbf{T})}(M)$.
\end{itemize}
\end{Corollary}

We want a set of equations that tell us whether an array $\mathbf{x}$ indexed by $\varkappa^0(\mathbf{T})$ can be extended to an array indexed by $\varkappa^{02}(\mathbf{T})$ satisfying the K-hexahedron equations. Below, we define a notion of coherence generalizing the notion of coherence from Section~\ref{main-results}.

\begin{Definition}\label{coherence_gen_def}Let $\varkappa$ be a $3$-dimensional cubical complex. We say that $\mathbf{x} = (x_s)_{s \in \varkappa}$ is a~\emph{coherent} solution of the Kashaev equation if $\mathbf{x}$ satisfies the Kashaev equation (i.e., $K^C(\mathbf{x}) = 0$ for every $3$-dimensional cube $C$ of $\varkappa$), and for any interior vertex $v$ of $\varkappa$:
\begin{gather} \label{coherence_condition_for_wir}
\prod_{C \ni v} K_v^C(\mathbf{x}) = \prod_{S \ni v} (x_v x_{v_2} + x_{v_1} x_{v_3}),
\end{gather}
where
\begin{itemize}\itemsep=0pt
\item the first product is over $3$-dimensional cubes $C$ incident to the vertex $v$,
\item the second product is over $2$-dimensional faces $S$ incident to the vertex $v$, and
\item $v, v_1, v_2, v_3$ are the vertices of such a face $S$ listed in cyclic order.
\end{itemize}
\end{Definition}

\begin{Remark}
The property of being coherent solution of the Kashaev equation is defined for all $3$-dimensional cubical complexes, not only for $3$-dimensional directed cubical complexes, as no choice of direction needs to be made in each cube.
\end{Remark}

\begin{Theorem} \label{main_thm_cyclic_zonotope}
Let $\mathbf{T}$ be a pile of $\Diamond$-tilings of $\mathbf{P}_n$, with $\varkappa = \varkappa(\mathbf{T})$.
\begin{enumerate}\itemsep=0pt
\item[$(a)$] Any coherent solution $\mathbf{x} = (x_s) \in (\C^*)^{\varkappa^0}$ of the Kashaev equation satisfying the property that
\begin{gather} \label{face_nonzero}
x_{v_1} x_{v_3} + x_{v_2} x_{v_4} \not= 0
\end{gather}
for all faces of $\varkappa$ with vertices $v_1$, $v_2$, $v_3$, $v_4$ in cyclic order, can be extended to $\mathbf{\tilde{x}} = (x_s)_{s \in \varkappa^{02}(\mathbf{T})}$ satisfying the K-hexahedron equations.
\item[$(b)$] Conversely, suppose that $\mathbf{\tilde{x}} = (x_s) \in (\C^*)^{\varkappa^{02}}$ $($with $x_s \not= 0$ for $s \in \varkappa^0)$ satisfies the K-hexahedron equations. Then the restriction of $\mathbf{\tilde{x}}$ to $\varkappa^0$ is a coherent solution of the Kashaev equation.
\end{enumerate}
\end{Theorem}

Theorem~\ref{main_thm_cyclic_zonotope} is proved in Section~\ref{cc_proof_section}, where we obtain results (Proposition~\ref{cc_easy_direction} and Theorem~\ref{cc_main_generalization}) generalizing both Theorems~\ref{main_thm_cyclic_zonotope} and~\ref{galoisfromintro}.

As an immediate corollary of Corollary~\ref{sym_matrix_k_hex} and Theorem~\ref{main_thm_cyclic_zonotope}, we obtain the following:

\begin{Corollary} \label{cor_with_t_min}Let $\mathbf{T}$ be a pile of $\Diamond$-tilings of $\mathbf{P}_n$ containing $T_{\textup{min},n}$, with $\varkappa = \varkappa(\mathbf{T})$. Let $\mathbf{x} = (x_s) \in (\C^*)^{\varkappa^0}$ be a standard array satisfying condition~\eqref{nonzero_face_condition}. Then the following are equivalent:
\begin{itemize}\itemsep=0pt
\item $\mathbf{x}$ is a coherent solution of the Kashaev equation,
\item there exists a symmetric matrix $M$ such that $\mathbf{x} = \mathbf{x}_{\varkappa(\mathbf{T})}(M)$.
\end{itemize}
\end{Corollary}

Next, we consider the problem of checking whether an array of $2^n$ numbers could correspond to the principal minors of some symmetric matrix.

\begin{Definition}Given an $n \times n$ symmetric matrix $M$, let $\mathbf{\bar{x}}(M) = (x_I)_{I \in [n]}$, where
\begin{gather} \label{xi_in_terms_of_pminor}
x_I = (-1)^{\lfloor |I| / 2 \rfloor} M_I^I
\end{gather}
for $I \subseteq [n]$. Given a pile $\mathbf{T}$ of $\Diamond$-tilings of $\mathbf{P}_n$, with $\varkappa = \varkappa(\mathbf{T})$, and an array $\mathbf{\bar{x}} = (x_I)_{I \subseteq [n]}$, let $\mathbf{x}_{\varkappa(\mathbf{T})}(\mathbf{\bar{x}}) = (x_s)_{s \in \varkappa^0}$ where $x_s = x_I$ when vertex $s$ is labeled by~$I \subseteq [n]$.
\end{Definition}

\begin{Definition}Given an array $\mathbf{\bar{x}} = (x_J)_{J \subseteq [n]}$, $I \subseteq [n]$, and distinct $i, j \in [n]$, set
\begin{gather*}
L_{I, \{i,j\}}(\mathbf{\bar{x}}) = x_I x_{I \Delta \{i,j\}} + x_{I \Delta \{i\}} x_{I \Delta \{j\}},
\end{gather*}
where $\Delta$ denotes the symmetric difference.
\end{Definition}

\begin{Corollary} \label{all_pminors_corollary}Fix an array $\mathbf{\bar{x}} = (x_I)_{I \subseteq [n]}$ with nonzero entries, satisfying the conditions that $L_{I, \{i,j\}} \not= 0$ for any $I \subseteq [n]$ and distinct $i,j \in [n]$, and $x_{\varnothing} = 1$. Then the following are equivalent:
\begin{itemize}\itemsep=0pt
\item there exists a symmetric matrix $M$ such that $\mathbf{\bar{x}} = \mathbf{\bar{x}}(M)$,
\item for some pile $\mathbf{T}$ of $\Diamond$-tilings of $\mathbf{P}_n$ in which every $I \subseteq [n]$ labels at least one vertex of $\varkappa(\mathbf{T})$, $\mathbf{x}_{\varkappa(\mathbf{T})}(\mathbf{\bar{x}})$ is a coherent solution of the Kashaev equation,
\item for all piles $\mathbf{T}$ of $\Diamond$-tilings of $\mathbf{P}_n$, $\mathbf{x}_{\varkappa(\mathbf{T})}(\mathbf{\bar{x}})$ is a coherent solution of the Kashaev equation.
\end{itemize}
\end{Corollary}

Corollary~\ref{all_pminors_corollary} can be deduced from Theorems~\ref{tiles_to_matrix_bij},~\ref{sym_mat_net}, Corollary~\ref{sym_matrix_k_hex_arb}, and Theorem~\ref{main_thm_cyclic_zonotope}; we provide such a proof in Section~\ref{matrix_proofs}.

Corollary~\ref{all_pminors_corollary} provides us a set of equations (namely, the equations for $\mathbf{x}_{\varkappa(\mathbf{T})}(\mathbf{\bar{x}})$ to be a~coherent solution of the Kashaev equation) to test whether a set of values matches the set of principal minors for some symmetric matrix. However, in general, the conditions~\eqref{coherence_condition_for_wir} for different interior vertices of $\varkappa(\mathbf{T})$ have very different forms.

\begin{Definition}Given an array $\mathbf{\bar{x}} = (x_J)_{J \subseteq [n]}$, $I \subseteq [n]$, and distinct $i,j,k \in [n]$, set
\begin{gather}
\label{minor_kc} K^{I, \{i,j,k\}}(\mathbf{\bar{x}}) = K^C(\mathbf{x}),\\
\label{minor_kcv} K_{I, \{i,j,k\}}(\mathbf{\bar{x}}) = K^C_v(\mathbf{x}),
\end{gather}
where $C$ is a $3$-dimensional cube, and $\mathbf{x}$ is the array indexed by the vertices of $C$ shown in Fig.~\ref{minor_labeled_cube}, and $v$ is the vertex of $C$ at which $\mathbf{x}$ has entry $x_I$.
\begin{figure}[ht]\centering
\begin{tikzpicture}[scale=1.5]
\draw (0,0)--(2.5,0)--(2.5,2.5)--(0,2.5)--(0,0);
\draw (1,1)--(3.5,1)--(3.5,3.5)--(1,3.5)--(1,1);
\draw (0,0)--(1,1);
\draw (2.5,0)--(3.5,1);
\draw (2.5, 2.5)--(3.5, 3.5);
\draw (0,2.5)--(1,3.5);

\filldraw (0,0) circle (1.5pt);
\filldraw (2.5,2.5) circle (1.5pt);
\filldraw (3.5,1) circle (1.5pt);
\filldraw (1,3.5) circle (1.5pt);

\filldraw[black] (2.5, 0) circle (1.5pt);
\filldraw[white] (2.5, 0) circle (1pt);
\filldraw[black] (0, 2.5) circle (1.5pt);
\filldraw[white] (0, 2.5) circle (1pt);
\filldraw[black] (1, 1) circle (1.5pt);
\filldraw[white] (1, 1) circle (1pt);
\filldraw[black] (3.5, 3.5) circle (1.5pt);
\filldraw[white] (3.5, 3.5) circle (1pt);

\draw (0,0) node[anchor=east]{$x_I$};
\draw (0,2.5) node[anchor=east]{$x_{I \Delta \{j\}}$};
\draw (1,1) node[anchor=east]{$x_{I \Delta \{k\}}$};
\draw (1,3.5) node[anchor=east]{$x_{I \Delta \{j,k\}}$};
\draw (2.5,0) node[anchor=west]{$x_{I \Delta \{i\}}$};
\draw (2.5,2.5) node[anchor=west]{$x_{I \Delta \{i,j\}}$};
\draw (3.5,1) node[anchor=west]{$x_{I \Delta \{i,k\}}$};
\draw (3.5,3.5) node[anchor=west]{$x_{I \Delta \{i,j,k\}}$};
\end{tikzpicture}
\caption{The array $\mathbf{x}$ in equations~\eqref{minor_kc}--\eqref{minor_kcv}, where $\Delta$ denotes the symmetric difference. In equation~\eqref{minor_kcv}, the vertex $v$ is the lower left vertex, with value $x_I$.} \label{minor_labeled_cube}
\end{figure}
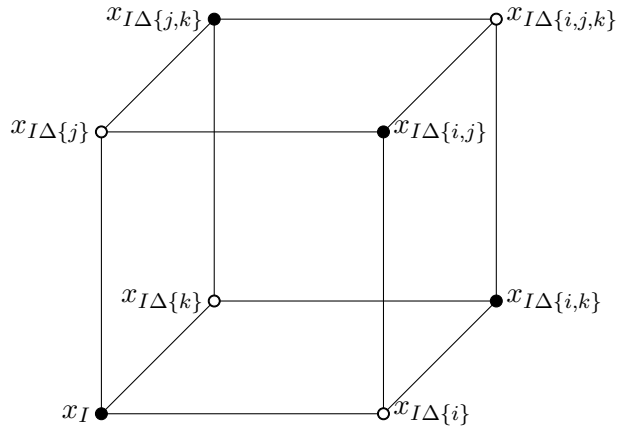
\end{Definition}

Let's focus on the $n = 4$ case. Using the $\Diamond$-tilings from Fig.~\ref{oct_cycle_figure}, consider the pile $\mathbf{T} = (T_0, T_1, \dots, T_7, T_0, T_1, T_2, T_3)$. Note that for every $I \subseteq [4]$, there is a tiling in $\mathbf{T}$ where $I$ labels a~vertex. (In fact, this property holds for the pile $(T_i, \dots, T_{i + 5})$ for any $i \in [8]$, with indices taken mod $8$.) Note that $\mathbf{x}_{\varkappa(\mathbf{T})}(\mathbf{\bar{x}})$ satisfies the Kashaev equation if and only if
\begin{gather} \label{kash_for_matrices}
K^{I, \{i,j,k\}}(\mathbf{\bar{x}}) = 0
\end{gather}
for all $I \subseteq [4]$ and distinct $i,j,k \in [4]$. Note that every interior vertex of $\varkappa(\mathbf{T})$ is incident with four $3$-dimensional cubes, each pair of which shares one $2$-dimensional face. Hence, for each $I \subseteq [4]$ labeling an interior vertex of $\varkappa(\mathbf{T})$ (namely, when $I$ is one of $\{1,3\}$, $\{1,4\}$, $\{1,2,4\}$, $\{1,3,4\}$, $\{2\}$, $\{2,3\}$, $\{2,4\}$, or $\{3\}$), the following condition holds if $\mathbf{\bar{x}} = \mathbf{\bar{x}}(M)$ for some $4 \times 4$ symmetric matrix $M$:
\begin{gather} \label{coh_in_n4_cond}
\prod_{J \in \binom{[4]}{3}} K_{I, J}(\mathbf{\bar{x}}) = \prod_{J \in \binom{[4]}{2}} L_{I, J}(\mathbf{\bar{x}}).
\end{gather}
Furthermore, it is straightforward to check that if $\mathbf{\bar{x}} = \mathbf{\bar{x}}(M)$ for some $4 \times 4$ symmetric mat\-rix~$M$, then equation~\eqref{coh_in_n4_cond} holds for all $I \subseteq [4]$. Hence, the result below follows from Corollary~\ref{all_pminors_corollary}.

\begin{Corollary} \label{n4_with_coh_cond}
Fix an array $\mathbf{\bar{x}} = (x_I)_{I \subseteq [4]}$, satisfying the conditions that $L_{I, \{i,j\}} \not= 0$ for any $I \subseteq [4]$ and distinct $i,j \in [4]$, and $x_{\varnothing} = 1$. Then the following are equivalent:
\begin{itemize}\itemsep=0pt
\item there exists a symmetric matrix $M$ such that $\mathbf{\bar{x}} = \mathbf{\bar{x}}(M)$,
\item for all piles $\mathbf{T}$ of $\Diamond$-tilings of $\mathbf{P}_4$, $\mathbf{x}_{\varkappa(\mathbf{T})}(\mathbf{\bar{x}})$ is a coherent solution of the Kashaev equation,
\item for all $I \subseteq [4]$ and distinct $i,j,k \in [4]$, equation~\eqref{kash_for_matrices} holds, and for all $I \subseteq [4]$, equation~\eqref{coh_in_n4_cond} holds.
\end{itemize}
\end{Corollary}

Equations~\eqref{kash_for_matrices}--\eqref{coh_in_n4_cond} are far more manageable than the coherence conditions that can arise in an arbitrary cubical complex $\varkappa(\mathbf{T})$. The good news is that the only coherence conditions that we need to check for any $n \ge 4$ are of the form of equation~\eqref{coh_in_n4_cond}!

\begin{Theorem} \label{main_all_minors_simple}
Fix an array $\mathbf{\bar{x}} = (x_I)_{I \subseteq [n]}$, satisfying the conditions that $L_{I, \{i,j\}} \not= 0$ for any $I \subseteq [n]$ and distinct $i,j \in [n]$, and $x_{\varnothing} = 1$. Then the following are equivalent:
\begin{itemize}\itemsep=0pt
\item there exists a symmetric matrix $M$ such that $\mathbf{\bar{x}} = \mathbf{\bar{x}}(M)$,
\item for all piles $\mathbf{T}$ of $\Diamond$-tilings of $\mathbf{P}_n$, $\mathbf{x}_{\varkappa(\mathbf{T})}(\mathbf{\bar{x}})$ is a coherent solution of the Kashaev equation,
\item for all $I \subseteq [n]$ and distinct $i,j,k \in [n]$, equation~\eqref{kash_for_matrices} holds, and for all $I \subseteq [n]$ and $A \in \binom{[n]}{4}$,
\begin{gather} \label{4_cube_coherence}
\prod_{J \in \binom{A}{3}} K_{I, J}(\mathbf{\bar{x}}) = \prod_{J \in \binom{A}{2}} L_{I, J}(\mathbf{\bar{x}}).
\end{gather}
\end{itemize}
\end{Theorem}

We prove Theorem~\ref{main_all_minors_simple} in Section~\ref{matrix_proofs}.

\begin{Remark}
We compare Theorem~\ref{main_all_minors_simple} to a similar result of Oeding {\cite[Corollary~1.4]{oeding}}. He considers the natural action of $\big(\SL_2(\C)^{\times n}\big) \ltimes S_n$ (where $S_n$ is the symmetric group on $n$ elements) on $\C^{2^{[n]}}$, and proves that for $\mathbf{\bar{x}} = (x_I)_{I \subseteq [n]}$ the following are equivalent:
\begin{itemize}\itemsep=0pt
\item there exists a symmetric matrix $M$ such that $x_I = M_I^I$ for all $I \subseteq [n]$,
\item all images of $\mathbf{\bar{x}}$ under $(\SL_2(\C)^{\times n}) \ltimes S_n$ satisfy
\begin{gather} \label{hyperdet_id_oeding_version}
2\big(a^2 + b^2 + c^2 + d^2\big) - (a + b + c + d)^2 + 4 (s + t) = 0,
\end{gather}
where
\begin{gather*}
a = x_{\varnothing} x_{\{1,2,3\}}, \qquad b = x_{\{1\}} x_{\{2,3\}},\qquad c = x_{\{2\}} x_{\{1,3\}}, \qquad d = x_{\{3\}} x_{\{1,2\}},\\
s = x_{\varnothing} x_{\{1,2\}} x_{\{1,3\}} x_{\{2,3\}},\qquad t = x_{\{1\}} x_{\{2\}} x_{\{3\}} x_{\{1,2,3\}}.
\end{gather*}
\end{itemize}
The left-hand side of equation~\eqref{hyperdet_id_oeding_version} can be identified as Cayley's $2 \times 2 \times 2$ \emph{hyperdeterminant}. Equivalently, equation~\eqref{hyperdet_id_oeding_version} is just equation~\eqref{kash_for_matrices} for $I = \varnothing$ and $\{i,j,k\} = \{1,2,3\}$ with the appropriate changes of sign (because we don't put additional signs on the principal minors in this setting). For the above equivalence, Oeding does not impose an assumption of genericity on $\mathbf{\bar{x}}$. Consider the subgroup $H \subseteq \SL_2(\C)$ defined by
\begin{gather*}
H = \left\{ \left( \begin{matrix} 1 & 0 \\ 0 & 1 \end{matrix} \right), \left( \begin{matrix} 0 & 1 \\ -1 & 0 \end{matrix} \right), \left( \begin{matrix} -1 & 0 \\ 0 & -1 \end{matrix} \right), \left( \begin{matrix} 0 & -1 \\ 1 & 0 \end{matrix} \right) \right\} \cong \Z_4.
\end{gather*}
In this language of~\cite{oeding}, our condition that a (signed) array $\mathbf{\bar{x}}$ satisfies equation~\eqref{kash_for_matrices} for all $I \subseteq [n]$ and distinct $i,j,k \in [n]$ can be restated as the condition that all images of the ``unsigned version'' of $\mathbf{\bar{x}}$, under the action of the group $(H^{\times n}) \ltimes S_n$, satisfy equation~\eqref{hyperdet_id_oeding_version}. Thus, Theorem~\ref{main_all_minors_simple} requires an additional assumption of genericity and an additional equation (equation~\eqref{4_cube_coherence}) compared to Oeding's criterion, but uses a weaker version of the requirement in the second bullet point above.
\end{Remark}

\section[S-Holomorphicity in ${\mathbb Z}^2$]{S-Holomorphicity in $\boldsymbol{{\mathbb Z}^2}$} \label{s_holo_section}

In this section, we discuss a certain equation (see~\eqref{qc_eq0}) which shares many properties with the Kashaev equation. We also study a related system of equations (see~\eqref{2dh1}--\eqref{2dhs2}) which plays the role analogous to the K-hexahedron equations. The equations studied herein arise in discrete complex analysis and in the study of the Ising model (see~\cite{chelkaksmirnov} and Remark~\ref{sholo_remark}). The presentation of results in this section follows a plan similar to that of Section~\ref{main-results}. The results in this section are proved in Section~\ref{klikerecurrences} as special cases of a general axiomatic framework.

\begin{Definition} \label{2dkashdef}Given a unit square $C$ with vertices in $\Z^2$ and an array $\mathbf{x} \in \C^{\Z^2}$, define
\begin{gather}
Q^C(\mathbf{x}) = z_{00}^2 + z_{10}^2 + z_{01}^2 + z_{11}^2 - 2(z_{00} z_{10} + z_{10} z_{11} + z_{11} z_{01} + z_{01} z_{00})\nonumber\\
\hphantom{Q^C(\mathbf{x}) =}{} - 6 (z_{00} z_{11} + z_{10} z_{01}), \label{qcdef}
\end{gather}
where $z_{00}$, $z_{10}$, $z_{01}$, $z_{11}$ denote the components of $\mathbf{x}$ at the vertices of $C$, as shown in Fig.~\ref{labeledsquare}. Notice that the right-hand side of~\eqref{qcdef} is invariant under the symmetries of the square. In other words, reindexing the $4$ values using an isomorphic labeling of the square does not change the definition of~$Q^C(\mathbf{x})$.
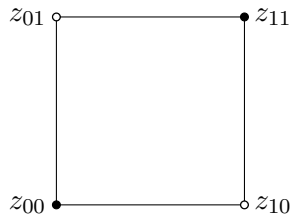
\begin{figure}[ht]\centering
\begin{tikzpicture}
\draw (0,0)--(2.5,0)--(2.5,2.5)--(0,2.5)--(0,0);
\filldraw (0,0) circle (1.5pt);
\filldraw (2.5,2.5) circle (1.5pt);
\filldraw[black] (2.5, 0) circle (1.5pt);
\filldraw[white] (2.5, 0) circle (1pt);
\filldraw[black] (0, 2.5) circle (1.5pt);
\filldraw[white] (0, 2.5) circle (1pt);
\draw (0,0) node[anchor=east]{$z_{00}$};
\draw (0,2.5) node[anchor=east]{$z_{01}$};
\draw (2.5,0) node[anchor=west]{$z_{10}$};
\draw (2.5,2.5) node[anchor=west]{$z_{11}$};
\end{tikzpicture}
\caption{Notation used in Definition~\ref{2dkashdef}.} \label{labeledsquare}
\end{figure}
\end{Definition}

\begin{Definition}Let $\mathbf{x} \in \C^{\Z^2}$. Let $v$ and $w$ be two opposite vertices in a unit square $C$ in $\Z^2$. We set
\begin{gather*}
Q^C_v(\mathbf{x}) = \frac{1}{4 \sqrt{2}} \frac{\partial Q^C}{\partial x_w}(\mathbf{x}) = \frac{1}{2 \sqrt{2}}(z_{11} - z_{10} - z_{01} - 3 z_{00}),
\end{gather*}
where we use a labeling of the components of $\mathbf{x}$ on the vertices of $C$ as in Fig.~\ref{labeledsquare}, with $z_{00}$ corresponding to the component of $\mathbf{x}$ at $v$.
\end{Definition}

\begin{Definition}
Given $v \in \Z^2$ and $i_1, i_2 \in \{-1, 1\}$, define $C_v(i_1, i_2)$ to be the unique unit square containing the vertices $v$ and $v + (i_1, i_2)$.
\end{Definition}

\begin{Proposition} \label{A2B2forQC}
Suppose that $\mathbf{x} = (x_s)_{s \in \Z^2} \in \C^{\Z^2}$ satisfies
\begin{gather} \label{qc_eq0}
Q^C(\mathbf{x}) = 0
\end{gather}
for every unit square $C$ in $\Z^2$ $($see~\eqref{qcdef}$)$. Then for any $v \in \Z^2$,
\begin{gather} \label{newprecond}
\left( \prod_{C \ni v} Q^C_v(\mathbf{x}) \right)^2 = \left( \prod_{S \ni v} (x_v + x_{v_1}) \right)^2,
\end{gather}
where
\begin{itemize}\itemsep=0pt
\item the first product is over the $4$ unit squares $C$ incident to the vertex $v$,
\item the second product is over the $4$ edges $S$ incident to $v$, and
\item $v$ and $v_1$ are the vertices of such an edge $S$.
\end{itemize}
Moreover, the following strengthening of~\eqref{newprecond} holds:
\begin{gather*}
\big( Q_v^{C_v(1,1)}(\mathbf{x}) Q_v^{C_v(-1,-1)}(\mathbf{x}) \big)^2 = \big( Q_v^{C_v(1,-1)}(\mathbf{x}) Q_v^{C_v(-1,1)}(\mathbf{x}) \big)^2 = \prod_{S \ni v} (x_v + x_{v_1}),
\end{gather*}
where the rightmost product is the same as in~\eqref{newprecond}.
\end{Proposition}

\begin{Remark} \label{get_minus_when_reverse}
The right-hand side of equation~\eqref{qcdef} is a quadratic polynomial in each of the variables $z_{ij}$. Setting this expression equal to zero and solving for $z_{11}$ in terms of $z_{00}$,~$z_{10}$,~$z_{10}$, we obtain
\begin{gather} \label{solvingforz11}
z_{11} = 3 z_{00} + z_{10} + z_{01} \pm 2 \sqrt{2} \sqrt{(z_{00} + z_{10})(z_{00} + z_{01})},
\end{gather}
where $\sqrt{(z_{00} + z_{10})(z_{00} + z_{01})}$ denotes either of the two square roots. Notice that if $z_{00}, z_{10}, z_{01} > 0$, then both solutions for $z_{11}$ in~\eqref{solvingforz11} are real; moreover, the larger of these two solutions is positive. However, solving the equation
\begin{gather} \label{solvingforz11big}
z_{11} = 3 z_{00} + z_{10} + z_{01} + 2 \sqrt{2} \sqrt{(z_{00} + z_{10})(z_{00} + z_{01})}
\end{gather}
for $z_{00}$, with $z_{10}, z_{01}, z_{11} > 0$, may result in a unique negative solution. (For example, take $z_{10} = z_{01} = z_{11} = 1$.) Hence, unlike the positive Kashaev recurrence, the equation~\eqref{solvingforz11big} only defines a recurrence on $\pos$-valued arrays in one direction.
\end{Remark}

\begin{Definition}
For $U \subseteq \Z$, let $\Z^2_U$ denote the set
\begin{gather*}
\Z^2_U = \big\{(i,j) \in \Z^2\colon i + j \in U\big\}.
\end{gather*}
\end{Definition}

\begin{Theorem} \label{s_holo_pos_thm}Suppose that $\mathbf{x} = (x_s) \in (\pos)^{\Z^2_{\{0,1,2,\dots\}}}$ satisfies~\eqref{solvingforz11big} for all $v \in \Z^2_{\{0,1,2,\dots\}}$, where we use the notation $z_{ij} = x_{v + (i,j)}$. Then for any $v \in \Z^2_{\{2,3,4,\dots\}}$, we have
\begin{gather} \label{cohforqs}
\prod_{C \ni v} Q^C_v(\mathbf{x}) = - \prod_{S \ni v} (x_v + x_{v_1}),
\end{gather}
where the notational conventions are the same as in equation~\eqref{newprecond}.
\end{Theorem}

\begin{Definition} \label{edge_e_def}Let
\begin{gather*}
E = \Z^2 + \left\{ \left(\frac{1}{2},0\right), \left(0,\frac{1}{2}\right) \right\}.
\end{gather*}
In order words, $E$ is the set of centers of edges in the tiling of $\R^2$ with unit squares.
\end{Definition}

We next state the analogue of Theorem~\ref{galoisfromintro}.

\begin{Theorem} \label{coh_for_QC_theorem} \quad
\begin{enumerate}\itemsep=0pt
\item[$(a)$] Assume that an array $\mathbf{x} = (x_s) \in \C^{\Z^2}$
\begin{itemize}\itemsep=0pt
\item satisfies the equation $Q^C(\mathbf{x}) = 0$ for all unit squares $C$,
\item satisfies equation~\eqref{cohforqs} for all $v \in \Z^2$, and
\item satisfies
\begin{gather*}
x_v + x_{v + e_i} \not= 0
\end{gather*}
for all $v \in \Z^2$ and $i \in \{1,2\}$.
\end{itemize}
Then $\mathbf{x}$ can be extended to an array $\mathbf{\tilde{x}} = (x_s) \in \C^{\Z^2 \cup E}$ satisfying the recurrence
\begin{gather}
\label{2dh1} z_{11} = 3 z_{00} + z_{10} + z_{01} + 2 \sqrt{2} z_{\frac{1}{2} 0} z_{0 \frac{1}{2}},\\
\label{2dh2} z_{\frac{1}{2} 1} = z_{\frac{1}{2} 0} + \sqrt{2} z_{0 \frac{1}{2}},\\
\label{2dh3} z_{1 \frac{1}{2}} = z_{0 \frac{1}{2}} + \sqrt{2} z_{\frac{1}{2} 0},
\end{gather}
together with the conditions
\begin{gather}
\label{2dhs1} z_{\frac{1}{2} 0}^2 = z_{00} + z_{10},\\
\label{2dhs2} z_{0 \frac{1}{2}}^2= z_{00} + z_{01},
\end{gather}
where we use the notation $z_{ij} = x_{v + (i, j)}$, for all $v \in \Z^2$.

\item[$(b)$] Conversely, suppose $\mathbf{\tilde{x}} = (x_s) \in \C^{\Z^2 \cup E}$ satisfies~\eqref{2dh1}--\eqref{2dhs2}. Then the restriction $\mathbf{x}$ of $\mathbf{\tilde{x}}$ to~$\Z^2$ satisfies $Q^C(\mathbf{x}) = 0$ for all unit squares $C$, and satisfies~\eqref{cohforqs} for all~$v \in \Z^2$.
\end{enumerate}
\end{Theorem}

\begin{Remark}Comparing Theorem~\ref{galoisfromintro} to Theorem~\ref{coh_for_QC_theorem}, we see that equations~\eqref{2dh1}--\eqref{2dhs2} play a role analogous to that of the K-hexahedron equations.
\end{Remark}

\begin{Remark} \label{sholo_remark}The equations \eqref{2dh1}--\eqref{2dhs2} appear in discrete complex analysis in the context of s-holomorphicity \cite{chelkaksmirnov}. Consider the labeling $\ell \colon E \rightarrow \C$ described in Fig.~\ref{labelingsquares}. Given $a, b \in \Z$, define the \emph{height} of $a + b {\rm i} \in \C$ to be $a + b$. If we orient the edge with midpoint $s \in E$ from the even height vertex to the odd height vertex, then $\ell(s)$ is a square root of the complex number associated with the directed edge. An \emph{s-holomorphic function} on the tiling of $\R^2$ with unit squares is a~complex-valued function $F$ on the faces of the tiling such that for any two faces~$f_1$,~$f_2$ sharing an edge with midpoint~$s$, we have
\begin{gather*}
\operatorname{Re}[\ell(s) F(f_1)] = \operatorname{Re}[\ell(s) F(f_2)].
\end{gather*}
Hence, given an $s$-holomorphic function $F$, we can define $\mathbf{x}' = (x_s) \in \R^E$ by setting
\begin{gather} \label{edgefromsholo}
x_s = \operatorname{Re}[\ell(s) F(f)]
\end{gather}
for either of the faces $f$ using the edge corresponding to $s$. It is straightforward to check that $\mathbf{x}' = (x_s) \in \R^E$ corresponds to an s-holomorphic function $F$ by~\eqref{edgefromsholo} if and only if $\mathbf{x}'$ satisfies~\eqref{2dh2}--\eqref{2dh3}. If we extend $\mathbf{x}'$ to $\mathbf{\tilde{x}} = (x_s) \in \R^{\Z^2 \cup E}$ satisfying~\eqref{2dh1}--\eqref{2dhs2}, the function $H\colon \Z^2 \rightarrow \R$ defined by
\begin{gather*}
H(i, j) = \begin{cases} x_{(i, j)} & \text{if $i + j$ is even},\\ -x_{(i, j)} & \text{if $i + j$ is odd}, \end{cases}
\end{gather*}
corresponds to a certain discrete integral. For more on this recurrence and its connections to discrete complex analysis and the Ising model, see \cite{chelkaksmirnov}.
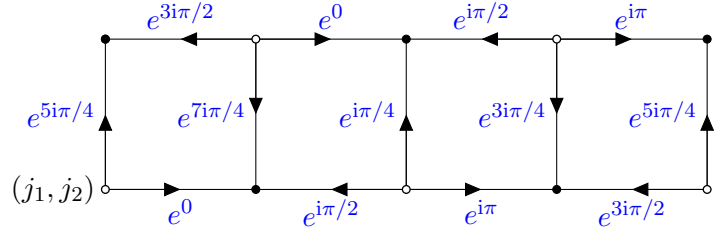
\begin{figure}[ht]\centering
\begin{tikzpicture}[>=triangle 45]
\draw (0,0)--(8,0)--(8,2)--(0,2)--(0,0);
\draw (2,0)--(2,2);
\draw (4,0)--(4,2);
\draw (6,0)--(6,2);

\draw[->] (0,0)--(1,0);
\draw[->] (4,0)--(5,0);
\draw[->] (2,2)--(3,2);
\draw[->] (6,2)--(7,2);
\draw[->] (4,0)--(3,0);
\draw[->] (8,0)--(7,0);
\draw[->] (2,2)--(1,2);
\draw[->] (6,2)--(5,2);
\draw[->] (0,0)--(0,1);
\draw[->] (4,0)--(4,1);
\draw[->] (8,0)--(8,1);
\draw[->] (2,2)--(2,1);
\draw[->] (6,2)--(6,1);

\filldraw[black] (0, 2) circle (1.5pt);
\filldraw[black] (0, 0) circle (1.5pt);
\filldraw[white] (0, 0) circle (1pt);
\filldraw[black] (2, 0) circle (1.5pt);
\filldraw[black] (2, 2) circle (1.5pt);
\filldraw[white] (2, 2) circle (1pt);
\filldraw[black] (0, 2) circle (1.5pt);
\filldraw[black] (4, 0) circle (1.5pt);
\filldraw[white] (4, 0) circle (1pt);
\filldraw[black] (4, 2) circle (1.5pt);
\filldraw[black] (6, 2) circle (1.5pt);
\filldraw[white] (6, 2) circle (1pt);
\filldraw[black] (6, 0) circle (1.5pt);
\filldraw[black] (8, 0) circle (1.5pt);
\filldraw[white] (8, 0) circle (1pt);
\filldraw[black] (8, 2) circle (1.5pt);

{ \color{blue}
\draw (1,0) node[anchor=north]{$e^0$};
\draw (3,0) node[anchor=north]{$e^{{\rm i} \pi/2}$};
\draw (5,0) node[anchor=north]{$e^{{\rm i} \pi}$};
\draw (7,0) node[anchor=north]{$e^{3 {\rm i} \pi / 2}$};
\draw (1,2) node[anchor=south]{$e^{3 {\rm i} \pi / 2}$};
\draw (3,2) node[anchor=south]{$e^{0}$};
\draw (5,2) node[anchor=south]{$e^{{\rm i} \pi/2}$};
\draw (7,2) node[anchor=south]{$e^{{\rm i} \pi}$};
\draw (0,1) node[anchor=east]{$e^{5 {\rm i} \pi / 4}$};
\draw (2,1) node[anchor=east]{$e^{7 {\rm i} \pi / 4}$};
\draw (4,1) node[anchor=east]{$e^{{\rm i} \pi / 4}$};
\draw (6,1) node[anchor=east]{$e^{3 {\rm i} \pi / 4}$};
\draw (8,1) node[anchor=east]{$e^{5 {\rm i} \pi / 4}$};
}
\draw (0,0) node[anchor=east]{$(j_1,j_2)$};
\end{tikzpicture}
\caption{Define the labeling $\ell\colon E \rightarrow \C$ invariant under translations by the vectors $(1,1)$ and $(4, 0)$ as follows. In the vicinity of a point $(j_1, j_2) \in \Z^2$ satisfying $j_1 - j_2 \equiv 0$ (mod $4$), the labeling is given by the values shown in blue in the figure.}\label{labelingsquares}
\end{figure}
\end{Remark}

\section{Further generalizations of the Kashaev equation} \label{gen_from_ca}

In this section, we provide two additional examples of equations with behavior similar to the Kashaev equation and its analogue~\eqref{qcdef}. In Section~\ref{klikerecurrences}, we shall develop a general framework which will allow us to prove all of the results in this section (as well as the results in Section~\ref{s_holo_section}).

Both recurrences considered in this section come with complex parameters that one can choose arbitrarily. For certain values of these parameters, the corresponding recurrences have cluster algebra-like behavior. We will explore the cluster algebra nature of these recurrences in future work.

We begin with the following, relatively simple, one-dimensional example:

\begin{Proposition} \label{A2B2propgen1d}Let $\alpha_1, \alpha_2, \alpha_3 \in \C$. Let $\mathbf{x} = (x_s) \in \C^{\Z}$ be an array such that
\begin{gather} \label{1dexamplekexplicit}
z_0^2 z_3^2 + \alpha_1 z_1^2 z_2^2 + \alpha_2 z_0 z_1 z_2 z_3 + \alpha_3 \big(z_0 z_2^3 + z_1^3 z_3\big) = 0
\end{gather}
for all $v \in \Z$, where $z_i = x_{v + i}$. Then
\begin{gather} \label{1d_example_A2_B2}
\big(2 z_0^2 z_3 + \alpha_2 z_0 z_1 z_2 + \alpha_3 z_1^3\big)^2 = \big(2 z_{-1} z_2^2 + \alpha_2 z_0 z_1 z_2 + \alpha_3 z_1^3\big)^2 = D,
\end{gather}
where
\begin{gather} \label{D_for_1dexamplekexplicit}
D = \alpha_3^2 z_1^6 + 2 \alpha_2 \alpha_3 z_0 z_1^4 z_2 + \big(\alpha_2^2 - 4 \alpha_1\big) z_0^2 z_1^2 z_2^2 - 4 \alpha_3 z_0^3 z_2^3,
\end{gather}
for all $v \in \Z$, where again $z_i = x_{v + i}$.
\end{Proposition}

\begin{Remark} \label{1d_example_a1_expl}The equation~\eqref{1dexamplekexplicit} involves $4$ entries of the array $\mathbf{x}$ indexed by $4$ consecutive integers $v$, $v + 1$, $v + 2$, $v + 3$. This equation is invariant under the central symmetry of the line segment $[v, v + 3]$. In other words, equation~\eqref{1dexamplekexplicit} is invariant under interchanging $z_0$ with~$z_3$, and~$z_1$ with~$z_2$. The value $D$ in~\eqref{D_for_1dexamplekexplicit} is the discriminant of equation~\eqref{1dexamplekexplicit} viewed as a quadratic equation in~$z_3$. The term squared on the left-hand side of~\eqref{1d_example_A2_B2} is the partial derivative of the left-hand side of~\eqref{1dexamplekexplicit} with respect to $z_3$. This expression plays the role of~$K^C_v$ in Proposition~\ref{A2B2proposition}.
\end{Remark}

\begin{Theorem} \label{pos_thm_1d_example}
Let $\alpha_2, \alpha_3 \le 0$, and $\alpha_1 \le \alpha_2^2 / 4$. Let $\mathbf{x} = (x_s) \in (\pos)^{\Z}$ satisfy the equation~\eqref{1dexamplekexplicit} for all $v \in \Z$, where, as before, we denote $z_i = x_{v + i}$. Moreover, assume that for all $v \in \Z$, the number $z_3 = x_{v + 3}$ is the larger of the two real solutions of~\eqref{1dexamplekexplicit}:
\begin{gather*}
z_3 = \frac{- \alpha_3 z_1^3 - \alpha_2 z_0 z_1 z_2 + \sqrt{D}}{2 z_0^2},
\end{gather*}
where $D$ is given by~\eqref{D_for_1dexamplekexplicit}. Then
\begin{gather*}
2 z_0^2 z_3 + \alpha_2 z_0 z_1 z_2 + \alpha_3 z_1^3 = 2 z_{-1} z_2^2 + \alpha_2 z_0 z_1 z_2 + \alpha_3 z_1^3,
\end{gather*}
or equivalently,
\begin{gather} \label{1d_example_coherence}
z_0^2 z_3 = z_{-1} z_2^2
\end{gather}
for all $v \in \Z$.
\end{Theorem}

We next show that an array $\mathbf{x}$ satisfying~\eqref{1dexamplekexplicit} satisfies condition~\eqref{1d_example_coherence} if and only if it can be extended to an array on a larger index set satisfying conditions resembling the K-hexahedron equations.

\begin{Theorem} \label{main_thm_1d_example}Let $\alpha_1, \alpha_2, \alpha_3 \in \C$.
\begin{enumerate}\itemsep=0pt
\item[$(a)$] For any array $\mathbf{x} = (x_s) \in (\C^*)^{\Z}$ satisfying~\eqref{1dexamplekexplicit} and~\eqref{1d_example_coherence}, there exists an array $\mathbf{y} = ( y_s)_{s \in \Z}$ such that $\mathbf{x}$ and $\mathbf{y}$ together satisfy the recurrence
\begin{gather} \label{1d_v_hex_prop}
z_3 = \frac{- \alpha_3 z_1^3 - \alpha_2 z_0 z_1 z_2 + w_1}{2 z_0^2},\\ \label{1d_f_hex_prop}
w_2 = \frac{\alpha_3^2 z_1^6 + \alpha_2 \alpha_3 z_0 z_1^4 z_2 + 2 \alpha_3 z_0^3 z_2^3 + w_1^2 + \big({-}2 \alpha_3 z_1^3 - \alpha_2 z_0 z_1 z_2\big) w_1}{2 z_0^3}
\end{gather}
together with the condition
\begin{gather} \label{1d_interval_spec}
w_1^2 = D,
\end{gather}
where $D$ is given by~\eqref{D_for_1dexamplekexplicit}, and we use the notation $z_i = x_{v + i}$ and $w_i = y_{v + i}$.
\item[$(b)$] Conversely, suppose $\mathbf{x} \in (\C^*)^{\Z}$ and $\mathbf{y} \in \C^{\Z}$ satisfy~\eqref{1d_v_hex_prop}--\eqref{1d_interval_spec}. Then $\mathbf{x}$ satisfies~\eqref{1dexamplekexplicit} and~\eqref{1d_example_coherence}.
\end{enumerate}
\end{Theorem}

\begin{Remark}The components of $\mathbf{y}$ are most naturally indexed by intervals of length~$2$ in $\Z$. Here, we index the components of $\mathbf{y}$ by the midpoints of those intervals.
\end{Remark}

\begin{Remark}If one sets $(\alpha_1, \alpha_2, \alpha_3) = (0,0,-4)$ or $(\alpha_1, \alpha_2, \alpha_3) = (-3,-6,-4)$, then the pairs of arrays $\mathbf{x}$, $\mathbf{y}$ satisfying~\eqref{1d_v_hex_prop}--\eqref{1d_interval_spec} are special cases of recurrences that arise from cluster algebras. The $(\alpha_1, \alpha_2, \alpha_3) = (-3,-6,-4)$ case is a special case of the K-hexahedron equations; namely, such pairs $\mathbf{x}, \mathbf{y}$ correspond to isotropic solutions $\mathbf{\tilde{z}} = (z_s) \in \C^{\lat}$ of the K-hexahedron equations where
\begin{gather*}
z_{(i,j,k)}= x_{i + j + k},\\
z_{(i,j,k) + \left( 0, \half, \half \right)} = z_{(i,j,k) + \left( \half, 0, \half \right)} = z_{(i,j,k) + \left( \half, \half, 0 \right)},\\
4 z_{(i,j,k) + \left( 0, \half, \half \right)}^3 = y_{i + j + k + 1},
\end{gather*}
for $(i,j,k) \in \Z^3$. However, the $(\alpha_1, \alpha_2, \alpha_3) = (0,0,-4)$ case is not a special case of the K-hexahedron equations. We will discuss the related cluster algebra recurrences in later work.
\end{Remark}

Next, we consider the following two-dimensional example:

\begin{Proposition} \label{A2B2_2d_example}Let $\alpha_1, \alpha_2 \in \C$. Let $\mathbf{x} = (x_s) \in (\C)^{\Z^2}$ be an array such that
\begin{gather}
0 = z_{00}^2 z_{12}^2 + z_{10}^2 z_{02}^2 + \frac{\alpha_2^2 - \alpha_1^2}{4} z_{01}^2 z_{11}^2 - \alpha_1 \big(z_{00} z_{02} z_{11}^2 + z_{10} z_{12} z_{01}^2\big)\nonumber\\
\hphantom{0=}{} - 2 z_{00} z_{10} z_{02} z_{12} - \alpha_2 (z_{00} z_{12} z_{01} z_{11} + z_{10} z_{02} z_{01} z_{11})\label{2d_example_eqn}
\end{gather}
for all $v \in \Z^2$, where $z_{ij} = x_{v + (i, j)}$. Given a $1 \times 2$ rectangle $B$ with vertices in~$\Z^2$, a vertex $w \in \Z^2$ at a corner of $B$, and the components of~$\mathbf{x}$ at the $6$ points of $\Z^2$ in $B$ labeled as in Fig.~{\rm \ref{1x2_labeled_figure}}, define $R^{B, 0}_{w}(\mathbf{x}) \in \C$ by
\begin{gather} \label{rb0_def}
R^{B,0}_{w}(\mathbf{x}) = z_{00}^2 z_{12} - \alpha_1 z_{10} z_{01}^2 - 2 z_{00} z_{10} z_{02} - \alpha_2 z_{00} z_{01} z_{11}.
\end{gather}
Given a $0 \times 2$ rectangle $($line segment$)$ $S$ with vertices in $\Z^2$, and the components of~$\mathbf{x}$ at the $3$ points of $\Z^2$ in $S$ labeled as in Fig.~{\rm \ref{1x2_labeled_figure}}, define $R^{S, 1}(\mathbf{x}) \in \C$ by
\begin{gather} \label{rs1_def}
R^{S, 1}(\mathbf{x}) = \alpha_1 z_{01}^2 + 4 z_{00} z_{02}.
\end{gather}
Given a $1 \times 1$ square $C$ with vertices in $\Z^2$, and the components of $\mathbf{x}$ at the $4$ points of $\Z^2$ in $C$ labeled as in Fig.~{\rm \ref{1x2_labeled_figure}}, define $R^{C, 2}(\mathbf{x}) \in \C$ by
\begin{gather} \label{rc2_def}
R^{C, 2}(\mathbf{x}) = \alpha_1\big(z_{00}^2 z_{11}^2 + z_{01}^2 z_{10}^2\big) + 2 \alpha_2 z_{00} z_{01} z_{10} z_{11}.
\end{gather}
Then for any $v \in \Z^2$,
\begin{gather} \label{2d_example_precoherence}
\left(\prod_{i = 1}^4 R^{B_i, 0}_{w_i}(\mathbf{x}) \right)^2 = \big( R^{S_1, 1}(\mathbf{x}) R^{S_2, 1}(\mathbf{x}) R^{C_1, 2}(\mathbf{x}) R^{C_2, 2}(\mathbf{x}) \big)^2,
\end{gather}
where
\begin{itemize}\itemsep=0pt
\item $(B_1, w_1)$, $(B_2, w_2)$, $(B_3, w_3)$, $(B_4, w_4)$ are the four $1 \times 2$ rectangle/corner pairs shown in Fig.~{\rm \ref{rectangles_in_coherence}},
\item $S_1$, $S_2$ are the two $0 \times 2$ rectangles $($line segments$)$ shown in Fig.~{\rm \ref{rectangles_in_coherence}}, and
\item $C_1$, $C_2$ are the two $1 \times 1$ squares shown in Fig.~{\rm \ref{rectangles_in_coherence}}.
\end{itemize}
Moreover, the following strengthening of~\eqref{2d_example_precoherence} holds:
\begin{gather*}
\big( R^{B_1, 0}_{w_1}(\mathbf{x}) R^{B_3, 0}_{w_3}(\mathbf{x}) \big)^2 = \big( R^{B_2, 0}_{w_2}(\mathbf{x}) R^{B_4, 0}_{w_4}(\mathbf{x}) \big)^2.
\end{gather*}

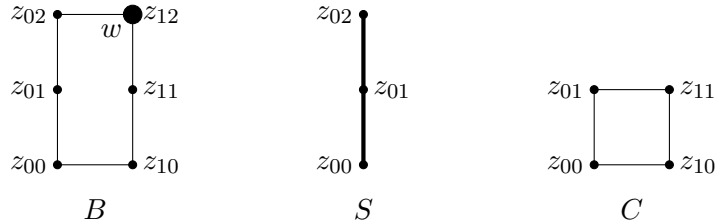
\begin{figure}[ht]\centering
\begin{tabular}{ c c c }
\begin{tikzpicture}
\draw (0,0)--(0,2)--(1,2)--(1,0)--(0,0);
\filldraw[black] (0, 0) circle (1.5pt);
\filldraw[black] (0,1) circle (1.5pt);
\filldraw[black] (0,2) circle (1.5pt);
\filldraw[black] (1, 0) circle (1.5pt);
\filldraw[black] (1, 1) circle (1.5pt);
\filldraw[black] (1, 2) circle (3.5pt);
\draw (0,0) node[anchor=east]{$z_{00}$};
\draw (0,1) node[anchor=east]{$z_{01}$};
\draw (0,2) node[anchor=east]{$z_{02}$};
\draw (1,0) node[anchor=west]{$z_{10}$};
\draw (1,1) node[anchor=west]{$z_{11}$};
\draw (1,2) node[anchor=west]{$z_{12}$};
\draw (1,2) node[anchor=north east]{$w$};
\end{tikzpicture} & \hspace{1cm}
\begin{tikzpicture}
\draw[ultra thick] (0,0)--(0,2);
%\fill[gray] (0,0) rectangle (0,2);
\filldraw[black] (0, 0) circle (1.5pt);
\filldraw[black] (0,1) circle (1.5pt);
\filldraw[black] (0,2) circle (1.5pt);
\draw (0,0) node[anchor=east]{$z_{00}$};
\draw (0,1) node[anchor=west]{$z_{01}$};
\draw (0,2) node[anchor=east]{$z_{02}$};
\end{tikzpicture} & \hspace{1cm}
\begin{tikzpicture}
\draw (0,0)--(1,0)--(1,1)--(0,1)--(0,0);
\filldraw[black] (0, 0) circle (1.5pt);
\filldraw[black] (0,1) circle (1.5pt);
\filldraw[black] (1, 0) circle (1.5pt);
\filldraw[black] (1, 1) circle (1.5pt);
\draw (0,0) node[anchor=east]{$z_{00}$};
\draw (0,1) node[anchor=east]{$z_{01}$};
\draw (1,0) node[anchor=west]{$z_{10}$};
\draw (1,1) node[anchor=west]{$z_{11}$};
\end{tikzpicture} \\ $B$ & \hspace{1cm} $S$ & \hspace{1cm} $C$
\end{tabular}
\caption{The components of $\mathbf{x}$ at a $1 \times 2$ rectangle $B$ with distinguished vertex $w$ in~\eqref{rb0_def}, at a $0 \times 2$ rectangle (line segment) $S$ in~\eqref{rs1_def}, and at a $1 \times 1$ square $C$ in~\eqref{rc2_def}.} \label{1x2_labeled_figure}
\end{figure}
\begin{figure}[ht]\centering
\begin{tabular}{ c c c c }
\begin{tikzpicture}[scale=0.8]

\draw (1,1)--(2,1)--(2,3)--(1,3)--(1,1);
\filldraw[black] (0, 0) circle (1.5pt);
\filldraw[black] (0,1) circle (1.5pt);
\filldraw[black] (0,2) circle (1.5pt);
\filldraw[black] (0,3) circle (1.5pt);
\filldraw[black] (1, 0) circle (1.5pt);
\filldraw[black] (1, 1) circle (3.5pt);
\filldraw[black] (1, 2) circle (1.5pt);
\filldraw[black] (1, 3) circle (1.5pt);
\filldraw[black] (2, 0) circle (1.5pt);
\filldraw[black] (2, 1) circle (1.5pt);
\filldraw[black] (2, 2) circle (1.5pt);
\filldraw[black] (2, 3) circle (1.5pt);

\filldraw[white] (2, 3) circle (1pt);

\draw (1.5,2) node{$B_1$};
\draw (1,1) node[anchor=north east]{$v$};
\draw (2,3) node[anchor=south west]{$w_1$};

\filldraw[white] (1,-1.5) circle (1pt);
\end{tikzpicture} & \hspace{1cm}
\begin{tikzpicture}[scale=0.8]
\draw (1,1)--(0,1)--(0,3)--(1,3)--(1,1);
\filldraw[black] (0, 0) circle (1.5pt);
\filldraw[black] (0,1) circle (1.5pt);
\filldraw[black] (0,2) circle (1.5pt);
\filldraw[black] (0,3) circle (1.5pt);
\filldraw[black] (1, 0) circle (1.5pt);
\filldraw[black] (1, 1) circle (3.5pt);
\filldraw[black] (1, 2) circle (1.5pt);
\filldraw[black] (1, 3) circle (1.5pt);
\filldraw[black] (2, 0) circle (1.5pt);
\filldraw[black] (2, 1) circle (1.5pt);
\filldraw[black] (2, 2) circle (1.5pt);
\filldraw[black] (2, 3) circle (1.5pt);

\filldraw[white] (0,3) circle (1pt);

\draw (0.5,2) node{$B_2$};
\draw (1,1) node[anchor=north west]{$v$};
\draw (0,3) node[anchor=south east]{$w_2$};

\filldraw[white] (1,-1.5) circle (1pt);
\end{tikzpicture} & \hspace{1cm}
\begin{tikzpicture}[scale=0.8]
\draw (0,0)--(1,0)--(1,2)--(0,2)--(0,0);
\filldraw[black] (0, 0) circle (1.5pt);
\filldraw[black] (0,1) circle (1.5pt);
\filldraw[black] (0,2) circle (1.5pt);
\filldraw[black] (0,3) circle (1.5pt);
\filldraw[black] (1, 0) circle (1.5pt);
\filldraw[black] (1, 1) circle (3.5pt);
\filldraw[black] (1, 2) circle (1.5pt);
\filldraw[black] (1, 3) circle (1.5pt);
\filldraw[black] (2, 0) circle (1.5pt);
\filldraw[black] (2, 1) circle (1.5pt);
\filldraw[black] (2, 2) circle (1.5pt);
\filldraw[black] (2, 3) circle (1.5pt);

\filldraw[white] (0,0) circle (1pt);

\draw (0.5,1) node{$B_3$};
\draw (1,1) node[anchor=west]{$v$};
\draw (0,0) node[anchor= north east]{$w_3$};

\filldraw[white] (1,-1.5) circle (1pt);
\end{tikzpicture} & \hspace{1cm}
\begin{tikzpicture}[scale=0.8]
\draw (1,0)--(2,0)--(2,2)--(1,2)--(1,0);
\filldraw[black] (0, 0) circle (1.5pt);
\filldraw[black] (0,1) circle (1.5pt);
\filldraw[black] (0,2) circle (1.5pt);
\filldraw[black] (0,3) circle (1.5pt);
\filldraw[black] (1, 0) circle (1.5pt);
\filldraw[black] (1, 1) circle (3.5pt);
\filldraw[black] (1, 2) circle (1.5pt);
\filldraw[black] (1, 3) circle (1.5pt);
\filldraw[black] (2, 0) circle (1.5pt);
\filldraw[black] (2, 1) circle (1.5pt);
\filldraw[black] (2, 2) circle (1.5pt);
\filldraw[black] (2, 3) circle (1.5pt);

\filldraw[white] (2,0) circle (1pt);

\draw (1.5,1) node{$B_4$};
\draw (1,1) node[anchor=east]{$v$};
\draw (2,0) node[anchor = north west]{$w_4$};

\filldraw[white] (1,-1.5) circle (1pt);
\end{tikzpicture} \\
\begin{tikzpicture}[scale=0.8]
\draw[ultra thick] (1,1)--(1,3);
\filldraw[black] (0, 0) circle (1.5pt);
\filldraw[black] (0,1) circle (1.5pt);
\filldraw[black] (0,2) circle (1.5pt);
\filldraw[black] (0,3) circle (1.5pt);
\filldraw[black] (1, 0) circle (1.5pt);
\filldraw[black] (1, 1) circle (3.5pt);
\filldraw[black] (1, 2) circle (1.5pt);
\filldraw[black] (1, 3) circle (1.5pt);
\filldraw[black] (2, 0) circle (1.5pt);
\filldraw[black] (2, 1) circle (1.5pt);
\filldraw[black] (2, 2) circle (1.5pt);
\filldraw[black] (2, 3) circle (1.5pt);

\draw (1,2) node[anchor=west]{$S_1$};
\draw (1,1) node[anchor=east]{$v$};

\end{tikzpicture} & \hspace{1cm}
\begin{tikzpicture}[scale=0.8]
\draw[ultra thick] (1,0)--(1,2);
\filldraw[black] (0, 0) circle (1.5pt);
\filldraw[black] (0,1) circle (1.5pt);
\filldraw[black] (0,2) circle (1.5pt);
\filldraw[black] (0,3) circle (1.5pt);
\filldraw[black] (1, 0) circle (1.5pt);
\filldraw[black] (1, 1) circle (3.5pt);
\filldraw[black] (1, 2) circle (1.5pt);
\filldraw[black] (1, 3) circle (1.5pt);
\filldraw[black] (2, 0) circle (1.5pt);
\filldraw[black] (2, 1) circle (1.5pt);
\filldraw[black] (2, 2) circle (1.5pt);
\filldraw[black] (2, 3) circle (1.5pt);

\draw (1,1) node[anchor=west]{$S_2$};
\draw (1,1) node[anchor=east]{$v$};

\end{tikzpicture} & \hspace{1cm}
\begin{tikzpicture}[scale=0.8]
\draw (1,1)--(2,1)--(2,2)--(1,2)--(1,1);
\filldraw[black] (0, 0) circle (1.5pt);
\filldraw[black] (0,1) circle (1.5pt);
\filldraw[black] (0,2) circle (1.5pt);
\filldraw[black] (0,3) circle (1.5pt);
\filldraw[black] (1, 0) circle (1.5pt);
\filldraw[black] (1, 1) circle (3.5pt);
\filldraw[black] (1, 2) circle (1.5pt);
\filldraw[black] (1, 3) circle (1.5pt);
\filldraw[black] (2, 0) circle (1.5pt);
\filldraw[black] (2, 1) circle (1.5pt);
\filldraw[black] (2, 2) circle (1.5pt);
\filldraw[black] (2, 3) circle (1.5pt);

\draw (1.5,1.5) node{$C_1$};
\draw (1,1) node[anchor=east]{$v$};

\end{tikzpicture} & \hspace{1cm}
\begin{tikzpicture}[scale=0.8]
\draw (1,1)--(0,1)--(0,2)--(1,2)--(1,1);
\filldraw[black] (0, 0) circle (1.5pt);
\filldraw[black] (0,1) circle (1.5pt);
\filldraw[black] (0,2) circle (1.5pt);
\filldraw[black] (0,3) circle (1.5pt);
\filldraw[black] (1, 0) circle (1.5pt);
\filldraw[black] (1, 1) circle (3.5pt);
\filldraw[black] (1, 2) circle (1.5pt);
\filldraw[black] (1, 3) circle (1.5pt);
\filldraw[black] (2, 0) circle (1.5pt);
\filldraw[black] (2, 1) circle (1.5pt);
\filldraw[black] (2, 2) circle (1.5pt);
\filldraw[black] (2, 3) circle (1.5pt);

\draw (0.5,1.5) node{$C_2$};
\draw (1,1) node[anchor=west]{$v$};

\end{tikzpicture}
\end{tabular}
\caption{On the top row, the rectangle/corner pairs $(B_i, w_i)$ for $i = 1, \dots, 4$, and on the bottom row, the line segments $S_1$, $S_2$ and the squares $C_1$, $C_2$ that appear in~\eqref{2d_example_precoherence}.} \label{rectangles_in_coherence}
\end{figure}
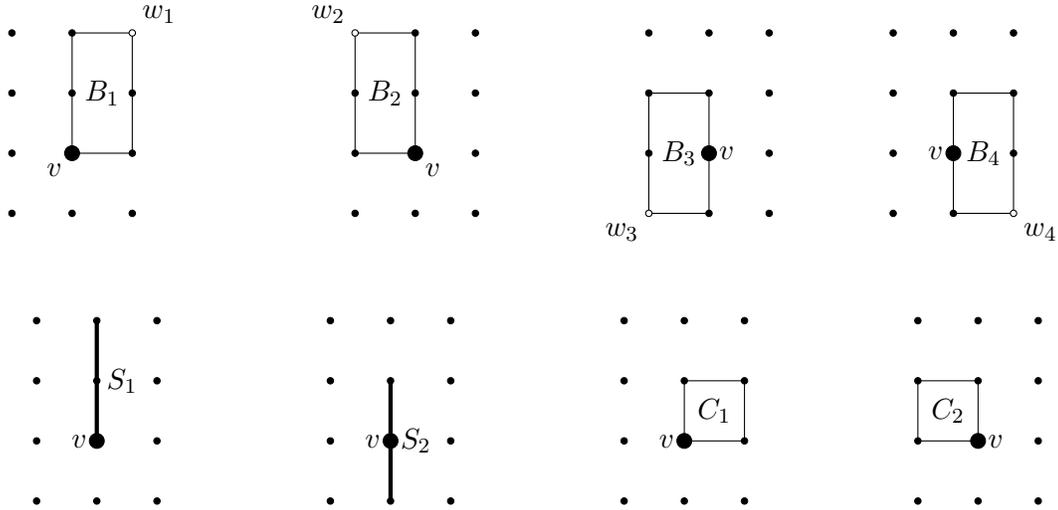
\end{Proposition}

\begin{Remark}Let $A = \{v, v + (0,1)\}$. In equation~\eqref{2d_example_precoherence}, $B_1$, $B_2$, $B_3$, $B_4$ are the four $1 \times 2$ rectangles with vertices in $\Z^2$ containing~$A$, $S_1$, $S_2$ are the two $0 \times 2$ rectangles (line segments) with vertices in $\Z^2$ containing~$A$, and $C_1$, $C_2$ are the two $1 \times 1$ unit squares with vertices in $\Z^2$ containing~$A$. Each $w_i$ is the farthest vertex from $A$ in~$B_i$.
\end{Remark}

\begin{Theorem} \label{pos_thm_2d_example}Let $\alpha_1, \alpha_2 \ge 0$. Define $R^{B, 0}_w$, $R^{S, 1}$, and $R^{C, 2}$ as in~\eqref{rb0_def}--\eqref{rc2_def}. Let $\mathbf{x} = (x_s) \in (\pos)^{\Z^2}$ satisfy the equation~\eqref{2d_example_eqn}, where, as before, we denote $z_{ij} = x_{v + (i,j)}$. Moreover, assume that for all $v \in \Z^2$, the number $z_{12} = x_{v + (1,2)}$ is the larger of the two real solutions of~\eqref{2d_example_eqn}:
\begin{gather*}
z_{12} = \frac{\alpha_1 z_{10} z_{01}^2 + 2 z_{00} z_{10} z_{02} + \alpha_2 z_{00} z_{01} z_{11} + \sqrt{D}}{2 z_{00}^2},
\end{gather*}
where
\begin{gather*}
D = \big(\alpha_1 z_{01}^2 + 4 z_{00} z_{02}\big)\big(\alpha_1\big(z_{00}^2 z_{11}^2 + z_{01}^2 z_{10}^2\big) + 2 \alpha_2 z_{00} z_{01} z_{10} z_{11}\big) = R^{S, 1}(\mathbf{x}) R^{C, 2}(\mathbf{x}),
\end{gather*}
and $S = v + (\{0\} \times \{0,1,2\})$, $C = v + \{0, 1\}^2$. Then
\begin{gather} \label{2d_example_coherence}
\prod_{i = 1}^4 R^{B_i, 0}_{w_i}(\mathbf{x}) = R^{S_1, 1}(\mathbf{x}) R^{S_2, 1}(\mathbf{x}) R^{C_1, 2}(\mathbf{x}) R^{C_2, 2}(\mathbf{x}),
\end{gather}
and
\begin{gather*}
R^{B_1, 0}_{w_1}(\mathbf{x}) R^{B_3, 0}_{w_3}(\mathbf{x}) = R^{B_2, 0}_{w_2}(\mathbf{x}) R^{B_4, 0}_{w_4}(\mathbf{x}),
\end{gather*}
for all $v \in \Z^2$, using the same notation as in equation~\eqref{2d_example_precoherence}.
\end{Theorem}

\begin{Theorem} \label{main_thm_2d_example}Let $\alpha_1, \alpha_2 \in \C$.
\begin{enumerate}[$(a)$]\itemsep=0pt
\item Suppose $\mathbf{x} = (x_s) \in (\C^*)^{\Z^2}$ satisfies~\eqref{2d_example_eqn} and~\eqref{2d_example_coherence}, and $R^{S, 1}(\mathbf{x}) \not= 0$ and $R^{C,2}(\mathbf{x}) \not= 0$ for any $0 \times 2$ rectangle $S$ with vertices in $\Z^2$ and $1 \times 1$ square $C$ with vertices in $\Z^2$. Then there exist arrays $\mathbf{y}_1 = (y_s)_{s \in \Z^2}$ and $\mathbf{y}_2 = \left( y_s \right)_{s \in \left( \Z^2 + \left( \half, \half \right) \right)}$ such that $\mathbf{x}$, $\mathbf{y}_1$, and $\mathbf{y}_2$ together satisfy the recurrence
\begin{gather} \label{2d_v_hex_prop}
z_{12} = \frac{\alpha_1 z_{10} z_{01}^2 + 2 z_{00} z_{10} z_{02} + \alpha_2 z_{00} z_{01} z_{11} + w_{01} w_{\half \half}}{2 z_{00}^2},\\ \label{2d_1f_hex_prop}
w_{11} = \frac{z_{10} w_{01} + w_{\half \half}}{z_{00}},\\ \label{2d_2f_hex_prop}
w_{\frac{3}{2} \half} = \frac{z_{01}(\alpha_1 z_{01} z_{10} + \alpha_2 z_{00} z_{11}) w_{01} + (\alpha_1 z_{01}^2 + 2 z_{00} z_{02}) w_{\half \half}}{2 z_{00}^2},
\end{gather}
together with the conditions
\begin{gather} \label{2d_strip_face_spec}
w_{01}^2 = \alpha_1 z_{01}^2 + 4 z_{00} z_{02},\\
\label{2d_square_face_spec} w_{\half \half}^2 = \alpha_1\big(z_{00}^2 z_{11}^2 + z_{01}^2 z_{10}^2\big) + 2 \alpha_2 z_{00} z_{01} z_{10} z_{11},
\end{gather}
where $z_{ij} = x_{v + (i,j)}$ and $w_{ij} = y_{v + (i,j)}$.
\item Conversely, suppose $\mathbf{x} \in (\C^*)^{\Z^2}$, $\mathbf{y}_1 \in \C^{\Z^2}$, and $\mathbf{y}_2 \in \C^{\Z^2 + \left( \half, \half \right)}$ satisfy~\eqref{2d_v_hex_prop}--\eqref{2d_square_face_spec}. Then~$\mathbf{x}$ satisfies~\eqref{2d_example_eqn} and~\eqref{2d_example_coherence}.
\end{enumerate}
\end{Theorem}

\begin{Remark}In Theorem~\ref{main_thm_2d_example}, the components of $\mathbf{y}_1$ are most naturally associated to $0 \times 2$ rectangles (line segments) with vertices in $\Z^2$ (although we index it by the center of the line segment in the theorem), and the components of $\mathbf{y}_2$ are most naturally associated to $1 \times 1$ unit squares with vertices in $\Z^2$ (although we index it by the center of the square in the theorem). If we think about the recurrence~\eqref{2d_v_hex_prop}--\eqref{2d_2f_hex_prop} in this way, each step of the recurrence uses the six vertices, two $0 \times 2$ rectangles, and two $1 \times 1$ unit squares contained in the $1 \times 2$ rectangle $v + \{0, 1\} \times \{0,1,2\}$, as is pictured in Fig.~\ref{expl_of_2d_example_rec}.
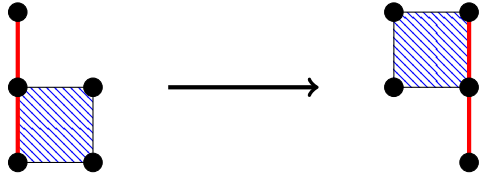
\begin{figure}[ht]\centering
\begin{tikzpicture}
\draw[pattern=north west lines, pattern color=blue] (0,0) rectangle (1,1);

\draw[red, ultra thick] (0,0)--(0,2);

\filldraw[black] (0, 0) circle (3.5pt);
\filldraw[black] (0,1) circle (3.5pt);
\filldraw[black] (0,2) circle (3.5pt);
\filldraw[black] (1, 0) circle (3.5pt);
\filldraw[black] (1, 1) circle (3.5pt);
%\filldraw[black] (1, 2) circle (1.5pt);

\draw[ultra thick, ->] (2,1)--(4,1);

\draw[pattern=north west lines, pattern color=blue] (5,1) rectangle (6,2);

\draw[red, ultra thick] (6,0)--(6,2);

%\filldraw[black] (0, 0) circle (3.5pt);
\filldraw[black] (5,1) circle (3.5pt);
\filldraw[black] (5,2) circle (3.5pt);
\filldraw[black] (6, 0) circle (3.5pt);
\filldraw[black] (6, 1) circle (3.5pt);
\filldraw[black] (6, 2) circle (3.5pt);
\end{tikzpicture}
\caption{The vertices/line segments/squares indexing the values in a step of the recurrence~\eqref{2d_v_hex_prop}--\eqref{2d_2f_hex_prop}. The black vertices index the values of $\mathbf{x}$, the red line segments index the values of $\mathbf{y}_1$, and the blue squares index the values of~$\mathbf{y}_2$.} \label{expl_of_2d_example_rec}
\end{figure}
\end{Remark}

\begin{Remark}If one sets $(\alpha_1, \alpha_2) = (4, 4)$ or $(\alpha_1, \alpha_2) = (4,0)$, then the arrays $\mathbf{x}$, $\mathbf{y}_1$, $\mathbf{y}_2$ satisfying~\eqref{2d_v_hex_prop}--\eqref{2d_square_face_spec} are special cases of recurrences that arise from cluster algebras. The $(\alpha_1, \alpha_2) = (4, 4)$ case is a special case of the K-hexahedron equations; namely, such $\mathbf{x}$, $\mathbf{y}_1$, $\mathbf{y}_2$ correspond to isotropic solutions $\mathbf{\tilde{z}} = (z_s) \in \C^{\lat}$ of the K-hexahedron equations where
\begin{gather*}
z_{(i,j,k)} = x_{(i, j + k)},\\
2 z_{(i,j,k)} = y_{(i, j + k + 1)},\\
z_{(i,j,k) + \left( \half, 0, \half \right)} = z_{(i,j,k) + \left( \half, \half, 0 \right)},\\
2 z_{(i,j,k) + \left( \half, 0, \half \right)}^2 = y_{\left( i + \half, j + k + \half \right)}
\end{gather*}
for $(i,j,k) \in \Z^3$. However, the $(\alpha_1, \alpha_2) = (4,0)$ case is not a special case of the K-hexahedron equations. We will discuss the related cluster algebra recurrences in future work. Using machinery from cluster algebras, we can prove Laurentness results for these recurrences similar to the one given by Kenyon and Pemantle in~{\cite[Theorem~7.8]{mainkp}}.
\end{Remark}

\section{Proofs of results from Section~\ref{main-results}} \label{mainproof}

In this section, we prove Proposition~\ref{A2B2proposition} and Theorems~\ref{galoisfromintro}--\ref{extensions_that_agree} (of which all other results in Section~\ref{main-results} are corollaries).

\begin{Lemma} \label{KCvgenxlemma}Let $C$ be a cube with vertices $V(C)$ labeled as in Fig.~{\rm \ref{labeledcubeforprop}}, and let $\mathbf{x} = (x_s) \in \C^{V(C)}$. Then
\begin{gather*}
\big(K_v^C(\mathbf{x})\big)^2 = \frac{x_v^2}{4} K^C(\mathbf{x}) + (x_v x_d + x_b x_c)(x_v x_e + x_a x_c)(x_v x_f + x_a x_b).
\end{gather*}
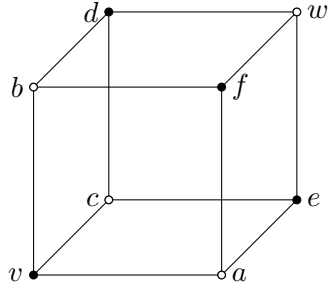
\begin{figure}[ht]\centering
\begin{tikzpicture}
\draw (0,0)--(2.5,0)--(2.5,2.5)--(0,2.5)--(0,0);
\draw (1,1)--(3.5,1)--(3.5,3.5)--(1,3.5)--(1,1);
\draw (0,0)--(1,1);
\draw (2.5,0)--(3.5,1);
\draw (2.5, 2.5)--(3.5, 3.5);
\draw (0,2.5)--(1,3.5);

\filldraw (0,0) circle (1.5pt);
\filldraw (2.5,2.5) circle (1.5pt);
\filldraw (3.5,1) circle (1.5pt);
\filldraw (1,3.5) circle (1.5pt);

\filldraw[black] (2.5, 0) circle (1.5pt);
\filldraw[white] (2.5, 0) circle (1pt);
\filldraw[black] (0, 2.5) circle (1.5pt);
\filldraw[white] (0, 2.5) circle (1pt);
\filldraw[black] (1, 1) circle (1.5pt);
\filldraw[white] (1, 1) circle (1pt);
\filldraw[black] (3.5, 3.5) circle (1.5pt);
\filldraw[white] (3.5, 3.5) circle (1pt);

\draw (0,0) node[anchor=east]{$v$};
\draw (0,2.5) node[anchor=east]{$b$};
\draw (1,1) node[anchor=east]{$c$};
\draw (1,3.5) node[anchor=east]{$d$};
\draw (2.5,0) node[anchor=west]{$a$};
\draw (2.5,2.5) node[anchor=west]{$f$};
\draw (3.5,1) node[anchor=west]{$e$};
\draw (3.5,3.5) node[anchor=west]{$w$};
\end{tikzpicture}
\caption{Labels for the vertices of a cube $C$.} \label{labeledcubeforprop}
\end{figure}
\end{Lemma}

\begin{proof}The proof follows from a straightforward computation.
\end{proof}

\begin{proof}[Proof of Proposition~\ref{A2B2proposition}.]Let $\mathbf{x} = (x_s) \in \C^{\Z^3}$ satisfy the Kashaev equation. Given a unit cube $C$ of $\Z^3$ labeled as in Fig.~\ref{labeledcubeforprop}, by Lemma~\ref{KCvgenxlemma}, we have
\begin{gather*}
\big(K_v^C(\mathbf{x})\big)^2 = \frac{x_v^2}{4} K^C(\mathbf{x}) + (x_v x_d + x_b x_c)(x_v x_e + x_a x_c)(x_v x_f + x_a x_b)\\
\hphantom{\big(K_v^C(\mathbf{x})\big)^2}{} = (x_v x_d + x_b x_c)(x_v x_e + x_a x_c)(x_v x_f + x_a x_b).
\end{gather*}
Taking the product over unit cubes $C$ of $\Z^3$ containing $v$, we obtain
\begin{gather*} %\label{A2B2maincalc}
\left( \prod_{C \ni v} K^C_v(\mathbf{x}) \right)^2= \prod_{C \ni v} \prod_{C \supset S \ni v} (x_v x_{v_2} + x_{v_1} x_{v_3}) = \left( \prod_{S \ni v} (x_v x_{v_2} + x_{v_1} x_{v_3}) \right)^2;
\end{gather*}
here we use that the double product counts each unit square containing $v$ twice. Similarly,
\begin{gather*}
\left( \prod_{\substack{C = C_v(i_1, i_2, i_3)\\ i_1, i_2, i_3 \in \{-1,1\} \\ i_1 i_2 i_3 = 1}} K_v^C(\mathbf{x}) \right)^2 = \prod_{\substack{C = C_v(i_1, i_2, i_3)\\ i_1, i_2, i_3 \in \{-1,1\} \\ i_1 i_2 i_3 = 1}} \prod_{C \supset S \ni v} (x_v x_{v_2} + x_{v_1} x_{v_3}) = \prod_{S \ni v} (x_v x_{v_2} + x_{v_1} x_{v_3})\\
\hphantom{\left( \prod_{\substack{C = C_v(i_1, i_2, i_3)\\ i_1, i_2, i_3 \in \{-1,1\} \\ i_1 i_2 i_3 = 1}} K_v^C(\mathbf{x}) \right)^2}{}
= \prod_{\substack{C = C_v(i_1, i_2, i_3)\\ i_1, i_2, i_3 \in \{-1,1\} \\ i_1 i_2 i_3 = -1}} \prod_{C \supset S \ni v} (x_v x_{v_2} + x_{v_1} x_{v_3})
 = \left( \!\prod_{\substack{C = C_v(i_1, i_2, i_3)\\ i_1, i_2, i_3 \in \{-1,1\} \\ i_1 i_2 i_3 = -1}}\! K_v^C(\mathbf{x})\! \right)^2\!,
\end{gather*}
because each double product counts each unit square containing $v$ twice.
\end{proof}

For the proof of Theorem~\ref{galoisfromintro}(b), we will need the following lemma.

\begin{Lemma} \label{cornersofcubelemma}Suppose $\mathbf{\tilde{x}} = (x_s) \in \C^\lat$ $($with $x_s \not= 0$ for $s \in \Z^3)$ satisfies the K-hexahedron equations. Let $\mathbf{x} \in (\C^*)^{\Z^3}$ be the restriction of $\mathbf{\tilde{x}}$ to $\Z^3$. Let $v \in \Z^3$ and $\mathbf{i} = (i_1, i_2, i_3) \in \{-1, 1\}^3$, and let $C$ be the unique unit cube in $\Z^3$ containing vertices~$v$ and $v + \mathbf{i}$. Then
\begin{gather*}
K_v^C(\mathbf{x}) = \pm x_{v + \frac{1}{2} (0,i_2,i_3)} x_{v + \frac{1}{2} (i_1,0,i_3)} x_{v + \frac{1}{2} (i_1,i_2,0)},
\end{gather*}
where the sign is positive if $\mathbf{i} \in \{\pm(1,1,1)\}$, and negative otherwise.
\end{Lemma}

\begin{proof}Let $R = \C[z_{ijk} \colon 0 \le i,j,k \le 1, (i,j,k) \in \lat]$. With the equations~\eqref{likehex1d}--\eqref{likehex4d} in mind, we define $p_1, p_2, p_3, p_4 \in R$ by
\begin{gather*}
p_1 = \zohh \zzzz - \zhzh \zhhz - \zzhh \zozz,\\
p_2 = \zhoh \zzzz - \zzhh \zhhz - \zhzh \zzoz,\\
p_3 = \zhho \zzzz - \zzhh \zhzh - \zhhz \zzzo,\\
p_4 = \zooo \zzzz^2 - A - 2 \zzhh \zhzh \zhhz,
\end{gather*}
where $A$ is the expression given in~\eqref{aandbs}. With~\eqref{khexspec} in mind, we define $q_1, \dots, q_6 \in R$ by
\begin{alignat*}{3}
& q_1= \zzhh^2 - \zzoo \zzzz - \zzoz \zzzo, \qquad && q_4 = \zohh^2 - \zooo \zozz - \zooz \zozo,&\\
& q_2= \zhzh^2 - \zozo \zzzz - \zozz \zzzo, \qquad && q_5 = \zhoh^2 - \zooo \zzoz - \zooz \zzoo,&\\
& q_3= \zhhz^2 - \zooz \zzzz - \zozz \zzoz, \qquad && q_6 = \zhho^2 - \zooo \zzzo - \zozo \zzoo.&
\end{alignat*}
Let $I = \langle p_1, p_2, p_3, p_4, q_1, \dots, q_6 \rangle$ be the ideal in $R$ generated by these elements. Let $v'$ be the ``bottom'' vertex in the cube $C$, i.e., $v' = v + \left( \min(0,i_1), \min(0,i_2), \min(0,i_3) \right)$. Because $\mathbf{\tilde{x}}$ satisfies the K-hexahedron equations, it follows that if $g \in I$, then specializing $z_{ijk} = x_{v' + (i,j,k)}$ in $g$ yields $0$. Given $j_\ell = \min(i_\ell, 0) + 1$ and $k_\ell = 1 - j_\ell$ for~$\ell = 1,2,3$, let
\begin{gather*}
K = \frac{1}{2} \big(z_{k_1k_2k_3} z_{j_1j_2j_3}^2 - z_{j_1j_2j_3} (z_{k_1j_2j_3} z_{j_1k_2k_3} + z_{j_1k_2j_3} z_{k_1j_2k_3} + z_{j_1j_2k_3} z_{k_1k_2j_3})\big)\\
\phantom{K=}{} - z_{k_1j_2j_3} z_{j_1k_2j_3} z_{j_1j_2k_3}.
\end{gather*}
Note that specializing $z_{ijk} = x_{v' + (i,j,k)}$ in $K$ yields $K^C_v(\mathbf{x})$. It can be checked that
\begin{gather*}
\zzzz^3 \big( K - z_{j_1\frac{1}{2}\frac{1}{2}} z_{\frac{1}{2}j_2\frac{1}{2}} z_{\frac{1}{2}\frac{1}{2}j_3} \big) \in I
\end{gather*}
if $\mathbf{i} \in \{\pm(1,1,1)\}$, and that
\begin{gather*}
\zzzz^2 \big( K + z_{j_1\frac{1}{2}\frac{1}{2}} z_{\frac{1}{2}j_2\frac{1}{2}} z_{\frac{1}{2}\frac{1}{2}j_3} \big) \in I
\end{gather*}
otherwise. Because $x_{v'} \not= 0$, the lemma follows.
\end{proof}

\begin{proof}[Proof of Theorem~\ref{galoisfromintro}(b).]Let $\mathbf{x} \in (\C^*)^{\Z^3}$ be the restriction of $\mathbf{\tilde{x}}$ to $\Z^3$. Applying Lem\-ma~\ref{cornersofcubelemma} for the $8$ cubes incident to a vertex $v \in \Z^3$, we get:
\begin{gather*}
\prod_{C \ni v} K^C_v(\mathbf{x}) = (-1)^6 \prod_{(i_1,i_2,i_3) \in \{-1,1\}^3} x_{v + \frac{1}{2}(0,i_2,i_3)} x_{v + \frac{1}{2}(i_1,0,i_3)} x_{v + \frac{1}{2}(i_1,i_2,0)}\\
\hphantom{\prod_{C \ni v} K^C_v(\mathbf{x})}{} = \prod_{S \ni v} x_S^2 = \prod_{S \ni v} (x_v x_{v_2} + x_{v_1} x_{v_3}),
\end{gather*}
so $\mathbf{x}$ is a coherent solution of the Kashaev equation.
\end{proof}

\begin{Proposition} \label{reversed_K_hex}Let $\mathbf{\tilde{x}} = (x_s) \in \C^{\lat}$, with $x_s \not= 0$ for $s \in \Z^3$. The following are equivalent:
\begin{itemize}\itemsep=0pt
\item $\mathbf{\tilde{x}}$ satisfies the K-hexahedron equations,
\item the array $(x_{-s})_{s \in \lat}$ satisfies the K-hexahedron equations.
\end{itemize}
\end{Proposition}

\begin{proof}The proof follows from a straightforward computation.
\end{proof}

Proposition~\ref{reversed_K_hex} enables us to run the K-hexahedron equations ``in reverse''. We note that the property that we show for the K-hexahedron equations in Proposition~\ref{reversed_K_hex} holds for the original hexahedron recurrence.

\begin{Remark} \label{addressing_directionality_of_hex}Here, we address the comments in Remark~\ref{directionality_of_hex}. Let $\mathbf{\tilde{x}}$ be an array of $14$ numbers indexed by the vertices and $2$-dimensional faces of a cube $C$ with a distinguished ``top'' vertex $v$, where the components of $\mathbf{\tilde{x}}$ indexed by the $8$ vertices of the cube are nonzero. Suppose $\mathbf{\tilde{x}}$ satisfies the K-hexahedron equations. Proposition~\ref{reversed_K_hex} tells us that $\mathbf{\tilde{x}}$ would satisfy the K-hexahedron equations if we took the vertex $w$ opposite $v$ to be the ``top'' vertex of $C$. On the other hand, Lemma~\ref{cornersofcubelemma} tells us that if the components of $\mathbf{\tilde{x}}$ indexed by the faces are nonzero, and~$w$ is a~vertex of $C$ other than $v$ or the vertex opposite it, then $\mathbf{\tilde{x}}$ would not satisfy the K-hexahedron equations if we place the vertex $w$ at the ``top'' of~$C$. This argument also implies the analogous statement for the hexahedron recurrence.
\end{Remark}

We next work toward a proof of Theorem~\ref{galoisfromintro}(a).

\begin{Lemma} \label{cubeconditionprop}Fix $v \in \Z^3$ and $\mathbf{\tilde{x}} = (x_s) \in \C^{\lat}$ satisfying the equations~\eqref{likehex1d}--\eqref{likehex4d}, with $x_s \not= 0$ for $s \in \Z^3$. Then the following are equivalent:
\begin{itemize}\itemsep=0pt
\item the following equations hold:
\begin{gather*}
x_{v + \left( 0, \half, \half \right)}^2 = x_{v} x_{v + ( 0, 1, 1)} + x_{v + ( 0, 1, 0)} x_{v + ( 0, 0, 1)},\\
x_{v + \left( \half, 0, \half \right)}^2 = x_{v} x_{v + ( 1, 0, 1)} + x_{v + ( 1, 0, 0)} x_{v + ( 0, 0, 1)},\\
x_{v + \left( \half, \half, 0 \right)}^2 = x_{v} x_{v + ( 1, 1, 0)} + x_{v + ( 1, 0, 0)} x_{v + ( 0, 1, 0)}.
\end{gather*}
\item the following equations hold:
\begin{gather*}
x_{v + \left( 1, \half, \half \right)}^2 = x_{v + (1, 0, 0) } x_{v + ( 1, 1, 1)} + x_{v + (1, 1, 0)} x_{v + (1, 0, 1)},\\
x_{v + \left( \half, 1, \half \right)}^2 = x_{v + (0, 1, 0)} x_{v + ( 1, 1, 1)} + x_{v + (1, 1, 0)} x_{v + (0, 1, 1)},\\
x_{v + \left( \half, \half, 1 \right)}^2 = x_{v + (0, 0, 1)} x_{v + ( 1, 1, 1)} + x_{v + (1, 0, 1)} x_{v + (0, 1, 1)}.
\end{gather*}
\end{itemize}
\end{Lemma}

\begin{proof}This is a straightforward verification.
\end{proof}

\begin{Definition}For $U \subseteq \Z$, we denote
\begin{gather*}
\Z^3_U = \big\{(i,j,k) \in \Z^3 \colon i + j + k \in U\big\}.
\end{gather*}
For $U, V \subseteq \Z$, we denote
\begin{gather*}
\lat_{U,V} = \Z^3_U \cup \big\{ (i,j,k) \in \lat - \Z^3 \colon i + j + k \in V \big\}.
\end{gather*}
In other words, $\lat_{U, V}$ contains the integer points at heights in $U$, and the half-integer points of~$\lat$ at heights in~$V$. In particular, we will be interested in $\zti = \Z^3_{\{0,1,2\}}$ and $\lati = \lat_{\{0,1,2\},\{1\}}$.
\end{Definition}

\begin{Corollary} \label{conditionprop}Suppose that $\mathbf{\tilde{x}} = (x_s) \in \C^{\lat}$ $($with $x_s \not= 0$ for $s \in \Z^3)$ satisfies the recurrence~\eqref{likehex1d}--\eqref{likehex4d}. If $\mathbf{\tilde{x}}$ satisfies~\eqref{khexspec} for all $s \in \lati - \zti$, then $\mathbf{\tilde{x}}$ satisfies~\eqref{khexspec} for all $s \in \lat - \Z^3$; in other words, $\mathbf{\tilde{x}}$ satisfies the K-hexahedron equations.
\end{Corollary}

\begin{proof}This follows directly from Lemma~\ref{cubeconditionprop}.
\end{proof}

\begin{Remark}By Proposition~\ref{reversed_K_hex}, given $\xtin = (x_s) \in \C^{\lati}$ with $x_s \not= 0$ for all $s \in \zti$, there exists at most one extension $\mathbf{\tilde{x}} = (x_s) \in \C^{\lat}$ of $\xtin$ to $\lat$ (with $x_s \not= 0$ for all $s \in \Z^3$) satisfying the K-hexahedron equations. We say ``at most one'' instead of ``one'' because in the course of running the recurrence~\eqref{likehex1d}--\eqref{likehex4d}, we might get a zero value at an integer point.
\end{Remark}

\begin{Definition} \label{generic_subarrays}We say that an array $\xtin$ indexed by $\lati$ that satisfies equation~\eqref{khexspec} for all $s \in \lati - \zti$ is \emph{generic} if $\xtin$ can be extended to an array $\mathbf{\tilde{x}}$ indexed by $\lat$ satisfying the K-hexahedron equations, with all components of $\mathbf{\tilde{x}}$ nonzero. Similarly, we say that an array $\xin$ indexed by $\zti$ is generic if every extension of $\xin$ to an array $\xtin$ indexed by $\lati$ satisfying equation~\eqref{khexspec} for all $s \in \lati - \zti$ is generic.
\end{Definition}

\begin{Definition}Let $\xtin$ be a generic array indexed by $\lati$ that satisfies equation~\eqref{khexspec} for $s \in \lati - \zti$. We denote by $(\xtin)^{\uparrow L}$ the unique extension of $\xtin$ to $\lat$ satisfying the K-hexahedron equations.
\end{Definition}

\begin{Lemma} \label{sign_prop_khex_cube}Let $C = [0, 1]^3$ be a unit cube. Fix values $t_{\left( 0, \half, \half \right)}, t_{\left( \half, 0, \half \right)}, \allowbreak t_{\left( \half, \half, 0 \right)} \allowbreak \in \{-1,1\}$. Suppose $\mathbf{\tilde{x}} = (x_s)_{s \in C \cap \lat}$ and $\mathbf{\tilde{y}} = (y_s)_{s \in C \cap \lat}$ are arrays of $14$ complex numbers indexed by the vertices and faces of $C$ such that
\begin{itemize}\itemsep=0pt
\item $\mathbf{\tilde{x}}$ and $\mathbf{\tilde{y}}$ both satisfy the K-hexahedron equations,
\item $x_s, y_s \not= 0$ for $s \in \{0,1\}^3$,
\item $y_{s} = x_{s}$ for all $s \in \{0,1\}^3 \setminus \{(1,1,1)\}$,
\item $y_{s} = t_{s} x_{s}$ for all $s \in \left\{ \left( 0, \half, \half \right), \left( \half, 0, \half \right), \left( \half, \half, 0 \right) \right\}$, and
\item $t_{\left( 0, \half, \half \right)} t_{\left( \half, 0, \half \right)} t_{\left( \half, \half, 0 \right)} = 1$.
\end{itemize}
Then the following equations hold:
\begin{gather}
\label{sp1} y_{\left( 1, \half, \half \right)}= t_{\left( 0, \half, \half \right)} x_{\left( 1, \half, \half \right)},\\
\label{sp2} y_{\left( \half, 1, \half \right)}= t_{\left( \half, 0, \half \right)} x_{\left( \half, 1, \half \right)},\\
\label{sp3} y_{\left( \half, \half, 1 \right)}= t_{\left( \half, \half, 0 \right)} x_{\left( \half, \half, 1 \right)},\\
\label{spm} y_{\left( 1,1,1 \right)} = x_{\left( 1,1,1 \right)}.
\end{gather}
\end{Lemma}

\begin{proof}Note that
\begin{gather} \label{1hh_sign_prop}
y_{\left( 1, \half, \half \right)} = \frac{t_{\left( \half, 0, \half \right)} t_{\left( \half, \half, 0 \right)} x_{\left( \half, 0, \half \right)} x_{\left( \half, \half, 0 \right)} + t_{\left( 0, \half, \half \right)} x_{\left( 0, \half, \half \right)} x_{(1, 0, 0)}}{x_{\left( 0,0,0 \right)}^2}.
\end{gather}
Either $t_{\left( 0, \half, \half \right)} = 1$ and $t_{\left( \half, 0, \half \right)} t_{\left( \half, \half, 0 \right)} = 1$, or $t_{\left( 0, \half, \half \right)} = -1$ and $t_{\left( \half, 0, \half \right)} t_{\left( \half, \half, 0 \right)} = -1$. By equation~\eqref{1hh_sign_prop}, $y_{\left( 1, \half, \half \right)} = t_{\left( 0, \half, \half \right)} x_{\left( 1, \half, \half \right)}$ in either case. Equations~\eqref{sp2}--\eqref{sp3} hold by the same argument. Furthermore, note that
\begin{gather*}
y_{\left( 1, 1, 1\right)} - x_{\left( 1, 1, 1\right)} = 2 \big( t_{\left( 0, \half, \half \right)} t_{v + \left( \half, 0, \half \right)} t_{v + \left( \half, \half, 0 \right)} - 1 \big) \frac{x_{\left( 0, \half, \half \right)} x_{\left( 0, \half, \half \right)} x_{\left( 0, \half, \half \right)}}{x_{\left( 0, 0, 0 \right)}^2} = 0,
\end{gather*}
so equation~\eqref{spm} holds.
\end{proof}

The idea behind Lemma~\ref{sign_prop_khex_cube} is that if $t_{\left( 0, \half, \half \right)} t_{v + \left( \half, 0, \half \right)} t_{v + \left( \half, \half, 0 \right)} = 1$, then each sign $t_s$ propagates to the face of $C$ opposite $s$. With this in mind, we make the following definition.

\begin{Definition} \label{face_equiv_classes}Define an equivalence relation on $\lat - \Z^3$ by setting $s_1 \sim s_2$ if and only if
\begin{itemize}\itemsep=0pt
\item $s_1 = \left( a_1, b + \half, c + \half \right)$ and $s_2 = \left( a_2, b + \half, c + \half \right)$,
\item $s_1 = \left( a + \half, b_1, c + \half \right)$ and $s_2 = \left( a + \half, b_2, c + \half \right)$, or
\item $s_1 = \left( a + \half, b + \half, c_1 \right)$ and $s_2 = \left( a + \half, b + \half, c_2 \right)$.
\end{itemize}
Let $\feq$ denote the set of equivalence classes under this equivalence relation. Denote by $[s] \in \feq$ the equivalence class of $s \in \lat - \Z^3$. If we think of $\lat - \Z^3$ as the set of unit squares in $\Z^3$, then the equivalence relation $\sim$ is generated by the equivalences $s_1 \sim s_2$ where $s_1$ and $s_2$ are opposite faces of a unit cube in~$\Z^3$. Geometrically, we can think of an element $[s] \in \feq$ as the line through the point $s$ which is perpendicular to the corresponding unit square. Note that $\lati - \zti$ is a set of representatives of~$\feq$.
\end{Definition}

\begin{Definition} \label{ttoudef}Define an action of $\{-1,1\}^{\feq}$ on arrays indexed by $\lati$ as follows: given $\mathbf{t} = (t_s) \in \{-1,1\}^{\feq}$ and $\xtin = (x_s)_{s \in \lati}$, set $\mathbf{t} \cdot \xtin = (\tilde{x}_s)_{s \in \lati}$, where
\begin{gather*}
\tilde{x}_s = \begin{cases} x_s & \text{if } s \in \zti, \\ t_{[s]} x_s & \text{if } \lati - \zti. \end{cases}
\end{gather*}
\end{Definition}

\begin{Remark}Assume that $\xin$ is a generic array indexed by $\zti$. Let $\xtin$ be any (generic) array indexed by $\lati$ that restricts to $\xin$ and satisfies equation~\eqref{khexspec} for~$s \in \lati - \zti$. Note that
\begin{gather*}
\big\{(\mathbf{t} \cdot \xtin)^{\uparrow L} \colon \mathbf{t} \in \{-1,1\}^{\feq}\big\}
\end{gather*}
is the set of arrays indexed by $\lat$ which satisfy the K-hexahedron equations and restrict to $\xin$.
\end{Remark}

\begin{Definition} \label{def_of_psi_z3}For $\mathbf{t} = (t_s) \in \{-1,1\}^{\feq}$, define $\psi(\mathbf{t}) = (u_s) \in \{-1,1\}^{\Z^3 + \left( \half, \half, \half \right)}$ by setting
\begin{gather} \label{def_of_psi_formula}
u_{s + \left( \half, \half, \half \right)} = t_{\left[s + \left( 0, \half, \half \right)\right]} t_{\left[s + \left( \half, 0, \half \right)\right]} t_{\left[s + \left( \half, \half, 0 \right)\right]}
\end{gather}
for $s \in \Z^3$. If we think of the elements of $\feq$ as lines in $\R^3$, then $u_s$ is the product of the components of $\mathbf{t}$ indexed by the $3$ lines in $\feq$ passing through the point $s$. If we think of $\feq$ as equivalence classes of unit squares in $\Z^3$, then $u_s$ is the product of the components of $\mathbf{t}$ indexed by the $3$ equivalence classes of $2$-dimensional faces of the unit cube centered at $s$.
\end{Definition}

\begin{Proposition} \label{kernel_of_psi}An array $\mathbf{t} = (t_s) \in \{-1,1\}^{\feq}$ is in the kernel of $\psi$ $($i.e., formula~\eqref{def_of_psi_formula} yields $1$ for all $s)$ if and only if there exist signs $\alpha_i, \beta_i, \gamma_i \in \{-1,1\}$, $i \in \Z$, such that
\begin{align} \label{param_t_1}
t_{\left[ \left( a, b + \half, c + \half \right) \right]} = \beta_b \gamma_c,\\
t_{\left[ \left( a + \half, b, c + \half \right) \right]} = \alpha_a \gamma_c,\\ \label{param_t_3}
t_{\left[ \left( a + \half, b + \half, c \right) \right]} = \alpha_a \beta_b
\end{align}
for all $(a, b, c) \in \Z^3$.
\end{Proposition}

\begin{proof}
Suppose there exist constants $\alpha_i, \beta_i, \gamma_i \in \{-1,1\}$ for $i \in \Z$ such that $\mathbf{t}$ satisfies equations~\eqref{param_t_1}--\eqref{param_t_3}. Let $\mathbf{u} = (u_s) = \psi(\mathbf{t})$. Then for any $(a,b,c) \in \Z^3$, $u_{\left( a, b, c \right) + \left( \half, \half, \half \right)} = \alpha_a^2 \beta_b^2 \gamma_c^2 = 1$, as desired.

Next, suppose that $\mathbf{t}$ is in the kernel of $\psi$. It is straightforward to check that the following identities for $\mathbf{t}$ hold for all $(a,b,c) \in \Z^3$:
\begin{gather*}
t_{\left[ \left( a , b + \half, c+ \half \right) \right]}= t_{\left[ \left( 0, \half, \half \right) \right]} t_{\left[ \left( 0, b + \half, \half \right) \right]} t_{\left[ \left( 0, \half, c + \half \right) \right]},\\
t_{\left[ \left( a + \half, b, c+ \half \right) \right]} = t_{\left[ \left( a + \half, \half, 0 \right) \right]} t_{\left[ \left( 0, \half, c + \half \right) \right]},\\
t_{\left[ \left( a + \half, b+ \half, c \right) \right]}= t_{\left[ \left( 0, \half, \half \right) \right]} t_{\left[ \left( 0, b + \half, \half \right) \right]} t_{\left[ \left( a + \half, \half, 0 \right) \right]}.
\end{gather*}
Hence, setting
\begin{gather*}
\alpha_a = t_{\left[ \left( a + \half, \half, 0 \right) \right]},\qquad
\beta_b = t_{\left[ \left( 0, \half, \half \right) \right]} t_{\left[ \left( 0, b + \half, \half \right) \right]},\qquad
\gamma_c = t_{\left[ \left( 0, \half, c + \half \right) \right]},
\end{gather*}
it follows that $\mathbf{t}$ satisfies~\eqref{param_t_1}--\eqref{param_t_3}.
\end{proof}

\begin{Definition}For $(a_1, \dots, a_d), (b_1, \dots, b_d) \in \R^d$, the notation $(a_1, \dots, a_d) \le (b_1, \dots, b_d)$ will mean that $a_i \le b_i$ for $i = 1, \dots, d$, and $(a_1, \dots, a_d) < (b_1, \dots, b_d)$ will mean that $(a_1, \dots, a_d) \le (b_1, \dots, b_d)$ but $(a_1, \dots, a_d) \not= (b_1, \dots, b_d)$.
\end{Definition}

\begin{Lemma} \label{change_one_value_lemma}Let $\xtin$ be a generic array indexed by $\lati$ satisfying equation~\eqref{khexspec} for $s \in \lati - \zti$. Let $\mathbf{t} \in \{-1, 1\}^{\feq}$, and $\mathbf{u} = \left(u_s \right)_{s \in \Z^3 + \left( \half, \half, \half \right)} = \psi(\mathbf{t})$. Denote $(\xtin)^{\uparrow L} = (x_s)_{s \in \lat}$, and $(\mathbf{t} \cdot \xtin)^{\uparrow L} = (y_s)_{s \in \lat}$. Suppose $v \in \Z^3_{\{3,4,5,\dots\}}$ satisfies the condition that $u_{w - \left( \half, \half, \half \right)} = 1$ for all $w \in \Z^3_{\{3,4,5, \dots\}}$ with $w \le v$. Then:
\begin{enumerate}\itemsep=0pt
\item[$(a)$] $y_v = x_v$,
\item[$(b)$] $y_{v - s} = t_{[v - s]} x_{v - s}$ for $s \in \left\{ \left( 0, \half, \half \right), \left( \half, 0, \half \right), \left( \half, \half, 0 \right) \right\}$.
\end{enumerate}
\end{Lemma}

\begin{proof}We prove statements~(a) and~(b) together for $w \in \Z^3_{\{3,4,5,\dots\}}$ with $w \le v$ by induction. Assume by induction that we have proved statements~(a) and~(b) for all $w \in \Z^3_{\{3,4,5,\dots\}}$ with $w < v$. By construction, $x_{w} = y_{w}$ for all $w \in \zti$ and statement~(b) holds for $w \in \Z^3_{\{2\}}$. Hence, $y_{v - s'} = x_{v - s'}$ for $s' \in \{0, 1\}^3 \setminus \{(0,0,0)\}$, and $y_{v - s} = t_{[v - s]} x_{v - s}$ for $s \in \left\{ \left( 1, \half, \half \right), \left( \half, 1, \half \right)\right.$, $\left.\left( \half, \half, 1 \right) \right\}$. Because $u_{v - \left( \half, \half, \half \right)} = 1$, statements~(a) and~(b) follow from Lemma~\ref{sign_prop_khex_cube}.
\end{proof}

\begin{Lemma} \label{prod_over_cubes_lemma}An array $\mathbf{u} = (u_s) \in \{-1, 1\}^{\Z^3 + \left( \half, \half, \half \right)}$ is in the image of $\psi$ $($see Definition~{\rm \ref{def_of_psi_z3})} if and only if for every $v \in \Z^3$,
\begin{gather} \label{prod_over_cubes}
\prod_{a_1, a_2, a_3 \in \{-1,1\}} u_{v + (a_1, a_2, a_3)/2} = 1.
\end{gather}
\end{Lemma}

\begin{Remark}Using the interpretation of $\Z^3 + \left( \half, \half, \half \right)$ as the set of unit cubes of~$\Z^3$, the product on the left-hand side of~\eqref{prod_over_cubes} is a product over unit cubes incident with~$v$.
\end{Remark}

\begin{proof}[Proof of Lemma~\ref{prod_over_cubes_lemma}.]First, suppose $\mathbf{u} = \psi(\mathbf{t})$, where $\mathbf{t} = (t_s) \in \{-1, 1\}^{\feq}$. For any $v \in \Z^3$, consider the cube $C = v + \left[ -\half, \half \right]^3$. Note that $v + (a_1, a_2, a_3)/2$ for $a_1, a_2, a_3 \in \{-1, 1\}$ are the vertices of $C$. Using the interpretation of $\feq$ as lines in $\R^3$, the edges of $C$ are contained in the lines in $\feq$. If $s_1$, $s_2$, $s_3$ are the $3$ lines in $\feq$ that intersect at~$v + (a_1, a_2, a_3)/2$, then $u_{v + (a_1, a_2, a_3)/2} = t_{s_1} t_{s_2} t_{s_3}$. Each edge of $C$ is incident with $2$ vertices of $C$. Hence, expanding the left-hand side of~\eqref{prod_over_cubes} in terms of components of $\mathbf{t}$, we obtain{\samepage
\begin{gather*}
\prod_{a_1, a_2, a_3 \in \{-1,1\}} u_{v + (a_1, a_2, a_3)/2} = \prod_s t_s^2= 1,
\end{gather*}
where the second product is over the lines $s \in \feq$ determined by the edges of $C$.}

Next, suppose that condition~\eqref{prod_over_cubes} holds. It is clear that $\mathbf{u}$ is uniquely determined by its components at $S = \left\{ \left( v_1 + \half, v_2 + \half, v_3 + \half \right) \in \Z^3 + \left( \half, \half, \half \right)\colon v_1 v_2 v_3 = 0 \right\}$ and condition~\eqref{prod_over_cubes}. For $(v_1, v_2, v_3) \in \lati - \zti$, set
\begin{gather*} %\label{choice_of_t_given_u}
t_{[(v_1, v_2, v_3)]} = \begin{cases}
u_{\left(\half, v_2, v_3\right)} & \text{if } v_1 \in \Z,\\
u_{\left(v_1, \half, v_3\right)} u_{\left(\half,\half,v_3\right)} & \text{if } v_2 \in \Z \text{ and } v_1 \not= \frac{1}{2},\\
u_{\left(v_1, v_2, \half\right)} u_{\left(v_1, \half, \half\right)} u_{\left(\half,v_2, \half\right)} & \text{if } v_3 \in \Z \text{ and } v_1, v_2 \not= \frac{1}{2},\\
1 & \text{if } v_2 \in \Z \text{ and } v_1 = \frac{1}{2},\\
1 & \text{if } v_3 \in \Z \text{ and either } v_1 = \frac{1}{2} \text{ or } v_2 = \frac{1}{2}.
\end{cases}
\end{gather*}
Set $\mathbf{t} = (t_s) \in \{-1,1\}^{\feq}$. It is straightforward to check that $\psi(\mathbf{t})$ agrees with $\mathbf{u}$ at $S$. Hence, because $\psi(\mathbf{t})$ and $\mathbf{u}$ both satisfy condition~\eqref{prod_over_cubes}, it follows that $\mathbf{u} = \psi(\mathbf{t})$.
\end{proof}

\begin{Lemma} \label{levels_0_thru_5}Let $\mathbf{\hat{x}}$ be an array indexed by $\Z^3_{\{0,1,2,3,4,5\}}$. Assume that $\mathbf{\hat{x}}$ satisfies the Kashaev equation, and, moreover, its restriction to $\zti$ is generic. Then there exists an array $\mathbf{\tilde{x}}$ indexed by $\lat$ satisfying the K-hexahedron equations and extending~$\mathbf{\hat{x}}$.
\end{Lemma}

\begin{proof}For $i = 3,4,5$, we will show by induction on $i$ that there exists an array $\xtin$ indexed by $\lati$ satisfying equation~\eqref{khexspec} for $s \in \lati - \zti$ such that $(\xtin)^{\uparrow L}$ agrees with $\mathbf{\hat{x}} = (x_s)_{s \in \Z^3_{\{0,1,2,3,4,5\}}}$ on $\Z^3_{\{0,1,\dots, i\}}$. Let $\xtin'$ be an array indexed by $\lati$ satisfying equation~\eqref{khexspec} for $s \in \lati - \zti$ such that $(\xtin')^{\uparrow L} = (y_s)_{s \in \lat}$ agrees with $\mathbf{\hat{x}}$ on $\Z^3_{\{0,1,\dots, i - 1\}}$. (For $i = 3$, we can obtain $\xtin'$ by taking an arbitrary extension of $\xin$ to $\lati$ satisfying equation~\eqref{khexspec} for $s \in \lati - \zti$. For $i = 4,5$, we have shown that $\xtin'$ exists by induction.) Choose $\mathbf{\tilde{u}} = (u_s) \in \{-1, 1\}^{\Z^3_{\{3,4,5\}} - \left( \half, \half, \half \right)}$ so that
\begin{itemize}\itemsep=0pt
\item $u_{s - \left( \half, \half, \half \right)} = 1$ if $s \in \Z^3_{\{i\}}$ and $y_s = x_s$,
\item $u_{s - \left( \half, \half, \half \right)} = -1$ if $s \in \Z^3_{\{i\}}$ and $y_s \not= x_s$,
\item $u_{s - \left( \half, \half, \half \right)} = 1$ if $i = 4$ and $s \in \Z^3_{\{3\}}$, or $i = 5$ and $s \in \Z^3_{\{3,4\}}$.
\end{itemize}
Extend $\mathbf{\tilde{u}}$ to $\mathbf{u} \in \{-1, 1\}^{\Z^3 + \left( \half, \half, \half \right)}$ by condition~\eqref{prod_over_cubes}. By Lemma~\ref{prod_over_cubes_lemma}, there exists $\mathbf{t} \in \{-1, 1\}^{\feq}$ such that $\mathbf{u} = \psi(\mathbf{t})$. Set $\xtin = \mathbf{t} \cdot \xtin'$. Then, by Lemma~\ref{change_one_value_lemma}, $(\xtin)^{\uparrow L}$ agrees with~$\mathbf{\hat{x}}$ on $\Z^3_{\{0,1, \dots, i\}}$, as desired.
\end{proof}

We can now prove a weaker version of Theorem~\ref{galoisfromintro}(a), under the additional constraint of genericity.

\begin{Corollary} \label{main_with_generic}Let $\mathbf{x} \in (\C^*)^{\Z^3}$ be a coherent solution of the Kashaev equation, whose restriction to $\zti$ is generic. Then $\mathbf{x}$ can be extended to $\mathbf{\tilde{x}} \in (\C^*)^{\lat}$ satisfying the K-hexahedron equations.
\end{Corollary}

\begin{proof}Let $\mathbf{x} \in (\C^*)^{\Z^3}$ be a coherent solution of the Kashaev equation, whose restriction to~$\zti$ is generic. By Lemma~\ref{levels_0_thru_5}, there exists an array $\mathbf{\tilde{x}} \in (\C^*)^{\lat}$ satisfying the K-hexahedron equations that agrees with $\mathbf{x}$ on $\Z^3_{\{0,1,2,3,4,5\}}$. Let $\mathbf{x}'$ be the restriction of $\tilde{\mathbf{x}}$ to $\Z^3$. By Theorem~\ref{galoisfromintro}(b), $\mathbf{x}'$ is a coherent solution of the Kashaev equation. There is a unique coherent solution of the Kashaev equation agreeing with $\mathbf{x}$ at $\Z^3_{\{0,1,2,3,4,5\}}$, as condition~\eqref{coherencecondition} gives the remaining values as rational expressions in the values at $\Z^3_{\{0,1,2,3,4,5\}}$ (see Remark~\ref{rationalin26others}). (As $\mathbf{x}$ is generic, $\mathbf{x}$ must satisfy condition~\eqref{conditionnonzerofaces}, and so $K^C_v(\mathbf{x}) \not= 0$ for all unit cubes $C$ in $\Z^3$ and vertices $v \in C$. Hence, the denominator of this rational expression is nonzero.) Hence, $\mathbf{x}' = \mathbf{x}$, as desired.
\end{proof}

\begin{proof}[Proof of Theorem~\ref{galoisfromintro}(a)]We need to loosen the genericity condition in Corollary~\ref{main_with_generic} to the conditions that $\mathbf{x}$ satisfies~\eqref{conditionnonzerofaces} and has nonzero entries.

Let $\mathbf{x} \in (\C^*)^{\Z^3}$ be a coherent solution of the Kashaev equation satisfying~\eqref{conditionnonzerofaces}. Let $A_j = [-j,j]^3 \cap \Z^3$ and $B_j = [-j,j]^3 \cap \lat$ for $j \in \Z_{\ge 0}$. We claim that if, for all $j$, there exist $\mathbf{\tilde{x}}_j \in (\C^*)^{B_j}$ satisfying the K-hexahedron equations that agree with $\mathbf{x}$ on $A_j$, then there exists $\mathbf{\tilde{x}} \in (\C^*)^{\lat}$ satisfying the K-hexahedron equations that agrees with $\mathbf{x}$ on $\Z^3$. Construct an infinite tree $T$ as follows:
\begin{itemize}\itemsep=0pt
\item The vertices of $T$ are solutions of the K-hexahedron equation indexed by $B_j$ that agree with $\mathbf{x}$ on $A_j$ (over $j \in \Z_{\ge 0}$).
\item Add an edge between $\mathbf{\tilde{x}}_j \in (\C^*)^{B_j}$ and $\mathbf{\tilde{x}}_{j + 1} \in (\C^*)^{B_{j + 1}}$ if $\mathbf{\tilde{x}}_{j + 1}$ restricts to $\mathbf{\tilde{x}}_j$.
\end{itemize}
Thus, $T$ is an infinite tree in which every vertex has finite degree. By K\"onig's infinity lemma (see~{\cite[Theorem~16.3]{wilson}}), there exists an infinite path $\mathbf{\tilde{x}}_0, \mathbf{\tilde{x}}_1, \dots$ in $T$ with $\mathbf{\tilde{x}}_j \in (\C^*)^{B_j}$. Thus, there exists $\mathbf{\tilde{x}} \in (\C^*)^{\lat}$ restricting to $\mathbf{\tilde{x}}_j$ for all $j \in \Z_{\ge 0}$, so $\mathbf{\tilde{x}}$ is a solution of the K-hexahedron equations that agrees with $\mathbf{x}$ on $\Z^3$.

Given $j \in \Z_{\ge 0}$, we claim that there exists $\mathbf{\tilde{x}} \in (\C^*)^{B_j}$ satisfying the K-hexahedron equations that agree with $\mathbf{x}$ on $A_j$. It is straightforward to show that there exists a sequence $\mathbf{x}_1, \mathbf{x}_2, \ldots \in (\C^*)^{\Z^3}$ of coherent solutions of the Kashaev equation that converge pointwise to $\mathbf{x}$ whose restrictions to $\zti$ are generic. By Corollary~\ref{main_with_generic}, there exist $\mathbf{\tilde{x}}_1, \mathbf{\tilde{x}}_2, \ldots \in (\C^*)^{\lat}$ satis\-fying the K-hexahedron equations such that $\mathbf{\tilde{x}}_i$ restricts to $\mathbf{x}_i$. However, the sequence $\mathbf{\tilde{x}}_1, \mathbf{\tilde{x}}_2, \dots$ does not necessarily converge (see Theorem~\ref{extensions_that_agree}). Let $\mathbf{\tilde{x}}_1', \mathbf{\tilde{x}}_2', \dots \in (\C^*)^{B_j}$ be the restrictions of~$\mathbf{\tilde{x}}_1, \mathbf{\tilde{x}}_2, \dots$ to~$B_j$. There exists a subsequence of $\mathbf{\tilde{x}}_1', \mathbf{\tilde{x}}_2', \dots$ that converges to some $\mathbf{\tilde{x}} \in (\C^*)^{B_j}$. (For each $s \in B_j \setminus A_j$, we can partition the sequence $\mathbf{\tilde{x}}_1', \mathbf{\tilde{x}}_2', \dots$ into two sequences, each of which converges at $s$. Because $B_j$ is finite, the claim follows.) The array $\mathbf{\tilde{x}}$ must satisfy the K-hexahedron equations and agree with $\mathbf{x}$ on $A_j$, so we are done.
\end{proof}

We shall now work towards a proof of Theorem~\ref{extensions_that_agree}.

\begin{Lemma} \label{form_when_t_is_in_kernel}Let $\mathbf{\tilde{x}} \in (\C^*)^{\lat}$ be a solution of the K-hexahedron equations. Let $\xtin \in (\C^*)^{\lati}$ denote the restriction of $\mathbf{\tilde{x}}$ to $\lati$. Let $\mathbf{t} = (t_s) \in \{-1,1\}^{\feq}$ be in the kernel of $\psi$. Then $(\mathbf{t} \cdot \xtin)^{\uparrow \lat} = (y_s)_{s \in \lat}$, where
\begin{gather*}
y_s = \begin{cases} x_s & \text{if } s \in \Z^3, \\ t_{[s]} x_s & \text{if } s \in \lat - \Z^3. \end{cases}
\end{gather*}
\end{Lemma}

\begin{proof}This follows from Lemma~\ref{change_one_value_lemma} and Proposition~\ref{reversed_K_hex}.
\end{proof}

\begin{Lemma} \label{agree_when_t_is_in_kernel}Let $\mathbf{\tilde{x}} \in (\C^*)^{\lat}$ be a solution of the K-hexahedron equations. Let $\xtin \in (\C^*)^{\lati}$ denote the restriction of $\mathbf{\tilde{x}}$ to $\lati$. For $\mathbf{t} \in \{-1,1\}^{\feq}$, the following are equivalent:
\begin{itemize}\itemsep=0pt
\item $\mathbf{\tilde{x}}$ and $(\mathbf{t} \cdot \xtin)^{\uparrow \lat}$ agree on $\Z^3$,
\item $\mathbf{t}$ is in the kernel of $\psi$ $($see Proposition~{\rm \ref{kernel_of_psi})}.
\end{itemize}
\end{Lemma}

\begin{proof}If $\mathbf{t}$ is in the kernel of $\psi$, then $\mathbf{\tilde{x}}$ and $(\mathbf{t} \cdot \xtin)^{\uparrow \lat}$ agree on $\Z^3$ by Lemma~\ref{form_when_t_is_in_kernel}.

If $\mathbf{t}$ is not in the kernel of $\psi$, let $\mathbf{u} = (u_s)_{s \in \Z^3 + \left( \half, \half, \half \right)} = \psi(\mathbf{t})$. Write $\mathbf{\tilde{x}} = (x_s)_{s \in \lat}$ and $(\mathbf{t} \cdot \xtin)^{\uparrow \lat} = (y_s)_{s \in \lat}$. Choose $v \in \Z^3_{\{3,4,5,\dots\}}$ such that $u_{v - \left( \half, \half, \half \right)} = -1$ and $u_{w - \left( \half, \half, \half \right)} = 1$ for all $w \in \Z^3_{\{3,4,5,\dots\}}$ with $w < v$. (Such a choice of $v$ exists because if $u_{v - \left( \half, \half, \half \right)} = 1$ for all $v \in \Z^3_{\{3,4,5,\dots,\}}$, then $u_{v - \left( \half, \half, \half \right)} = 1$ for all $v \in \Z^3$ because $\mathbf{u}$ must satisfy equation~\eqref{prod_over_cubes} for all $v \in \Z^3$.) Then by Lemma~\ref{change_one_value_lemma}, $y_{v - s} = x_{v - s}$ for $s \in \{0,1\}^3 - \{(0,0,0)\}$, and $y_{v - s} = t_{[v - s]} x_{v - s}$ for $s \in \left\{ \left( 1, \half, \half \right), \left( \half, 1, \half \right), \left( \half, \half, 1 \right) \right\}$. Hence,
\begin{gather*}
y_v - x_v = -4 \frac{x_{v - \left( 1, \half, \half \right)} x_{v - \left( \half, 1, \half \right)} x_{v - \left( \half, \half, 1 \right)}}{x_{v - \left( 1,1,1 \right)}^2} \not= 0,
\end{gather*}
so $y_v \not= x_v$.
\end{proof}

\begin{proof}[Proof of Theorem~\ref{extensions_that_agree}]Let $\xtin \in (\C^*)^{\lati}$ denote the restriction of $\mathbf{\tilde{x}}$ to $\lati$.

Suppose that for some signs $\alpha_i, \beta_i, \gamma_i \in \{-1,1\}$, $i \in \Z$, $\mathbf{\tilde{y}} \in (\C^*)^{\lat}$ satisfies equations~\eqref{agree_on_vertices}--\eqref{agree_on_L_3} for $\left( a, b, c \right) \in \Z^3$. Define $\mathbf{t} = (t_s) \in \{-1,1\}^{\feq}$ by equations~\eqref{param_t_1}--\eqref{param_t_3}, so $\mathbf{t}$ is in the kernel of $\psi$ by Proposition~\ref{kernel_of_psi}. Hence, by Lemma~\ref{form_when_t_is_in_kernel}, $\mathbf{\tilde{y}} = (\mathbf{t} \cdot \xtin)^{\uparrow \lat}$, so $\mathbf{\tilde{y}}$ satisfies the K-hexahedron equations, proving part~(a).

Next, if $\mathbf{\tilde{y}} \in (\C^*)^{\lat}$ satisfies the K-hexahedron equations, then $\mathbf{\tilde{y}} = (\mathbf{t} \cdot \xtin)^{\uparrow \lat}$ for some~$\mathbf{t} = (t_s) \in \{-1,1\}^{\feq}$. By Lemma~\ref{agree_when_t_is_in_kernel}, $\mathbf{t}$ is in the kernel of $\psi$. Hence, by Proposition~\ref{kernel_of_psi}, there exist signs $\alpha_i, \beta_i, \gamma_i \in \{-1,1\}$, $i \in \Z$, such that $\mathbf{t}$ is given by equations~\eqref{param_t_1}--\eqref{param_t_3}. Hence, by Lemma~\ref{form_when_t_is_in_kernel}, $\mathbf{\tilde{y}}$ satisfies equations~\eqref{agree_on_vertices}--\eqref{agree_on_L_3} for $\left( a, b, c \right) \in \Z^3$, proving part~(b).
\end{proof}

\section{Coherence for cubical complexes} \label{cc_proof_section}

In this section, we generalize Proposition~\ref{A2B2proposition} and Theorem~\ref{galoisfromintro} from $\Z^3$ to certain classes of $3$-dimensional cubical complexes. Proposition~\ref{A2B2proposition} generalizes to arbitrary $3$-dimensional cubical complexes embedded in $\R^3$ (see Proposition~\ref{A2B2_gen}), while Theorem~\ref{galoisfromintro}(b) generalizes to directed cubical complexes corresponding to piles of quadrangulations of a polygon (see Proposition~\ref{cc_easy_direction}). Theorem~\ref{galoisfromintro}(a) does not hold for arbitrary directed cubical complexes corresponding to piles of quadrangulations of a polygon. It turns out that an additional property of a cubical complex is required, which we call~\emph{comfortable-ness}. This property is satisfied by the standard tiling of~$\R^3$ with unit cubes, as well as by cubical complexes corresponding to piles of $\Diamond$-tilings of~$\mathbf{P}_n$ (see Proposition~\ref{pseudo_comfortable_prop}). Let $\varkappa$ be the directed cubical complex corresponding to a pile of quadrangulations of a polygon. In Theorems~\ref{cc_main_generalization}--\ref{when_theorem_fails}, we show that Theorem~\ref{galoisfromintro}(a) holds for~$\varkappa$ if and only if~$\varkappa$ is comfortable. The proof of Theorem~\ref{cc_main_generalization} is nearly identical to the proof of Theorem~\ref{galoisfromintro}(a) in Section~\ref{mainproof}.

First, we note that Proposition~\ref{A2B2proposition} generalizes to arbitrary $3$-dimensional cubical complexes embedded in $\R^3$ as follows:

\begin{Proposition} \label{A2B2_gen}Let $\varkappa$ be a $3$-dimensional cubical complex embedded in $\R^3$. Suppose that $\mathbf{x} = (x_s)_{s \in \varkappa^0}$ satisfies the Kashaev equation. Then for any interior vertex $v \in \varkappa^0$ $($see Definition~{\rm \ref{interior_point_definition})},
\begin{gather*}
\left( \prod_{C \ni v} K_v^C(\mathbf{x}) \right)^2 = \left( \prod_{S \ni v} (x_v x_{v_2} + x_{v_1} x_{v_3}) \right)^2,
\end{gather*}
where
\begin{itemize}\itemsep=0pt
\item the first product is over $3$-dimensional cubes $C$ incident to the vertex $v$,
\item the second product is over $2$-dimensional faces $S$ incident to the vertex $v$, and
\item $v$, $v_1$, $v_2$, $v_3$ are the vertices of such a face $S$ listed in cyclic order.
\end{itemize}
\end{Proposition}

\begin{proof}The proof is almost identical to the proof of Proposition~\ref{A2B2proposition} in Section~\ref{mainproof}.
\end{proof}

\begin{Remark}With Proposition~\ref{A2B2_gen} in mind, we can think of the notion of coherence from Definition~\ref{coherence_gen_def} as follows. Let $\mathbf{T} = (T_0, \dots, T_\ell)$ be a pile of quadrangulations of a polygon with $\varkappa = \varkappa(\mathbf{T})$. Start with an arbitrary array $\xin$ indexed by $\varkappa^0(T_0)$ whose entries are ``sufficiently generic''. We want to extend $\xin$ to an array $\mathbf{x}$ indexed by $\varkappa^0$ that is a coherent solution of the Kashaev equation. Building $\mathbf{x}$ inductively, suppose we have defined the values of $\mathbf{x}$ at $\varkappa^0(T_0, \dots, T_{i - 1})$, and we need to define the value $x_w$ of $\mathbf{x}$ at the new vertex $w$ in $T_i$. Let $C \in \varkappa^3$ be the cube corresponding to the flip between $T_{i - 1}$ and $T_i$, and let $v$ be the bottom vertex of~$C$, i.e., let $v$ be the unique vertex in $T_{i - 1}$ but not $T_i$. In order that $\mathbf{x}$ continue to satisfy the Kashaev equation, there are $2$ possible values for $x_w$, say $a$ and $b$, so that $K^C(\mathbf{x}) = 0$. If the vertex $v$ is in $T_0$, i.e., $v$ is not an interior vertex of $\varkappa$, then we can either set $x_w = a$ or $x_w = b$, and $\mathbf{x}$ will continue to be a coherent solution of the Kashaev equation. Now, suppose $v$ is not in $T_0$, i.e., $v$ is an interior vertex of $\varkappa$. Because we have chosen $\xin$ to be ``sufficiently generic'', the value of $\prod_{C \ni v} K_v^C(\mathbf{x})$ depends on whether we set $x_w = a$ or $x_w = b$. Proposition~\ref{A2B2_gen} tells us that for one of the $2$ possible values, say $x_w = a$, equation~\eqref{coherence_condition_for_wir} holds, while for the other value, $x_w = b$, the following equation holds:
\begin{gather*} \label{anti_coherence_condition_for_wir}
\prod_{C \ni v} K_v^C(\mathbf{x}) = -\prod_{S \ni v} (x_v x_{v_2} + x_{v_1} x_{v_3}).
\end{gather*}
Hence, the condition of coherence tells us which of the $2$ solutions is the ``correct'' one when $v$ is an interior vertex of $\varkappa$.
\end{Remark}

We now prove the following generalization of Theorem~\ref{galoisfromintro}(b).

\begin{Proposition} \label{cc_easy_direction}Let $\mathbf{T}$ be a pile of quadrangulations of a polygon. Let $\mathbf{\tilde{x}} = (x_s)_{s \in \varkappa^{02}(\mathbf{T})}$ be an array $($with $x_s \not= 0$ for all~$s \in \varkappa^0(\mathbf{T}))$ satisfying the K-hexahedron equations. Then the restriction of $\mathbf{\tilde{x}}$ to $\varkappa^0(\mathbf{T})$ is a coherent solution of the Kashaev equation.
\end{Proposition}

\begin{proof}The proof follows almost exactly the same as the proof of Theorem~\ref{galoisfromintro}(b). For an interior vertex $v$ of $\varkappa$, there is exactly one cube $C$ for which $v$ is the top vertex, and exactly one cube $C$ for which $v$ is the bottom vertex. Let $\mathbf{x}$ be the restriction of $\mathbf{\tilde{x}}$ to $\varkappa^0$. By Lemma~\ref{cornersofcubelemma}, taking the product over the cubes incident to $v$,
\begin{gather*}
\prod_{C \ni v} K^C_v(\mathbf{x})= (-1)^2 \prod_{S \in \varkappa^2 \colon S \ni v} x_S = \prod_{S \ni v} x_S^2 = \prod_{S \ni v} (x_v x_{v_2} + x_{v_1} x_{v_3}),
\end{gather*}
so the restriction of $\mathbf{\tilde{x}}$ to $\mathbf{x}$ is a coherent solution of the Kashaev equation.
\end{proof}

The following statement generalizes Theorem~\ref{ABtheorem}:

\begin{Corollary}Let $\mathbf{T}$ be a pile of quadrangulations of a polygon. Let $\mathbf{x} = (x_s)_{s \in \varkappa^{0}(\mathbf{T})}$ be an array satisfying the positive Kashaev recurrence. Then $\mathbf{x}$ is a coherent solution of the Kashaev equation.
\end{Corollary}

\begin{proof}This follows immediately from Proposition~\ref{cc_easy_direction} because $\mathbf{x}$ can be extended to an array indexed by $\varkappa^{02}(\mathbf{T})$ satisfying the K-hexahedron equations by choosing the positive solutions from equation~\eqref{khexspec} for $s \in \varkappa^2(\mathbf{T})$.
\end{proof}

\begin{Remark}The converse of Proposition~\ref{cc_easy_direction} (i.e., the counterpart of Theorems~\ref{galoisfromintro}(a) and \ref{main_thm_cyclic_zonotope}(a)) does not hold for an arbitrary choice of $\mathbf{T}$. In other words, there exist piles $\mathbf{T}$ and arrays~$\mathbf{x}$ indexed by $\varkappa^0(\mathbf{T})$ with nonzero components that are coherent solutions of the Kashaev equation, where $\mathbf{x}$ cannot be extended to an array indexed by $\varkappa^{02}(\mathbf{T})$ satisfying the K-hexahedron equations.
\end{Remark}

In order for a converse of Proposition~\ref{cc_easy_direction} (equivalently, a generalization of Theorems~\ref{galoisfromintro}(a) and~\ref{main_thm_cyclic_zonotope}(a)) to hold, one must impose an additional condition on the underlying cubical complexes; see Definition~\ref{com_def} below.

\begin{Definition}\label{com_def}Let $\varkappa$ be a three-dimensional cubical complex that can be embedded into~$\R^3$, cf.~Definition~\ref{interior_point_definition}. (While this embeddability condition can be relaxed, it is satisfied in all subsequent applications. In fact, $\varkappa$ will always be the cubical complex associated to a pile of quadrangulations.) Let $\sim$ be the equivalence relation on $\varkappa^2$ generated by the equivalences $s_1 \sim s_2$ for all pairs $(s_1, s_2)$ involving opposite faces of some $3$-dimensional cube in~$\varkappa^3$. Let $\cueq$ denote the set of equivalence classes under this equivalence relation. Denote by $[s] \in \cueq$ the equivalence class of~$s \in \varkappa^2$. By analogy with Definition~\ref{ttoudef}, denote by $\psi_\varkappa\colon \{-1, 1\}^{\cueq} \rightarrow \{-1, 1\}^{\varkappa^3}$ the map sending an array $\mathbf{t} = (t_{[s]})_{[s] \in \cueq}$ to the array $\psi_\varkappa(\mathbf{t}) = (u_C)_{C \in \varkappa^3}$ defined by $u_C = t_{[a]} t_{[b]} t_{[c]}$, where $a$, $b$, $c$ are representatives of the three pairs of opposite $2$-dimensional faces of $C$. We say that the cubical complex $\varkappa$ is \emph{comfortable} if the following statements are equivalent for every $\mathbf{u} = (u_C) \in \{-1,1\}^{\varkappa^3}$:
\begin{itemize}\itemsep=0pt
\item[(C1)] $\mathbf{u}$ is in the image of $\psi_\varkappa$,
\item[(C2)] for every interior vertex $v \in \varkappa^0$ (cf.\ Definition~\ref{interior_point_definition}), we have
\begin{gather*}
\prod_{C \ni v} u_C = 1,
\end{gather*}
the product over $3$-dimensional cubes $C \in \varkappa^3$ containing $v$.
\end{itemize}
By Lemma~\ref{prod_over_cubes_lemma}, the standard tiling of $\R^3$ by unit cubes is comfortable.
\end{Definition}

\begin{Remark} \label{comf_cond_remark}In Definition~\ref{com_def}, the statement (C1) always implies (C2). Indeed, if $\mathbf{u} = \psi_\varkappa(\mathbf{t})$ with $\mathbf{t} = \left( t_{[s]} \right)_{[s] \in \cueq} \in \{-1, 1\}^{\cueq}$, then for any interior vertex $v \in \varkappa^0$,
\begin{gather*}
\prod_{v \in C \in \varkappa^3} u_C = \prod_{v \in s \in \varkappa^2} t_{[s]}^2 = 1.
\end{gather*}
Thus, in checking comfortableness, we simply must check that~(C2) implies~(C1).
\end{Remark}

We next state four results (Propositions~\ref{pseudo_comfortable_prop}--\ref{non_example_prop} and Theorems~\ref{cc_main_generalization}--\ref{when_theorem_fails}) which the rest of this section is dedicated to proving. The reader may want to review Definitions~\ref{divide_definition}--\ref{ps_arr_definition} before proceeding with the following proposition.%ref

\begin{Proposition} \label{pseudo_comfortable_prop}
Let $\mathbf{T}$ be a pile of quadrangulations of a polygon. Suppose that the divide associated to each quadrangulation in $\mathbf{T}$ is a pseudoline arrangement. Then $\varkappa = \varkappa(\mathbf{T})$ is comfortable. In particular, if $\mathbf{T}$ is a pile of $\Diamond$-tilings of the polygon~$\mathbf{P}_n$, then $\varkappa = \varkappa(\mathbf{T})$ is comfortable.
\end{Proposition}

\begin{Proposition} \label{non_example_prop}
There exists a pile $\mathbf{T}$ of quadrangulations of some polygon such that the cubical complex $\varkappa = \varkappa(\mathbf{T})$ is not comfortable.
\end{Proposition}

{
%\color{brown}
We next state a generalization of Theorems~\ref{galoisfromintro} and~\ref{main_thm_cyclic_zonotope}.

\begin{Theorem} \label{cc_main_generalization}
Let $\mathbf{T}$ be a pile of quadrangulations of a polygon such that $\varkappa = \varkappa(\mathbf{T})$ is comfortable. Any coherent solution of the Kashaev equation $\mathbf{x} = (x_s)_{s \in \varkappa^0}$ with nonzero components satisfying condition~\eqref{face_nonzero} can be extended to an array $\mathbf{\tilde{x}} = (x_s)_{s \in \varkappa^{02}}$ satisfying the K-hexahedron equations.
\end{Theorem}

However, Theorems~\ref{galoisfromintro}~and~\ref{main_thm_cyclic_zonotope} don't generalize to cubical complexes that are not comfortable.
}

\begin{Theorem} \label{when_theorem_fails}
Let $\mathbf{T}$ be a pile of quadrangulations of a polygon such that $\varkappa = \varkappa(\mathbf{T})$ is not comfortable. Then there exists a coherent solution of the Kashaev equation $\mathbf{x}$ indexed by $\varkappa^0$ which cannot be extended to an array indexed by $\varkappa^{02}$ satisfying the K-hexahedron equations.
\end{Theorem}

Note that Theorem~\ref{main_thm_cyclic_zonotope}(a) follows directly from Proposition~\ref{pseudo_comfortable_prop} and Theorem~\ref{cc_main_generalization}. In Remark~\ref{cc_main_gen_implies_original} below, we explain that Theorem~\ref{galoisfromintro}(a) follows from Proposition~\ref{pseudo_comfortable_prop} and Theorem~\ref{cc_main_generalization} as well.

\begin{Remark} \label{cc_main_gen_implies_original}
Together, Theorem~\ref{cc_main_generalization} and Proposition~\ref{pseudo_comfortable_prop} imply Theorem~\ref{galoisfromintro}(a). For each cube $[-j,j]^3 \in \R^3$, project the ``bottom'' faces (i.e., $\{-j\} \times [-j,j] \times [-j,j]$, $[-j,j] \times \{-j\} \times [-j,j]$, $[-j,j] \times [-j,j] \times \{-j\}$) onto $\R^2$ to obtain a quadrangulation $T_j$ of a region $R_j$, as shown in Fig.~\ref{hexagon_build}. The divide associated to each quadrangulation $T_i$ is a pseudoline arrangement. Hence, by Proposition~\ref{pseudo_comfortable_prop}, for any pile $\mathbf{T}_i$ including $T_i$, $\varkappa(\mathbf{T}_i)$ is comfortable. Choose $\mathbf{T}_i$, so that we can associate the vertices of $\varkappa(\mathbf{T}_i)$ with $\{-j, \dots, j\}^3$, so that $\bigcup_{j = 1}^\infty \varkappa^0(\mathbf{T}_i) = \Z^3$. Repeating the K\"onig's infinity lemma argument from the end of the proof of Theorem~\ref{galoisfromintro}(a), Theorem~\ref{cc_main_generalization} implies Theorem~\ref{galoisfromintro}.
\begin{figure}[ht]\centering
\begin{tabular}{ c c }
\begin{tikzpicture}
\draw (0,0)--(0.866,0.5)--(0.866,1.5)--(0, 2)--(-0.866,1.5)--(-0.866,0.5)--(0,0);
\draw (0,1)--(0,2);
\draw (0,1)--(0.866,0.5);
\draw (0,1)--(-0.866,0.5);
\draw ({0.866/2},{0.5/2})--({-0.866/2},{(0.5 + 1)/2})--({-0.866/2},{(1.5 + 2)/2});
\draw ({-0.866/2},{0.5/2})--({0.866/2},{(0.5 + 1)/2})--({0.866/2},{(1.5 + 2)/2});
\draw (-0.866,{(0.5 + 1.5)/2})--(0,{(1 + 2)/2})--(0.866,{(0.5 + 1.5)/2});
\end{tikzpicture} &
\begin{tikzpicture}[scale=2]
\draw (0,0)--(0.866,0.5)--(0.866,1.5)--(0, 2)--(-0.866,1.5)--(-0.866,0.5)--(0,0);
\draw (0,1)--(0,2);
\draw (0,1)--(0.866,0.5);
\draw (0,1)--(-0.866,0.5);
\draw ({0.866/4},{0.5/4})--({-3*0.866/4},{(3*0.5 + 1)/4})--({-3*0.866/4},{(3*1.5 + 2)/4});
\draw ({0.866/2},{0.5/2})--({-0.866/2},{(0.5 + 1)/2})--({-0.866/2},{(1.5 + 2)/2});
\draw ({3*0.866/4},{3*0.5/4})--({-0.866/4},{(0.5 + 3*1)/4})--({-0.866/4},{(1.5 + 3*2)/4});

\draw ({-0.866/4},{0.5/4})--({3*0.866/4},{(3*0.5 + 1)/4})--({3*0.866/4},{(3*1.5 + 2)/4});
\draw ({-0.866/2},{0.5/2})--({0.866/2},{(0.5 + 1)/2})--({0.866/2},{(1.5 + 2)/2});
\draw ({-3*0.866/4},{3*0.5/4})--({0.866/4},{(0.5 + 3*1)/4})--({0.866/4},{(1.5 + 3*2)/4});

\draw (-0.866,{(3*0.5 + 1.5)/4})--(0,{(3*1 + 2)/4})--(0.866,{(3*0.5 + 1.5)/4});
\draw (-0.866,{(0.5 + 1.5)/2})--(0,{(1 + 2)/2})--(0.866,{(0.5 + 1.5)/2});
\draw (-0.866,{(0.5 + 3*1.5)/4})--(0,{(1 + 3*2)/4})--(0.866,{(0.5 + 3*1.5)/4});
\end{tikzpicture}\\
quadrangulation $T_1$ of $R_1$ & quadrangulation $T_2$ of $R_2$
\end{tabular}
\caption{The quadrangulations $T_j$ of regions $R_j$ described in Remark~\ref{cc_main_gen_implies_original}.} \label{hexagon_build}
\end{figure}
\end{Remark}

The rest of this section is dedicated to proving Propositions~\ref{pseudo_comfortable_prop}--\ref{non_example_prop} and Theorems~\ref{cc_main_generalization}--\ref{when_theorem_fails}.

We begin by proving Proposition~\ref{pseudo_comfortable_prop}.

\begin{Lemma} \label{subseq_comf}Let $\mathbf{T} = (T_0, \dots, T_\ell)$ be a pile of quadrangulations of a polygon such that $\varkappa(\mathbf{T})$ is comfortable. Given $0 \le i \le j \le \ell$, let $\mathbf{T}' = (T_i, \dots, T_j)$. Then $\varkappa(\mathbf{T}')$ is comfortable.
\end{Lemma}

\begin{proof}
It suffices to check that (C2) implies (C1) for $\varkappa(\mathbf{T}')$ (see Remark~\ref{comf_cond_remark}). Note that any $\mathbf{u} = (u_C)_{C \in \varkappa^3(\mathbf{T}')}$ satisfying (C2) can be extended to $\mathbf{\tilde{u}} = (u_C)_{C \in \varkappa^3(\mathbf{T})}$ satisfying (C2). Identifying $\cueq(\mathbf{T})$ and $\cueq(\mathbf{T}')$, the fact that there exists $\mathbf{t}$ such that $\psi_{\varkappa(\mathbf{T})}(\mathbf{t}) = \mathbf{\tilde{u}}$ implies that $\psi_{\varkappa(\mathbf{T}')}(\mathbf{t}) = \mathbf{u}$, as desired.
\end{proof}

We can now prove Proposition~\ref{pseudo_comfortable_prop} in the special case where $\mathbf{T}$ be a pile of $\Diamond$-tilings of $\mathbf{P}_n$.

\begin{Lemma} \label{pn_comfortable}
Let $\mathbf{T}$ be a pile of $\Diamond$-tilings of $\mathbf{P}_n$. Then $\varkappa = \varkappa(\mathbf{T})$ is comfortable.
\end{Lemma}

\begin{proof}
Labeling the vertices of $\varkappa$ by subsets of $[n]$ (as in Section~\ref{principal_minors_section}), we can label the cubes in~$\varkappa^3$ by $3$-element subsets of $[n]$ by taking the symmetric difference of the labels of any opposite vertices in the cube. Note that we can extend $\mathbf{T}$ to a longer pile $\mathbf{T}'$ so that for every $A \in \binom{[n]}{3}$, at least one cube of $\varkappa(\mathbf{T}')$ is labeled by $A$. Hence, by Proposition~\ref{subseq_comf}, it suffices to prove the theorem under the additional assumption that each set in $\binom{[n]}{3}$ labels at least one cube in $\varkappa^3$.

Let $A_1$ be the set of $\mathbf{u} \in \{-1,1\}^{\varkappa^3}$ satisfying~(C1), and $A_2$ be the set of $\mathbf{u}$ satisfying~(C2). Because $A_1 \subseteq A_2$, it suffices to show that $\left| A_1 \right| \ge \left| A_2 \right|$ in order to prove that $A_1 = A_2$. We claim that both $A_1$ and $A_2$ have size $2^{\binom{n - 1}{2}}$.

First, we claim that $\left| A_1 \right| \ge 2^{\binom{n - 1}{2}}$. Identify each element $S \in \cueq$ with a $2$-element subset of~$[n]$ by taking the symmetric difference of the labels of any pair of opposite vertices of any tile in~$S$. Note that if $\mathbf{u} = \psi_\varkappa(\mathbf{t})$, and a cube $C$ labeled by~$\{i,j,k\}$, then~$u_C = t_{\{i,j\}} t_{\{i,k\}} t_{\{j,k\}}$. Define a map of vector spaces $f \colon \{-1,1\}^{\binom{[n]}{2}} \rightarrow \{-1, 1\}^{\binom{[n]}{3}}$ where $f \big( ( t_S)_{S \in \binom{[n]}{2}} \big) = ( u_C )_{C \in \binom{[n]}{3}}$ with
\begin{gather*}
u_{\{i,j,k\}} = t_{\{i,j\}} t_{\{i,k\}} t_{\{j,k\}}.
\end{gather*}
If we fix $t_{\{1, 2\}} = \cdots = t_{\{1,n\}} = 1$, then $u_{\{1,j,k\}} = t_{\{j,k\}}$, so the rank of $f$ is at least the number of $2$-element subsets of $\{2,\dots, n\}$, i.e., $\binom{n - 1}{2}$. Hence, it follows that~$\left| A_1 \right| \ge 2^{\binom{n - 1}{2}}$.

Thus, in order to prove the proposition, we must show that $\left| A_2 \right| \le 2^{\binom{n - 1}{2}}$. Note that there are $\binom{n - 1}{2}$ vertices in the interior of any $\Diamond$-tiling of $\mathbf{P}_n$. In choosing $\mathbf{u}$ satisfying~(C2), we can make an arbitrary choice of sign for any cube that shares its bottom vertex with $T_0$, but the signs of the remaining cubes is determined by condition~(C2). Hence, because at most $\binom{n - 1}{2}$ cubes can share their bottom vertices with $T_0$ (the bottom of a cube cannot be on the boundary of $T_0$), there are at most $2^{\binom{n - 1}{2}}$ such $\mathbf{u}$ satisfying condition~(C2), proving our claim.
\end{proof}

We can now prove Proposition~\ref{pseudo_comfortable_prop} in its full generality.

\begin{proof}[Proof of Proposition~\ref{pseudo_comfortable_prop}]
Let $\mathbf{T} = (T_0, \dots, T_\ell)$. We claim that we can ``embed'' the quadrangulations $T_0, \dots, T_\ell$ in $\Diamond$-tilings of $\mathbf{P}_n$. Let $D_0, \dots, D_\ell$ be the divides associated to $T_0, \dots, T_\ell$. Because $D_0, \dots, D_\ell$ are pseudoline arrangements connected by braid moves, we can extend $D_0,\dots,D_\ell$ to pseudoline arrangements $\tilde{D}_0, \dots, \tilde{D}_\ell$, still connected by braid moves, in which every pair of branches intersects exactly once. By Proposition~\ref{pn_pseudoline}, there exists a pile $\mathbf{\tilde{T}} = (\tilde{T}_0, \dots, \tilde{T}_\ell)$ of $\Diamond$-tilings of $\mathbf{P}_n$, for which the divides associated to $\tilde{T}_0, \dots, \tilde{T}_\ell$ are $\tilde{D}_0, \dots, \tilde{D}_\ell$.

By Lemma~\ref{pn_comfortable}, $\varkappa(\mathbf{\tilde{T}})$ is comfortable. The cubical complex $\varkappa(\mathbf{\tilde{T}})$ consists of~$\varkappa = \varkappa(\mathbf{T})$, unioned with $2$-dimensional faces that are not part of any $3$-dimensional cube. Hence, it follows that $\varkappa$ is comfortable as well.
\end{proof}

\begin{proof}[Proof of Proposition~\ref{non_example_prop}]We describe a pile $\mathbf{T} = (T_0, \dots, T_8)$ of quadrangulations of a~squa\-re such that $\varkappa = \varkappa(\mathbf{T})$ is not comfortable. Let $T_0$ be as in Fig.~\ref{non-example_figure}. It is easier to understand this example by looking at the divides associated to $T_0, \dots, T_8$, displayed in Fig.~\ref{non-example_figure2}. Note that the divides associated to these quadrangulations are not pseudoline arrangements. Note that $\varkappa$ has no interior vertices. Hence, every $\mathbf{u} \in \{-1, 1\}^{\varkappa^3}$ satisfies~(C2). However, it is not difficult to check that if $\mathbf{u}$ satisfies~(C1), then the sign on a given cube is determined by the sign on the other~$7$. Hence, $\varkappa$ is not comfortable.
\end{proof}

\begin{figure}[ht]\centering
\begin{tikzpicture}
\draw (0,0)--(5,0)--(5,5)--(0,5)--(0,0);
\draw (2,2)--(2,3)--(3,3)--(3,2)--(2,2);
\draw (1,2)--(1,3)--(2,3)--(2,2)--(1,2);
\draw (3,2)--(3,3)--(4,3)--(4,2)--(3,2);
\draw (2,1)--(2,2)--(3,2)--(3,1)--(2,1);
\draw (2,3)--(2,4)--(3,4)--(3,3)--(2,3);

\draw (0,0)--(2,1);
\draw (0,0)--(1,2);

\draw (5,0)--(3,1);
\draw (5,0)--(4,2);

\draw (5,5)--(3,4);
\draw (5,5)--(4,3);

\draw (0,5)--(1,3);
\draw (0,5)--(2,4);

\draw[blue, ultra thick] (2.5,0)--(2.5,5);
\draw[blue, ultra thick] (0,2.5)--(5,2.5);

\draw[blue, ultra thick] (2.5,2.5) ellipse (1 and 2);
\draw[blue, ultra thick] (2.5,2.5) ellipse (2 and 1);
\end{tikzpicture}
\caption{The quadrangulation $T_0$ from the proof of Proposition~\ref{non_example_prop}, with the associated divide drawn on top in blue.} \label{non-example_figure}
\end{figure}
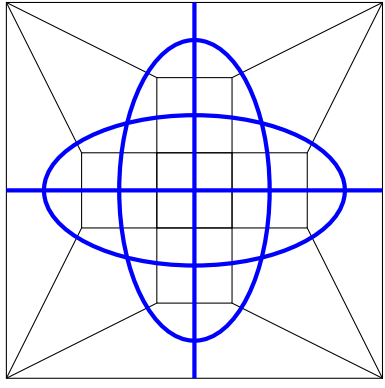
\begin{figure}[ht]\centering
\begin{tabular}{ c c c }
\begin{tikzpicture}[scale=2]
\draw (1,0)--(1,2);
\draw (0,1)--(2,1);
\draw (1,1) ellipse (0.7 and 0.25);
\draw (1,1) ellipse (0.25 and 0.7);
\end{tikzpicture} &
\begin{tikzpicture}[scale=2]
\draw plot [smooth] coordinates {(1,0) (1.8,0.9) (1,1) (1,2)};
\draw (0,1)--(2,1);
\draw (1,1) ellipse (0.7 and 0.25);
\draw (1,1) ellipse (0.25 and 0.7);
\end{tikzpicture} &
\begin{tikzpicture}[scale=2]
\draw plot [smooth] coordinates {(1,0) (1.8,0.9) (1,1) (1,2)};
\draw plot [smooth] coordinates {(2,1) (1.1,1.8) (1,1) (0,1)};
\draw (1,1) ellipse (0.7 and 0.25);
\draw (1,1) ellipse (0.25 and 0.7);
\end{tikzpicture} \\
$T_0$ & $T_1$ & $T_2$ \\
\begin{tikzpicture}[scale=2]
\draw plot [smooth] coordinates {(1,0) (1.8,0.9) (1,1) (0.2,1.1) (1,2)};
\draw plot [smooth] coordinates {(2,1) (1.1,1.8) (1,1) (0,1)};
\draw (1,1) ellipse (0.7 and 0.25);
\draw (1,1) ellipse (0.25 and 0.7);
\end{tikzpicture} &
\begin{tikzpicture}[scale=2]
\draw plot [smooth] coordinates {(1,0) (1.8,0.9) (1,1) (0.2,1.1) (1,2)};
\draw plot [smooth] coordinates {(2,1) (1.1,1.8) (1,1) (0.9,0.2) (0,1)};
\draw (1,1) ellipse (0.7 and 0.25);
\draw (1,1) ellipse (0.25 and 0.7);
\end{tikzpicture} &
\begin{tikzpicture}[scale=2]
\draw plot [smooth] coordinates {(1,0) (1.8,0.9) (1.1,1.6) (1,1) (0.2,1.1) (1,2)};
\draw plot [smooth] coordinates {(2,1) (1.1,1.8) (1,1) (0.9,0.2) (0,1)};
\draw (1,1) ellipse (0.7 and 0.25);
\draw (1,1) ellipse (0.25 and 0.7);
\end{tikzpicture} \\
$T_3$ & $T_4$ & $T_5$ \\
\begin{tikzpicture}[scale=2]
\draw plot [smooth] coordinates {(1,0) (1.8,0.9) (1.1,1.6) (1,1) (0.2,1.1) (1,2)};
\draw plot [smooth] coordinates {(2,1) (1.1,1.8) (0.4,1.1) (1,1) (0.9,0.2) (0,1)};
\draw (1,1) ellipse (0.7 and 0.25);
\draw (1,1) ellipse (0.25 and 0.7);
\end{tikzpicture} &
\begin{tikzpicture}[scale=2]
\draw plot [smooth] coordinates {(1,0) (1.8,0.9) (1.1,1.6) (1,1) (0.9,0.4) (0.2,1.1) (1,2)};
\draw plot [smooth] coordinates {(2,1) (1.1,1.8) (0.4,1.1) (1,1) (0.9,0.2) (0,1)};
\draw (1,1) ellipse (0.7 and 0.25);
\draw (1,1) ellipse (0.25 and 0.7);
\end{tikzpicture} &
\begin{tikzpicture}[scale=2]
\draw plot [smooth] coordinates {(1,0) (1.8,0.9) (1.1,1.6) (1,1) (0.9,0.4) (0.2,1.1) (1,2)};
\draw plot [smooth] coordinates {(2,1) (1.1,1.8) (0.4,1.1) (1,1) (1.6,0.9) (0.9,0.2) (0,1)};
\draw (1,1) ellipse (0.7 and 0.25);
\draw (1,1) ellipse (0.25 and 0.7);
\end{tikzpicture}\\
$T_6$ & $T_7$ & $T_8$
\end{tabular}
\caption{The divides associated to the quadrangulations $T_0, \dots, T_8$ from the proof of Proposition~\ref{non_example_prop}.} \label{non-example_figure2}
\end{figure}
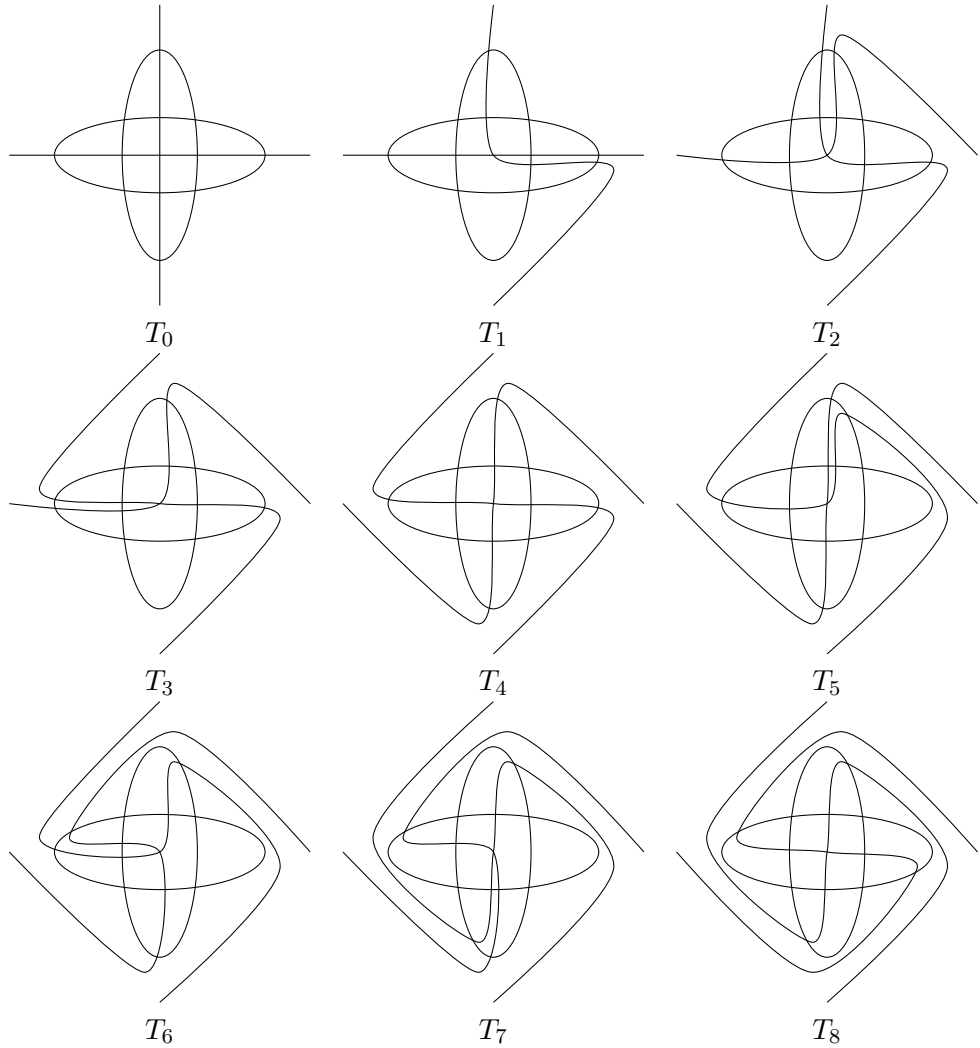

The rest of this section is dedicated to the proofs of Theorems~\ref{cc_main_generalization}--\ref{when_theorem_fails}.

\begin{Definition}Let $\mathbf{T} = (T_0, \dots, T_\ell)$ be a pile of quadrangulations of a polygon, with $\varkappa = \varkappa(\mathbf{T})$, and $\mathbf{x} = (x_s)_{s \in \varkappa^0}$. We say that $\mathbf{x}$ is \emph{generic} if for all extensions of $\xin$ (the restriction of $\mathbf{x}$ to $\varkappa^0(T_0)$) to an array $\mathbf{\tilde{x}}$ indexed by $\varkappa^{02}(\mathbf{T})$ satisfying the K-hexahedron equations, the entries of $\mathbf{\tilde{x}}$ are all nonzero.
\end{Definition}

\begin{Definition}Let $\mathbf{T} = (T_0, \dots, T_\ell)$ be a pile of quadrangulations of a polygon, with $\varkappa = \varkappa(\mathbf{T})$. Given an array $\xtin = (x_s)_{s \in \varkappa^{02}(T_0)}$ and $\mathbf{t} = \left( t_{[s]} \right)_{[s] \in \cueq}$, set $\mathbf{t} \cdot \xtin = (y_s)_{s \in \varkappa^{02}(T_0)}$, where
\begin{gather*}
y_s = \begin{cases} x_s & \text{if } s \in \varkappa^0(T_0), \\ t_{[s]} x_s & \text{if } s \in \varkappa^2(T_0). \end{cases}
\end{gather*}
Given a generic array $\xtin$ indexed by $\varkappa^{02}(T_0)$, define $(\xtin)^{\uparrow \varkappa^{02}}$ to be the unique extension of~$\xtin$ to an array indexed by $\varkappa^{02}$ satisfying the K-hexahedron equations. Define $(\xtin)^{\uparrow \varkappa^0}$ to be the restriction of $(\xtin)^{\uparrow \varkappa^{02}}$ to $\varkappa^0$.
\end{Definition}

\begin{Lemma} \label{change_one_cube_lemma}Let $\mathbf{T} = (T_0, \dots, T_\ell)$ be a pile of quadrangulations of a polygon, with $\varkappa = \varkappa(\mathbf{T})$. Fix a generic array $\xtin$ indexed by $\varkappa^{02}(T_0)$ satisfying equation~\eqref{khexspec} for $s \in \varkappa^2(T_0)$, and $\mathbf{t} \in \{-1, 1\}^{\cueq}$. Then the following are equivalent:
\begin{itemize}\itemsep=0pt
\item $\psi_\varkappa(\mathbf{t})$ has value $1$ on $C_1, \dots, C_{i - 1}$, and value $-1$ on $C_i$,
\item $(\xtin)^{\uparrow \varkappa^0}$ and $(\mathbf{t} \cdot \xtin)^{\uparrow \varkappa^0}$ agree at $\varkappa^0(T_0), \dots, \varkappa^0(T_{i - 1})$ but not at $\varkappa^0(T_i)$.
\end{itemize}
\end{Lemma}

\begin{proof}The proof follows directly from Lemma~\ref{sign_prop_khex_cube}.
\end{proof}

\begin{Lemma} \label{min_vert_on_bottom}
Let $\mathbf{T} = (T_0, \dots, T_\ell)$ be a pile of quadrangulations of a polygon, with $\varkappa = \varkappa(\mathbf{T})$. Let $\mathbf{x}$ and $\mathbf{x}'$ be generic and distinct coherent solutions of the Kashaev equation, both indexed by~$\varkappa^0$, such that $\mathbf{x}$ and $\mathbf{x}'$ agree at $\varkappa^0(T_0)$. Let $i$ be the minimum value such that $\mathbf{x}$ and $\mathbf{x}'$ do not agree at $\varkappa^0(T_i)$. Then the cube $C_i$ shares its bottom vertex with $T_0$.
\end{Lemma}

\begin{proof}Assume (for contradiction) that $C_i$ doesn't share its bottom vertex with $T_0$. Then the bottom vertex of $C_i$ must be an interior vertex of $\varkappa$. Hence, by the coherence and genericity of~$\mathbf{x}$ and $\mathbf{x}'$, the values of $\mathbf{x}$ and $\mathbf{x}'$ are uniquely determined by their values at $\varkappa^0(T_0), \dots, \varkappa^0(T_{i - 1})$, which are the same for $\mathbf{x}$ and $\mathbf{x}'$. Hence, $\mathbf{x}$ and $\mathbf{x}'$ agree at the top vertex of $C_i$, and thus agree at $\varkappa^0(T_i)$, a contradiction.
\end{proof}

We can now prove a weaker version of Theorem~\ref{cc_main_generalization}, under the additional constraint of genericity.

\begin{Corollary} \label{main_cc_with_gen_cond}Let $\mathbf{T}$ be a pile of quadrangulations of a polygon such that $\varkappa = \varkappa(\mathbf{T})$ is comfortable. Any generic, coherent solution of the Kashaev equation $\mathbf{x} = (x_s)_{s \in \varkappa^0}$ can be extended to $\mathbf{\tilde{x}} = (x_s)_{s \in \varkappa^{02}}$ satisfying the K-hexahedron equations.
\end{Corollary}

\begin{proof}Choose an arbitrary extension of $\xin$, the restriction of $\mathbf{x}$ to $\varkappa^0(T_0)$ to an array $\xtin'$ indexed by $\varkappa^{02}(T_0)$ satisfying equation~\eqref{khexspec} for $s \in \varkappa^2(T_0)$. The result follows once we can show that there exists $\mathbf{t} \in \{-1, 1\}^{\cueq}$ such that $(\mathbf{t} \cdot \xtin')^{\uparrow \varkappa^{02}}$ agrees with $\mathbf{x}$ on $\varkappa^0$.

We proceed by induction, and assume that there exists $\mathbf{t}$ such that $(\mathbf{t} \cdot \xtin')^{\uparrow \varkappa^{02}}$ agrees with~$\mathbf{x}$ on $\varkappa^0(T_j)$ for $j = 0, \dots, i - 1$. If $(\mathbf{t} \cdot \xtin')^{\uparrow \varkappa^{02}}$ agrees with $\mathbf{x}$ on $\varkappa^0(T_j)$ for $j = 0, \dots, i$, we are done. Suppose that $(\mathbf{t} \cdot \xtin')^{\uparrow \varkappa^{02}}$ does not agree with $\mathbf{x}$ on~$\varkappa^0(T_i)$. We need to find $\mathbf{t}_i$ such that $(\mathbf{t}_i \mathbf{t} \cdot \xtin')^{\uparrow \varkappa^{02}}$ agrees with $\mathbf{x}$ on $\varkappa^0(T_j)$ for $j = 0, \dots, i$. By Lemma~\ref{change_one_cube_lemma}, this is equivalent to finding $\mathbf{t}_i$ such that $\psi_\varkappa(\mathbf{t}_i)$ is $1$ on $C_1, \dots, C_{i - 1}$, and $-1$ on $C_i$. By Lemma~\ref{min_vert_on_bottom}, the cube~$C_i$ shares its bottom vertex with $T_0$. Hence, there exists $\mathbf{u} = (u_s) \in \{-1,1\}^{\varkappa^3}$ satisfying (C2) such that $u_{C_1} = \cdots = u_{C_{i - 1}} = 1$ and $u_{C_i} = -1$. (For example, choose $\mathbf{u}$ so that $u_{C_i} = -1$, and $u_C = 1$ for all other cubes~$C$ that share a bottom vertex with~$T_0$. Then the remaining values are determined by condition~(C2).) Because $\varkappa$ is comfortable, there exists $\mathbf{t}_i$ such that $\psi_\varkappa(\mathbf{t}_i) = \mathbf{u}$, as desired.
\end{proof}

\begin{proof}[Proof of Theorem~\ref{cc_main_generalization}.]We need to loosen the condition that $\mathbf{x}$ is generic from Corollary~\ref{main_cc_with_gen_cond} to the conditions that $\mathbf{x}$ has nonzero components and satisfies condition~\eqref{face_nonzero}.

Let $\mathbf{x} \in (\C^*)^{\varkappa^0}$ be a coherent solution of the Kashaev equation with nonzero components that satisfies condition~\eqref{face_nonzero}. It is straightforward to show that there exists a sequence $\mathbf{x}_1, \mathbf{x}_2, \ldots \in (\C^*)^{\varkappa^0}$ of generic, coherent solutions of the Kashaev equation that converge pointwise to $\mathbf{x}$. By Corollary~\ref{main_cc_with_gen_cond}, there exist $\mathbf{\tilde{x}}_1, \mathbf{\tilde{x}}_2, \ldots \in (\C^*)^{\varkappa^{02}}$ satisfying the K-hexahedron equations such that $\mathbf{\tilde{x}}_i$ restricts to $\mathbf{x}_i$. There exists a subsequence of $\mathbf{\tilde{x}}_1, \mathbf{\tilde{x}}_2, \dots$ that converges to an array~$\mathbf{\tilde{x}}$. (For each $s \in \varkappa^2(\mathbf{T})$, we can partition the sequence $\mathbf{\tilde{x}}_1, \mathbf{\tilde{x}}_2, \dots$ into two sequences, each of which converges at $s$. Because $\varkappa^2(\mathbf{T})$ is finite, the claim follows.) The array $\mathbf{\tilde{x}}$ must satisfy the K-hexahedron equations and restrict to $\mathbf{x}$, so we are done.
\end{proof}

In order to complete the proof of Theorem~\ref{when_theorem_fails}, we will need the following technical lemma.

\begin{Lemma} \label{not_comf_cube_lemma}
Let $\mathbf{T} = (T_0, \dots, T_\ell)$ be a pile of quadrangulations of a polygon such that $\varkappa(\mathbf{T})$ is not comfortable, but $\varkappa(T_0, \dots, T_{\ell - 1})$ is comfortable. Let $C_\ell$ be the cube of $\varkappa$ corresponding to the flip from $T_{\ell - 1}$ to $T_\ell$.
\begin{enumerate}\itemsep=0pt
\item[$(a)$] Let $v$ be the bottom vertex of the cube $C_\ell$, i.e., let $v$ be the vertex of $T_{\ell - 1}$ not in~$T_\ell$. Then $v$ is in $T_0$.
\item[$(b)$] Let $\mathbf{w} = (w_{C})_{C \in \varkappa^3}$ where $w_{C_\ell} = -1$, and $w_{C} = 1$ for $C \not= C_\ell$. Then $\mathbf{w}$ is not in the image of $\psi_\varkappa$.
\end{enumerate}
\end{Lemma}

\begin{proof}Let $\varkappa' = \varkappa(T_0, \dots, T_{\ell - 1})$. Let
\begin{itemize}\itemsep=0pt
\item $a_1$ be the number of $\mathbf{u} \in \{-1,1\}^{(\varkappa')^3}$ satisfying~(C1),
\item $a_2$ be the number of $\mathbf{u} \in \{-1,1\}^{(\varkappa')^3}$ satisfying~(C2),
\item $b_1$ be the number of $\mathbf{u} \in \{-1,1\}^{\varkappa^3}$ satisfying~(C1), and
\item $b_2$ be the number of $\mathbf{u} \in \{-1,1\}^{\varkappa^3}$ satisfying~(C2).
\end{itemize}
Because $a_1$, $a_2$, $b_1$, $b_2$ enumerate the elements of vector fields over $\mathbb{F}_2$, all four quantities must be powers of $2$. Because $\varkappa'$ is comfortable, $a_1 = a_2$. Because $\varkappa$ is not comfortable, $b_1 < b_2$. It is clear that $b_1 \le 2 a_1$ and $b_2 \le 2 a_2$. Hence, it follows that~$a_1 = a_2 = b_1 = b_2 / 2$.

Assume (for contradiction) that $v$ is not in $T_0$, so $v$ is in the interior of $\varkappa$. But then if $\mathbf{u} = (u_{C})_{C \in \varkappa^3}$ satisfies~(C2),
\begin{gather*}
u_{C_\ell} = \prod_{C \in \varkappa^3\colon w \in C \not= C_\ell} u_{C},
\end{gather*}
so $a_2 = b_2$, a contradiction. Hence, we have proved~(a).

Because $a_1 = b_1$, it follows that for each $\mathbf{u}' \in \{-1,1\}^{(\varkappa')^3}$ satisfying~(C1), there exists exactly one $\mathbf{u} \in \{-1,1\}^{\varkappa^3}$ satisfying~(C1) that restricts to $\mathbf{u}'$. Because $\mathbf{u} = (u_C)_{C \in \varkappa^3}$ and $\mathbf{u}' = (u_C)_{C \in (\varkappa')^3}$ where $u_C = 1$ for all $C$ satisfy~(C1), $\mathbf{w}$ cannot satisfy~(C1), proving~(b).
\end{proof}

\begin{proof}[Proof of Theorem~\ref{when_theorem_fails}] Without loss of generality, we assume that $\varkappa(T_0, \dots, T_{\ell - 1})$ is comfortable. (If not, let $m$ be minimum so that $\varkappa(T_0, {\dots}, T_m)$ is not comfortable, but $\varkappa(T_0, {\dots}, T_{m{-}1})$ is comfortable. If we can prove the theorem for $\varkappa(T_0, \dots, T_m)$, it follows that it holds for $\varkappa(\mathbf{T})$.)

We now construct an array $\mathbf{x}$ satisfying the desired conditions. Let $C$ be the cube of $\varkappa$ corresponding to the flip from $T_{\ell - 1}$ to $T_\ell$, and let $v$ be the top vertex of $C$ (i.e., $v$ is the new vertex in $T_i$). Choose arbitrary positive values for $\xin$. Extend $\xin$ to $\mathbf{x}$ by the positive Kashaev recurrence until we reach $v$, where we choose the other value such that $K^C(\mathbf{x}) = 0$.

By construction, $\mathbf{x}$ restricted to $\varkappa(T_0, \dots, T_{\ell - 1})$ satisfies the positive Kashaev recurrence, and hence is a coherent solution of the Kashaev equation. By Lemma~\ref{not_comf_cube_lemma}(a), no vertices of $C$ are in the interior of $\varkappa$. Hence, $\mathbf{x}$ is a coherent solution of the Kashaev equation.

Next, we show that $\mathbf{x}$ cannot be extended to an array indexed by $\varkappa^{02}$ satisfying the K-hexahedron equations. Let $\mathbf{x}_{\text{pK}}$ be the array satisfying the positive Kashaev recurrence that restricts to $\xin$ at $T_0$ (so $\mathbf{x}_{\text{pK}}$ agrees with $\mathbf{x}$ everywhere except~$v$). Let $\mathbf{\tilde{x}}_{\text{pK}}$ be an extension of~$\mathbf{x}_{\text{pK}}$ to $\varkappa^{02}$ satisfying the K-hexahedron equations. Assume (for contradiction) that there exists $\mathbf{t} \in \{-1,1\}^{\tilde{\varkappa}_2}$ such that $\mathbf{\tilde{x}}(\mathbf{t} \cdot (\mathbf{\tilde{x}}_{\text{pK}})_0)$ restricts to~$\mathbf{x}$. Hence, by Lemma~\ref{change_one_cube_lemma}, $\psi_{\varkappa}(\mathbf{t})$ has value $-1$ at $C$, and value $1$ everywhere else. But Lemma~\ref{not_comf_cube_lemma}(b) says that array is not in the image of~$\psi_\varkappa$, a contradiction. Hence, no such~$\mathbf{t}$ exists, so~$\mathbf{x}$ cannot be extended to an array indexed by~$\varkappa^{02}$ satisfying the K-hexahedron equations.
\end{proof}

\section{Proofs of Corollary~\ref{all_pminors_corollary} and Theorem~\ref{main_all_minors_simple}} \label{matrix_proofs}

This section contains the proofs of Corollary~\ref{all_pminors_corollary} and Theorem~\ref{main_all_minors_simple}.

We use the following lemma in proving Corollary~\ref{all_pminors_corollary}.

\begin{Lemma} \label{same_label_values_agree}
Let $\mathbf{T}$ be a pile of $\Diamond$-tilings of $\mathbf{P}_n$, with $\varkappa = \varkappa(\mathbf{T})$. Let $\mathbf{x} = (x_s) \in (\C^*)^{\varkappa^0}$ be a coherent solution of the Kashaev equation satisfying condition~\eqref{nonzero_face_condition}. Suppose $s_1, s_2 \in \varkappa^0$ are labeled by the same subset of $[n]$. Then $x_{s_1} = x_{s_2}$.
\end{Lemma}

\begin{proof}
Note that due to the homogeneity of the Kashaev equation and the coherence equations (equation~\eqref{coherence_gen_def}), we can rescale the components of $\mathbf{x}$ to obtain a standard array. Hence, we can assume that $\mathbf{x}$ is standard.

By Theorem~\ref{main_thm_cyclic_zonotope}(a), we can extend $\mathbf{x}$ to an array $\mathbf{\tilde{x}}$ indexed by $\varkappa^{02}$ satisfying the K-hexahedron equations. Note that we can choose a sequence $\mathbf{\tilde{x}}_1, \mathbf{\tilde{x}}_2, \dots$ of standard arrays indexed by $\varkappa^{02}$ satisfying the K-hexahedron equations converging to $\mathbf{\tilde{x}}$ such that the restriction of~$\mathbf{\tilde{x}}_i$ to~$\varkappa^{02}(T)$ for any tiling $T$ in $\mathbf{T}$ is generic. By Theorems~\ref{tiles_to_matrix_bij} and~\ref{sym_mat_net}, there exist symmetric $n \times n$ matrices such that $\mathbf{\tilde{x}}_i = \mathbf{\tilde{x}}_{\varkappa(\mathbf{T})}(M_i)$. Hence, the components of $\mathbf{\tilde{x}}_i$ at $s_1$ and $s_2$ must agree, so the components of $\mathbf{\tilde{x}}$ at $s_1$ and $s_2$ must agree.
\end{proof}

\begin{proof}[Proof of Corollary~\ref{all_pminors_corollary}.]
The first bullet point implies the second two by Corollary~\ref{sym_matrix_k_hex_arb}, and it is obvious that the third implies the second. Thus, we need to show that the second bullet point implies the first.

Next, suppose $\mathbf{T} = (T_0, \dots, T_\ell)$ is a pile of $\Diamond$-tilings of $\mathbf{P}_n$ in which every $I \subseteq [n]$ labels at least one vertex of $\varkappa(\mathbf{T})$, and $\mathbf{x} = \mathbf{x}_{\varkappa(\mathbf{T})}(\mathbf{\bar{x}})$ is a coherent solution of the Kashaev equation. Let $\mathbf{T}' = (T_0, \dots, T_\ell, \dots, T_{\ell'})$ be an extension of $\mathbf{T}$ where $\mathbf{T}'$ contains the tiling $T_{\textup{min},n}$. By Lemma~\ref{same_label_values_agree}, $\mathbf{x}_{\varkappa(\mathbf{T}')}(\mathbf{\bar{x}})$ is the unique extension of $\mathbf{x}$ to $\varkappa^0(\mathbf{T}')$. By Theorem~\ref{main_thm_cyclic_zonotope}(a), there exists an array $\mathbf{\tilde{x}}$ indexed by $\varkappa^{02}(\mathbf{T}')$ extending $\mathbf{x}_{\varkappa(\mathbf{T}')}(\mathbf{\bar{x}})$ that satisfies the K-hexahedron equations. By Theorem~\ref{tmin_laurent}, Proposition~\ref{flip_hex}, and Corollary~\ref{sym_for_min_tiling} (all of which are due to Kenyon and Pemantle \cite{otherhex}), there exists a unique symmetric matrix $M$ such that $\mathbf{\tilde{x}} = \mathbf{\tilde{x}}_{\varkappa(\mathbf{T}')}(M)$, so $M$ satisfies condition~\eqref{xi_in_terms_of_pminor}.
\end{proof}

Next, we shall work towards a proof of Theorem~\ref{main_all_minors_simple}.

\begin{Proposition} \label{4_cube_coherence_condition_holds}
Let $M$ be an $n \times n$ symmetric matrix, and let $\mathbf{\bar{x}} = \mathbf{\bar{x}}(M)$. Then for all $I \subseteq [n]$ and $A \in \binom{[n]}{4}$, equation~\eqref{4_cube_coherence} holds.
\end{Proposition}

\begin{proof}
Note that it suffices to prove Proposition~\ref{4_cube_coherence_condition_holds} for generic, symmetric $M$, because any symmetric matrix can be written as a limit of generic, symmetric matrices. Fix a generic, symmetric $n \times n$ matrix $M$ for the rest of the proof.

For $I \subset [n]$ and distinct $i,j \in [n]$, let
\begin{gather*}
x_{I, \{i,j\}} = (-1)^{\lfloor (\left| I' \right| +1) / 2 \rfloor} M_{I' \cup \{i\}}^{I' \cup \{j\}},
\end{gather*}
where $I' = I \setminus \{i,j\}$. Note that if $\mathbf{T}$ is a pile of $\Diamond$-tilings of $\mathbf{P}_n$, a cube of $\varkappa(\mathbf{T})$ containing vertices labeled by $I$ and $I \cup \{i,j,k\}$ for $i,j,k \not\in I$ has top/bottom vertices labeled by $I \cup \{j\}$ and $I \cup \{i,k\}$. Hence, by Lemma~\ref{cornersofcubelemma} and the fact that $\mathbf{\tilde{x}}_{\varkappa(\mathbf{T})}(M)$ satisfies the K-hexahedron equations, where $\mathbf{T}$ is any pile of $\Diamond$-tilings of $\mathbf{P}_n$, it follows that
\begin{gather} \label{kij_as_prod_ap_minors}
K_{I, \{i,j,k\}}(\mathbf{\bar{x}}) = \pm x_{I, \{i,j\}} x_{I, \{\{i,k\}} x_{I, \{j,k\}},
\end{gather}
where the plus sign appears on the right-hand side of equation~\eqref{kij_as_prod_ap_minors} if either
\begin{itemize}\itemsep=0pt
\item $i, k \in I$ and $j \not\in I$, or
\item $j \in I$ and $i,k \not\in I$,
\end{itemize}
and the minus sign appears otherwise.

Let $I \subseteq [n]$ and $A \in \binom{[n]}{4}$. It is straightforward to check that an even number of $\{i < j < k\} \in \binom{A}{3}$ satisfy neither of the bullet points above. Hence, by equation~\eqref{kij_as_prod_ap_minors}, it follows that
\begin{gather*}
\prod_{J \in \binom{A}{3}} K_{I, J}(\mathbf{\bar{x}}) = \prod_{J \in \binom{A}{2}} x_{I,J}^2 = \prod_{J \in \binom{A}{2}} L_{I,J}(\mathbf{\bar{x}}),
\end{gather*}
as desired.
\end{proof}

The reader may want to review Example~\ref{two_piles_c4} before proceeding with the following lemma.%ref

\begin{Lemma} \label{replacing_z43_lemma}Fix an array $\mathbf{\bar{x}} = (x_I )_{I \subseteq [4]}$, satisfying the conditions that
\begin{itemize}\itemsep=0pt
\item $x_I \not= 0$ for all $I \subseteq [4]$,
\item for any $I \subseteq [4]$ and distinct $i,j \in [4]$, $L_{I, \{i,j\}} \not= 0$,
\item for all $I \subseteq [4]$ and distinct $i,j,k \in [4]$, equation~\eqref{kash_for_matrices} holds,
\item for all $I \subseteq [4]$, equation~\eqref{coh_in_n4_cond} holds.
\end{itemize}
Let $\mathbf{T}_1 = (T_{1,0}, \dots, T_{1,4}), \mathbf{T}_2 = (T_{2,0}, \dots, T_{2,4}) \in \mathcal{C}(4)$ be the two distinct piles in~$\mathcal{C}(4)$. Let $\mathbf{\tilde{x}}_1 \in (\C^*)^{\varkappa^{02}(\mathbf{T}_1)}$ be an extension of $\mathbf{x}_{\varkappa(\mathbf{T}_1)}(\mathbf{\bar{x}})$ satisfying the K-hexahedron equations. Then there exists an extension $\mathbf{\tilde{x}}_2 \in (\C^*)^{\varkappa^{02}(\mathbf{T}_2)}$ of $\mathbf{x}_{\varkappa(\mathbf{T}_2)}(\mathbf{\bar{x}})$ satisfying the K-hexahedron equations that agrees with $\mathbf{\tilde{x}}_1$ on $\varkappa^{02}(T_{1,0}) = \varkappa^{02}(T_{2,0})$.
\end{Lemma}

\begin{proof}
By the homogeneity of equations~\eqref{kash_for_matrices}--\eqref{coh_in_n4_cond}, we can rescale the components of $\mathbf{\bar{x}}$ so that $x_\varnothing = 1$. Hence, it follows from Corollary~\ref{n4_with_coh_cond} that there exists an extension $\mathbf{\tilde{x}}_1' \in (\C^*)^{\varkappa^{02}(\mathbf{T}_1)}$ of $\mathbf{x}_{\varkappa(\mathbf{T}_1)}(\mathbf{\bar{x}})$ and an extension $\mathbf{\tilde{x}}_2' \in (\C^*)^{\varkappa^{02}(\mathbf{T}_2)}$ of $\mathbf{x}_{\varkappa(\mathbf{T}_2)}(\mathbf{\bar{x}})$ that agree on $\varkappa^{02}(T_{1,0}) = \varkappa^{02}(T_{2,0})$.

Let $\xtin'$ be the restriction of $\mathbf{\tilde{x}}_1'$ to $\varkappa^{02}(T_{1,0})$. Let $\mathbf{t} \in \{-1,1\}^{\varkappa^2(T_{1,0})}$ (where we associate~$\varkappa^2(T_{1,0})$ with $\tilde{\varkappa}_2(\mathbf{T}_1)$) so that $\mathbf{\tilde{x}}_1 = (\mathbf{t} \cdot \xtin')^{\uparrow \varkappa^{02}(\mathbf{T}_1)}$. By Lemma~\ref{change_one_cube_lemma}, because $\mathbf{\tilde{x}}_1'$ and $\mathbf{\tilde{x}}_1$ agree on $\varkappa^0(\mathbf{T}_1)$, $\psi_{\varkappa(\mathbf{T}_1)}(\mathbf{t})$ has value $1$ at every cube of $\varkappa(\mathbf{T}_1)$. Because $\psi_{\varkappa(\mathbf{T}_1)}(\mathbf{t})$ has value $1$ at every cube of $\varkappa(\mathbf{T}_1)$, $\psi_{\varkappa(\mathbf{T}_2)}(\mathbf{t})$ has value $1$ at every cube of $\varkappa(\mathbf{T}_2)$. Hence, $(\mathbf{t} \cdot \xtin')^{\uparrow \varkappa^2(\mathbf{T}_2)}$ agrees with~$\mathbf{\tilde{x}}_1$ on $\varkappa^{02}(T_{1,0}) = \varkappa^{02}(T_{2,0})$ and restricts to $\mathbf{x}_{\varkappa(\mathbf{T}_2)}(\mathbf{\bar{x}})$.
\end{proof}

The reader may want to review Definition~\ref{three_dim_flip} before proceeding with the following definition.%ref

\begin{Definition}
Let $\mathbf{T}_1 = (T_{1,0}, \dots, T_{1,\ell})$ and $\mathbf{T}_2 = (T_{2,0}, \dots, T_{2,\ell})$ be two piles, such that the directed cubical complexes $\varkappa(\mathbf{T}_1)$ and $\varkappa(\mathbf{T}_2)$ are related by a flip. Label the vertices of $\varkappa(\mathbf{T}_1)$ and $\varkappa(\mathbf{T}_2)$ involved in the $3$-flip by subsets of $[4]$, as in Fig.~\ref{labeled_3_flip_fig}. Let $\mathbf{x}_1 \in (\C^*)^{\varkappa^0(\mathbf{T}_1)}$ and $\mathbf{x}_2 \in (\C^*)^{\varkappa^0(\mathbf{T}_2)}$ be arrays satisfying condition~\eqref{nonzero_face_condition}. We say that the pair $(\mathbf{x}_1, \mathbf{x}_2)$ is \emph{K-flipped} when
\begin{itemize}\itemsep=0pt
\item $\mathbf{x}_1$ and $\mathbf{x}_2$ agree everywhere, except at the vertex at which $\varkappa(\mathbf{T}_1)$ and $\varkappa(\mathbf{T}_2)$ differ,
\item writing $\mathbf{\bar{x}} = (x_I)_{I \subseteq [4]}$, where $x_I$ is the component of $\mathbf{x}_1$ and/or $\mathbf{x}_2$ at the vertex labeled by~$I$, $\mathbf{\bar{x}}$ satisfies equation~\eqref{kash_for_matrices} for all $I \subseteq [4]$ and distinct $i,j,k \in [4]$ and equation~\eqref{coh_in_n4_cond} for all $I \subseteq [4]$.
\end{itemize}
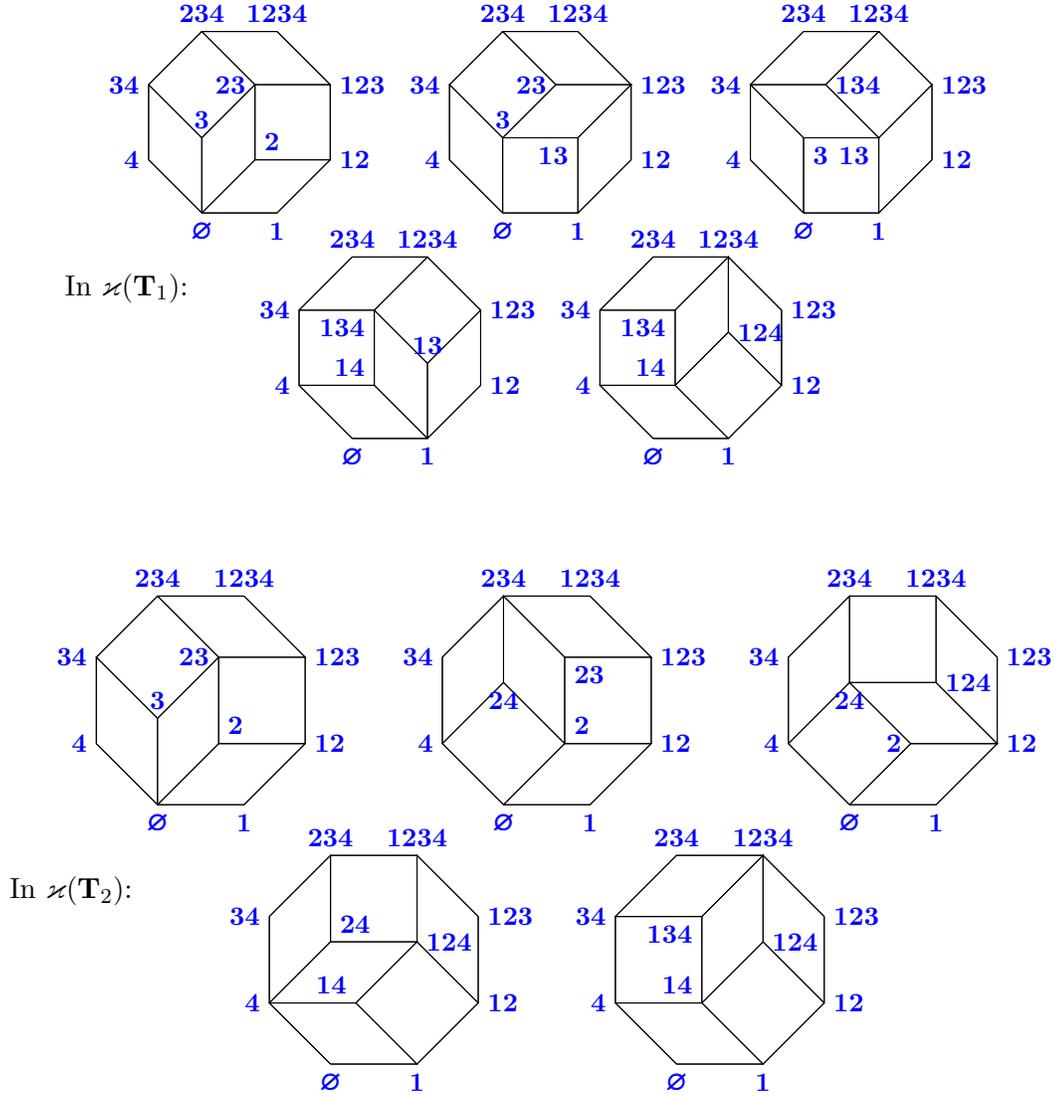
\begin{figure}[ht]\centering

\begin{tabular}{ c }
\begin{tikzpicture}[scale=1.0]%[scale=1.15]

\draw (-1,-1) node{In $\varkappa(\mathbf{T}_1)$:};

\draw (0,0)--(1,0)--({1+sqrt(2)/2},{sqrt(2)/2})--({1+sqrt(2)/2},{1+sqrt(2)/2})--({1},{1+sqrt(2)})--({0},{1+sqrt(2)})--({-sqrt(2)/2},{1+sqrt(2)/2})--({-sqrt(2)/2},{sqrt(2)/2})--(0,0);

\draw (0,0)--(1,0)--({1+sqrt(2)/2},{sqrt(2)/2})--({sqrt(2)/2},{sqrt(2)/2})--(0,0);
\draw (0,0)--({sqrt(2)/2},{sqrt(2)/2})--({sqrt(2)/2},{1+sqrt(2)/2})--(0,1)--(0,0);
\draw (0,0)--(0,1)--({-sqrt(2)/2},{1 + sqrt(2)/2})--({-sqrt(2)/2},{sqrt(2)/2})--(0,0);
\draw[shift={({sqrt(2)/2},{sqrt(2)/2})}] (0,0)--(1,0)--(1,1)--(0,1)--(0,0);
\draw[shift={(0,1)}] (0,0)--({sqrt(2)/2},{sqrt(2)/2})--(0,{sqrt(2)})--({-sqrt(2)/2},{sqrt(2)/2})--(0,0);
\draw[shift={({sqrt(2)/2},{1+sqrt(2)/2})}] (0,0)--(1,0)--({1 - sqrt(2)/2},{sqrt(2)/2})--({- sqrt(2)/2},{sqrt(2)/2})--(0,0);

{
\color{blue}
\small
\draw (0,0) node[anchor= north]{$\bm{\varnothing}$};
\draw (1,0) node[anchor=north]{$\mathbf{1}$};
\draw ({1+sqrt(2)/2},{sqrt(2)/2}) node[anchor=west]{$\mathbf{12}$};
\draw ({1+sqrt(2)/2},{1+sqrt(2)/2}) node[anchor=west]{$\mathbf{123}$};
\draw ({1},{1+sqrt(2)}) node[anchor=south]{$\mathbf{1234}$};
\draw ({0},{1+sqrt(2)}) node[anchor=south]{$\mathbf{234}$};
\draw ({-sqrt(2)/2},{1+sqrt(2)/2}) node[anchor=east]{$\mathbf{34}$};
\draw ({-sqrt(2)/2},{sqrt(2)/2}) node[anchor=east]{$\mathbf{4}$};

\draw ({sqrt(2)/2},{sqrt(2)/2}) node[anchor=south west]{$\mathbf{2}$};
\draw (0,{1}) node[anchor=south]{$\mathbf{3}$};
\draw ({sqrt(2)/2},{1+sqrt(2)/2}) node[anchor=east]{$\mathbf{23}$};
}

\begin{scope}[shift={(4,0)}]
\draw (0,0)--(1,0)--({1+sqrt(2)/2},{sqrt(2)/2})--({1+sqrt(2)/2},{1+sqrt(2)/2})--({1},{1+sqrt(2)})--({0},{1+sqrt(2)})--({-sqrt(2)/2},{1+sqrt(2)/2})--({-sqrt(2)/2},{sqrt(2)/2})--(0,0);

\draw[shift={(0,1)}] (0,0)--(1,0)--({1+sqrt(2)/2},{sqrt(2)/2})--({sqrt(2)/2},{sqrt(2)/2})--(0,0);
\draw[shift={(1,0)}] (0,0)--({sqrt(2)/2},{sqrt(2)/2})--({sqrt(2)/2},{1+sqrt(2)/2})--(0,1)--(0,0);
\draw (0,0)--(0,1)--({-sqrt(2)/2},{1 + sqrt(2)/2})--({-sqrt(2)/2},{sqrt(2)/2})--(0,0);
\draw (0,0)--(1,0)--(1,1)--(0,1)--(0,0);
\draw[shift={(0,1)}] (0,0)--({sqrt(2)/2},{sqrt(2)/2})--(0,{sqrt(2)})--({-sqrt(2)/2},{sqrt(2)/2})--(0,0);
\draw[shift={({sqrt(2)/2},{1+sqrt(2)/2})}] (0,0)--(1,0)--({1 - sqrt(2)/2},{sqrt(2)/2})--({- sqrt(2)/2},{sqrt(2)/2})--(0,0);

{
\color{blue}
\small
\draw (0,0) node[anchor= north]{$\bm{\varnothing}$};
\draw (1,0) node[anchor=north]{$\mathbf{1}$};
\draw ({1+sqrt(2)/2},{sqrt(2)/2}) node[anchor=west]{$\mathbf{12}$};
\draw ({1+sqrt(2)/2},{1+sqrt(2)/2}) node[anchor=west]{$\mathbf{123}$};
\draw ({1},{1+sqrt(2)}) node[anchor=south]{$\mathbf{1234}$};
\draw ({0},{1+sqrt(2)}) node[anchor=south]{$\mathbf{234}$};
\draw ({-sqrt(2)/2},{1+sqrt(2)/2}) node[anchor=east]{$\mathbf{34}$};
\draw ({-sqrt(2)/2},{sqrt(2)/2}) node[anchor=east]{$\mathbf{4}$};

\draw (1,1) node[anchor=north east]{$\mathbf{13}$};
\draw (0,{1}) node[anchor=south]{$\mathbf{3}$};
\draw ({sqrt(2)/2},{1+sqrt(2)/2}) node[anchor=east]{$\mathbf{23}$};
}
\end{scope}

\begin{scope}[shift={(8,0)}]

\draw (0,0)--(1,0)--({1+sqrt(2)/2},{sqrt(2)/2})--({1+sqrt(2)/2},{1+sqrt(2)/2})--({1},{1+sqrt(2)})--({0},{1+sqrt(2)})--({-sqrt(2)/2},{1+sqrt(2)/2})--({-sqrt(2)/2},{sqrt(2)/2})--(0,0);

\draw[shift={({-sqrt(2)/2},{1+sqrt(2)/2})}] (0,0)--(1,0)--({1+sqrt(2)/2},{sqrt(2)/2})--({sqrt(2)/2},{sqrt(2)/2})--(0,0);
\draw[shift={(1,0)}] (0,0)--({sqrt(2)/2},{sqrt(2)/2})--({sqrt(2)/2},{1+sqrt(2)/2})--(0,1)--(0,0);
\draw (0,0)--(0,1)--({-sqrt(2)/2},{1 + sqrt(2)/2})--({-sqrt(2)/2},{sqrt(2)/2})--(0,0);
\draw (0,0)--(1,0)--(1,1)--(0,1)--(0,0);
\draw[shift={(1,1)}] (0,0)--({sqrt(2)/2},{sqrt(2)/2})--(0,{sqrt(2)})--({-sqrt(2)/2},{sqrt(2)/2})--(0,0);
\draw[shift={(0,1)}] (0,0)--(1,0)--({1 - sqrt(2)/2},{sqrt(2)/2})--({- sqrt(2)/2},{sqrt(2)/2})--(0,0);

{
\color{blue}
\small
\draw (0,0) node[anchor= north]{$\bm{\varnothing}$};
\draw (1,0) node[anchor=north]{$\mathbf{1}$};
\draw ({1+sqrt(2)/2},{sqrt(2)/2}) node[anchor=west]{$\mathbf{12}$};
\draw ({1+sqrt(2)/2},{1+sqrt(2)/2}) node[anchor=west]{$\mathbf{123}$};
\draw ({1},{1+sqrt(2)}) node[anchor=south]{$\mathbf{1234}$};
\draw ({0},{1+sqrt(2)}) node[anchor=south]{$\mathbf{234}$};
\draw ({-sqrt(2)/2},{1+sqrt(2)/2}) node[anchor=east]{$\mathbf{34}$};
\draw ({-sqrt(2)/2},{sqrt(2)/2}) node[anchor=east]{$\mathbf{4}$};

\draw (1,1) node[anchor=north east]{$\mathbf{13}$};
\draw (0,{1}) node[anchor=north west]{$\mathbf{3}$};
\draw ({1-sqrt(2)/2},{1+sqrt(2)/2}) node[anchor=west]{$\mathbf{134}$};
}

\end{scope}

\begin{scope}[shift={(2,-3)}]

\draw (0,0)--(1,0)--({1+sqrt(2)/2},{sqrt(2)/2})--({1+sqrt(2)/2},{1+sqrt(2)/2})--({1},{1+sqrt(2)})--({0},{1+sqrt(2)})--({-sqrt(2)/2},{1+sqrt(2)/2})--({-sqrt(2)/2},{sqrt(2)/2})--(0,0);

\draw[shift={({-sqrt(2)/2},{1+sqrt(2)/2})}] (0,0)--(1,0)--({1+sqrt(2)/2},{sqrt(2)/2})--({sqrt(2)/2},{sqrt(2)/2})--(0,0);
\draw[shift={(1,0)}] (0,0)--({sqrt(2)/2},{sqrt(2)/2})--({sqrt(2)/2},{1+sqrt(2)/2})--(0,1)--(0,0);
\draw[shift={(1,0)}] (0,0)--(0,1)--({-sqrt(2)/2},{1 + sqrt(2)/2})--({-sqrt(2)/2},{sqrt(2)/2})--(0,0);
\draw[shift={({-sqrt(2)/2},{sqrt(2)/2})}] (0,0)--(1,0)--(1,1)--(0,1)--(0,0);
\draw[shift={(1,1)}] (0,0)--({sqrt(2)/2},{sqrt(2)/2})--(0,{sqrt(2)})--({-sqrt(2)/2},{sqrt(2)/2})--(0,0);
\draw (0,0)--(1,0)--({1 - sqrt(2)/2},{sqrt(2)/2})--({- sqrt(2)/2},{sqrt(2)/2})--(0,0);

{
\color{blue}
\small
\draw (0,0) node[anchor= north]{$\bm{\varnothing}$};
\draw (1,0) node[anchor=north]{$\mathbf{1}$};
\draw ({1+sqrt(2)/2},{sqrt(2)/2}) node[anchor=west]{$\mathbf{12}$};
\draw ({1+sqrt(2)/2},{1+sqrt(2)/2}) node[anchor=west]{$\mathbf{123}$};
\draw ({1},{1+sqrt(2)}) node[anchor=south]{$\mathbf{1234}$};
\draw ({0},{1+sqrt(2)}) node[anchor=south]{$\mathbf{234}$};
\draw ({-sqrt(2)/2},{1+sqrt(2)/2}) node[anchor=east]{$\mathbf{34}$};
\draw ({-sqrt(2)/2},{sqrt(2)/2}) node[anchor=east]{$\mathbf{4}$};

\draw (1,1) node[anchor=south]{$\mathbf{13}$};
\draw ({1 - sqrt(2)/2},{sqrt(2)/2}) node[anchor=south east]{$\mathbf{14}$};
\draw ({1-sqrt(2)/2},{1+sqrt(2)/2}) node[anchor=north east]{$\mathbf{134}$};
}

\end{scope}

\begin{scope}[shift={(6,-3)}]
\draw (0,0)--(1,0)--({1+sqrt(2)/2},{sqrt(2)/2})--({1+sqrt(2)/2},{1+sqrt(2)/2})--({1},{1+sqrt(2)})--({0},{1+sqrt(2)})--({-sqrt(2)/2},{1+sqrt(2)/2})--({-sqrt(2)/2},{sqrt(2)/2})--(0,0);

\draw[shift={({-sqrt(2)/2},{1+sqrt(2)/2})}] (0,0)--(1,0)--({1+sqrt(2)/2},{sqrt(2)/2})--({sqrt(2)/2},{sqrt(2)/2})--(0,0);
\draw[shift={({1 - sqrt(2)/2},{sqrt(2)/2})}] (0,0)--({sqrt(2)/2},{sqrt(2)/2})--({sqrt(2)/2},{1+sqrt(2)/2})--(0,1)--(0,0);
\draw[shift={({1+sqrt(2)/2},{sqrt(2)/2})}] (0,0)--(0,1)--({-sqrt(2)/2},{1 + sqrt(2)/2})--({-sqrt(2)/2},{sqrt(2)/2})--(0,0);
\draw[shift={({-sqrt(2)/2},{sqrt(2)/2})}] (0,0)--(1,0)--(1,1)--(0,1)--(0,0);
\draw[shift={(1,0)}] (0,0)--({sqrt(2)/2},{sqrt(2)/2})--(0,{sqrt(2)})--({-sqrt(2)/2},{sqrt(2)/2})--(0,0);
\draw (0,0)--(1,0)--({1 - sqrt(2)/2},{sqrt(2)/2})--({- sqrt(2)/2},{sqrt(2)/2})--(0,0);

{
\color{blue}
\small
\draw (0,0) node[anchor= north]{$\bm{\varnothing}$};
\draw (1,0) node[anchor=north]{$\mathbf{1}$};
\draw ({1+sqrt(2)/2},{sqrt(2)/2}) node[anchor=west]{$\mathbf{12}$};
\draw ({1+sqrt(2)/2},{1+sqrt(2)/2}) node[anchor=west]{$\mathbf{123}$};
\draw ({1},{1+sqrt(2)}) node[anchor=south]{$\mathbf{1234}$};
\draw ({0},{1+sqrt(2)}) node[anchor=south]{$\mathbf{234}$};
\draw ({-sqrt(2)/2},{1+sqrt(2)/2}) node[anchor=east]{$\mathbf{34}$};
\draw ({-sqrt(2)/2},{sqrt(2)/2}) node[anchor=east]{$\mathbf{4}$};

\draw (1,{sqrt(2)}) node[anchor=west]{$\mathbf{124}$};
\draw ({1 - sqrt(2)/2},{sqrt(2)/2}) node[anchor=south east]{$\mathbf{14}$};
\draw ({1-sqrt(2)/2},{1+sqrt(2)/2}) node[anchor=north east]{$\mathbf{134}$};
}
\end{scope}
\end{tikzpicture} \vspace{1cm}\\
\begin{tikzpicture}[scale=1.15]

\draw (-1,-1) node{In $\varkappa(\mathbf{T}_2)$:};

\draw (0,0)--(1,0)--({1+sqrt(2)/2},{sqrt(2)/2})--({1+sqrt(2)/2},{1+sqrt(2)/2})--({1},{1+sqrt(2)})--({0},{1+sqrt(2)})--({-sqrt(2)/2},{1+sqrt(2)/2})--({-sqrt(2)/2},{sqrt(2)/2})--(0,0);

\draw (0,0)--(1,0)--({1+sqrt(2)/2},{sqrt(2)/2})--({sqrt(2)/2},{sqrt(2)/2})--(0,0);
\draw (0,0)--({sqrt(2)/2},{sqrt(2)/2})--({sqrt(2)/2},{1+sqrt(2)/2})--(0,1)--(0,0);
\draw (0,0)--(0,1)--({-sqrt(2)/2},{1 + sqrt(2)/2})--({-sqrt(2)/2},{sqrt(2)/2})--(0,0);
\draw[shift={({sqrt(2)/2},{sqrt(2)/2})}] (0,0)--(1,0)--(1,1)--(0,1)--(0,0);
\draw[shift={(0,1)}] (0,0)--({sqrt(2)/2},{sqrt(2)/2})--(0,{sqrt(2)})--({-sqrt(2)/2},{sqrt(2)/2})--(0,0);
\draw[shift={({sqrt(2)/2},{1+sqrt(2)/2})}] (0,0)--(1,0)--({1 - sqrt(2)/2},{sqrt(2)/2})--({- sqrt(2)/2},{sqrt(2)/2})--(0,0);

{
\color{blue}
\small
\draw (0,0) node[anchor= north]{$\bm{\varnothing}$};
\draw (1,0) node[anchor=north]{$\mathbf{1}$};
\draw ({1+sqrt(2)/2},{sqrt(2)/2}) node[anchor=west]{$\mathbf{12}$};
\draw ({1+sqrt(2)/2},{1+sqrt(2)/2}) node[anchor=west]{$\mathbf{123}$};
\draw ({1},{1+sqrt(2)}) node[anchor=south]{$\mathbf{1234}$};
\draw ({0},{1+sqrt(2)}) node[anchor=south]{$\mathbf{234}$};
\draw ({-sqrt(2)/2},{1+sqrt(2)/2}) node[anchor=east]{$\mathbf{34}$};
\draw ({-sqrt(2)/2},{sqrt(2)/2}) node[anchor=east]{$\mathbf{4}$};

\draw ({sqrt(2)/2},{sqrt(2)/2}) node[anchor=south west]{$\mathbf{2}$};
\draw (0,{1}) node[anchor=south]{$\mathbf{3}$};
\draw ({sqrt(2)/2},{1+sqrt(2)/2}) node[anchor=east]{$\mathbf{23}$};
}

\begin{scope}[shift={(4,0)}]
\draw (0,0)--(1,0)--({1+sqrt(2)/2},{sqrt(2)/2})--({1+sqrt(2)/2},{1+sqrt(2)/2})--({1},{1+sqrt(2)})--({0},{1+sqrt(2)})--({-sqrt(2)/2},{1+sqrt(2)/2})--({-sqrt(2)/2},{sqrt(2)/2})--(0,0);

\draw (0,0)--(1,0)--({1+sqrt(2)/2},{sqrt(2)/2})--({sqrt(2)/2},{sqrt(2)/2})--(0,0);
\draw[shift={({- sqrt(2)/2},{sqrt(2)/2})}] (0,0)--({sqrt(2)/2},{sqrt(2)/2})--({sqrt(2)/2},{1+sqrt(2)/2})--(0,1)--(0,0);
\draw[shift={({sqrt(2)/2},{sqrt(2)/2})}] (0,0)--(0,1)--({-sqrt(2)/2},{1 + sqrt(2)/2})--({-sqrt(2)/2},{sqrt(2)/2})--(0,0);
\draw[shift={({sqrt(2)/2},{sqrt(2)/2})}] (0,0)--(1,0)--(1,1)--(0,1)--(0,0);
\draw (0,0)--({sqrt(2)/2},{sqrt(2)/2})--(0,{sqrt(2)})--({-sqrt(2)/2},{sqrt(2)/2})--(0,0);
\draw[shift={({sqrt(2)/2},{1+sqrt(2)/2})}] (0,0)--(1,0)--({1 - sqrt(2)/2},{sqrt(2)/2})--({- sqrt(2)/2},{sqrt(2)/2})--(0,0);

{
\color{blue}
\small
\draw (0,0) node[anchor= north]{$\bm{\varnothing}$};
\draw (1,0) node[anchor=north]{$\mathbf{1}$};
\draw ({1+sqrt(2)/2},{sqrt(2)/2}) node[anchor=west]{$\mathbf{12}$};
\draw ({1+sqrt(2)/2},{1+sqrt(2)/2}) node[anchor=west]{$\mathbf{123}$};
\draw ({1},{1+sqrt(2)}) node[anchor=south]{$\mathbf{1234}$};
\draw ({0},{1+sqrt(2)}) node[anchor=south]{$\mathbf{234}$};
\draw ({-sqrt(2)/2},{1+sqrt(2)/2}) node[anchor=east]{$\mathbf{34}$};
\draw ({-sqrt(2)/2},{sqrt(2)/2}) node[anchor=east]{$\mathbf{4}$};

\draw ({sqrt(2)/2},{1+sqrt(2)/2}) node[anchor=north west]{$\mathbf{23}$};
\draw ({sqrt(2)/2},{sqrt(2)/2}) node[anchor=south west]{$\mathbf{2}$};
\draw ({0},{sqrt(2)}) node[anchor=north]{$\mathbf{24}$};
}
\end{scope}

\begin{scope}[shift={(8,0)}]

\draw (0,0)--(1,0)--({1+sqrt(2)/2},{sqrt(2)/2})--({1+sqrt(2)/2},{1+sqrt(2)/2})--({1},{1+sqrt(2)})--({0},{1+sqrt(2)})--({-sqrt(2)/2},{1+sqrt(2)/2})--({-sqrt(2)/2},{sqrt(2)/2})--(0,0);

\draw (0,0)--(1,0)--({1+sqrt(2)/2},{sqrt(2)/2})--({sqrt(2)/2},{sqrt(2)/2})--(0,0);
\draw[shift={({- sqrt(2)/2},{sqrt(2)/2})}] (0,0)--({sqrt(2)/2},{sqrt(2)/2})--({sqrt(2)/2},{1+sqrt(2)/2})--(0,1)--(0,0);
\draw[shift={({1+sqrt(2)/2},{sqrt(2)/2})}] (0,0)--(0,1)--({-sqrt(2)/2},{1 + sqrt(2)/2})--({-sqrt(2)/2},{sqrt(2)/2})--(0,0);
\draw[shift={({0},{sqrt(2)})}] (0,0)--(1,0)--(1,1)--(0,1)--(0,0);
\draw (0,0)--({sqrt(2)/2},{sqrt(2)/2})--(0,{sqrt(2)})--({-sqrt(2)/2},{sqrt(2)/2})--(0,0);
\draw[shift={({sqrt(2)/2},{sqrt(2)/2})}] (0,0)--(1,0)--({1 - sqrt(2)/2},{sqrt(2)/2})--({- sqrt(2)/2},{sqrt(2)/2})--(0,0);

{
\color{blue}
\small
\draw (0,0) node[anchor= north]{$\bm{\varnothing}$};
\draw (1,0) node[anchor=north]{$\mathbf{1}$};
\draw ({1+sqrt(2)/2},{sqrt(2)/2}) node[anchor=west]{$\mathbf{12}$};
\draw ({1+sqrt(2)/2},{1+sqrt(2)/2}) node[anchor=west]{$\mathbf{123}$};
\draw ({1},{1+sqrt(2)}) node[anchor=south]{$\mathbf{1234}$};
\draw ({0},{1+sqrt(2)}) node[anchor=south]{$\mathbf{234}$};
\draw ({-sqrt(2)/2},{1+sqrt(2)/2}) node[anchor=east]{$\mathbf{34}$};
\draw ({-sqrt(2)/2},{sqrt(2)/2}) node[anchor=east]{$\mathbf{4}$};

\draw (1,{sqrt(2)}) node[anchor=west]{$\mathbf{124}$};
\draw ({sqrt(2)/2},{sqrt(2)/2}) node[anchor=east]{$\mathbf{2}$};
\draw ({0},{sqrt(2)}) node[anchor=north]{$\mathbf{24}$};
}

\end{scope}

\begin{scope}[shift={(2,-3)}]

\draw (0,0)--(1,0)--({1+sqrt(2)/2},{sqrt(2)/2})--({1+sqrt(2)/2},{1+sqrt(2)/2})--({1},{1+sqrt(2)})--({0},{1+sqrt(2)})--({-sqrt(2)/2},{1+sqrt(2)/2})--({-sqrt(2)/2},{sqrt(2)/2})--(0,0);

\draw[shift={({-sqrt(2)/2},{sqrt(2)/2})}] (0,0)--(1,0)--({1+sqrt(2)/2},{sqrt(2)/2})--({sqrt(2)/2},{sqrt(2)/2})--(0,0);
\draw[shift={({- sqrt(2)/2},{sqrt(2)/2})}] (0,0)--({sqrt(2)/2},{sqrt(2)/2})--({sqrt(2)/2},{1+sqrt(2)/2})--(0,1)--(0,0);
\draw[shift={({1+sqrt(2)/2},{sqrt(2)/2})}] (0,0)--(0,1)--({-sqrt(2)/2},{1 + sqrt(2)/2})--({-sqrt(2)/2},{sqrt(2)/2})--(0,0);
\draw[shift={({0},{sqrt(2)})}] (0,0)--(1,0)--(1,1)--(0,1)--(0,0);
\draw[shift={(1,0)}] (0,0)--({sqrt(2)/2},{sqrt(2)/2})--(0,{sqrt(2)})--({-sqrt(2)/2},{sqrt(2)/2})--(0,0);
\draw (0,0)--(1,0)--({1 - sqrt(2)/2},{sqrt(2)/2})--({- sqrt(2)/2},{sqrt(2)/2})--(0,0);

{
\color{blue}
\small
\draw (0,0) node[anchor= north]{$\bm{\varnothing}$};
\draw (1,0) node[anchor=north]{$\mathbf{1}$};
\draw ({1+sqrt(2)/2},{sqrt(2)/2}) node[anchor=west]{$\mathbf{12}$};
\draw ({1+sqrt(2)/2},{1+sqrt(2)/2}) node[anchor=west]{$\mathbf{123}$};
\draw ({1},{1+sqrt(2)}) node[anchor=south]{$\mathbf{1234}$};
\draw ({0},{1+sqrt(2)}) node[anchor=south]{$\mathbf{234}$};
\draw ({-sqrt(2)/2},{1+sqrt(2)/2}) node[anchor=east]{$\mathbf{34}$};
\draw ({-sqrt(2)/2},{sqrt(2)/2}) node[anchor=east]{$\mathbf{4}$};

\draw (1,{sqrt(2)}) node[anchor=west]{$\mathbf{124}$};
\draw ({1 - sqrt(2)/2},{sqrt(2)/2}) node[anchor=south east]{$\mathbf{14}$};
\draw ({0},{sqrt(2)}) node[anchor=south west]{$\mathbf{24}$};
}

\end{scope}

\begin{scope}[shift={(6,-3)}]
\draw (0,0)--(1,0)--({1+sqrt(2)/2},{sqrt(2)/2})--({1+sqrt(2)/2},{1+sqrt(2)/2})--({1},{1+sqrt(2)})--({0},{1+sqrt(2)})--({-sqrt(2)/2},{1+sqrt(2)/2})--({-sqrt(2)/2},{sqrt(2)/2})--(0,0);

\draw[shift={({-sqrt(2)/2},{1+sqrt(2)/2})}] (0,0)--(1,0)--({1+sqrt(2)/2},{sqrt(2)/2})--({sqrt(2)/2},{sqrt(2)/2})--(0,0);
\draw[shift={({1 - sqrt(2)/2},{sqrt(2)/2})}] (0,0)--({sqrt(2)/2},{sqrt(2)/2})--({sqrt(2)/2},{1+sqrt(2)/2})--(0,1)--(0,0);
\draw[shift={({1+sqrt(2)/2},{sqrt(2)/2})}] (0,0)--(0,1)--({-sqrt(2)/2},{1 + sqrt(2)/2})--({-sqrt(2)/2},{sqrt(2)/2})--(0,0);
\draw[shift={({-sqrt(2)/2},{sqrt(2)/2})}] (0,0)--(1,0)--(1,1)--(0,1)--(0,0);
\draw[shift={(1,0)}] (0,0)--({sqrt(2)/2},{sqrt(2)/2})--(0,{sqrt(2)})--({-sqrt(2)/2},{sqrt(2)/2})--(0,0);
\draw (0,0)--(1,0)--({1 - sqrt(2)/2},{sqrt(2)/2})--({- sqrt(2)/2},{sqrt(2)/2})--(0,0);

{
\color{blue}
\small
\draw (0,0) node[anchor= north]{$\bm{\varnothing}$};
\draw (1,0) node[anchor=north]{$\mathbf{1}$};
\draw ({1+sqrt(2)/2},{sqrt(2)/2}) node[anchor=west]{$\mathbf{12}$};
\draw ({1+sqrt(2)/2},{1+sqrt(2)/2}) node[anchor=west]{$\mathbf{123}$};
\draw ({1},{1+sqrt(2)}) node[anchor=south]{$\mathbf{1234}$};
\draw ({0},{1+sqrt(2)}) node[anchor=south]{$\mathbf{234}$};
\draw ({-sqrt(2)/2},{1+sqrt(2)/2}) node[anchor=east]{$\mathbf{34}$};
\draw ({-sqrt(2)/2},{sqrt(2)/2}) node[anchor=east]{$\mathbf{4}$};

\draw (1,{sqrt(2)}) node[anchor=west]{$\mathbf{124}$};
\draw ({1 - sqrt(2)/2},{sqrt(2)/2}) node[anchor=south east]{$\mathbf{14}$};
\draw ({1-sqrt(2)/2},{1+sqrt(2)/2}) node[anchor=north east]{$\mathbf{134}$};
}
\end{scope}
\end{tikzpicture}
\end{tabular}
\caption{Labeling the vertices involved in a flip between $\varkappa(\mathbf{T}_1)$ and $\varkappa(\mathbf{T}_2)$ in Lemma~\ref{flips_work_in_larger_complex} with subsets of $[4]$ (in blue).} \label{labeled_3_flip_fig}
\end{figure}
\end{Definition}

\begin{Lemma} \label{flips_work_in_larger_complex}Let $\mathbf{T}_1 = (T_{1,0}, \dots, T_{1, \ell})$ and $\mathbf{T}_2 = (T_{2, 0}, \dots, T_{2, \ell})$ be two piles, such that the directed cubical complexes $\varkappa(\mathbf{T}_1)$ and $\varkappa(\mathbf{T}_2)$ are related by a flip. Let $\mathbf{x}_1 \in (\C^*)^{\varkappa^0(\mathbf{T}_1)}$ and $\mathbf{x}_2 \in (\C^*)^{\varkappa^0(\mathbf{T}_2)}$ be arrays satisfying condition~\eqref{nonzero_face_condition}, such that the pair $(\mathbf{x}_1, \mathbf{x}_2)$ is K-flipped.
\begin{enumerate}\itemsep=0pt
\item[$(a)$] Then $\mathbf{x}_1$ is a coherent solution of the Kashaev equation if and only if $\mathbf{x}_2$ is a coherent solution of the Kashaev equation.
\item[$(b)$] Suppose $\mathbf{x}_1$ and $\mathbf{x}_2$ are both coherent solutions of the Kashaev equation. Let $\mathbf{\tilde{x}}_1 \!\in\! (\C^*)^{\varkappa^{02}(\mathbf{T}_1)}$ be an extension of $\mathbf{x}_1$ to $\varkappa^{02}(\mathbf{T}_1)$, and let $\xtin$ be the restriction of $\mathbf{\tilde{x}}_1$ to $\varkappa^{02}(T_{1,0})$. Then, identifying $\varkappa^{02}(T_{1,0})$ and $\varkappa^{02}(T_{2,0})$, $(\xtin)^{\uparrow \varkappa^{02}(\mathbf{T}_2)}$ restricts to $\mathbf{x}_2$.
\end{enumerate}
\end{Lemma}

\begin{proof}This result follows from Lemma~\ref{replacing_z43_lemma}.
\end{proof}

\begin{Lemma} \label{bubbling_up_cubes}There exists a sequence $\mathbf{T}_0, \dots, \mathbf{T}_{\binom{n}{3} - \binom{n - 1}{2}} \in \mathcal{C}(n)$ of piles, where we write $\mathbf{T}_i = (T_{i, 0}, \dots, T_{i, \binom{n}{3}})$ for $i = 0, \dots, \binom{n}{3} - \binom{n - 1}{2}$, such that
\begin{itemize}\itemsep=0pt
\item the directed cubical complexes $\varkappa(\mathbf{T}_{i - 1})$ and $\varkappa(\mathbf{T}_i)$ are related by a flip for $i = 1, \dots, \ell$,
\item the directed cubes of $\varkappa(\mathbf{T}_0)$ corresponding to the flips between $T_{0, i - 1}$ and $T_{0, i}$ for $i = 1, \dots, \binom{n - 1}{2}$ share their bottom vertex with $T_{0,0}$,
\item for $i = 1, \dots, \binom{n}{3} - \binom{n - 1}{2}$, $T_{0, j} = \cdots = T_{i - 1, j}$ for $j = i + \binom{n - 1}{2}, \dots, \binom{n}{3}$,
\item for $i = 1, \dots, \binom{n}{3} - \binom{n - 1}{2}$, the directed cube of $\varkappa(\mathbf{T}_{i - 1})$ corresponding to the flip between $T_{i - 1, i - 1 + \binom{n}{3} - \binom{n - 1}{2}}$ and $T_{i - 1, i - 1 + \binom{n}{3} - \binom{n - 1}{2}}$ is the top of the four cubes of $\varkappa(\mathbf{T}_{i - 1})$ involved in the flip between $\varkappa(\mathbf{T}_{i - 1})$ and $\varkappa(\mathbf{T}_{i})$.
\end{itemize}
\end{Lemma}

\begin{Remark}
The idea behind Lemma~\ref{bubbling_up_cubes} is as follows. Let $C_i$ be the cube of $\varkappa(\mathbf{T}_0)$ corresponding to the flip between $T_{0, i - 1}$ and $T_{0, i}$. The second bullet point states that the cubes $C_1, \dots, C_{\binom{n - 1}{2}}$ share their bottom vertex with $T_{0,0}$, and hence do not have their bottom vertices in the interior of $\varkappa(\mathbf{T}_0)$. As a consequence of the remaining bullet points, there is a sequence of $\binom{n}{3} - \binom{n - 1}{2}$ flips on the directed cubical complex $\varkappa(\mathbf{T}_0)$ in which $C_{\binom{n - 1}{2} + 1}, \dots, C_{\binom{n}{3}}$ (in that order) are the top cubes involved in the flips.
\end{Remark}

\begin{proof}[Proof of Lemma~\ref{bubbling_up_cubes}]
Define a total order $<_{\textup{lex}}$ on $\binom{[n]}{k}$, where $\{i_1 < \cdots < i_k \} <_{\textup{lex}} \{ i_1' < \cdots < i_k' \}$ when there exists $j$ such that $i_\ell = i_\ell'$ for $\ell < j$, and $i_j < i_j'$. Set $\{\alpha_1 <_{\textup{lex}} \cdots <_{\textup{lex}} \alpha_{\binom{n}{3}}\} = \binom{[n]}{3}$. Note that the permutation $\sigma_0 = (\alpha_1, \dots, \alpha_{\binom{n}{3}})$ of $\binom{[n]}{3}$ is admissible (see Definition~\ref{admissible_def}). Let $\mathbf{T}_0$ be the pile corresponding to $(\alpha_1, \dots, \alpha_{\binom{n}{3}})$ (see Theorem~\ref{admissible_prop}). Note that $1 \in \alpha_i$ for $i = 1, \dots, \binom{n - 1}{2}$, so the second bullet point holds.

We now construct the piles $\mathbf{T}_1, \dots, \mathbf{T}_{\binom{n}{3} - \binom{n - 1}{2}}$ inductively as follows. For $i = 1,\dots, \binom{n}{3} - \binom{n - 1}{2}$, the admissible permutation $\sigma_i$ corresponding to $\mathbf{T}_i$ should have the following properties:
\begin{itemize}\itemsep=0pt
\item The inversion set of $\sigma_i$ is $\{\{1\} \cup \alpha_{\binom{n - 1}{2} + 1}, \dots, \{1\} \cup \alpha_{\binom{n - 1}{2} + i}\}$. Hence, the inversion sets of $\sigma_{i - 1}$ and $\sigma_i$ differ by the element $\{1\} \cup \alpha_{\binom{n - 1}{2} + i}$, so $\varkappa(\mathbf{T}_{i - 1})$ and $\varkappa(\mathbf{T}_i)$ are related by a flip. Hence, the first bullet point holds.
\item Writing $\sigma_{i - 1} = (\beta_1, \dots, \beta_{\binom{n}{3}})$, $\beta_j = \alpha_j$ for $j = i + \binom{n - 1}{2}, \dots, \binom{n}{3}$. Hence, the third bullet point holds. Because $\beta_{i + \binom{n - 1}{2}} = \alpha_{i + \binom{n - 1}{2}}$ and the flip between $\varkappa(\mathbf{T}_{i - 1})$ and $\varkappa(\mathbf{T}_i)$ consists of the inclusion of $\{1\} \cup \alpha_{i + \binom{n - 1}{2}}$ to the inversion set, the fourth bullet point follows.
\end{itemize}
For $i = 1, \dots, \binom{n}{3} - \binom{n - 1}{2}$, write $\sigma_{i - 1} = (\beta_1, \dots, \beta_{\binom{n}{3}})$. We want to obtain $\sigma_i$. Write $\alpha_{i + \binom{n - 1}{2}} = \beta_{i + \binom{n - 1}{2}} = \{ i_1 < i_2 < i_3 \}$. Let $(\gamma_1, \dots, \gamma_{i + \binom{n - 1}{2} - 4})$ be the subsequence of $(\beta_1, \dots, \beta_{i + \binom{n - 1}{2} - 4})$ excluding $\{1, i_1, i_2\}$, $\{1, i_1, i_3\}$, and $\{1, i_2, i_3\}$. Setting
\begin{gather*}
\sigma_i = \big( \gamma_1, \dots, \gamma_{i + \binom{n - 1}{2} - 4}, \{ i_1, i_2, i_3 \}, \{1, i_2, i_3\}, \{1, i_1, i_3\}, \{1, i_1, i_2\}, \beta_{i + \binom{n - 1}{2}}, \dots, \beta_{\binom{n}{3}} \big),
\end{gather*}
it is straightforward to check that $\sigma_i$ is an admissible permutation with the desired proper\-ties.
\end{proof}

\begin{Remark}
The pile $\mathbf{T}_0$ constructed in the proof of Lemma~\ref{bubbling_up_cubes} is a representative for the smallest element of the third higher Bruhat order. The sequence $\big(\varkappa(\mathbf{T}_0), \dots, \varkappa\big(\mathbf{T}_{\binom{n}{3} - \binom{n - 1}{2}}\big)\big)$, where $\mathbf{T}_0, \dots, \mathbf{T}_{\binom{n}{3} - \binom{n - 1}{2}}$ are the piles constructed in the proof of Lemma~\ref{bubbling_up_cubes}, are the first $\binom{n}{3} - \binom{n - 1}{2} + 1$ elements for a representative for the smallest element of the fourth higher Bruhat order. See~\cite{manin_sch} or~\cite{ziegler} for further discussion of higher Bruhat orders.
\end{Remark}

\begin{Lemma} \label{works_for_one_fcs}
Let $\mathbf{\bar{x}} = (x_I)_{I \subseteq [n]}$ be an array satisfying the conditions that $L_{I, \{i,j\}} \not= 0$ for any $I \subseteq [n]$ and distinct $i,j \in [n]$, and $x_{\varnothing} = 1$. Suppose that for all $I \subseteq [n]$ and distinct $i,j,k \in [n]$, equation~\eqref{kash_for_matrices} holds, and for all $I \subseteq [n]$ and $A \in \binom{[n]}{4}$, equation~\eqref{4_cube_coherence} holds. Then there exists $\mathbf{T} \in \mathcal{C}(n)$ such that $\mathbf{x}_{\varkappa(\mathbf{T})}(\mathbf{\bar{x}})$ is a coherent solution of the Kashaev equation.
\end{Lemma}

\begin{proof}
Let $\mathbf{T}_0, \dots, \mathbf{T}_{\binom{n}{3} - \binom{n - 1}{2}} \in \mathcal{C}(n)$ be a sequence of piles satisfying the conditions of Lem\-ma~\ref{bubbling_up_cubes}, where we write $\mathbf{T}_i = (T_{i, 0}, \dots, T_{i, \binom{n}{3}})$ for $i = 0, \dots, \binom{n}{3} - \binom{n - 1}{2}$. We will show that $\mathbf{x}_{\varkappa(\mathbf{T}_0)}(\mathbf{\bar{x}})$ is a coherent solution of the Kashaev equation.

We claim that $\mathbf{x}_{\varkappa(T_{0,0}, \dots, T_{0,j})}(\mathbf{\bar{x}})$ is a coherent solution of the Kashaev equation for $j = 1, \dots, \binom{n}{3}$ and proceed by induction. Because equation~\eqref{kash_for_matrices} holds for all $I \subseteq [n]$ and distinct $i, j, k \in [n]$, $\mathbf{x}_{\varkappa(T_{0,0}, \dots, T_{0,j})}(\mathbf{\bar{x}})$ satisfies the Kashaev equation. Hence, we only have to check coherence, i.e., we need to check that equation~\eqref{coherence_condition_for_wir} holds for every interior vertex of $\varkappa(T_{0,0}, \dots, T_{0,j})$. %
For $j = 1, \dots, \binom{n - 1}{2}$, none of the vertices $\varkappa^0(T_{0,0}, \dots, T_{0,j})$ are interior vertices of $\varkappa(T_{0,0}, \dots, T_{0,j})$, so $\mathbf{x}_{\varkappa(T_{0,0}, \dots, T_{0,j})}(\mathbf{\bar{x}})$ is a coherent solution of the Kashaev equation. By our inductive hypothesis, $\mathbf{x}_{\varkappa(T_{0,0}, \dots, T_{0,j - 1})}(\mathbf{\bar{x}})$ is a coherent solution of the Kashaev equation. By construction, $\varkappa(T_{i - 1, 0}, \dots, T_{i - 1, j - 1})$ and $\varkappa(T_{i, 0}, \dots, T_{i, j - 1})$ are related by a flip for $i = 1, \dots, j - \binom{n - 1}{2} - 1$. Hence, the pairs $\big( \mathbf{x}_{\varkappa(T_{i - 1, 0}, \dots, T_{i - 1, j - 1})}(\mathbf{\bar{x}}), \mathbf{x}_{\varkappa(T_{i, 0}, \dots, T_{i, j - 1})}(\mathbf{\bar{x}}) \big)$ are K-flipped for $i = 1, \dots, j - \binom{n - 1}{2} - 1$ by the conditions of the lemma. By repeated applications of Lemma~\ref{flips_work_in_larger_complex}(a), it follows that $\mathbf{x}_{\varkappa \big(T_{j - \binom{n - 1}{2} - 1, 0}, \dots, T_{j - \binom{n - 1}{2} - 1, j - 1} \big) }(\mathbf{\bar{x}})$ is a coherent solution of the Kashaev equation. By construction, the cube of $\varkappa \big(T_{j - \binom{n - 1}{2} - 1, 0}, \dots, T_{j - \binom{n - 1}{2} - 1, j} \big)$ corresponding to the flip between $T_{j - \binom{n - 1}{2} - 1, j - 1}$ and $T_{j - \binom{n - 1}{2} - 1, j}$ is the top of four cubes where a flip can take place. Hence, because equation~\eqref{4_cube_coherence} holds for all $I \subseteq [n]$ and $A \in \binom{[n]}{4}$, $\mathbf{x}_{\varkappa \big(T_{j - \binom{n - 1}{2} - 1, 0}, \dots, T_{j - \binom{n - 1}{2} - 1, j} \big) }(\mathbf{\bar{x}})$ is a coherent solution of the Kashaev equation. By construction, $\varkappa(T_{i - 1, 0}, \dots, T_{i - 1, j })$ and $\varkappa(T_{i - 1, 0}, \dots, T_{i - 1, j})$ are related by a flip for $i = 1, \dots, j - \binom{n - 1}{2} - 1$, so the pairs $\big( \mathbf{x}_{\varkappa(T_{i - 1, 0}, \dots, T_{i - 1, j })}(\mathbf{\bar{x}}), \mathbf{x}_{\varkappa(T_{i, 0}, \dots, T_{i, j})}(\mathbf{\bar{x}}) \big)$ are K-flipped for $i = 1, \dots, j - \binom{n - 1}{2} - 1$. Thus, by repeated applications of Lemma~\ref{flips_work_in_larger_complex}(a), it follows that $\mathbf{x}_{\varkappa (T_{0, 0}, \dots, T_{0, j} ) }(\mathbf{\bar{x}})$ is a coherent solution of the Kashaev equation.
\end{proof}

\begin{proof}[Proof of Theorem~\ref{main_all_minors_simple}.]
By Corollary~\ref{all_pminors_corollary}, the first two bullet points are equivalent, and by Corollary~\ref{all_pminors_corollary} and Proposition~\ref{4_cube_coherence_condition_holds}, the first bullet point implies the third bullet point. Hence, we just need to show that the third bullet point implies the first.

Suppose that the third bullet point holds. By Lemma~\ref{works_for_one_fcs}, there exists $\mathbf{T} \in \mathcal{C}(n)$ such that $\mathbf{x}_{\varkappa(\mathbf{T})}(\mathbf{\bar{x}})$ is a coherent solution of the Kashaev equation. Hence, by Theorem~\ref{main_thm_cyclic_zonotope}, there exists an extension $\mathbf{\tilde{x}} \in (\C^*)^{\varkappa^{02}(\mathbf{T})}$ of $\mathbf{x}_{\varkappa(\mathbf{T})}(\mathbf{\bar{x}})$ to $\varkappa^{02}(\mathbf{T})$ that satisfies the K-hexahedron equations. Hence, there exists a symmetric matrix $M$ such that $\mathbf{\tilde{x}} = \mathbf{\tilde{x}}_{\varkappa(\mathbf{T})}(M)$. Let $\xtin$ be the restriction of $\mathbf{\tilde{x}}$ to $T_{\min ,n}$.

Given any $\mathbf{T}' \in \mathcal{C}(n)$, by Proposition~\ref{bn3_connected_by_flips}, there exists a sequence of piles $\mathbf{T}_0, \dots, \mathbf{T}_\ell$ with $\mathbf{T} = \mathbf{T}_0$ and $\mathbf{T}' = \mathbf{T}_\ell$ such that $\varkappa(\mathbf{T}_{i - 1})$ and $\varkappa(\mathbf{T}_i)$ are related by a flip for $i = 1, \dots, \ell$. Hence, $(\mathbf{x}_{\varkappa(\mathbf{T}_{i - 1})}(\mathbf{\bar{x}}), \mathbf{x}_{\varkappa(\mathbf{T}_{i})}(\mathbf{\bar{x}}))$ is K-flipped, so by repeated applications of Lemma~\ref{flips_work_in_larger_complex}(a) and~(b), $(\xtin)^{\uparrow \varkappa^{02}(\mathbf{T}')}$ restricts to $\mathbf{x}_{\varkappa(\mathbf{T}')}(\mathbf{\bar{x}})$. Because every $I \subseteq [n]$ labels a vertex in $\varkappa(\mathbf{T}')$ for some $\mathbf{T}' \in \mathcal{C}(n)$, $\mathbf{\bar{x}} = \mathbf{\bar{x}}(M)$, as desired.
\end{proof}

\section{Generalizations of the Kashaev equation} \label{klikerecurrences}

In this section, we describe an axiomatic setup for equations similar to the Kashaev equation and the examples from Sections~\ref{s_holo_section}--\ref{gen_from_ca}. This allows us to prove all of the results from Sections~\ref{s_holo_section}--\ref{gen_from_ca}. This section is organized as follows:
\begin{itemize}\itemsep=0pt
\item Proposition~\ref{A2B2propgen} and Lemma~\ref{equivalent_coherence_statement_lemma} generalize Propositions~\ref{A2B2proposition}, \ref{A2B2forQC}, \ref{A2B2propgen1d}, and~\ref{A2B2_2d_example}.
\item In Definition~\ref{def_of_adapted}, given a polynomial equation resembling the Kashaev equation, we describe how to obtain a set of equations with the same properties as the K-hexahedron equations.
\item In Definition~\ref{definition_of_prop_signs}, we define certain signs that appear in our generalized ``coherence'' equation~\eqref{coherencegen}.
\item Theorem~\ref{maintheoremgen}, the main result in this section, generalizes Theorems~\ref{galoisfromintro}, \ref{coh_for_QC_theorem}, \ref{main_thm_1d_example}, \linebreak and~\ref{main_thm_2d_example}. The proof of Theorem~\ref{maintheoremgen} is nearly identical to the proof of Theorem~\ref{galoisfromintro} from Section~\ref{mainproof}.
\end{itemize}

\begin{Definition} \label{gensectionmaindefs}
For $d \ge 1$ and $\mathbf{a} = (a_1, \dots, a_d) \in \Z^d$, we denote by
\begin{gather*}
[\mathbf{a}] = \big\{ ( b_1, \dots, b_d ) \in \Z^d \colon 0 \le | b_i | \le | a_i | \text{ and } a_i b_i \ge 0 \text{ for all $i$} \big\}
\end{gather*}
the set of integer points in the $\left| a_1 \right| \times \cdots \times \left| a_d \right|$ box with opposite vertices $(0, \dots, 0)$ and $(a_1, \dots, a_d)$. For $\mathbf{b} = (b_1, \dots, b_d), \mathbf{c} = (c_1, \dots, c_d) \in \Z^d$, we write $\mathbf{b} \odot \mathbf{c} = (b_1 c_1, \dots, b_d c_d) \in \Z^d$. Denote by $\mathbf{1} = (1, \dots, 1) \in \Z^d$ the all $1$'s vector, and set $\mathbf{1}_i = (0, \dots, 0, 1, 0, \dots, 0) \in \Z^d$, with $1$ in the $i$th place, for $i = 1, \dots, d$. Let
\begin{gather*}
\mathbf{z}_{[\mathbf{a}]} = \{z_{\mathbf{i}} \colon \mathbf{i} \in [\mathbf{a}]\}
\end{gather*}
be a set of indeterminates. For $i = 1, \dots, d$, let $\pi_{\mathbf{a}, i} \colon \mathbf{z}_{[\mathbf{a}]} \rightarrow \mathbf{z}_{[\mathbf{a}]}$ be the involution defined by
\begin{gather*}
z_{(j_1, \dots, j_d)} \mapsto z_{(j_1, \dots, a_i - j_i, \dots, j_d)},
\end{gather*}
i.e., we ``flip'' the index of each variable in its $i$th coordinate. The action of $\pi_{\mathbf{a}, i}$ extends from~$\mathbf{z}_{[\mathbf{a}]}$ to the polynomial ring $\C[\mathbf{z}_{[\mathbf{a}]}]$. Given an array $\mathbf{x} = (x_s) \in \C^{\Z^d}$ and integer vectors $v \in \Z^d$, $\mathbf{a} \in \Z_{\ge 0}^d$, and $\bm{\alpha} \in \{-1, 1\}^d$, we denote by $\mathbf{x}_{v + [\mathbf{a} \odot \bm{\alpha}]} \in \C^{[\mathbf{a}]}$ the array whose entries are
\begin{gather*}
(\mathbf{x}_{v + [\mathbf{a} \odot \bm{\alpha}]})_{\mathbf{i}} = x_{v + \mathbf{i} \odot \bm{\alpha}}, \qquad \text{for} \quad \mathbf{i} \in [\mathbf{a}].
\end{gather*}
In particular,
\begin{gather*}
(\mathbf{x}_{v + [\mathbf{a}]})_{\mathbf{i}} = x_{v + \mathbf{i}}, \qquad \text{for} \quad \mathbf{i} \in [\mathbf{a}].
\end{gather*}
Thus, given a polynomial $f \in \C[\mathbf{z}_{[\mathbf{a}]}]$, the number $f(\mathbf{x}_{v + [\bm{\alpha} \odot \mathbf{a}]}) \in \C$ is obtained by setting $z_{\mathbf{i}} = x_{v + \bm{\alpha} \odot \mathbf{i}}$ for each variable $z_{\mathbf{i}}$ for $\mathbf{i} \in [\mathbf{a}]$. We say that $\mathbf{x} \in \C^{\Z^d}$ \emph{satisfies $f$} if $f(\mathbf{x}_{v + [\mathbf{a}]}) = 0$ for all $v \in \Z^d$.
\end{Definition}

\begin{Proposition} \label{A2B2propgen}Let $\mathbf{a} = (a_1, \dots, a_d) \in \Z^d_{\ge 1}$, and a polynomial $f \in \C[\mathbf{z}_{[\mathbf{a}]}]$ satisfy the following conditions:
\begin{enumerate}\itemsep=0pt
\item[$(1)$] $f$ is invariant under the action of $\pi_{\mathbf{a}, i}$ for $i = 1, \dots, d$,
\item[$(2)$] $f$ has degree $2$ with respect to the variable $z_{\mathbf{a}}$; as a quadratic polynomial in $z_{\mathbf{a}}$, $f$ has discriminant $D$ which factors as a product $D = f_1 \cdots f_d$, where each polynomial $f_i \in \C[\mathbf{z}_{[\mathbf{a} - \mathbf{1}_i]}]$ is invariant under the action of $\pi_{\mathbf{a} - \mathbf{1}_i, j}$ for $j = 1, \dots, d$.
\end{enumerate}
Then for any $\mathbf{x} = (x_s) \in \C^{\Z^d}$ satisfying $f$, we have, for all $v \in \Z^d$:
\begin{align} \label{A2B2propgeneqn}
\left( \prod_{\bm{\alpha} \in \{-1,0\}^d}\! \frac{\partial f}{\partial z_{\mathbf{a}}} (\mathbf{x}_{v - (\mathbf{a} - \mathbf{1}) \odot \bm{\alpha} + [(\mathbf{1} + 2 \bm{\alpha}) \odot \mathbf{a}]}) \right)^2 = \left( \prod_{i = 1}^d \prod_{\substack{\bm{\beta} = (\beta_1, \dots, \beta_d) \in \{-1,0\}^d\\ \beta_i = 0}}\!\! f_i(\mathbf{x}_{v + \bm{\beta} + [\mathbf{a}]}) \right)^2\!\!.\!\!\!
\end{align}
Moreover, for all $v \in Z^d$, we have
\begin{gather}
\left( \prod_{\substack{\bm{\alpha} = (\alpha_1, \dots, \alpha_d) \\ \alpha_1, \dots, \alpha_d \in \{-1, 0\} \\ \alpha_1 + \cdots + \alpha_d \text{ even}}} \frac{\partial f}{\partial z_{\mathbf{a}}} (\mathbf{x}_{v - (\mathbf{a} - \mathbf{1}) \odot \bm{\alpha} + [(\mathbf{1} + 2 \bm{\alpha}) \odot \mathbf{a}]}) \right)^2\nonumber\\
\qquad{} = \left( \prod_{\substack{\bm{\alpha} = (\alpha_1, \dots, \alpha_d) \\ \alpha_1, \dots, \alpha_d \in \{-1, 0\} \\ \alpha_1 + \cdots + \alpha_d \text{ odd}}} \frac{\partial f}{\partial z_{\mathbf{a}}} (\mathbf{x}_{v - (\mathbf{a} - \mathbf{1}) \odot \bm{\alpha} + [(\mathbf{1} + 2 \bm{\alpha}) \odot \mathbf{a}]}) \right)^2\nonumber\\
\qquad{} = \prod_{i = 1}^d \prod_{\substack{\bm{\beta} = (\beta_1, \dots, \beta_d) \in \{-1,0\}^d\\ \beta_i = 0}} f_i(\mathbf{x}_{v + \bm{\beta} + [\mathbf{a}]}).\label{A2B2propgeneqnalt}
\end{gather}
\end{Proposition}

\begin{Remark}The subscripts $v - (\mathbf{a} - \mathbf{1}) \odot \bm{\alpha} + [(\mathbf{1} + 2 \bm{\alpha}) \odot \mathbf{a}]$ appearing on the left-hand side of~\eqref{A2B2propgeneqn} run over all boxes of size $a_1 \times \cdots \times a_d$ containing~$v + [\mathbf{a} - \mathbf{1}]$. The subscripts $v - (\mathbf{a} - \mathbf{1}) \odot \bm{\alpha} + [(\mathbf{1} + 2 \bm{\alpha}) \odot \mathbf{a}]$ appearing on the right-hand side of~\eqref{A2B2propgeneqn} run over $i = 1, \dots, d$ and boxes of size $a_1 \times \cdots \times a_{i - 1} \times (a_i - 1) \times a_{i + 1} \times \cdots \times a_d$ containing~$v + [\mathbf{a} - \mathbf{1}]$. In particular, when $a_1 = \cdots = a_d = 1$, all of these products are over boxes of a certain size containing the vertex~$v$. For example, in the case $\mathbf{a} = (1,2)$ (like in Proposition~\ref{A2B2_2d_example}), the boxes we are considering on the left-hand side of~\eqref{A2B2propgeneqn} are given in the top row of Fig.~\ref{rectangles_in_coherence}, while the boxes we are considering on the right-hand side of~\eqref{A2B2propgeneqn} are given in the bottom row of Fig.~\ref{rectangles_in_coherence}.
\end{Remark}

Before proving Proposition~\ref{A2B2propgen}, we give several examples of polynomials discussed in previous sections that satisfy conditions~(1)--(2) from Proposition~\ref{A2B2propgen}. In the examples below, we write $z_{i_1 \cdots i_d} = z_{(i_1, \dots, i_d)}$ for $(i_1, \dots, i_d) \in \Z^d$.

\begin{Example} \label{kash_sat_ex}
Our first example is the Kashaev equation. Let $\mathbf{a} = (1,1,1) \in \Z^3$, and let
\begin{gather*}
f = 2\big(a^2 + b^2 + c^2 + d^2\big) - (a + b + c + d)^2 - 4(s + t) \in \C[\mathbf{z}_{[\mathbf{a}]}],
\end{gather*}
where
\begin{gather*}
a = \zzzz \zooo, \qquad b = \zozz \zzoo, \qquad c = \zzoz \zozo, \qquad d = \zzzo \zooz,\\
s = \zzzz \zzoo \zozo \zooz,\qquad t = \zozz \zzoz \zzzo \zooo.
\end{gather*}
The polynomial $f$ is invariant not only under the action of the $\pi_{\mathbf{a}, i}$, but under all symmetries of the cube. Its discriminant (as a polynomial in $\zooo$) $D$ factors as a product $D = f_1 f_2 f_3$, with
\begin{gather*}
f_1 = 16 (\zzzz \zzoo + \zzoz \zzzo),\qquad
f_2 = \zzzz \zozo + \zozz \zzzo,\qquad
f_3 = \zzzz \zooz + \zozz \zzoz.\
\end{gather*}
Hence, $f$ satisfies conditions~(1)--(2) from Proposition~\ref{A2B2propgen}. Therefore, Proposition~\ref{A2B2proposition} is a~special case of Proposition~\ref{A2B2propgen}.
\end{Example}

\begin{Example} \label{sholo_sat_ex}Let $\mathbf{a} = (1,1) \in \Z^2$, and let
\begin{gather*}
f = z_{00}^2 + z_{10}^2 + z_{01}^2 + z_{11}^2 - 2(z_{00} z_{10} + z_{10} z_{11} + z_{11} z_{01} + z_{01} z_{00}) \\
\hphantom{f=}{} - 6 (z_{00} z_{11} + z_{10} z_{01}) \in \C[\mathbf{z}_{[\mathbf{a}]}].
\end{gather*}
The polynomial $f$ is invariant not only under the action of the $\pi_{\mathbf{a}, i}$, but under all symmetries of the square. Its discriminant (as a polynomial in $z_{11}$) $D$ factors as a product $D = f_1 f_2$, with
\begin{gather*}
f_1= 32 (z_{00} + z_{01}), \qquad f_2 = z_{00} + z_{10}.
\end{gather*}
Hence, $f$ satisfies conditions~(1)--(2) from Proposition~\ref{A2B2propgen}. Therefore, Proposition~\ref{A2B2forQC} is a~special case of Proposition~\ref{A2B2propgen}.
\end{Example}

\begin{Example} \label{1d_sat_ex}
Fix $\alpha_1, \alpha_2, \alpha_3 \in \C$. Let $\mathbf{a} = (3) \in \Z^1$, and let
\begin{gather*}
f = z_0^2 z_3^2 + \alpha_1 z_1^2 z_2^2 + \alpha_2 z_0 z_1 z_2 z_3 + \alpha_3 \big(z_0 z_2^3 + z_1^3 z_3\big) \in \C[\mathbf{z}_{[\mathbf{a}]}].
\end{gather*}
The polynomial $f$ is invariant under the action of $\pi_{\mathbf{a}, 1}$, i.e., replacing $z_i$ by $z_{3 - i}$. Its discriminant (as a polynomial in $z_{11}$) $D = f_1$, with
\begin{gather*}
f_1 = \alpha_3^2 z_1^6 + 2 \alpha_2 \alpha_3 z_0 z_1^4 z_2 + \big(\alpha_2^2 - 4 \alpha_1\big) z_0^2 z_1^2 z_2^2 - 4 \alpha_3 z_0^3 z_2^3.
\end{gather*}
Hence, $f$ satisfies conditions~(1)--(2) from Proposition~\ref{A2B2propgen}. Therefore, Proposition~\ref{A2B2propgen1d} is a~special case of Proposition~\ref{A2B2propgen}.
\end{Example}

\begin{Example} \label{2d_sat_ex}Fix $\alpha_1, \alpha_2 \in \C$. Let $\mathbf{a} = (1,2) \in \Z^2$, and let
\begin{gather*}
f = z_{00}^2 z_{12}^2 + z_{10}^2 z_{02}^2 + \frac{\alpha_2^2 - \alpha_1^2}{4} z_{01}^2 z_{11}^2 - \alpha_1 \big(z_{00} z_{02} z_{11}^2 + z_{10} z_{12} z_{01}^2\big)\\
\hphantom{f=}{} - 2 z_{00} z_{10} z_{02} z_{12} - \alpha_2 (z_{00} z_{12} z_{01} z_{11} + z_{10} z_{02} z_{01} z_{11}) \in \C[\mathbf{z}_{[\mathbf{a}]}].
\end{gather*}
The polynomial $f$ is invariant under the action of $\pi_{\mathbf{a}, 1}$ and $\pi_{\mathbf{a}, 2}$. Its discriminant (as a~polynomial in $z_{12}$) factors as a product $D = f_1 f_2$, with
\begin{gather*}
f_1 = \alpha_1 z_{01}^2 + 4 z_{00} z_{02},\qquad f_2 = \alpha_1\big(z_{00}^2 z_{11}^2 + z_{01}^2 z_{10}^2\big) + 2 \alpha_2 z_{00} z_{01} z_{10} z_{11}.
\end{gather*}
Hence, $f$ satisfies conditions~(1)--(2) from Proposition~\ref{A2B2propgen}. Therefore, Proposition~\ref{A2B2_2d_example} is a~special case of Proposition~\ref{A2B2propgen}.
\end{Example}

\begin{proof}[Proof of Proposition~\ref{A2B2propgen}.]Let $g$ denote the coefficient of $z_{\mathbf{a}}^2$ in $f$ (viewed as a polynomial in~$z_{\mathbf{a}}$). It is easy to check that
\begin{gather} \label{prod_fi_identity}
D = f_1 \cdots f_d = \left( \frac{\partial f}{\partial z_{\mathbf{a}}} \right)^2 - 4 f g.
\end{gather}
Because $\mathbf{x}$ satisfies $f$, we have
\begin{gather*}
\left(\frac{\partial f}{\partial z_{\mathbf{a}}}\right)^2 (\mathbf{x}_{v - (\mathbf{a} - \mathbf{1}) \odot \bm{\alpha} + [(1 + 2 \bm{\alpha})\odot \mathbf{a}]}) = (f_1 \cdots f_d) (\mathbf{x}_{v - (\mathbf{a} - \mathbf{1}) \odot \bm{\alpha} + [(1 + 2 \bm{\alpha}) \odot\mathbf{a}]}).
\end{gather*}
Because each $f_i$ is invariant under the action of $\pi_{\mathbf{a} - \mathbf{1}_i, j}$ for $j = 1, \dots, d$, we have
\begin{gather} \label{indivderivsquared}
\left(\frac{\partial f}{\partial z_{\mathbf{a}}}\right)^2 (\mathbf{x}_{v - (\mathbf{a} - \mathbf{1}) \odot \bm{\alpha} + [(1 + 2 \bm{\alpha}) \odot \mathbf{a}]}) = \prod_{i = 1}^d f_i(\mathbf{x}_{v + \bm{\alpha} \odot (\mathbf{1} - \mathbf{1}_i) + [\mathbf{a}]}).
\end{gather}
Given $i \in \{1, \dots, d\}$ and $\bm{\beta} = (\beta_1, \dots, \beta_d) \in \{-1, 0\}^d$ with $\beta_i = 0$, there exist exactly two $\bm{\alpha} = (\alpha_1, \dots, \alpha_d) \in \{-1, 0\}^d$ with $(\mathbf{1} - \mathbf{1}_i) \odot \bm{\alpha} = \bm{\beta}$: one with $\alpha_i = 0$ and the other with $\alpha_i = 1$. Hence, $\alpha_1 + \cdots + \alpha_d$ is even for one such choice of $\bm{\alpha}$, and odd for the other. Taking the product over $\bm{\alpha} = (\alpha_1, \dots, \alpha_d) \in \{-1, 0\}^d$ with $\alpha_1 + \cdots + \alpha_d$ odd (or even) in~\eqref{indivderivsquared}, we obtain~\eqref{A2B2propgeneqnalt}. Equation~\eqref{A2B2propgeneqn} follows.
\end{proof}

\begin{Lemma} \label{equivalent_coherence_statement_lemma}
Let $\mathbf{a} = (a_1, \dots, a_d) \in \Z^d_{\ge 1}$, and let $f \in \C[\mathbf{z}_{[\mathbf{a}]}]$ be a polynomial satisfying conditions~$(1)$--$(2)$ from Proposition~{\rm \ref{A2B2propgen}}. Let $\mathbf{x} = (x_s) \in \C^{\Z^d}$ be an array satisfying $f$. Fix $v \in \Z^d$ and $\gamma \in \{-1,1\}$. Then
\begin{gather*}
\prod_{\bm{\alpha} \in \{-1,0\}^d} \frac{\partial f}{\partial z_{\mathbf{a}}} (\mathbf{x}_{v - (\mathbf{a} - \mathbf{1}) \odot \bm{\alpha} + [(\mathbf{1} + 2 \bm{\alpha}) \odot \mathbf{a}]}) = \gamma \prod_{i = 1}^d \prod_{\substack{\bm{\beta} = (\beta_1, \dots, \beta_d) \in \{-1,0\}^d\\ \beta_i = 0}} f_i(\mathbf{x}_{v + \bm{\beta} + [\mathbf{a}]})
\end{gather*}
if and only if
\begin{gather*}
\prod_{\substack{\bm{\alpha} = (\alpha_1, \dots, \alpha_d) \\ \alpha_1, \dots, \alpha_d \in \{-1, 0\} \\ \alpha_1 + \cdots + \alpha_d \text{ even}}} \frac{\partial f}{\partial z_{\mathbf{a}}} (\mathbf{x}_{v - (\mathbf{a} - \mathbf{1}) \odot \bm{\alpha} + [(\mathbf{1} + 2 \bm{\alpha}) \odot \mathbf{a}]}) = \gamma \prod_{\substack{\bm{\alpha} = (\alpha_1, \dots, \alpha_d) \\ \alpha_1, \dots, \alpha_d \in \{-1, 0\} \\ \alpha_1 + \cdots + \alpha_d \text{ odd}}} \frac{\partial f}{\partial z_{\mathbf{a}}} (\mathbf{x}_{v - (\mathbf{a} - \mathbf{1}) \odot \bm{\alpha} + [(\mathbf{1} + 2 \bm{\alpha}) \odot \mathbf{a}]}).
\end{gather*}
\end{Lemma}

\begin{proof}This follows immediately from Proposition~\ref{A2B2propgen}.
\end{proof}

\begin{Definition} \label{fa_set_def}Given $\mathbf{a} = (a_1, \dots, a_d) \in \Z^d_{\ge 1}$ and $1 \le i \le d$, let
\begin{gather*}
F^{\mathbf{a}}_i = \big\{v + [\mathbf{a} - \mathbf{1}_i] \colon v \in \Z^d\big\}
\end{gather*}
denote the set of boxes of size $a_1 \times \cdots \times a_{i - 1} \times (a_i - 1) \times a_{i + 1} \times \cdots \times a_d$ in $\Z^d$. Set
\begin{gather*}
\fab = \bigcup_{i = 1}^d F^{\mathbf{a}}_i.
\end{gather*}
Given an array $\mathbf{\tilde{x}} = (x_s)_{s \in \Z^d \cup \fab}$ with $\mathbf{x} = (x_s)_{s \in \Z^d}$, $i \in \{1, \dots, d\}$, and $v \in \Z^d$, we remark that $\mathbf{x}_{v + [\mathbf{a} - \mathbf{1}_i]}$ (with $\mathbf{x}$ bold) refers to the array defined in Definition~\ref{gensectionmaindefs}, whereas $x_{v + [\mathbf{a} - \mathbf{1}_i]}$ (with $x$ not bold) refers to the component of $\mathbf{\tilde{x}}$ indexed by $v + [\mathbf{a} - \mathbf{1}_i] \in \fab$.
\end{Definition}

\begin{Definition}
For $\mathbf{a} \in \Z_{\ge 1}^d$, define $[\mathbf{a}^*]$ by
\begin{gather*}
[\mathbf{a}^*] = \left( [\mathbf{a}] \setminus \{\mathbf{a}\} \right) \cup \bigcup_{i = 1}^d \{ [\mathbf{a} - \mathbf{1}_i]\}.
\end{gather*}
In other words, the set $[\mathbf{a}^*]$ consists of $[\mathbf{a}] \setminus \{\mathbf{a}\} \subset \Z^d$, along with the sets $[\mathbf{a} - \mathbf{1}_i] \in \fab$ for $i = 1, \dots, d$.%
\end{Definition}

We want to develop a generalization of the K-hexahedron equations for arrays indexed by $\Z^d \cup \fab$. Suppose that $f \in \C[\mathbf{z}_{[\mathbf{a}]}]$ is a polynomial satisfying conditions~(1)--(2) from Proposition~\ref{A2B2propgen}, with the polynomials $f_1, \dots, f_d$ from condition~(2) fixed. Let $\gtf$ and $\hof$ be the coefficients of $z_{\mathbf{a}}^2$ and $z_{\mathbf{a}}$ in $f$, viewed as a polynomial in $z_{\mathbf{a}}$. We consider arrays $\mathbf{\tilde{x}} = (x_s) \in \C^{\Z^d \cup \fab}$ such that
\begin{gather}
\label{newvertexgen} x_{v + \mathbf{a}} = \frac{-\hof(\mathbf{x}_{v + [\mathbf{a}]}) + \prod\limits_{i = 1}^d x_{v + [\mathbf{a} - \mathbf{1}_i]}}{2\gtf(\mathbf{x}_{v + [\mathbf{a}]})} \qquad \text{for all $v \in \Z^d$},\\
\label{newfacegen} x_{v + \mathbf{1}_i + [\mathbf{a} - \mathbf{1}_i]} = r_i(x_{v + s} \colon s \in [\mathbf{a}^*]) \qquad \text{for $i = 1, \dots, d$} \text{ and for all $v \in \Z^d$},\\
\label{squarecond} x_{v + [\mathbf{a} - \mathbf{1}_i]}^2 = f_i(\mathbf{x}_{v + [\mathbf{a} - \mathbf{1}_i]}) \qquad \text{for $i = 1, \dots, d$} \text{ and for all $v \in \Z^d$},
\end{gather}
where $r_1,\dots,r_d$ are some rational functions in the variables $z_s$ for $s \in [\mathbf{a}^*]$. Note that if $\mathbf{\tilde{x}}$ satisfies conditions~\eqref{newvertexgen} and~\eqref{squarecond}, then by the quadratic formula, its restriction $\mathbf{x} = (x_s)_{s \in \Z^d}$ satisfies~$f$. In the following definition, we formulate the properties that our tuple of rational functions $(r_1,\dots,r_d)$ should have in order for the subsequent developments to follow. %

\begin{Definition} \label{def_of_adapted}Let $\mathbf{a} = (a_1, \dots, a_d) \in \Z^d_{\ge 1}$, and let $f \in \C[\mathbf{z}_{[\mathbf{a}]}]$ be a polynomial satisfying conditions (1)--(2) from Proposition~\ref{A2B2propgen}. Fix the polynomials $f_1, \dots, f_d$ from condition~(2) of Proposition~\ref{A2B2propgen}. Let $g$ be the coefficient of $z_{\mathbf{a}}^2$ in~$f$, viewed as a polynomial in $z_{\mathbf{a}}$. For $i = 1, \dots, d$, let $r_i = \frac{p_i}{q_i}$ be rational functions in the variables $z_s$ for $s \in [\mathbf{a}^*]$, with $p_i$, $q_i$ polynomials in these variables. We say that $(r_1, \dots, r_d)$ is \emph{adapted to $(f; f_1, \dots, f_d)$} if there exist signs $\beta_1,\dots,\beta_d \in \{-1,1\}$ such that the following properties hold for $i = 1, \dots, d$:
\begin{itemize}\itemsep=0pt
\item the denominator $q_i$ of $r_i$ is of the form
\begin{gather*} \label{denom_form_of_ri}
q_i = g^{b_i} \prod_{j \in \{1, \dots, d\} \setminus \{ i \}} z_{[\mathbf{a} - \mathbf{1}_j]}^{b_{ij}},
\end{gather*}
where $b_i \in \Z_{\ge 0}$ and $b_{ij} \in \{0, 1\}$,
\item for all arrays $\mathbf{\tilde{x}} = (x_s) \in \C^{\Z^d \cup \fab}$ satisfying~\eqref{newvertexgen}, \eqref{squarecond}, and
\begin{gather}
\label{nonzeroqi} q_i(x_{v + s} \colon s \in [\mathbf{a}^*]) \not= 0 \qquad \text{for all $v \in \Z^d$},\\
\label{nonzeromaind} g(\mathbf{x}_{v + [\mathbf{a}]}) \not= 0 \qquad \text{for all $v \in \Z^d$},
\end{gather}
the following condition holds:
\begin{gather} \label{ada_expression_for_ri}
r_i(x_{v + s} \colon s \in [\mathbf{a}^*]) = \beta_i \frac{\left( \frac{\partial f}{\partial z_{\mathbf{a}}} \right)(\mathbf{x}_{v + a_i \mathbf{1}_i + [\mathbf{a} \odot (\mathbf{1} - 2 \mathbf{1}_i)]})}{\prod\limits_{j \in \{1, \dots, d\} - \{i\}} x_{v + [\mathbf{a} - \mathbf{1}_j]}}.
\end{gather}
\end{itemize}
Note that one can obtain a tuple $(r_1,\dots,r_d)$ adapted to $(f; f_1,\dots,f_d)$ by choosing the signs $\beta_1,\dots,\beta_d \in \{-1,1\}$ and using condition~\eqref{newvertexgen} to replace all instances of $x_{v + \mathbf{a}}$ in~\eqref{ada_expression_for_ri}.
\end{Definition}

In the following proposition, we show that with $(r_1,\dots,r_d)$ adapted to $(f ; f_1, \dots, f_d)$, the recurrence~\eqref{newvertexgen}--\eqref{newfacegen} ``propagates'' the condition~\eqref{squarecond}.

\begin{Proposition}Let $\mathbf{a} = (a_1, \dots, a_d) \in \Z^d_{\ge 1}$, and let $f \in \C[\mathbf{z}_{[\mathbf{a}]}]$ be a polynomial satisfying conditions $(1)$--$(2)$ from Proposition~{\rm \ref{A2B2propgen}}. Fix the polynomials $f_1, \dots, f_d$ from condition~$(2)$ of Proposition~{\rm \ref{A2B2propgen}}. Let $g$ be the coefficient of $z_{\mathbf{a}}^2$ in $f$, viewed as a polynomial in~$z_{\mathbf{a}}$. Let $(r_1,\dots,r_d)$ be a $d$-tuple of rational functions in the variables $z_s$ for $s \in [\mathbf{a}^*]$ adapted to $(f; f_1,\dots,f_d)$. Fix $v \in \Z^d$. Let $\mathbf{\tilde{x}} = (x_s) \in \C^{\Z^d \cup \fab}$ be an array satisfying conditions~\eqref{newvertexgen}--\eqref{newfacegen}, \eqref{nonzeroqi}--\eqref{nonzeromaind}, and
\begin{gather*}
x_{v + [\mathbf{a} - \mathbf{1}_i]}^2 = f_i(\mathbf{x}_{v + [\mathbf{a} - \mathbf{1}_i]}) \qquad \text{for $i = 1,\dots,d$}.
\end{gather*}
Then
\begin{gather*}
x_{v + \mathbf{1}_i + [\mathbf{a} - \mathbf{1}_i]}^2 = f_i(\mathbf{x}_{v + \mathbf{1}_i + [\mathbf{a} - \mathbf{1}_i]}) \qquad \text{for $i = 1,\dots,d$}.
\end{gather*}
\end{Proposition}

\begin{proof}By identity~\eqref{prod_fi_identity},
\begin{gather*}
\left( \frac{\partial f}{\partial z_{\mathbf{a}}} \right)^2(\mathbf{x}_{v + a_i \mathbf{1}_i + [\mathbf{a} \odot (\mathbf{1} - 2 \mathbf{1}_i)]}) = (f_1 \cdots f_d)(\mathbf{x}_{v + a_i \mathbf{1}_i + [\mathbf{a} \odot (\mathbf{1} - 2 \mathbf{1}_i)]})\\
\hphantom{\left( \frac{\partial f}{\partial z_{\mathbf{a}}} \right)^2(\mathbf{x}_{v + a_i \mathbf{1}_i + [\mathbf{a} \odot (\mathbf{1} - 2 \mathbf{1}_i)]})}{}
= f_i(\mathbf{x}_{v + \mathbf{1}_i + [\mathbf{a} - \mathbf{1}_i]}) \prod_{j \in \{1, \dots, d\} - \{i\}} f_j(\mathbf{x}_{v + [\mathbf{a} - \mathbf{1}_j]}).
\end{gather*}
Hence,
\begin{gather*}
x_{v + \mathbf{1}_i + [\mathbf{a} - \mathbf{1}_i]}^2 = \beta_i^2 \frac{\left( \frac{\partial f}{\partial z_{\mathbf{a}}} \right)^2(\mathbf{x}_{v + a_i \mathbf{1}_i + [\mathbf{a} \odot (\mathbf{1} - 2 \mathbf{1}_i)]})}{\prod\limits_{j \in \{1, \dots, d\} - \{i\}} x_{v + [\mathbf{a} - \mathbf{1}_j]}^2}\\
\hphantom{x_{v + \mathbf{1}_i + [\mathbf{a} - \mathbf{1}_i]}^2}{}
= \frac{f_i(\mathbf{x}_{v + \mathbf{1}_i + [\mathbf{a} - \mathbf{1}_i]}) \prod\limits_{j \in \{1, \dots, d\} - \{i\}} f_j(\mathbf{x}_{v + [\mathbf{a} - \mathbf{1}_j]})}{\prod\limits_{j \in \{1, \dots, d\} - \{i\}} f_j(\mathbf{x}_{v + [\mathbf{a} - \mathbf{1}_j]})}= f_i(\mathbf{x}_{v + \mathbf{1}_i + [\mathbf{a} - \mathbf{1}_i]}),
\end{gather*}
as desired.
\end{proof}

We now describe $d$-tuples of rational functions adapted to $(f;f_1,\dots,f_d)$ for the four polynomials $f$ in the Examples~\ref{kash_sat_ex}--\ref{2d_sat_ex}.

\begin{Example} \label{kash_ada_ex}
Continuing with Example~\ref{kash_sat_ex}, let us write
\begin{gather*}
z_{i_1\left( i_2 + \half \right)\left( i_3 + \half \right)} = z_{(i_1,i_2,i_3) + [\mathbf{a} - \mathbf{1}_1]},\qquad
z_{\left( i_1 + \half \right) i_2 \left( i_3 + \half \right)} = z_{(i_1,i_2,i_3) + [\mathbf{a} - \mathbf{1}_2]},\\
z_{\left( i_1 + \half \right)\left( i_2 + \half \right) i_3} = z_{(i_1,i_2,i_3) + [\mathbf{a} - \mathbf{1}_3]}.
\end{gather*}
Set
\begin{gather*}
r_1(z_s \colon s \in [\mathbf{a}^*]) = \frac{4 \zhzh \zhhz + \zzhh \zozz}{\zzzz},\qquad
r_2(z_s \colon s \in [\mathbf{a}^*]) = \frac{\zzhh \zhhz + 4 \zhzh \zzoz}{4 \zzzz},\\
r_3(z_s \colon s \in [\mathbf{a}^*]) = \frac{\zzhh \zhzh + 4 \zhhz \zzzo}{4 \zzzz}.
\end{gather*}
It can be checked that $(r_1, r_2, r_3)$ is adapted to $(f;f_1,f_2,f_3)$ by following the construction at the end of Definition~\ref{def_of_adapted} with $\beta_1 = \beta_2 = \beta_3 = -1$ and using condition~\eqref{squarecond}. Note that~$r_1$,~$r_2$,~$r_3$ matches the right-hand-sides of~\eqref{likehex1d}--\eqref{likehex3d} and the K-hexahedron equations~\eqref{likehex1d}--\eqref{likehex4d} and~\eqref{khexspec} are the same as conditions~\eqref{newvertexgen}--\eqref{squarecond} if each $z_{v + [\mathbf{a} - \mathbf{1}_1]}$ for $v \in \Z^3$ is rescaled by a factor of $4$.
\end{Example}

\begin{Example} \label{sholo_ada_ex}
Continuing with Example~\ref{sholo_sat_ex}, let us write
\begin{gather*}
z_{i_1\left( i_2 + \half \right)} = z_{(i_1,i_2) + [\mathbf{a} - \mathbf{1}_1]},\qquad
z_{\left( i_1 + \half \right) i_2} = z_{(i_1,i_2) + [\mathbf{a} - \mathbf{1}_2]}.
\end{gather*}
Set
\begin{gather*}
r_1(z_s \colon s \in [\mathbf{a}^*]) = z_{0 \frac{1}{2}} + 8 z_{\frac{1}{2} 0},\qquad
r_2(z_s \colon s \in [\mathbf{a}^*]) = z_{\frac{1}{2} 0} + \frac{1}{4} z_{0 \frac{1}{2}}.
\end{gather*}
It can be checked that $(r_1, r_2)$ is adapted to $(f; f_1,f_2)$ by following the construction at the end of Definition~\ref{def_of_adapted} with $\beta_1 = \beta_2 = -1$ and using condition~\eqref{squarecond}. Note that $r_1$, $r_2$ matches the right-hand-sides of~\eqref{2dh2}--\eqref{2dh3} and the conditions~\eqref{2dh1}--\eqref{2dhs2} are the same as conditions~\eqref{newvertexgen}--\eqref{squarecond} if each $z_{v + [\mathbf{a} - \mathbf{1}_1]}$ for $v \in \Z^2$ is rescaled by a factor of $4 \sqrt{2}$.
\end{Example}

\begin{Example} \label{1d_ada_ex}Continuing with Example~\ref{1d_sat_ex}, let us write
\begin{gather*}
w_{i + 1} = z_{i + [\mathbf{a} - \mathbf{1}_1]}.
\end{gather*}
Set
\begin{gather*}
r_1(z_s \colon s \in [\mathbf{a}^*]) = \frac{\alpha_3^2 z_1^6 + \alpha_2 \alpha_3 z_0 z_1^4 z_2 + 2 \alpha_3 z_0^3 z_2^3 + w_1^2 + \big({-}2 \alpha_3 z_1^3 - \alpha_2 z_0 z_1 z_2\big) w_1}{2 z_0^3}.
\end{gather*}
It can be checked that $(r_1)$ is adapted to $(f; f_1)$ by following the construction at the end of Definition~\ref{def_of_adapted} with $\beta_1 = 1$ and using condition~\eqref{squarecond}. Note that $r_1$ matches the right-hand side of equation~\eqref{1d_f_hex_prop}, and conditions~\eqref{1d_v_hex_prop}--\eqref{1d_interval_spec} are the same as conditions~\eqref{newvertexgen}--\eqref{squarecond}.
\end{Example}

\begin{Example} \label{2d_ada_ex}Continuing with Example~\ref{2d_sat_ex}, let us write
\begin{gather*}
w_{i_1(i_2 +1)} = z_{(i_1, i_2) + [\mathbf{a} - \mathbf{1}_1]},\qquad
w_{\left( i_1 + \half \right) \left( i_2 +\half \right)} = z_{(i_1, i_2) + [\mathbf{a} - \mathbf{1}_2]}.
\end{gather*}
Set
\begin{gather*}
r_1(z_s \colon s \in [\mathbf{a}^*]) = \frac{z_{10} w_{01} + w_{\half \half}}{z_{00}},\\
r_2(z_s \colon s \in [\mathbf{a}^*]) = \frac{z_{01}(\alpha_1 z_{01} z_{10} + \alpha_2 z_{00} z_{11}) w_{01} + \big(\alpha_1 z_{01}^2 + 2 z_{00} z_{02}\big) w_{\half \half}}{2 z_{00}^2}.
\end{gather*}
It can be checked that $(r_1,r_2)$ is adapted to $(f; f_1,f_2)$ by following the construction at the end of Definition~\ref{def_of_adapted} with $\beta_1 = \beta_2 = -1$ and using condition~\eqref{squarecond}. Note that $r_1$, $r_2$ matches the right-hand side of equation~\eqref{2d_1f_hex_prop}--\eqref{2d_2f_hex_prop}, and conditions~\eqref{2d_v_hex_prop}--\eqref{2d_square_face_spec} are the same as conditions~\eqref{newvertexgen}--\eqref{squarecond}.
\end{Example}

\begin{Lemma} \label{runinarbdirection}Let $\mathbf{a} = (a_1, \dots, a_d) \in \Z^d_{\ge 1}$. Let $f \in \C[\mathbf{z}_{[\mathbf{a}]}]$ be a polynomial that is irreducible over $\C$ and satisfies conditions $(1)$--$(2)$ from Proposition~{\rm \ref{A2B2propgen}}. Fix the polynomials $f_1, \dots, f_d$ from condition~$(2)$ of Proposition~{\rm \ref{A2B2propgen}}. Let $(r_1, \dots, r_d)$ be a tuple of rational functions in the variables $z_s$ for $s \in [\mathbf{a}^*]$ that is adapted to $(f; f_1,\dots,f_d)$. Then for all $\bm{\alpha} \in \{-1, 1\}^d$, there exists a unique sign $\gamma_{\bm{\alpha}} \in \{-1,1\}$ such the following condition holds for all arrays $\mathbf{\tilde{x}} = (x_s) \in \C^{\Z^d \cup \fab}$ satisfying~\eqref{newvertexgen}--\eqref{squarecond} and~\eqref{nonzeroqi}--\eqref{nonzeromaind}:
\begin{gather} \label{arbdirectionderivative}
\frac{\partial f}{\partial z_{\mathbf{a}}}(\mathbf{x}_{v + [\bm{\alpha} \odot \mathbf{a}]}) = \gamma_{\bm{\alpha}} \prod_{i = 1}^d x_{v + [\bm{\alpha} \odot (\mathbf{a} - \mathbf{1}_i)]} \qquad \text{for all $v \in \Z^d$}.
\end{gather}
\end{Lemma}

\begin{Definition} \label{definition_of_prop_signs}For $f$ and $(r_1,\dots,r_d)$ as in Lemma~\ref{runinarbdirection}, we call the signs $(\gamma_{\bm{\alpha}})_{\bm{\alpha} \in \{-1,1\}^d}$ given in Lemma~\ref{runinarbdirection} the \emph{propagation signs} corresponding to $(f; f_1,\dots,f_d;r_1,\dots,r_d)$.
\end{Definition}

The proof of Lemma~\ref{runinarbdirection} relies on the following lemma.

\begin{Lemma} \label{derivativeproductsign}Let $\mathbf{a} = (a_1, \dots, a_d) \in \Z^d_{\ge 1}$. Let $f \in \C[\mathbf{z}_{[\mathbf{a}]}]$ be a polynomial that is irreducible over $\C$ and satisfies conditions $(1)$--$(2)$ from Proposition~{\rm \ref{A2B2propgen}}. Let $j \not= k \in \{1, \dots, d\}$. Then there exists $\alpha_{jk} \in \{-1,1\}$ such that for all $\bm{\alpha} \in \{0, 1\}^d$,
\begin{gather*}
\frac{\partial f}{\partial z_{\mathbf{a} \odot \bm{\alpha}}} \frac{\partial f}{\partial z_{\mathbf{a} \odot \bm{\alpha} + (\mathbf{a} - 2 \mathbf{a} \odot \bm{\alpha}) \odot (\mathbf{1}_j + \mathbf{1}_k)}} - \alpha_{jk} \frac{\partial f}{\partial z_{\mathbf{a} \odot \bm{\alpha} + (\mathbf{a} - 2 \mathbf{a} \odot \bm{\alpha}) \odot \mathbf{1}_j}} \frac{\partial f}{\partial z_{\mathbf{a} \odot \bm{\alpha} + (\mathbf{a} - 2 \mathbf{a} \odot \bm{\alpha}) \odot \mathbf{1}_k}}
\end{gather*}
is a multiple of $f$.
\end{Lemma}

\begin{proof}Because $f$ satisfies conditions~(1)--(2) from Proposition~\ref{A2B2propgen},
\begin{gather*}
\left( \frac{\partial f}{\partial z_{\mathbf{a}}} \right)^2 \left( \frac{\partial f}{\partial z_{\mathbf{a} \odot (\mathbf{1} - \mathbf{1}_j - \mathbf{1}_k)}} \right)^2 - \left( \frac{\partial f}{\partial z_{\mathbf{a} \odot (\mathbf{1} - \mathbf{1}_j)}} \right)^2 \left( \frac{\partial f}{\partial z_{\mathbf{a} \odot (\mathbf{1} - \mathbf{1}_k)}} \right)^2\\
\qquad{}\equiv (f_1 \cdots f_d) (\pi_{\mathbf{a}, j} \pi_{\mathbf{a}, k} (f_1 \cdots f_d)) - (\pi_{\mathbf{a}, j} (f_1 \cdots f_d)) (\pi_{\mathbf{a}, k} (f_1 \cdots f_d))\\
\qquad{} = (1 - 1) f_j f_k (\pi_{\mathbf{a}, j}(f_j)) (\pi_{\mathbf{a}, k}(f_k)) \prod_{i \in \{1, \dots, d\} \setminus \{j,k\}} f_j^2 = 0
\end{gather*}
mod $f$. Hence, by the irreducibility of $f$, there exists $\alpha_{jk} \in \{-1,1\}$ such that
\begin{gather*}
\frac{\partial f}{\partial z_{\mathbf{a}}} \frac{\partial f}{\partial z_{\mathbf{a} \odot (\mathbf{1} - \mathbf{1}_j - \mathbf{1}_k)}} - \alpha_{jk} \frac{\partial f}{\partial z_{\mathbf{a} \odot (\mathbf{1} - \mathbf{1}_j)}} \frac{\partial f}{\partial z_{\mathbf{a} \odot (\mathbf{1} - \mathbf{1}_k)}}
\end{gather*}
is a multiple of $f$. The full lemma follows from condition~(1) of Proposition~\ref{A2B2propgen}.
\end{proof}

\begin{proof}[Proof of Lemma~\ref{runinarbdirection}]We proceed by induction on the number of $-1$s in $\bm{\alpha}$. When $\bm{\alpha} = \mathbf{1}$, then $\gamma_{\bm{\alpha}} = 1$ by condition~\eqref{newvertexgen}. If $\bm{\alpha}$ contains one $-1$, say $\bm{\alpha} = \mathbf{1} - \mathbf{1}_i$, then $\gamma_{\bm{\alpha}} = \beta_i$, where $\beta_i$ is the sign from Definition~\ref{def_of_adapted}.

Suppose $\ell \ge 2$, and $\bm{\alpha}$ has $\ell$ $-1$s, including $-1$s at positions $j$ and $k$. Then by Lemma~\ref{derivativeproductsign} and our inductive hypothesis,
\begin{gather*}
\frac{\partial f}{\partial z_{\mathbf{a}}}(\mathbf{x}_{v + [\bm{\alpha} \odot \mathbf{a}]}) = \frac{\alpha_{jk} \frac{\partial f}{\partial z_{\mathbf{a} \odot (\mathbf{1} - \mathbf{1}_j)}}(\mathbf{x}_{v + [\bm{\alpha} \odot \mathbf{a}]}) \frac{\partial f}{\partial z_{\mathbf{a} \odot (\mathbf{1} - \mathbf{1}_k)}}(\mathbf{x}_{v + [\bm{\alpha} \odot \mathbf{a}]})}{\frac{\partial f}{\partial z_{\mathbf{a} \odot (\mathbf{1} - \mathbf{1}_j - \mathbf{1}_k)}}(\mathbf{x}_{v + [\bm{\alpha} \odot \mathbf{a}]})}\\
\hphantom{\frac{\partial f}{\partial z_{\mathbf{a}}}(\mathbf{x}_{v + [\bm{\alpha} \odot \mathbf{a}]})}{}
= \alpha_{jk} \gamma_{\bm{\alpha} + 2(\mathbf{1}_j + \mathbf{1}_k)} \gamma_{\bm{\alpha} + 2 \mathbf{1}_j} \gamma_{\bm{\alpha} + 2 \mathbf{1}_k} \prod_{i = 1}^d x_{v + [\bm{\alpha} \odot (\mathbf{a} - \mathbf{1}_i)]},
\end{gather*}
so setting $\gamma_{\bm{\alpha}} = \alpha_{jk} \gamma_{\bm{\alpha} + 2(\mathbf{1}_j + \mathbf{1}_k)} \gamma_{\bm{\alpha} + 2 \mathbf{1}_j} \gamma_{\bm{\alpha} + 2 \mathbf{1}_k} \in \{-1,1\}$, we obtain the desired result.
\end{proof}

\begin{Example} \label{kash_sign_ex}Let us continue with Examples~\ref{kash_sat_ex} and~\ref{kash_ada_ex}. Following the argument in the proof of Lemma~\ref{runinarbdirection}, it can be shown that $\gamma_{\bm{\alpha}} = 1$ if $\bm{\alpha} = \pm \mathbf{1}$, and $\gamma_{\bm{\alpha}} = -1$ otherwise. Note that this fact is equivalent to Lemma~\ref{cornersofcubelemma}.
\end{Example}

\begin{Example} \label{sholo_sign_ex}Let us continue with Examples~\ref{sholo_sat_ex} and~\ref{sholo_ada_ex}. Following the argument in the proof of Lemma~\ref{runinarbdirection}, it can be shown that $\gamma_{\bm{\alpha}} = 1$ if $\bm{\alpha} = \mathbf{1}$, and $\gamma_{\bm{\alpha}} = -1$ otherwise. In particular, $\gamma_{-\mathbf{1}} = 1$. Hence, if $\mathbf{\tilde{x}} = (x_s) \in (\C^*)^{\Z^2 \cup F_{\mathbf{a}}}$ satisfies conditions~\eqref{2dh1}--\eqref{2dhs2}, it follows that $(x_{-s})_{s \in \Z^2 \cup F_{\mathbf{a}}}$ cannot satisfy conditions~\eqref{2dh1}--\eqref{2dhs2}.
\end{Example}

\begin{Example} \label{1d_sign_ex}Let us continue with Examples~\ref{1d_sat_ex} and~\ref{1d_ada_ex}. It is straightforward to show that $\gamma_{(1)} = \gamma_{(-1)} = 1$.
\end{Example}

\begin{Example} \label{2d_sign_ex}Let us continue with Examples~\ref{2d_sat_ex} and~\ref{2d_ada_ex}. Following the argument in the proof of Lemma~\ref{runinarbdirection}, it can be shown that $\gamma_{\bm{\alpha}} = 1$ if $\bm{\alpha} = \pm \mathbf{1}$, and $\gamma_{\bm{\alpha}} = -1$ otherwise.
\end{Example}

We now state the main theorem of this section.

\begin{Theorem} \label{maintheoremgen}Let $\mathbf{a} = (a_1, \dots, a_d) \in \Z^d_{\ge 1}$. Let $f \in \C[\mathbf{z}_{[\mathbf{a}]}]$ be a polynomial that is irreducible over $\C$ and satisfies conditions $(1)$--$(2)$ from Proposition~{\rm \ref{A2B2propgen}}. Fix the polynomials $f_1, \dots, f_d$ from condition~$(2)$ of Proposition~{\rm \ref{A2B2propgen}}. Let $\gtf$ and $\hof$ be the coefficients of $z_{\mathbf{a}}^2$ and $z_{\mathbf{a}}$ in $f$, viewed as a~polynomial in $z_{\mathbf{a}}$. Let $(r_1, \dots, r_d)$ be a tuple of rational functions in the variables~$z_s$ for $s \in [\mathbf{a}^*]$ that is adapted to $(f; f_1,\dots,f_d)$. Let $(\gamma_{\bm{\alpha}})_{\bm{\alpha} \in \{-1,1\}^d}$ be the propagation signs corresponding to $(f; f_1,\dots,f_d; r_1,\dots,r_d)$.
\begin{enumerate}\itemsep=0pt
\item[$(a)$] Let $\mathbf{x} = (x_s)_{s \in \Z^d}$ be an array such that
\begin{gather}
\mathbf{x} \text{ satisfies } f,\\
\label{nonzeropartial} \frac{\partial f}{\partial z_{\mathbf{a}}}(\mathbf{x}_{v + [\mathbf{a}]}) \not= 0 \qquad \text{for all $v \in \Z^d$ if $d > 1$},\\
\label{nonzerodenominator} \gtf(\mathbf{x}_{v + [\mathbf{a}]}) \not= 0 \qquad \text{for all $v \in \Z^d$},\\
\prod_{\bm{\alpha} \in \{-1,0\}^d} \frac{\partial f}{\partial z_{\mathbf{a}}} (\mathbf{x}_{v - (\mathbf{a} - \mathbf{1}) \odot \bm{\alpha} + [(\mathbf{1} + 2 \bm{\alpha}) \odot \mathbf{a}]})\label{coherencegen}\\
\qquad{} = \left( \prod_{\bm{\alpha} \in \{-1,1\}^d} \gamma_{\bm{\alpha}} \right) \prod_{i = 1}^d \prod_{\bm{\beta} = (\beta_1, \dots, \beta_d) \in \{-1,0\}^d \colon \beta_i = 0} f_i(\mathbf{x}_{v + \bm{\beta} + [\mathbf{a}]}) \qquad \text{for all $v \in \Z^d$}.\nonumber
\end{gather}
Then $\mathbf{x}$ can be extended to an array $\mathbf{\tilde{x}} = (x_s) \in \C^{\Z^d \cup \fab}$ satisfying~\eqref{newvertexgen}--\eqref{squarecond}.
\item[$(b)$] Conversely, if $\mathbf{\tilde{x}} = (x_s) \in \C^{\Z^d \cup \fab}$ satisfies conditions~\eqref{newvertexgen}--\eqref{squarecond} and~\eqref{nonzeroqi}--\eqref{nonzeromaind}, then the restriction of $\mathbf{\tilde{x}}$ to $\Z^d$ satisfies $f$ and the condition~\eqref{coherencegen}.
\end{enumerate}
\end{Theorem}

\begin{Remark}The equation~\eqref{coherencegen} is a generalization of the coherence condition (equation~\eqref{coherencecondition}) for the Kashaev equation.
\end{Remark}

The following proposition states that the sign $\prod\limits_{\bm{\alpha} \in \{-1,1\}^d} \gamma_{\bm{\alpha}}$ in~\eqref{coherencegen} is independent of the choice of $(r_1,\dots,r_d)$ if $d \ge 2$.

\begin{Proposition} \label{gamma_quant_ind}
Let $\mathbf{a} = (a_1, \dots, a_d) \in \Z^d_{\ge 1}$ with $d \ge 2$. Let $f \in \C[\mathbf{z}_{[\mathbf{a}]}]$ be a polynomial that is irreducible over $\C$ and satisfies conditions $(1)$--$(2)$ from Proposition~{\rm \ref{A2B2propgen}}. Then there exists a~sign $\gamma \in \{-1, 1\}$ such that for any tuple of rational functions $(r_1, \dots, r_d)$ in the variables $z_s$ for $s \in [\mathbf{a}^*]$ that is adapted to $(f; f_1,\dots,f_d)$, we have
\begin{gather*}
\gamma = \prod_{\bm{\alpha} \in \{-1,1\}^d} \gamma_{\bm{\alpha}},
\end{gather*}
where $(\gamma_{\bm{\alpha}})_{\bm{\alpha} \in \{-1,1\}^d}$ are the propagation signs corresponding to~$(f; f_1, \dots, f_d; r_1, \dots, r_d)$.
\end{Proposition}

\begin{Remark}
By Lemma~\ref{equivalent_coherence_statement_lemma}, given $f \in \C[\mathbf{z}_{[\mathbf{a}]}]$ satisfying conditions~(1)--(2) from Proposition~\ref{A2B2propgen} and an array $\mathbf{x} = (x_s)_{s \in \Z^d}$ satisfying $f$, the following are equivalent:
\begin{itemize}\itemsep=0pt
\item $\mathbf{x}$ satisfies condition~\eqref{coherencegen},
\item $\mathbf{x}$ satisfies
\begin{gather}
\prod_{\substack{\bm{\alpha} = (\alpha_1, \dots, \alpha_d) \\ \alpha_1, \dots, \alpha_d \in \{-1, 0\} \\ \alpha_1 + \cdots + \alpha_d \text{ even}}} \frac{\partial f}{\partial z_{\mathbf{a}}} (\mathbf{x}_{v - (\mathbf{a} - \mathbf{1}) \odot \bm{\alpha} + [(\mathbf{1} + 2 \bm{\alpha}) \odot \mathbf{a}]})\nonumber\\
\qquad{} = \gamma \prod_{\substack{\bm{\alpha} = (\alpha_1, \dots, \alpha_d) \\ \alpha_1, \dots, \alpha_d \in \{-1, 0\} \\ \alpha_1 + \cdots + \alpha_d \text{ odd}}} \frac{\partial f}{\partial z_{\mathbf{a}}} (\mathbf{x}_{v - (\mathbf{a} - \mathbf{1}) \odot \bm{\alpha} + [(\mathbf{1} + 2 \bm{\alpha}) \odot \mathbf{a}]})\qquad \text{for all $v \in \Z^d$}. \label{alt_coherencegen}
\end{gather}
\end{itemize}
Hence, one can replace condition~\eqref{coherencegen} in Theorem~\ref{maintheoremgen} by condition~\eqref{alt_coherencegen}.
\end{Remark}

\begin{Example}Continuing with Examples~\ref{kash_sat_ex}, \ref{kash_ada_ex}, and~\ref{kash_sign_ex}, Theorem~\ref{galoisfromintro} is a special case of Theorem~\ref{maintheoremgen}. Theorem~\ref{ABtheorem} is a special case of Theorem~\ref{maintheoremgen}(b), where we require all values of $\mathbf{\tilde{x}}$, including the values indexed by $\fab$, to be positive.
\end{Example}

\begin{Example}Continuing with Examples~\ref{sholo_sat_ex}, \ref{sholo_ada_ex}, and~\ref{sholo_sign_ex}, Theorem~\ref{coh_for_QC_theorem} is a special case of Theorem~\ref{maintheoremgen}. Theorem~\ref{s_holo_pos_thm} is a special case of Theorem~\ref{maintheoremgen}(b), where we require $x_s > 0$ for $s \in \Z^2_{\{0,1,2,\dots\}}$, $s \in \Z^2_{\{0,1,2,\dots\}} + [\mathbf{a} - \mathbf{1}_1]$, and $s \in \Z^2_{\{0,1,2,\dots\}} + [\mathbf{a} - \mathbf{1}_2]$.
\end{Example}

\begin{Example}Continuing with Examples~\ref{1d_sat_ex}, \ref{1d_ada_ex}, and~\ref{1d_sign_ex}, Theorem~\ref{main_thm_1d_example} is a special case of Theorem~\ref{maintheoremgen}. Theorem~\ref{pos_thm_1d_example} is a special case of Theorem~\ref{maintheoremgen}(b), where we require all values of $\mathbf{\tilde{x}}$, including the values indexed by $\fab$, to be positive.%
\end{Example}

\begin{Example}Continuing with Examples~\ref{2d_sat_ex}, \ref{2d_ada_ex}, and~\ref{2d_sign_ex}, Theorem~\ref{main_thm_2d_example} is a special case of Theorem~\ref{maintheoremgen}. Theorem~\ref{pos_thm_2d_example} is a special case of Theorem~\ref{maintheoremgen}(b), where we require all values of $\mathbf{\tilde{x}}$, including the values indexed by $\fab$, to be positive.%
\end{Example}

Before we prove Theorem~\ref{maintheoremgen}, we first prove Proposition~\ref{gamma_quant_ind}. Proposition~\ref{gamma_quant_ind} follows from the lemma below.

\begin{Lemma} \label{flip_sign_gamma_alpha_lemma}Let $\mathbf{a} = (a_1, \dots, a_d) \in \Z^d_{\ge 1}$. Let $f \in \C[\mathbf{z}_{[\mathbf{a}]}]$ be a polynomial that is irreducible over $\C$ and satisfies conditions $(1)$--$(2)$ from Proposition~$\ref{A2B2propgen}$. Fix $j \in \{1, \dots, d\}$. Let $(r_1, \dots, r_d)$ and $(\tilde{r}_1, \dots, \tilde{r}_d)$ be tuples of rational functions in the variables $z_s$ for $s \in [\mathbf{a}^*]$ that are adapted to $(f; f_1,\dots,f_d)$, such that $\tilde{r}_j = -r_j$ and $\tilde{r}_i = r_i$ for $i \not= j$. Let $(\gamma_{\bm{\alpha}})_{\bm{\alpha} \in \{-1,1\}^d}$ and $(\tilde{\gamma}_{\bm{\alpha}} \in \{-1,1\})_{\bm{\alpha} \in \{-1,1\}^d}$ be the propagation signs corresponding to $(f; f_1,\dots,f_d;r_1,\dots,r_d)$ and $(f;f_1,\dots,f_d; \tilde{r}_1,\dots,\tilde{r}_d)$, respectively. Given $\bm{\alpha} = (\alpha_1, \dots, \alpha_d) \in \{-1,1\}^d$, the following are equivalent:
\begin{itemize}\itemsep=0pt
\item $\gamma_{\bm{\alpha}} = \tilde{\gamma}_{\bm{\alpha}}$,
\item $\alpha_j = 1$.
\end{itemize}
\end{Lemma}

\begin{proof}
We proceed by induction on the number of $-1$s in $\bm{\alpha}$. Note that $\tilde{\gamma}_{\mathbf{1}} = \gamma_{\mathbf{1}} = 1$, $\tilde{\gamma}_{\mathbf{1} - 2 \mathbf{1}_j} = - \gamma_{\mathbf{1} - 2 \mathbf{1}_j}$, and $\tilde{\gamma}_{\mathbf{1} - 2 \mathbf{1}_i} = \gamma_{\mathbf{1} - 2 \mathbf{1}_i}$ for $i \not= j$. Suppose $\bm{\alpha}$ has at least two~$-1$s. As we showed in the proof to Lemma~\ref{runinarbdirection}, if $k_1$ and $k_2$ are distinct values such that $\alpha_{k_1} = \alpha_{k_2} = 1$, then $\gamma_{\bm{\alpha}} = \alpha_{k_1 k_2} \gamma_{\bm{\alpha} + 2 (\mathbf{1}_{k_1} + \mathbf{1}_{k_2})} \gamma_{\bm{\alpha} + 2 \mathbf{1}_{k_1}} \gamma_{\bm{\alpha} + 2 \mathbf{1}_{k_2}}$ and $\tilde{\gamma}_{\bm{\alpha}} = \alpha_{k_1 k_2} \tilde{\gamma}_{\bm{\alpha} + 2 (\mathbf{1}_{k_1} + \mathbf{1}_{k_2})} \tilde{\gamma}_{\bm{\alpha} + 2 \mathbf{1}_{k_1}} \tilde{\gamma}_{\bm{\alpha} + 2 \mathbf{1}_{k_2}}$. If $\alpha_j = 1$, let $k_1$, $k_2$ be distinct values such that $\alpha_{k_1} = \alpha_{k_2} = -1$, so
\begin{gather*}
\tilde{\gamma}_{\bm{\alpha}} = \alpha_{k_1 k_2} \tilde{\gamma}_{\bm{\alpha} + 2 (\mathbf{1}_{k_1} + \mathbf{1}_{k_2})} \tilde{\gamma}_{\bm{\alpha} + 2 \mathbf{1}_{k_1}} \tilde{\gamma}_{\bm{\alpha} + 2 \mathbf{1}_{k_2}} = \alpha_{k_1 k_2} \gamma_{\bm{\alpha} + 2 (\mathbf{1}_{k_1} + \mathbf{1}_{k_2})} \gamma_{\bm{\alpha} + 2 \mathbf{1}_{k_1}} \gamma_{\bm{\alpha} + 2 \mathbf{1}_{k_2}} = \gamma_{\bm{\alpha}}.
\end{gather*}
If $\alpha_j = -1$, let $k \not= j$ be a value such that $\alpha_k = -1$, so
\begin{gather*}
\tilde{\gamma}_{\bm{\alpha}} = \alpha_{jk} \tilde{\gamma}_{\bm{\alpha} + 2 (\mathbf{1}_{j} + \mathbf{1}_{k})} \tilde{\gamma}_{\bm{\alpha} + 2 \mathbf{1}_{j}} \tilde{\gamma}_{\bm{\alpha} + 2 \mathbf{1}_{k}} = -\alpha_{jk} \gamma_{\bm{\alpha} + 2 (\mathbf{1}_{j} + \mathbf{1}_{k})} \gamma_{\bm{\alpha} + 2 \mathbf{1}_{j}} \gamma_{\bm{\alpha} + 2 \mathbf{1}_{k}} = -\gamma_{\bm{\alpha}}.\tag*{\qed}
\end{gather*}
\renewcommand{\qed}{}
\end{proof}

\begin{proof}[Proof of Proposition~\ref{gamma_quant_ind}]It suffices to prove the proposition for tuples of rational functions $(r_1, \dots, r_d)$ and $(\tilde{r}_1, \dots, \tilde{r}_d)$ satisfying the conditions of Lemma~\ref{flip_sign_gamma_alpha_lemma}. Let $(\gamma_{\bm{\alpha}})_{\bm{\alpha} \in \{-1,1\}^d}$ and $(\tilde{\gamma}_{\bm{\alpha}} \in \{-1,1\})_{\bm{\alpha} \in \{-1,1\}^d}$ be the propagation signs corresponding to $(f; f_1,\dots,f_d; \allowbreak r_1,\dots,r_d)$ and $(f; f_1,\dots,f_d; \tilde{r}_1,\dots,\tilde{r}_d)$, respectively. Then by Lemma~\ref{flip_sign_gamma_alpha_lemma},
\begin{gather*}
\prod_{\bm{\alpha} \in \{-1,1\}^d} \tilde{\gamma}_{\bm{\alpha}} = (-1)^{2^{d - 1}} \prod_{\bm{\alpha} \in \{-1,1\}^d} \gamma_{\bm{\alpha}} = \prod_{\bm{\alpha} \in \{-1,1\}^d} \gamma_{\bm{\alpha}},
\end{gather*}
because $d \ge 2$.
\end{proof}

The remainder of this section is dedicated to the proof of Theorem~\ref{maintheoremgen}. For the rest of this section, we fix all quantities given in Theorem~\ref{maintheoremgen}, and set
\begin{gather*}
a = a_1 + \cdots + a_d.
\end{gather*}

\begin{proof}[Proof of Theorem~\ref{maintheoremgen}(b).]Let $\mathbf{x}$ be the restriction of $\mathbf{\tilde{x}}$ to $\Z^d$. It is clear that $\mathbf{x}$ must satisfy $f$. By Lemma~\ref{runinarbdirection},
\begin{gather*}
\frac{\partial f}{\partial z_{\mathbf{a}}} (\mathbf{x}_{v - (\mathbf{a} - \mathbf{1}) \odot \bm{\alpha} + [(\mathbf{1} + 2 \bm{\alpha}) \odot \mathbf{a}]}) = \gamma_{\mathbf{1} + 2 \bm{\alpha}} \prod_{i = 1}^d x_{v - (\mathbf{a} - \mathbf{1}) \odot \bm{\alpha} + [(1 + 2 \bm{\alpha}) \odot(\mathbf{a} - \mathbf{1}_i)]}\\
\hphantom{\frac{\partial f}{\partial z_{\mathbf{a}}} (\mathbf{x}_{v - (\mathbf{a} - \mathbf{1}) \odot \bm{\alpha} + [(\mathbf{1} + 2 \bm{\alpha}) \odot \mathbf{a}]})}{} = \gamma_{\mathbf{1} + 2 \bm{\alpha}} \prod_{i = 1}^d x_{v + \bm{\alpha} \odot (\mathbf{1} - \mathbf{1}_i) + [\mathbf{a} - \mathbf{1}_i]}.
\end{gather*}
Taking the product over $\bm{\alpha} \in \{-1,0\}^d$, we get
\begin{gather*}
 \prod_{\bm{\alpha} \in \{-1,0\}^d} \frac{\partial f}{\partial z_{\mathbf{a}}} (\mathbf{x}_{v - (\mathbf{a} - \mathbf{1}) \odot \bm{\alpha} + [(\mathbf{1} + 2 \bm{\alpha}) \odot \mathbf{a}]})\\
 \qquad {} = \left( \prod_{\bm{\alpha} \in \{-1, 1\}^d} \gamma_{\bm{\alpha}} \right) \prod_{i = 1}^d \prod_{\bm{\beta} = (\beta_1, \dots, \beta_d) \in \{-1,0\}^d \colon \beta_i = 0} x_{v + \bm{\beta} + [\mathbf{a} - \mathbf{1}_i]}^2\\
\qquad {}= \left( \prod_{\bm{\alpha} \in \{-1, 1\}^d} \gamma_{\bm{\alpha}} \right) \prod_{i = 1}^d \prod_{\bm{\beta} = (\beta_1, \dots, \beta_d) \in \{-1,0\}^d \colon \beta_i = 0} f_i(\mathbf{x}_{v + \bm{\beta} + [\mathbf{a}]}).\tag*{\qed}
\end{gather*}\renewcommand{\qed}{}
\end{proof}

The proof of Theorem~\ref{maintheoremgen}(a) below is nearly identical to the proof of Theorem~\ref{galoisfromintro}(a).

\begin{Definition}For $U \subseteq \Z$, let $\Z^d_U$ denote the set
\begin{gather*}
\Z^d_U = \big\{(i_1, \dots, i_d) \in \Z^d \colon i_1 + \cdots i_d \in U\big\}.
\end{gather*}
For $U \subseteq \Z$, we will also use the notation
\begin{gather*}
F_{\mathbf{a}, U} = \{v + [\mathbf{a} - \mathbf{1}_i] \colon v \in \Z_U, i \in \{1, \dots, d\}\}.
\end{gather*}
In particular, we will be interested in $\zadi = \Z^d_{\{0, \dots, a- 1\}}$ and $\fati = F_{\mathbf{a}, \{0\}}$.
\end{Definition}

\begin{Definition}We say that an array $\xtin$ indexed by $\zadi \cup \fati$ satisfying condition~\eqref{squarecond} is \emph{generic} if there exists an extension of $\xtin$ to an array $\mathbf{\tilde{x}}$ indexed by $\Z^d \cup \fab$ satisfying equations~\eqref{newvertexgen}--\eqref{squarecond} where the restriction of $\mathbf{\tilde{x}}$ to $\Z^d$ satisfies conditions~\eqref{nonzeropartial}--\eqref{nonzerodenominator}. Similarly, we say that an array $\xin$ indexed by $\zadi$ is generic if every extension of $\xin$ to an array $\xtin$ indexed by $\zadi \cup \fati$ satisfying condition~\eqref{squarecond} is generic.
\end{Definition}

\begin{Definition}Let $\xtin$ be a generic array indexed by $\zadi \cup \fati$ satisfying condition~\eqref{squarecond}. We denote by $(\xtin)^{\uparrow \Z^d\! {\cup} \fab}$ the unique extension of $\xtin$ to $\Z^d\! {\cup} \fab$ where $(\xtin)^{\uparrow \Z^d\! {\cup} \fab}$ satisfies equations~\eqref{newvertexgen}--\eqref{squarecond}.
\end{Definition}

The next lemma generalizes Lemma~\ref{sign_prop_khex_cube}.

\begin{Lemma} \label{sign_prop_gen}Let $S = [\mathbf{a}] \cup \{ b \mathbf{1}_i + [\mathbf{a} - \mathbf{1}_i]\colon b \in \{0,1\}, i \in \{1,\dots, d\}\}$, i.e., $S$ is the set of vertices of $[\mathbf{a}] \subset \Z^d$ and boxes of $\fab$ completely contained in~$[\mathbf{a}]$. Fix values $t_i \in \{-1,1\}$ for $i = 1, \dots, d$. Suppose $\mathbf{\tilde{x}} = (x_s)_{s \in S}$ and $\mathbf{\tilde{y}} = (y_s)_{s \in S}$ are arrays of complex numbers such that
\begin{itemize}\itemsep=0pt
\item $\mathbf{\tilde{x}}$ and $\mathbf{\tilde{y}}$ both satisfy equations~\eqref{newvertexgen}--\eqref{squarecond}, with the denominators in equations~\eqref{newvertexgen}--\eqref{newfacegen} non-vanishing,
\item $y_{s} = x_{s}$ for $s \in [\mathbf{a}] - \{\mathbf{a}\}$,
\item $y_{[\mathbf{a} - \mathbf{1}_i]} = t_i x_{[\mathbf{a} - \mathbf{1}_i]}$ for $i = 1, \dots, d$, and
\item $\prod\limits_{i = 1}^d t_i = 1$.
\end{itemize}
Then the following equations hold:
\begin{gather*}
y_{\mathbf{1}_i + [\mathbf{a} - \mathbf{1}_i]} = t_i x_{\mathbf{1}_i + [\mathbf{a} - \mathbf{1}_i]} \qquad \text{for $i = 1, \dots, d$},\\ y_{\mathbf{a}} = x_{\mathbf{a}}.
\end{gather*}
\end{Lemma}

\begin{proof}Note that
\begin{gather*}
y_{\mathbf{a}} - x_{\mathbf{a}} = \left( \prod_{i = 1}^d t_i - 1 \right) \frac{\prod\limits_{i = 1}^d x_{[\mathbf{a} - \mathbf{1}_i]}}{2 \gtf(\mathbf{x}_{[\mathbf{a} - \mathbf{1}_i]})} = 0,
\end{gather*}
so $y_{\mathbf{a}} = x_{\mathbf{a}}$. By~\eqref{arbdirectionderivative},
\begin{gather*}\begin{split}&
\frac{\partial f}{\partial z_{\mathbf{a}}}(\mathbf{x}_{a_i \mathbf{1}_i + [(\mathbf{1} - 2 \mathbf{1}_i) \odot \mathbf{a}]}) = \gamma_{(\mathbf{1} - 2 \mathbf{1}_i)} \prod_{j = 1}^d x_{a_i \mathbf{1}_i + [((\mathbf{1} - 2 \mathbf{1}_i)) \odot (\mathbf{a} - \mathbf{1}_j)]}\\
& \hphantom{\frac{\partial f}{\partial z_{\mathbf{a}}}(\mathbf{x}_{a_i \mathbf{1}_i + [(\mathbf{1} - 2 \mathbf{1}_i) \odot \mathbf{a}]})}{}
= \gamma_{(\mathbf{1} - 2 \mathbf{1}_i)} x_{\mathbf{1}_i + [\mathbf{a} - \mathbf{1}_i]} \prod_{j \in \{1, \dots, d\} \setminus \{i\}} x_{[\mathbf{a} - \mathbf{1}_j]}
\end{split}
\end{gather*}
and
\begin{gather*}
\frac{\partial f}{\partial z_{\mathbf{a}}}(\mathbf{y}_{a_i \mathbf{1}_i + [(\mathbf{1} - 2 \mathbf{1}_i) \odot \mathbf{a}]}) = \gamma_{(\mathbf{1} - 2 \mathbf{1}_i)} \prod_{j = 1}^d y_{a_i \mathbf{1}_i + [((\mathbf{1} - 2 \mathbf{1}_i)) \odot (\mathbf{a} - \mathbf{1}_j)]}\\
\hphantom{\frac{\partial f}{\partial z_{\mathbf{a}}}(\mathbf{y}_{a_i \mathbf{1}_i + [(\mathbf{1} - 2 \mathbf{1}_i) \odot \mathbf{a}]})}{}
= \gamma_{(\mathbf{1} - 2 \mathbf{1}_i)} y_{\mathbf{1}_i + [\mathbf{a} - \mathbf{1}_i]} \prod_{j \in \{1, \dots, d\} \setminus \{i\}} y_{[\mathbf{a} - \mathbf{1}_j]}\\
\hphantom{\frac{\partial f}{\partial z_{\mathbf{a}}}(\mathbf{y}_{a_i \mathbf{1}_i + [(\mathbf{1} - 2 \mathbf{1}_i) \odot \mathbf{a}]})}{}
= \gamma_{(\mathbf{1} - 2 \mathbf{1}_i)} y_{\mathbf{1}_i + [\mathbf{a} - \mathbf{1}_i]} \prod_{j \in \{1, \dots, d\} \setminus \{i\}} t_j x_{[\mathbf{a} - \mathbf{1}_j]}\\
\hphantom{\frac{\partial f}{\partial z_{\mathbf{a}}}(\mathbf{y}_{a_i \mathbf{1}_i + [(\mathbf{1} - 2 \mathbf{1}_i) \odot \mathbf{a}]})}{}
= \gamma_{(\mathbf{1} - 2 \mathbf{1}_i)} t_i y_{\mathbf{1}_i + [\mathbf{a} - \mathbf{1}_i]} \prod_{j \in \{1, \dots, d\} \setminus \{i\}}x_{[\mathbf{a} - \mathbf{1}_j]}.
\end{gather*}
Because
\begin{gather*}
\frac{\partial f}{\partial z_{\mathbf{a}}}(\mathbf{x}_{a_i \mathbf{1}_i + [(\mathbf{1} - 2 \mathbf{1}_i) \odot \mathbf{a}]}) = \frac{\partial f}{\partial z_{\mathbf{a}}}(\mathbf{y}_{a_i \mathbf{1}_i + [(\mathbf{1} - 2 \mathbf{1}_i) \odot \mathbf{a}]}),
\end{gather*}
it follows that $x_{\mathbf{1}_i + [\mathbf{a} - \mathbf{1}_i]} = t_i y_{\mathbf{1}_i + [\mathbf{a} - \mathbf{1}_i]}$.
\end{proof}

\begin{Definition}
Define an equivalence relation on $\fab$ by setting $s_1 \sim s_2$ if and only if $s_1 = v + [\mathbf{a} - \mathbf{1}_i]$ and $s_2 = v + \beta \mathbf{1}_i + [\mathbf{a} - \mathbf{1}_i]$ for some $1 \le i \le d$, $v \in \Z^d$, and~$\beta \in \Z$. Let $\faeq$ denote the set of equivalence classes under this equivalence relation. Denote by $[s] \in \faeq$ the equivalence class of $s \in \fab$.
\end{Definition}

\begin{Definition}Define an action of $\{-1,1\}^{\faeq}$ on arrays indexed by $\zadi \cup \fati$ as follows: given $\mathbf{t} = (t_s)_{s \in \faeq} \in \{-1, 1\}^{\faeq}$ and $\xtin = (x_s)_{s \in \zadi \cup \fati}$, define $\mathbf{t} \cdot \xtin = (\tilde{x}_s)_{\zadi \cup \fati}$, where
\begin{gather*}
\tilde{x}_s = \begin{cases} x_s & \text{if $s \in \zadi$},\\ t_{[s]} x_s & \text{if $s \in \fati$}. \end{cases}
\end{gather*}
\end{Definition}

\begin{Definition} \label{gen_def_of_psi}For $\mathbf{t} = (t_s) \in \{-1, 1\}^{\faeq}$, define $\psi(\mathbf{t}) = (u_s) \in \{-1, 1\}^{\Z^d + \mathbf{a}/2}$ by
\begin{gather*}
u_{s + \mathbf{a}/2} = \prod_{i = 1}^d t_{[s + [\mathbf{a} - \mathbf{1}_i]]}
\end{gather*}
for $s \in \Z^d$.
\end{Definition}

The next lemma generalizes Lemma~\ref{change_one_value_lemma}.

\begin{Lemma} \label{use_of_psi_gen_lemma}Let $\xtin$ be a generic array indexed by $\zadi \cup \fati$ satisfying condition~\eqref{squarecond}. Let $\mathbf{t} \in \{-1,1\}^{\faeq}$, and $\mathbf{u} = (u_s)_{s \in \Z^d + \mathbf{a}/2} = \psi(\mathbf{t})$. Let $(\xtin)^{\uparrow \Z^d \cup \fab} = (x_s)_{s \in \Z^d \cup \fab}$, and \mbox{$(\mathbf{t} \cdot \xtin)^{\uparrow \Z^d \cup \fab} = (y_s)_{s \in \Z^d \cup \fab}$}. Suppose $v \in \Z^d_{\{a, a+1, \dots\}}$ satisfies the condition that $u_{w - \mathbf{a}/2} = 1$ for all $w \in \Z^d_{\{a, a+1, \dots \}}$ with $w \le v$. Then:
\begin{enumerate}\itemsep=0pt
\item[$(a)$] $y_v = x_v$,
\item[$(b)$] $y_{v - \mathbf{a} + \mathbf{1}_i + [\mathbf{a} - \mathbf{1}_i]} = t_{[v - \mathbf{a} + \mathbf{1}_i + [\mathbf{a} - \mathbf{1}_i]]} x_{v - \mathbf{a} + \mathbf{1}_i + [\mathbf{a} - \mathbf{1}_i]}$ for $i = 1, \dots, d$.
\end{enumerate}
\end{Lemma}

\begin{proof}We prove parts~(a) and~(b) together by induction. Assume that we have proved parts~(a) and~(b) for all $w \in \Z^d_{\{a, a+1, \dots\}}$ with $w < v$. By construction, $x_w = y_w$ for all $w \in \zadi$ and statement~(b) holds for all $w \in \Z^d_{\{a-1\}}$. Hence, $y_{v-s'} = x_{v-s'}$ for $s \in [\mathbf{a}] - \{\mathbf{0}\}$, and $y_{v - \mathbf{a} + [\mathbf{a} - \mathbf{1}_i]} = t_{[v - \mathbf{a} + [\mathbf{a} - \mathbf{1}_i]]} x_{v - \mathbf{a} + [\mathbf{a} - \mathbf{1}_i]}$ for $i = 1, \dots, d$. Because~$u_{v - \mathbf{a}/2} = 1$, statements~(a) and~(b) follow from Lemma~\ref{sign_prop_gen}.
\end{proof}

The next lemma generalizes Lemma~\ref{prod_over_cubes_lemma}.

\begin{Lemma} \label{gen_image_of_psi_lemma}An array $\mathbf{u} = (u_s) \in \{-1,1\}^{\Z^d + \mathbf{a}/2}$ is in the image of $\psi$ $($see Definition~{\rm \ref{gen_def_of_psi})} if and only if for every $v \in \Z^3$,
\begin{gather} \label{prod_over_hcube}
\prod_{\bm{\alpha} \in \{0,1\}^d} u_{v + \mathbf{a}/2 + \bm{\alpha}} = 1.
\end{gather}
\end{Lemma}

\begin{proof}First, suppose $\mathbf{u} = \psi(\mathbf{t})$, where $\mathbf{t} = (t_s) \in \{-1, 1\}^{\faeq}$. Then for any $v \in \Z^d$,
\begin{gather*}
\prod_{\bm{\alpha} \in \{0,1\}^d} u_{v + \mathbf{a}/2 + \bm{\alpha}} = \prod_{\bm{\alpha} \in \{0,1\}^d} \prod_{i = 1}^d t_{[v + [\mathbf{a} - \mathbf{1}_i]]} = \prod_{i = 1}^d \prod_{\substack{\bm{\alpha} = (\alpha_1, \dots, \alpha_d) \in \{0,1\}^d \\ \alpha_i = 0}} t_{[v + \bm{\alpha} + [\mathbf{a} - \mathbf{1}_i]]}^2 = 1.
\end{gather*}

Next, suppose that condition~\eqref{prod_over_hcube} holds. It is clear that $\mathbf{u}$ is uniquely determined by its components at $S = \left\{(v_1, \dots, v_d) + \mathbf{a}/2\colon v_1 \cdots v_d = 0\right\}$ and condition~\eqref{prod_over_hcube}. For $v = (v_1, \dots, v_d) \in \Z^d_{\{0\}}$ and $i \in \{1, \dots, d\}$, set
\begin{gather*}
t_{[v + [\mathbf{a} - \mathbf{1}_i]]} = \begin{cases} \prod\limits_{j = 1}^i u_{(v_1, \dots, v_{j - 1}, 0, v_{j + 1}, \dots, v_d) + \mathbf{a}/2} & \text{if } v_1, \dots, v_{i - 1} \not= 0, \\ 1 & \text{otherwise}. \end{cases}
\end{gather*}
Set $\mathbf{t} = (t_s) \in \{-1,1\}^{\faeq}$. It is straightforward to check that $\psi(\mathbf{t})$ agrees with $\mathbf{u}$ at~$S$. Hence, because $\psi(\mathbf{t})$ and $\mathbf{u}$ both satisfy condition~\eqref{prod_over_hcube}, it follows that~$\mathbf{u}=\psi(\mathbf{t})$.
\end{proof}

The next lemma generalizes Lemma~\ref{levels_0_thru_5}.

\begin{Lemma} \label{extended_from_levels_a+d-1}Let $\mathbf{\hat{x}}$ be an array indexed by $\Z^d_{\{0, 1, \dots, a + d - 1\}}$. Assume that $\mathbf{\hat{x}}$ satisfying~$f$, and, moreover, its restriction to $\zadi$ is generic. Then there exists an array $\mathbf{\tilde{x}}$ indexed by $\Z^d \cup \fab$ satisfying equations~\eqref{newvertexgen}--\eqref{squarecond} and extending $\mathbf{\hat{x}}$.
\end{Lemma}

\begin{proof}For $i = a, \dots, a + d - 1$, we will show by induction on $i$ that there exists an array $\xtin$ indexed by $\zadi \cup \fati$ satisfying~\eqref{squarecond} such that $(\xtin)^{\uparrow \Z^d \cup \fab}$ agrees with $\mathbf{\hat{x}} = (x_s)_{s \in \Z^d_{\{0, \dots, a +d - 1\}}}$ on $\Z^d_{\{0, \dots, i\}}$. Let $\xtin'$ be an array indexed by $\zadi \cup \fati$ satisfying~\eqref{squarecond} such that $(\xtin')^{\uparrow \Z^d \cup \fab} = (y_s)_{s \in \Z^d \cup \fab}$ agrees with $\mathbf{\hat{x}}$ on $\Z^d_{\{0,\dots,i-1\}}$. (For $i = a$, we can obtain~$\xtin'$ be taking an arbitrary extension of $\xin$ to $\zadi \cup \fati$ satisfying condition~\eqref{squarecond}. For $i > a$, we have shown that $\xtin'$ exists by induction.) Choose $\mathbf{\tilde{u}} = (u_s) \in \{-1,1\}^{\Z^d_{\{a, \dots, a + d - 1\}} - \mathbf{a}/2}$ so that
\begin{itemize}\itemsep=0pt
\item $u_{s - \mathbf{a}/2} = 1$ if $s \in \Z^d_{\{i\}}$ and $x_s = y_s$,
\item $u_{s - \mathbf{a}/2} = -1$ if $s \in \Z^d_{\{i\}}$ and $x_s = y_s$,
\item $u_{s - \mathbf{a}/2} = 1$ if $s \in \Z^d_{\{j\}}$ for $0 \le j < i$.
\end{itemize}
Extend $\mathbf{\tilde{u}}$ to $\mathbf{u} = (u_s)_{s \in \Z^d + \mathbf{a}/2}$ by condition~\eqref{squarecond}. By Lemma~\ref{gen_image_of_psi_lemma}, there exists $\mathbf{t} \in \{-1, 1\}^{\faeq}$ such that $\mathbf{u} = \psi(\mathbf{t})$. Set $\xtin = \mathbf{t} \cdot \xtin'$. Then by Lemma~\ref{use_of_psi_gen_lemma}, $(\xtin)^{\uparrow \Z^d \cup \fab}$ agrees with $\mathbf{\hat{x}}$ on~$\Z^d_{\{0, \dots, i\}}$, as desired.
\end{proof}

We can now prove a weaker version of Theorem~\ref{maintheoremgen}(a), under the additional constraint of genericity.

\begin{Corollary} \label{main_theorem_gen_generic}Let $\mathbf{x} = (x_s)_{s \in \Z^d}$ be an array that satisfies $f$ and condition~\eqref{coherencegen}, and whose restriction to $\zadi$ is generic. Then $\mathbf{x}$ can be extended to an array $\mathbf{\tilde{x}}$ indexed by $\Z^d \cup \fab$ satisfying equations~\eqref{newvertexgen}--\eqref{squarecond}.
\end{Corollary}

\begin{proof} Let $\mathbf{x} = (x_s)_{s \in \Z^d}$ be an array that satisfies $f$ and condition~\eqref{coherencegen}, and whose restriction to $\zadi$ is generic. By Lemma~\ref{extended_from_levels_a+d-1}, there exists an array $\mathbf{\tilde{x}}$ indexed by $\Z^d \cup \fab$ satisfying equations~\eqref{newvertexgen}--\eqref{squarecond} that agrees with $\mathbf{x}$ on~$\Z^d_{\{0, \dots, a + d - 1\}}$. Let $\mathbf{x}'$ be the restriction of~$\mathbf{\tilde{x}}$ to~$\Z^d$. By Theorem~\ref{maintheoremgen}(b), $\mathbf{x}'$ satisfies $f$ and~\eqref{coherencegen}. There is a unique solution of $f$ satisfying condition~\eqref{coherencegen} agreeing with $\mathbf{x}$ at $\Z^d_{\{0,\dots, a+d-1\}}$, as condition~\eqref{coherencegen} gives the remaining values as rational expressions in the values at $\Z^d_{\{0,\dots,a+d-1\}}$, where the denominators do not vanish because conditions~\eqref{nonzeropartial}--\eqref{nonzerodenominator} hold for $\mathbf{x}$ (as $\mathbf{x}$ is generic). Hence, $\mathbf{x}' = \mathbf{x}$, as desired.
\end{proof}

\begin{proof}[Proof of Theorem~\ref{maintheoremgen}(a).] We need to loosen the genericity condition in Corollary~\ref{main_theorem_gen_generic} to the condition that $\mathbf{x}$ satisfies~\eqref{nonzeropartial}--\eqref{nonzerodenominator}.

Let $\mathbf{x}$ satisfy $f$ along with conditions~\eqref{nonzeropartial}--\eqref{nonzerodenominator} and condition~\eqref{coherencegen}. Let~$A_j = [-j, j]^d \cap \Z^d$, and let $B_j = \left\{ s \in \fab \colon s \subseteq [-j,j]^d\right\}$. We claim that if there exist $\mathbf{\tilde{x}}_j \in \C^{A_j \cup B_j}$ sa\-tisfying equations~\eqref{newvertexgen}--\eqref{squarecond} that agree with $\mathbf{x}$ on $A_j$ for all $j$, then there exists $\mathbf{\tilde{x}} \in \C^{\Z^d \cup \fab}$ satisfying equations~\eqref{newvertexgen}--\eqref{squarecond} that agrees with $\mathbf{x}$ on $\Z^d$. Construct an infinite tree $T$ as follows:
\begin{itemize}\itemsep=0pt
\item The vertices of $T$ are arrays indexed by $A_j \cup B_j$ satisfying equations~\eqref{newvertexgen}--\eqref{squarecond} that agree with $\mathbf{x}$ on $A_j$ (over $j \in \Z_{\ge 0}$).
\item Add an edge between $\mathbf{\tilde{x}}_j \in \C^{A_j \cup B_j}$ and $\mathbf{\tilde{x}}_{j + 1} \in \C^{A_{j + 1} \cup B_{j + 1}}$ if $\mathbf{\tilde{x}}_{j + 1}$ restricts to $\mathbf{\tilde{x}}_j$.
\end{itemize}
Thus, $T$ is an infinite tree in which every vertex has finite degree. By K\"onig's infinity lemma, there exists an infinite path $\mathbf{\tilde{x}}_0, \mathbf{\tilde{x}}_1, \dots$ in $T$ with $\mathbf{\tilde{x}}_j \in \C^{A_j \cup B_j}$. Thus, there exists $\mathbf{\tilde{x}} \in \C^{\Z^d \cup \fab}$ restricting to $\mathbf{\tilde{x}}_j$ for all $j \in \Z_{\ge 0}$, so $\mathbf{\tilde{x}}$ satisfies equations~\eqref{newvertexgen}--\eqref{squarecond} and agrees with $\mathbf{x}$ on $\Z^d$.

Given $j \in \Z_{\ge 0}$, we claim that there exists $\mathbf{\tilde{x}} \in \C^{A_j \cup B_j}$ satisfying equations~\eqref{newvertexgen}--\eqref{squarecond} that agrees with $\mathbf{x}$ on $A_j$. Because $\mathbf{x}$ satisfies conditions~\eqref{nonzeropartial}--\eqref{nonzerodenominator}, there exists a sequence $\mathbf{x}_1, \mathbf{x}_2, \dots$ of arrays satisfying $f$ along with conditions~\eqref{nonzeropartial}--\eqref{nonzerodenominator} and condition~\eqref{coherencegen}, whose restrictions to $\zadi$ are generic. By Corollary~\ref{main_theorem_gen_generic}, there exist $\mathbf{\tilde{x}}_1, \mathbf{\tilde{x}}_2, \ldots \in \C^{\Z^d \cup \fab}$ satisfying equations~\eqref{newvertexgen}--\eqref{squarecond} such that $\mathbf{\tilde{x}}_i$ restricts to $\mathbf{x}_i$. However, the sequence $\mathbf{\tilde{x}}_1, \mathbf{\tilde{x}}_2, \dots$ does not necessarily converge. Let $\mathbf{\tilde{x}}_1', \mathbf{\tilde{x}}_2', \ldots \in \C^{A_j \cup B_j}$ be the restrictions of $\mathbf{\tilde{x}}_1, \mathbf{\tilde{x}}_2, \dots$ to $A_j \cup B_j$. There exists a subsequence of $\mathbf{\tilde{x}}_1', \mathbf{\tilde{x}}_2', \dots$ that converges to some $\mathbf{\tilde{x}} \in \C^{A_j \cup B_j}$. (For each~$s \in B_j$, we can partition the sequence $\mathbf{\tilde{x}}_1', \mathbf{\tilde{x}}_2', \dots$ into two sequences, each of which converges at $s$. Because $B_j$ is finite, the claim follows.) The array $\mathbf{\tilde{x}}$ must satisfy equations~\eqref{newvertexgen}--\eqref{squarecond} and agree with $\mathbf{x}$ on $A_j$, so we are done.
\end{proof}

\subsection*{Acknowledgements}

I would like to thank my Ph.D.\ advisor, Sergey Fomin, for his invaluable mathematical insights and the countless hours he dedicated to our meetings while I was writing this paper. I~am also grateful to Dmitry Chelkak for pointing out the connection with s-holomorphicity, and to Thomas Lam and John Stembridge for helpful discussions and editorial suggestions. Finally, I~would like to thank the anonymous referees for their thorough reading of this manuscript, and for their suggestions that improved the quality of this paper.

\pdfbookmark[1]{References}{ref}
\LastPageEnding

\end{document}